\documentclass[12pt,a4paper]{amsart}
\usepackage{amssymb}
\usepackage[all,cmtip]{xy}
\usepackage{hyperref}

\calclayout

\newcommand{\+}{\protect\nobreakdash-}
\renewcommand{\:}{\colon}

\newcommand{\rarrow}{\longrightarrow}
\newcommand{\larrow}{\longleftarrow}
\newcommand{\ot}{\otimes}
\renewcommand{\d}{\partial}
\newcommand{\bec}{\natural}

\DeclareFontFamily{U}{mathb}{\hyphenchar\font45}
\DeclareFontShape{U}{mathb}{m}{n}{
      <5> <6> <7> <8> <9> <10> gen * mathb
      <10.95> mathb10 <12> <14.4> <17.28> <20.74> <24.88> mathb12
      }{}
\DeclareSymbolFont{mathb}{U}{mathb}{m}{n}
\DeclareFontSubstitution{U}{mathb}{m}{n}
\DeclareMathSymbol{\blackdiamond}{0}{mathb}{"0C}

\newcommand{\bu}{{\text{\smaller\smaller$\scriptstyle\bullet$}}}
\newcommand{\cu}{{\text{\smaller$\scriptstyle\blackdiamond$}}}
\newcommand{\subcu}{{\text{\smaller$\scriptscriptstyle\blackdiamond$}}}
\newcommand{\ci}{{\text{\smaller\smaller$\scriptstyle\circ$}}}

\newcommand{\lrarrow}{\mskip.5\thinmuskip\relbar\joinrel\relbar\joinrel
 \rightarrow\mskip.5\thinmuskip\relax}
\newcommand{\llarrow}{\mskip.5\thinmuskip\leftarrow\joinrel\relbar
 \joinrel\relbar\mskip.5\thinmuskip\relax}

\DeclareMathOperator{\Hom}{Hom}
\DeclareMathOperator{\Ext}{Ext}
\DeclareMathOperator{\Tot}{Tot}
\DeclareMathOperator{\Spec}{Spec}
\DeclareMathOperator{\cone}{cone}
\DeclareMathOperator{\pdi}{pd}
\DeclareMathOperator{\idi}{id}
\DeclareMathOperator{\coker}{coker}

\newcommand{\sop}{{\mathsf{op}}}
\newcommand{\rop}{{\mathrm{op}}}
\newcommand{\id}{\mathrm{id}}
\newcommand{\Id}{\mathrm{Id}}

\newcommand{\mc}{\boldsymbol(\mathbf{mc}\boldsymbol)}

\newcommand{\rmodl}{{\operatorname{\mathrm{--mod}}}}

\newcommand{\qcoh}{{\operatorname{\mathrm{--qcoh}}}}
\newcommand{\qqcoh}{{\operatorname{\mathrm{--qqcoh}}}}

\newcommand{\cmodl}{{\operatorname{\mathit{--mod}}}}
\newcommand{\cmodr}{{\operatorname{\mathit{mod--}}}}
\newcommand{\cqcoh}{{\operatorname{\mathit{--qcoh}}}}

\newcommand{\smodl}{{\operatorname{\mathsf{--mod}}}}
\newcommand{\smodr}{{\operatorname{\mathsf{mod--}}}}
\newcommand{\sqcoh}{{\operatorname{\mathsf{--qcoh}}}}

\newcommand{\bmodl}{{\operatorname{\mathbf{--mod}}}}
\newcommand{\bmodr}{{\operatorname{\mathbf{mod--}}}}
\newcommand{\bqcoh}{{\operatorname{\mathbf{--qcoh}}}}

\newcommand{\bb}{{\mathsf{b}}}
\newcommand{\abs}{{\mathsf{abs}}}
\newcommand{\co}{{\mathsf{co}}}
\newcommand{\ctr}{{\mathsf{ctr}}}
\newcommand{\proj}{{\mathsf{proj}}}
\newcommand{\inj}{{\mathsf{inj}}}
\newcommand{\bproj}{{\mathbf{proj}}}
\newcommand{\binj}{{\mathbf{inj}}}
\newcommand{\fg}{{\mathsf{fg}}}
\newcommand{\bfg}{{\mathbf{fg}}}
\newcommand{\fp}{{\mathsf{fp}}}
\newcommand{\bfp}{{\mathbf{fp}}}
\newcommand{\fpinj}{{\mathsf{fpinj}}}
\newcommand{\fproj}{{\mathsf{fproj}}}
\newcommand{\bfpinj}{{\mathbf{fpinj}}}

\newcommand{\Ac}{{\mathsf{Ac}}}

\newcommand{\sA}{\mathsf A}
\newcommand{\sB}{\mathsf B}
\newcommand{\sC}{\mathsf C}
\newcommand{\sD}{\mathsf D}
\newcommand{\sE}{\mathsf E}
\newcommand{\sF}{\mathsf F}
\newcommand{\sG}{\mathsf G}
\newcommand{\sH}{\mathsf H}
\newcommand{\sK}{\mathsf K}
\newcommand{\sL}{\mathsf L}
\newcommand{\sM}{\mathsf M}
\newcommand{\sN}{\mathsf N}
\newcommand{\sP}{\mathsf P}
\newcommand{\sS}{\mathsf S}
\newcommand{\sT}{\mathsf T}
\newcommand{\sX}{\mathsf X}
\newcommand{\sY}{\mathsf Y}
\newcommand{\sZ}{\mathsf Z}

\newcommand{\cA}{\mathcal A}
\newcommand{\cB}{\mathcal B}
\newcommand{\cC}{\mathcal C}
\newcommand{\cG}{\mathcal G}
\newcommand{\cH}{\mathcal H}
\newcommand{\cK}{\mathcal K}
\newcommand{\cP}{\mathcal P}
\newcommand{\cZ}{\mathcal Z}

\newcommand{\bA}{\mathbf A}
\newcommand{\bB}{\mathbf B}
\newcommand{\bC}{\mathbf C}
\newcommand{\bE}{\mathbf E}
\newcommand{\bF}{\mathbf F}
\newcommand{\bG}{\mathbf G}

\newcommand{\biB}{\boldsymbol B}
\newcommand{\biL}{\boldsymbol L}
\newcommand{\biM}{\boldsymbol M}
\newcommand{\biN}{\boldsymbol N}
\newcommand{\biQ}{\boldsymbol Q}
\newcommand{\biR}{\boldsymbol R}
\newcommand{\biS}{\boldsymbol S}
\newcommand{\biX}{\boldsymbol X}

\newcommand\hathatRbu{\widehat{\widehat\biR}
 {\vphantom{\widehat\biR^t}}^\bu}
\newcommand\hathatBbu{\widehat{\widehat\biB}
 {\vphantom{\widehat\biB^t}}^\bu}

\newcommand{\boZ}{\mathbb Z}

\theoremstyle{plain}
\newtheorem{thm}{Theorem}[section]
\newtheorem{prop}[thm]{Proposition}
\newtheorem{lem}[thm]{Lemma}
\newtheorem{cor}[thm]{Corollary}
\theoremstyle{definition}
\newtheorem{rem}[thm]{Remark}
\newtheorem{ex}[thm]{Example}
\newtheorem{exs}[thm]{Examples}

\newcommand{\Section}[1]{\medskip\section{#1}\smallskip}
\begin{document}

\title{Exact DG-categories and fully faithful \\
triangulated inclusion functors}

\author{Leonid Positselski}

\address{Institute of Mathematics, Czech Academy of Sciences \\
\v Zitn\'a~25, 115~67 Praha~1 \\ Czech Republic}

\email{positselski@math.cas.cz}

\begin{abstract}
 We construct an ``almost involution'' assigning a new DG\+category
to a given one, and use this construction to recover, say,
the abelian category of graded modules over the graded ring $R^*$ from
the DG\+category of DG\+modules over a DG\+ring $(R^*,d)$.
 This provides an appropriate technical background for the definition
and discussion of abelian and exact DG\+categories.
 In the setting of exact DG\+categories, derived categories of
the second kind are defined in the maximal natural generality.
 We develop the related abstract category-theoretic language and
use it to formulate and prove several full-and-faithfulness theorems
for triangulated functors induced by the inclusions of fully exact
DG\+subcategories.
 Such functors are fully faithful for derived categories of
the second kind more often than for the conventional derived
categories.
 Examples and applications range from the categories of complexes in
abelian/exact categories to matrix factorization categories, and from
curved DG\+modules over curved DG\+rings to quasi-coherent CDG\+modules
over quasi-coherent CDG\+quasi-algebras over schemes.
\end{abstract}

\maketitle

\tableofcontents

\section*{Introduction}
\medskip

 Abelian and exact categories (in the sense of Quillen) are popular
subjects of contemporary research, and indeed they are crucially
important concepts in homological algebra.
 The same applies to differential graded categories.
 Still it often goes unmentioned that, for many examples of
DG\+categories appearing ``in nature'', there is an abelian or exact
category structure lurking nearby.

 In fact, there are usually \emph{two} such exact/abelian categories
related to a DG\+cat\-e\-gory.
 For example, if $\sA$ is an abelian category, then the DG\+category
of complexes in $\sA$ has two underlying abelian categories:
the abelian category of graded objects in $\sA$ (with homogeneous
maps of degree~$0$ as morphisms) and the abelian category of
complexes in $\sA$ (with closed morphisms of degree~$0$ between
the complexes playing the role of morphisms in the abelian category).
 Defining an \emph{abelian} or \emph{exact DG\+category} as
an abstract category-theoretic concept involves elaborating on
the connection between the two underlying abelian/exact categories.

 The main construction of this paper assigns to a DG\+category $\bA$
another DG\+category~$\bA^\bec$.
 Iterating this procedure twice, for a DG\+category $\bA$ with shifts
and cones, leads to a DG\+category $\bA^{\bec\bec}$ and a fully faithful
DG\+functor $\bec\bec\:\bA\rarrow\bA^{\bec\bec}$, which is
an equivalence whenever the original DG\+category $\bA$ is
idempotent-complete and all twists of its objects by Maurer--Cartan
cochains in their complexes of endomorphisms exist in~$\bA$.
 Otherwise, generally speaking, the passage from $\bA$ to
$\bA^{\bec\bec}$ adds all twists and some of their direct summands
(see Propositions~\ref{iterated-bec-construction}
and~\ref{iterated-bec-construction-revisited}).

 We use this construction for our purposes as a way to produce
a preadditive category that can serve the role of ``the underlying
category of graded objects'' of a DG\+category.
 The point is that, given a DG\+ring $\biR^\bu=(R^*,d)$, \emph{not every
graded $R^*$\+module admits a differential making it a DG\+module
over~$\biR^\bu$}, not even at the cost of adding an extra direct summand
(Examples~\ref{nonexistence-of-cdg-structure-examples}\+-%
\ref{nonexistence-of-derivation-examples}).
 Given the DG\+category $\bA=\biR^\bu\bmodl$, however, one can recover
the whole category of graded $R^*$\+modules by attaching to every graded
$R^*$\+module $M^*$ a contractible DG\+module $G^+(M^*)$ over $\biR^\bu$
freely generated by $M^*$, endowed with its canonical contracting
homotopy with zero square.
 In our notation, this means that the abelian category $R^*\smodl$
of graded $R^*$\+modules is equivalent to the category $\sZ^0(\bA^\bec)$
of closed morphisms of degree~$0$ in the DG\+category
$\bA^\bec=(\biR^\bu\bmodl)^\bec$
(see Example~\ref{CDG-ring-bec-example}).

 Given a DG\+category $\bA$ with shifts and cones, there is a triple
(in fact, a doubly infinite shift-periodic sequence) of faithful,
conservative adjoint functors between the additive categories
$\sZ^0(\bA)$ and $\sZ^0(\bA^\bec)$ of closed morphisms of degree~$0$
in $\bA$ and~$\bA^\bec$.
 These functors are interpreted as ``forgetting the differential''
(say, in a DG\+module) and ``freely generating a DG\+module by
a graded module''.

 A DG\+category $\bA$ is said to be \emph{abelian} if both
the additive categories $\sZ^0(\bA)$ and $\sZ^0(\bA^\bec)$ are abelian.
 The natural functors mentioned above then preserve and reflect
short exact sequences.
 An \emph{exact structure} on a DG\+category $\bE$ is the datum of
exact category structures on both the additive categories
$\sZ^0(\bE)$ and $\sZ^0(\bE^\bec)$ such that the same functors
preserve and reflect admissible short exact sequences.
 Complexes in abelian/exact categories, DG\+modules or curved
DG\+modules, and factorization categories provide natural examples
of abelian and exact DG\+categories.
 While an exact DG\+category need not have twists, any abelian
DG\+category has them; so for an abelian DG\+category, the DG\+functor
$\bec\bec\:\bA\rarrow\bA^{\bec\bec}$ is a DG\+equivalence.

 The first part of this paper (Sections~\ref{graded-and-DG-secn}\+-%
\ref{exact-dg-cats-secn}) is a greatly expanded version
of~\cite[Section~3.2 and Remark~3.5]{Pkoszul}.
 The theory whose key elements were very briefly hinted at
in~\cite{Pkoszul} is developed with full details here.
 After all the preparatory work is done and the theory is ready,
we eventually arrive, under mild assumptions, at a simple and
straightforward equivalent definition of an exact DG\+category
(see Remark~\ref{long-journey-rem}).
 But the theory serves to demonstrate that this simple definition is
the right one.

 What can one do with an exact DG\+category?
 Following~\cite[Remarks~3.5\+-3.7]{Pkoszul}, we argue that
exact DG\+categories provide the maximal natural generality setting
for the constructions of derived categories of the second kind,
such as the coderived, contraderived, and absolute derived categories.
 The concept of derived categories of the second kind goes back
to~\cite[Sections~2.1 and~4.1]{Psemi}
and~\cite[Sections~3.3 and~4.2]{Pkoszul}; we refer the reader
to~\cite[Remark~9.2]{PS4} and~\cite[Section~7]{Pksurv} for a historical
and philosophical discussion.

 Let us \emph{warn} the reader that in this paper we generally consider
the complexes of morphisms in DG\+categories up to closed isomorphism
of complexes and \emph{not up to quasi-isomorphism}, and accordingly
the DG\+categories themselves are considered up to equivalence and
\emph{not up to quasi-equivalence}.
 So we mostly deal with equivalences rather than quasi-equivalences of
DG\+categories; likewise, ``fully faithful DG\+functors'' are presumed
to induce termwise (closed) isomorphisms of the complexes of morphisms
rather than merely quasi-isomorphisms.
 This is the natural point of view in the context of derived categories
of the second kind.

 Similarly, by \emph{finite direct sums} in a DG\+category $\bA$ we
mean finite direct sums in the preadditive category $\sZ^0(\bA)$ of
closed morphisms in~$\bA$.
 A DG\+category $\bA$ is called \emph{additive} if it has finite
direct sums in this sense (i.~e., if the category $\sZ^0(\bA)$ is
additive).
 This is a \emph{stronger} condition than the existence of finite
direct sums in the homotopy category $\sH^0(\bA)$.
 The existence of finite direct sums in $\bA$ is also a \emph{stronger}
condition than the existence of finite direct sums in the category
$\bA^0$ of arbitrary (not necessarily closed) morphisms in~$\bA$.
 Likewise, by a \emph{complex in} $\bA$ we always mean a complex in
$\sZ^0(\sA)$; no notions such a homotopy complex in a DG\+category
are considered in this paper.
 See Section~\ref{dg-categories-subsecn} for a discussion.

 In the second part of the present paper
(Sections~\ref{derived-second-kind-secn}\+-%
\ref{compact-generation-secn}), we formulate and prove some of the most
important general properties of derived categories of the second kind
using the concepts and tools developed in the first part of the paper.
 In that, we follow the approaches of~\cite[Sections~3.5\+-3.7
and~3.11]{Pkoszul} and~\cite[Sections~1.4\+-1.6]{EP}.
 The language of abelian and exact DG\+categories unites the results
about
\begin{enumerate}
\item the coderived, contraderived, and absolute derived categories of
(complexes in) abelian and exact categories, and
\item the coderived, contraderived, and absolute derived categories of
DG\+modules over DG\+rings (and curved DG\+modules over curved
DG\+rings)
\end{enumerate}
in a common framework, which includes also (matrix) factorization
categories, quasi-coherent CDG\+modules over quasi-coherent
CDG\+quasi-algebras over schemes, etc.
 In particular, in Section~\ref{derived-second-kind-secn} we reproduce
the results of~\cite[Sections~3.5\+-3.7]{Pkoszul} about graded-injective
and graded-projective resolutions in the context of exact
DG\+categories, as promised in~\cite[Remarks~3.5\+-3.7]{Pkoszul}.

 Replacing an abelian or exact category with its (co)resolving
subcategory with respect to which all the objects of the ambient
category have finite (co)resolution dimension generally leaves
the derived category unchanged.
 This applies both to the conventional derived categories (as defined
for exact categories in~\cite{Neem}), including the conventional
unbounded derived categories, as explained
in~\cite[Proposition~5.14]{Sto} or~\cite[Proposition~A.5.8]{Pcosh},
and to derived categories of the second kind.
 For derived categories of the second kind, we establish this
property in the generality of exact DG\+categories (or more
precisely, exact DG\+pairs) in Section~\ref{finite-resol-dim-secn},
using the technique going back to~\cite[Section~7.2.2]{Psemi},
\cite[Section~3.2]{PP2}, and~\cite[Section~1.4]{EP} (it is also
sketched in~\cite[Proposition~A.5.8]{Pcosh}).
 Let us emphasize that checking the essential surjectivity of such
triangulated functors is relatively easy; it is proving
the full-and-faithfulness that requires the bulk of the work.
 See Theorem~\ref{finite-resolution-dimension-main-theorem}.

 The nature of the construction of the triangulated Verdier quotient
category is such that inclusions of full subcategories of abelian
or exact categories need not induce fully faithful functors between
their derived categories, generally speaking.
 Proving that, under various specific assumptions, the inclusion of
a full subcategory into an abelian, exact, or homotopy category remains
fully faithful after the passage to the derived categories is
a nontrivial task.
 The results of the paper~\cite{PSch}, where the inclusions of
the full subcategories of Noetherian or coherent objects into
the ambient abelian categories of infinitely generated objects are
considered, illustrate the observation that such functors between
derived categories of the second kind tend to be fully faithful
more often than for the conventional unbounded derived categories.
 Another (and more obvious) such illustration is provided by
the inclusions of the full subcategories of projective or injective
objects into an abelian or exact category.

 In this paper, we deal with derived categories of the second kind.
 So we leave aside the results about fully faithful inclusions of
conventional unbounded derived categories provable by the method
worked out in~\cite[Theorems~1.3 and~2.9, Corollaries~1.4
and~2.10]{Pmgm}, \cite[Theorem~6.4 and Proposition~6.5]{PMat},
and~\cite[Proposition~6.5]{Pper}.
 Instead, using the technique of~\cite[Section~1.5]{EP}
and~\cite[Theorem~4.2.1]{Pweak} and following the approach
of~\cite[Section~12]{Kel2} and~\cite[Section~A.2]{Pcosh}, we consider
what we call \emph{self-resolving} (or \emph{self-coresolving})
subcategories in exact categories.
 We work with exact DG\+pairs $(\bE,\sK)$; here $\bE$ is an exact
DG\+category playing the role of the DG\+category of (curved)
DG\+modules, while the exact category $\sK$ plays the role of
the category of graded modules over the underlying graded ring.

 In the context of exact DG\+categories or exact DG\+pairs, we show
that the inclusion of a self-resolving subcategory induces a fully
faithful functor between the absolute derived and the contraderived
categories, while the inclusion of a self-coresolving subcategory
induces fully faithful functors between the absolute derived and
coderived categories.
 In the case when the exact subcategory is actually resolving or
coresolving, it follows that the induced triangulated functor is
a triangulated equivalence.
 This is the material of our Section~\ref{full-and-faithfulness-secn}.
 The particular case of its results described
in~\cite[Proposition~A.3.1(b)]{Pcosh} has already found its uses in
the context of so-called pseudo-coderived and pseudo-contraderived
categories~\cite[Section~4]{PS2}, \cite[Section~1]{Pps}.
 See Theorems~\ref{full-and-faithfulness-main-theorem}
and~\ref{contraderived-resolving-equivalence-theorem} for the full
generality.

 It is a well-known fact that the theory of unbounded derived
categories becomes simpler under the assumption of finite homological
dimension of the abelian/exact category.
 The same applies to derived categories of the second kind, but
there is a caveat pertaining to DG\+modules.
 Let us explain the situation, starting with the simpler case of
complexes in exact categories.

 It is obvious from the definitions that any absolutely acyclic complex
in an exact category is acyclic.
 If the exact category has exact (co)product functors, then any
co/contraacyclic complex is acyclic as well.
 By the definition, any absolutely acyclic complex is both coacyclic
and contraacyclic.
 So the coderived and contraderived categories are ``sandwiched in
between'' the absolute derived and the conventional derived category,
in the sense of the existence of natural triangulated Verder
quotient functors acting from the absolute derived category to
the co/contraderived category and then to the conventional derived
category.
 By~\cite[Remark~2.1]{Psemi}, for an exact category of finite
homological dimension, all acyclic complexes are absolutely acyclic;
so the absolute derived category coincides with the conventional
derived category.
 Hence both of them coincide with the co/contraderived category if
the exact category has exact (co)product functors.

 Using~\cite[Theorems~3.4.1(d) and~3.6]{Pkoszul}, one can extend
these results to nonpositively cohomologically graded DG\+rings
with the underlying graded ring having finite global dimension;
and~\cite[Theorems~3.4.2(d) and~3.6]{Pkoszul} imply the similar
results for connected, simply connected nonnegatively cohomologically
graded DG\+rings with the underlying graded ring of finite global
dimension (see~\cite[Theorem~7.8(b)]{Pksurv}).
 There are other sets of assumptions under which the absolute derived
category of DG\+modules over a DG\+ring coincides with
the conventional derived category; e.g., this holds for cofibrant
DG\+algebras over a field~\cite[Theorem~9.4]{Pkoszul}.
 But this is \emph{not} true for nonnegatively cohomologically
graded DG\+rings with the underlying graded ring having finite global
dimension in general (see the references in~\cite[last paragraph
of Section~3.6]{Pkoszul} for a detailed discussion).

 Once again, in this paper we only deal with derived categories of
the second kind.
 For an exact DG\+pair $(\bE,\sK)$ such that the exact category $\sK$
has finite homological dimension, we show that the thick subcategory
of absolutely acyclic objects in the homotopy category of $\bE$ is
\emph{strongly generated} (in the sense of~\cite{BvdB}) by
the totalizations of short exact sequences (see
Proposition~\ref{finite-homol-dim-abs-acycl-strongly-generated-prop}).
 It follows that the absolute derived category of $\bE$ coincides with
the coderived and/or contraderived category whenever $\sK$ has finite
homological dimension and the exact DG\+category $\bE$ has exact
(co)product functors.
 These assertions are easy to prove under the assumption of enough
projective or injective objects in $\sK$, but the general case is quite
involved.
 This is the material of Section~\ref{finite-homol-dim-secn}, where
we spell out the details of the argument whose ideas go back
to~\cite[proof of Theorem~1.6]{EP}.
 See Theorem~\ref{finite-homological-dimension-main-theorem}.

 In the last Section~\ref{compact-generation-secn}, we mostly deal with
abelian DG\+categories.
 We consider Grothendieck abelian DG\+categories, locally finitely
presentable DG\+categories, locally Noetherian and locally coherent
DG\+categories.
 The main results concern compact generators in the coderived categories
of abelian DG\+categories from the latter two classes.
 For any locally Noetherian DG\+category $\bA$, we denote by
$\bA_\bfg\subset\bA$ the full abelian DG\+subcategory of finitely
generated (Noetherian) objects in $\bA$, and show that the absolute
derived category of $\bA_\bfg$ is a full triangulated subcategory
in the coderived category of $\bA$, consisting of compact objects and
generating the whole coderived category.
 The proof of this result, going back to D.~Arinkin and originally
recorded, with his kind permission, in~\cite[Theorem~3.11.2]{Pkoszul}
in the context of CDG\+modules over CDG\+rings, is reproduced in
the much more general context of locally Noetherian DG\+categories
in this paper.
 See Theorem~\ref{Noetherian-compact-generation-theorem}.

 For locally coherent DG\+categories, we obtain a similar compact
generation result under additional assumptions of what we call
``finiteness of fp\+dimensions'' of the two locally coherent
abelian categories associated with a locally coherent DG\+cat\-e\-gory.
 The coderived category $\sD^\co(\bA)$ of a locally coherent
(Grothendieck abelian) DG\+category $\bA$ contains the absolute
derived category $\sD^\abs(\bA_\bfp)$ of its full abelian
DG\+subcategory of finitely presentable (coherent) objects $\bA_\bfp$
as a full triangulated subcategory consisting of compact objects.
 Under the additional assumptions mentioned above, the full subcategory
$\sD^\abs(\bA_\bfp)$ generates the coderived category $\sD^\co(\bA)$.
 This is the result of
Theorem~\ref{coherent-finite-fp-dimen-compact-generation-theorem}.

 As a corollary, we obtain the following result for CDG\+modules
over a CDG\+ring $\biR^\cu=(R^*,d,h)$.
 Assume that the graded ring $R^*$ is graded left coherent and
there is an integer~$n$ such that all the homogeneous left ideals
in $R^*$ have less than $\aleph_n$ generators.
 Then the coderived category of left CDG\+modules
$\sD^\co(\biR^\cu\bmodl)$ coincides with the homotopy category of
graded-injective left CDG\+modules over $\biR^\cu$ (known as
the ``coderived category in the sense of Becker''~\cite{Bec,PS4}),
is compactly generated and, up to direct summands, the absolute
derived category of finitely presentable CDG\+modules
$\sD^\abs(\biR^\cu\bmodl_\bfp)$ is the full subcategory of
compact objects in $\sD^\co(\biR^\cu\bmodl)$.
 This is our
Corollary~\ref{coherent-aleph-n-Noetherian-cdg-ring-corollary}.

 Other than proving that, under various restrictive assumptions on
an abelian or exact DG\+category, its coderived or contraderived
category in our sense agrees with the one in the sense of Becker, we
do \emph{not} consider specifically co/contraderived categories in
the sense of Becker in this paper.
 In particular, no model structures are mentioned in this paper.
 This material is relegated to a separate paper~\cite{PS5}, joint
with J.~\v St\!'ov\'\i\v cek, where we purport to show that the theory
of Becker's derived categories of the second kind, as developed for
CDG\+modules in~\cite{Bec} and for abelian categories in~\cite{PS4},
finds its maximal natural generality in the context of locally
presentable abelian DG\+categories.
 More precisely, it seems that the natural generality for Becker's
coderived categories is that of Grothendieck abelian DG\+categories,
while for Becker's contraderived categories the natural context is
that of locally presentable abelian DG\+categories with enough
projective objects.
 In particular, the finite fp\+projective dimension and
$\aleph_n$\+Noetherianity assumptions in the above-mentioned results
in the (locally) coherent case can be dropped if one is willing to
consider Becker's coderived categories, as we show in~\cite{PS5}.

\medskip
\subsection*{Acknowledgement}
 I~am grateful to Bernhard Keller for stimulating interest and to
Jan \v St\!'ov\'\i\v cek for helpful discussions.
 I~also wish to thank two anonymous referees for several relevant 
suggestions, which led, in particular, to an important improvement
of the result of Example~\ref{minimal-exact-dg}.
 The author was supported by the GA\v CR project 20-13778S and
the Institute of Mathematics, Czech Academy of Sciences (research plan
RVO:~67985840).

\Section{Graded Categories and DG-Categories} \label{graded-and-DG-secn}

\subsection{Grading group datum}
 We will work with complexes (graded abelian groups, etc.)\ whose
components are indexed by a fixed abelian group~$\Gamma$.
 A symmetric bilinear form $\sigma\:\Gamma\times\Gamma\rarrow\boZ/2$
and an element $\boldsymbol1\in\Gamma$ are presumed to be chosen,
and the equation $\sigma(\boldsymbol1,\boldsymbol1)=1\bmod 2$ is
imposed.
 The differentials on complexes will raise the degree
by~$\boldsymbol1$, and the sign $(-1)^{\sigma(a,b)}$ will appear
when a factor of degree $a\in\Gamma$ and a factor of
degree $b\in\Gamma$ are moved around each other in a formula.

 The most standard example is $\Gamma=\boZ$, \ $\boldsymbol1=1$, and
$\sigma(a,b)=ab\bmod2$.
 Alternatively, taking $\Gamma=\boZ/2$, \ $\boldsymbol1=1\bmod2$, and
$\sigma(a,b)=ab$ means considering $2$\+periodic complexes.
  When working with complexes of vector spaces over $\mathbb F_2$,
the form~$\sigma$ is not needed, so one can even put $\Gamma=0$
and consider $1$\+periodic complexes.

 The point is that sign rules are necessary to make the theory of
complexes and DG\+categories work.
 For example, in order to spell out the definition of a DG\+category,
one needs the sign rule for the differential on the tensor product of
two complexes of abelian groups or vector spaces.
 The explanation is that, for the equation $d^2=0$ on the differential
in a complex to have proper meaning, the differential~$d$ is supposed
to change a suitable parity.
 Writing down sign rules for $1$\+periodic complexes, which have no
parity alternated by the differential, is problematic.
 However, over $\mathbb F_2$, one has $-1=1$, and no sign rules
are needed.

 The notation related to the grading group will be generally
suppressed in the sequel: we will write~$1$ instead of~$\boldsymbol1$,
\ $n$ instead of $n\cdot\boldsymbol1$ for $n\in\boZ$, \ $ab$ instead of
$\sigma(a,b)$, \ $a$ instead of $\sigma(\boldsymbol1,a)$, etc.
 The meaning will be clear from the context.
 The abelian group homomorphism $\boZ\rarrow\Gamma$ taking $1$ to
$\boldsymbol1$ will be presumed when constructing the total complexes
of bicomplexes graded by $\Gamma$ along one coordinate and by
the integers along the other one.
 The reader is referred to~\cite[Section~1.1]{PP2} for a more detailed
discussion.

\subsection{Graded categories}
 A \emph{$\Gamma$\+graded category} is a category enriched in
$\Gamma$\+graded abelian groups.
 In other words, a $\Gamma$\+graded category $\cA$ assigns to every
pair of objects $X$, $Y\in\cA$ a $\Gamma$\+graded abelian group
of morphisms $\Hom_\cA^*(X,Y)=\bigoplus_{i\in\Gamma}\Hom^i_\cA(X,Y)$.
 To every triple of objects $X$, $Y$, $Z\in\cA$, a morphism of
$\Gamma$\+graded abelian groups $\Hom_\cA^*(Y,Z)\ot_{\boZ}
\Hom_\cA^*(X,Y)\rarrow\Hom_\cA^*(X,Z)$, called the composition of
morphisms, is assigned.
 The composition of morphisms in $\cA$ must be associative as usual.
 The identity morphism $\id_X$ is an element of the group
$\Hom^0_\cA(X,X)$.

 A \emph{preadditive category} is a category enriched in abelian
groups.
 For any $\Gamma$\+graded category $\cA$, its underlying preadditive
category $\cA^0$ is defined by the rules that the objects of $\cA^0$
are the same as the objects of $\cA$ and $\Hom_{\cA^0}(X,Y)=
\Hom^0_\cA(X,Y)$ for all objects $X$, $Y\in\cA$ (with
the composition of morphisms in $\cA^0$ defined in the obvious way
in terms of the composition of morphisms in~$\cA$).
 A preadditive category is said to be \emph{additive} if the coproduct
of any finite collection of objects exists, or equivalently,
the product of any finite collection of objects exists.
 A $\Gamma$\+graded category is said to be \emph{additive} if its
underlying preadditive category is additive.

 A (covariant) \emph{graded functor} $F\:\cA\rarrow\cB$ between
$\Gamma$\+graded categories $\cA$ and $\cB$ is a rule assigning to
every object $X\in\cA$ an object $F(X)\in\cB$ and to every pair
of objects $X$, $Y\in\cA$ a morphism of $\Gamma$\+graded abelian
groups $\Hom_\cA^*(X,Y)\rarrow\Hom_\cB^*(F(X),F(Y))$ in a way
compatible with the compositions and the identity morphisms.
 A graded functor $F\:\cA\rarrow\cB$ induces an additive functor
between the underlying preadditive categories $F^0\:\cA^0\rarrow\cB^0$.

 A graded functor $F\:\cA\rarrow\cB$ is said to be \emph{fully
faithful} if it induces isomorphisms of graded abelian groups
$\Hom_\cA^*(X,Y)\simeq\Hom_\cB^*(F(X),F(Y))$ for all objects
$X$, $Y\in\cA$.
 A graded functor $F$ is said to be an \emph{equivalence of graded
categories} if $F$ is fully faithful and $F^0\:\cA^0\rarrow\cB^0$
is an essentially surjective additive functor.

 Given two objects $X$ and $Y$ in a $\Gamma$\+graded category $\cA$
and an element $i\in\Gamma$, one says that $Y$ is a \emph{shift} of
$X$ by~$[i]$ and writes $Y=X[i]$ if morphisms $f\in\Hom_\cA^{-i}(X,Y)$
and $g\in\Hom_\cA^i(Y,X)$ are given such that $fg=\id_Y$ and
$gf=\id_X$.
 Clearly, the shift $X[i]$ of a given object $X\in\cA$ by a given
degree $i\in\Gamma$ is defined uniquely up to a unique isomorphism of
degree~$0$.
 If all shifts exist in $\cA$, then they are well-defined as graded
functors (in fact, autoequivalences) $[i]\:\cA\rarrow\cA$.
 The induced functor between the underlying preadditive categories
will be denoted also by $[i]\:\cA^0\rarrow\cA^0$.
 
 A graded category with shifts is uniquely determined by its
underlying preadditive category with the shift functors acting on it.
 The graded abelian group $\Hom_\cA^*(X,Y)$ is recovered by the rule
$\Hom_\cA^i(X,Y)=\Hom_\cA^0(X,Y[i])=\Hom_{\cA^0}(X,Y[i])$ for
all objects $X$, $Y\in\cA$ and degrees $i\in\Gamma$.

 Let $\cA$ be a graded category.
 An object $X\in\cA$ is said to be the \emph{coproduct} of a family
of objects $X_\alpha\in\cA$ if a functorial isomorphism of graded
abelian groups $\Hom_\cA^*(X,Z)\simeq\prod_\alpha\Hom_\cA^*(X_\alpha,Z)$
is given for all the objects $Z\in\cA$ (where $\prod$ denotes
the product of graded abelian groups).
 Similarly, an object $Y\in\cA$ is said to be the \emph{product} of
a family of objects $Y_\alpha\in\cA$ if a functorial isomorphism of
graded abelian groups $\Hom_\cA^*(Z,Y)\simeq\prod_\alpha\Hom_\cA^*(Z,
Y_\alpha)$ is given for all the objects $Z\in\cA$.
 Here one requires functoriality with respect to all morphisms in $\cA$
(of all the degrees).
 If $\cA$ is a graded category with shifts, then an object $X$
is the coproduct of the objects $X_\alpha$ in $\cA$ if and only if it
is their coproduct in the category $\cA^0$, and similarly an object $Y$
is the product of the objects $Y_\alpha$ in $\cA$ if and only if it is
their product in~$\cA^0$.

 An additive category $\sA$ is said to be \emph{idempotent-complete}
if for every object $X\in\sA$ and every morphism $e\:X\rarrow X$
such that $e^2=e$ there exist two objects $Y$, $Z\in\sA$ and
an isomorphism $X\simeq Y\oplus Z$ for which the morphism~$e$ is
equal to the composition of the direct summand projection $X\rarrow Y$
with the direct summand inclusion $Y\rarrow X$.
 An additive graded category $\cA$ is said to be
\emph{idempotent-complete} if the additive category $\cA^0$ is
idempotent-complete.

\subsection{DG-categories} \label{dg-categories-subsecn}
 This section is a brief reminder of the material
of~\cite[Section~1.2]{Pkoszul}, intended mainly to fix the terminology
and notation.

 A \emph{DG\+category} is a category enriched in ($\Gamma$\+graded)
complexes.
 In other words, a DG\+category $\bA$ assigns to every pair of
objects $X$, $Y\in\bA$ a complex of abelian groups
$\Hom_\bA^\bu(X,Y)$.
 To every triple of objects $X$, $Y$, $Z\in\bA$, a morphism of
complexes of abelian groups $\Hom_\bA^\bu(Y,Z)\ot_\boZ
\Hom_\bA^\bu(X,Y)\rarrow\Hom_\bA^\bu(X,Z)$, called the composition
of morphisms, is assigned.
 The composition of morphisms in $\bA$ must be strictly associative,
i.~e., associative on the level of cochains.
 The identity morphism $\id_X$ is a cocycle in $\Hom_\bA^0(X,X)$.
 Generally, a \emph{closed morphism of degree~$i$} between objects
$X$ and $Y\in\bA$ is a cocycle in $\Hom_\bA^i(X,Y)$.

 To every DG\+category $\bA$ one can assign a preadditive category
$\bA^0$ and a graded category $\bA^*$.
 The objects of $\bA^0$ and $\bA^*$ are the same as the objects of
$\bA$, and the morphisms are defined by the rules
$\Hom_{\bA^0}(X,Y)=\Hom_\bA^0(X,Y)$ and $\Hom_{\bA^*}^i(X,Y)=
\Hom_\bA^i(X,Y)$ for all objects $X$, $Y\in\sA$.

 Furthermore, to any DG\+category $\bA$ one can assign a graded
category $\cZ(\bA)$ and a preadditive category $\sZ^0(\bA)$.
 The objects of $\cZ(\bA)$ are the same as the objects of $\bA$,
and the morphisms are defined by the rule that
$\Hom_{\cZ(\bA)}^*(X,Y)$ is the graded group of all cocycles in
the complex $\Hom_\bA^\bu(X,Y)$, i.~e., closed morphisms in~$\bA$.
 The preadditive category $\sZ^0(\bA)$ is defined as
$\sZ^0(\bA)=(\cZ(\bA))^0$, so $\Hom_{\sZ^0(\bA)}(X,Y)$ is
the group of cocycles in the term $\Hom_\bA^0(X,Y)$ of
the complex $\Hom_\bA^\bu(X,Y)$.

 Similarly, to any DG\+category $\bA$ one can assign a graded
category $\cH(\bA)$ and a preadditive category $\sH^0(\bA)$.
 The objects of $\cH(\bA)$ are the same as the objects of $\bA$,
and the morphisms are defined by the rule that
$\Hom_{\cH(\bA)}^*(X,Y)=H^*\Hom_\bA^\bu(X,Y)$ is the graded group of
cohomology of the complex $\Hom_\bA^\bu(X,Y)$.
 The preadditive category $\sH^0(\bA)$ is defined as
$\sH^0(\bA)=(\cH(\bA))^0$, so $\Hom_{\sH^0(\bA)}(X,Y)=
H^0\Hom_\bA^\bu(X,Y)$.

 By a \emph{direct sum} of a finite set of objects in a DG\+category
$\bA$ we mean their direct sum in the preadditive category $\sZ^0(\bA)$.
 So an object $Z\in\bA$ is a direct sum of two objects $X$, $Y\in\bA$
if closed morphisms of degree zero $\iota_X\:X\rarrow Z$,
\ $\pi_X\:Z\rarrow X$, \ $\iota_Y\:Y\rarrow Z$, and $\pi_Y\:Z\rarrow
Y$ are given, satisfying the usual equations for morphisms defining
a direct sum (biproduct) of two objects in an additive category.
 This is a \emph{stronger} condition than the object $Z$ being
a direct sum of $X$ and $Y$ in the homotopy category $\sH^0(\sA)$,
in that the equations need to be satisfied in $\bA$ and not only
in~$\sH^0(\bA)$.
 The object $Z$ being a direct sum of $X$ and $Y$ in $\bA$ is also
a \emph{stronger} condition than $Z$ being a direct sum of $X$ and $Y$
in $\bA^0$, in that the structure morphisms $\iota_X$, $\pi_X$,
$\iota_Y$, and~$\pi_Y$ need to be closed.

 A DG\+category $\bA$ is said to be \emph{additive} if the preadditive
category $\sZ^0(\bA)$ is additive, i.~e., if finite direct sums exist
in~$\bA$.
 In this case, the preadditive categories $\bA^0$ and $\sH^0(\bA)$
are additive as well.
 An additive DG\+category $\bA$ is said to be \emph{idempotent-complete}
if the additive category $\sZ^0(\bA)$ is idempotent-complete.

 Let $X$ be an object of a DG\+category~$\bA$.
 An element $a\in\Hom^1_\bA(X,X)$ is said to be a \emph{Maurer--Cartan
cochain} (or to \emph{satisfy the Maurer--Cartan equation}) if
$d(a)+a^2=0$ in $\Hom^2_\bA(X,X)$.
 An object $Y\in\bA$ is called the \emph{twist} of $X$ by
a Maurer--Cartan cochain $a\in\Hom^1_\bA(X,X)$ (denoted $Y=X(a))$ if
morphisms $f\in\Hom^0_\bA(X,Y)$ and $g\in\Hom^0_\bA(Y,X)$ are given
such that $fg=\id_Y$, \ $gf=\id_X$, and $d(f)=fa$.
 Then one has $d(g)=-ag$, hence $X=Y(-fag)$.

 Conversely, let $X$ and $Y$ be two objects in $\bA$, and
$f\in\Hom^0_\bA(X,Y)$ and $g\in\Hom^0_\bA(Y,X)$ be two morphisms
such that $fg=\id_Y$ and $gf=\id_X$.
 Put $a=gd(f)\in\Hom^1_\bA(X,X)$.
 Then $a$~is a Maurer--Cartan cochain and $Y=X(a)$.
 The twist $X(a)$ of a given object $X$ by a given Maurer--Cartan
cochain~$a$, if it exists, is defined uniquely up to a unique
closed isomorphism of degree~$0$.

 Given an object $X\in\bA$ and a degree $i\in\Gamma$, one says that
an object $Y\in\bA$ is a \emph{shift} of $X$ by~$[i]$ (denoted by
$Y=X[i]$) if morphisms $f\in\Hom^{-i}_\bA(X,Y)$ and
$g\in\Hom^i_\bA(Y,X)$ are given such that $fg=\id_Y$, \ $gf=\id_X$,
and $d(f)=0$ (equivalently, $d(g)=0$).
 Clearly, the shift $X[i]$ of a given object $X$ by a given degree~$i$,
if it exists, is defined uniquely up to a unique closed isomorphism
of degree~$0$.
 Any shift in a DG\+category $\bA$ is also a shift in the graded
categories $\bA^*$, \ $\cZ(\bA)$, and $\cH(\bA)$.

 An object $X\in\bA$ is said to be the \emph{coproduct} of a family
of objects $X_\alpha\in\bA$ if an isomorphism of complexes of
abelian groups $\Hom_\bA^\bu(X,Z)\simeq\prod_\alpha\Hom_\bA^\bu
(X_\alpha,Z)$ is specified for all the objects $Z\in\bA$, in a way
functorial with respect to all morphisms in~$\bA$.
 Equivalently, one can say that $X=\coprod_\alpha X_\alpha$ in $\bA$
if a closed morphism $X_\alpha\rarrow X$ of degree~$0$ is given for
every index~$\alpha$ making $X$ the coproduct of the objects
$X_\alpha$ in the graded category~$\bA^*$.
 The coproduct of a family of objects in $\bA$, if it exists, is
defined uniquely up to a unique closed isomorphism of degree~$0$.
 Any coproduct in a DG\+category $\bA$ is also a coproduct in
the graded categories $\bA^*$, \ $\cZ(\bA)$, and $\cH(\bA)$.

 Dually, an object $Y\in\bA$ is said to be the \emph{product} of
a family of objects $Y_\alpha\in\bA$ if an isomorphism of complexes
of abelian groups $\Hom_\bA^\bu(Z,Y)\simeq\prod_\alpha\Hom_\bA^\bu
(Z,Y_\alpha)$ is specified for all the objects $Z\in\bA$, in a way
functorial with respect to all morphisms in~$\bA$.
 Equivalently, one can say that $Y=\prod_\alpha Y_\alpha$ in $\bA$
if a closed morphism $Y\rarrow Y_\alpha$ of degree~$0$ is given for
every index~$\alpha$ making $Y$ the product of the objects
$Y_\alpha$ in the graded category~$\bA^*$.

 Finite products in $\bA$ coincide with the finite coproducts,
and with what we defined above as finite direct sums in~$\bA$;
they exist if and only if the DG\+category $\bA$ is additive.
 So finite (co)products in $\bA$ are the same thing as finite
(co)products in~$\sZ^0(\bA)$, essentially because finite direct sums
can be described in terms of structure morphisms and equations on
them, as mentioned above.
 However, one needs to make additional assumptions about $\bA$ in order
to prove that any infinite coproduct in $\sZ^0(\bA)$ is a coproduct
in~$\bA$; see Lemma~\ref{coproducts-in-cocycles-and-in-DG} below.

 An object $C\in\bA$ is said to be the \emph{cone} of a closed
morphism $f\:X\rarrow Y$ of degree~$0$ in $\bA$ if an isomorphism
between the complex of abelian groups $\Hom_\bA^\bu(Z,C)$ and
the cone of the morphism of complexes $\Hom_\bA^\bu(Z,f)\:
\Hom_\bA^\bu(Z,X)\rarrow\Hom_\bA^\bu(Z,Y)$ is specified for all objects
$Z\in\bA$, in a way functorial with respect to all morphisms in~$\bA$.
 Equivalently, $C$ is the cone of~$f$ if an isomorphism between
the complex of abelian groups $\Hom_\bA^\bu(C[-1],Z)$ and the cone
of the morphism of complexes $\Hom_\bA^\bu(f,Z)\:\Hom_\bA^\bu(Y,Z)
\rarrow\Hom_\bA^\bu(X,Z)$ is specified for all objects $Z\in\bA$,
in a way functorial with respect to all morphisms in~$\bA$.
 The cone of a closed morphism in $\bA$, if it exists, is defined
uniquely up to a unique closed isomorphism of degree~$0$.

 Let $f\:X\rarrow Y$ be a closed morphism of degree~$0$ in~$\bA$
and $C=\cone(f)$ be its cone.
 Then there are natural closed morphisms $Y\rarrow C\rarrow X[1]$
of degree~$0$ in~$\bA$.
 The short sequence $0\rarrow Y\rarrow C\rarrow X[1]\rarrow0$
is split exact in the preadditive category~$\bA^0$.
 Conversely, any short sequence $0\rarrow U\rarrow V\rarrow W
\rarrow0$ of closed morphisms of degree~$0$ in $\bA$ which is
split exact in $\bA^0$ arises from a closed morphism
$f\:W[-1]\rarrow U$ of degree~$0$ in $\bA$ in this way.

 In an additive DG\+category $\bA$ with shifts, the cone $C=\cone f$
of a closed morphism $f\:X\rarrow Y$ of degree~$0$ can be constructed
as a twist $(Y\oplus X[1])(a_f)$ of the direct sum $Y\oplus X[1]$ by
a suitable Maurer--Cartan cochain~$a_f$ produced from~$f$.
 The Maurer--Cartan cochain~$a_f$ has an additional property that
$d(a_f)=0=a_f^2$.
 Thus any additive DG\+category with shifts and twists has cones.
 Conversely, any DG\+category with shifts, cones, and a zero object
is additive.

 The notion of an additive DG\+category with shifts and
twists is closely related to the concepts of a ``twisted complex''
and the ``convolution of a twisted complex'' going back to
the paper~\cite[Section~1]{BK}, while the notion of an additive
DG\+category with shifts and cones corresponds to the discussion of
``one-sided twisted complexes'' in~\cite[Section~4]{BK}
(another relevant term is a ``pretriangulated DG\+category'').

 By a \emph{complex} in a DG\+category $\bA$ we will mean a sequence
of objects and morphisms
$$
 \dotsb\lrarrow X^{-1}\lrarrow X^0\rarrow X^1\lrarrow X^2\lrarrow\dotsb
$$
in $\bA$, indexed by the integers $n\in\boZ$ or elements of the grading
group $n\in\Gamma$.
 The differentials $X^n\rarrow X^{n+1}$ are presumed to be closed
morphisms of degree~$0$, and their compositions $X^{n-1}\rarrow X^n
\rarrow X^{n+1}$ must vanish as morphisms in~$\bA$.
 So a complex in $\bA$ is the same thing as a complex in $\sZ^0(\bA)$.

 The \emph{coproduct totalization} $\Tot^\sqcup(X^\bu)$ of a complex
$X^\bu$ in $\bA$ is an object $Y\in\bA$ such that an isomorphism between
the complex of abelian groups $\Hom_\bA^\bu(Y,Z)$ and the total complex
of the bicomplex $\Hom_\bA^\bu(X^\bu,Z)$, constructed by taking
infinite products of abelian groups along the diagonals, is specified
for all objects $Z\in\bA$, in a way functorial with respect to all
morphisms in~$\bA$.
 In a DG\+category $\bA$ with twists, shifts, and coproducts,
the coproduct totalization $\Tot^\sqcup(X^\bu)$ can be constructed as
a twist of the coproduct $\coprod_{n\in\boZ}X^n[-n]$ or
$\coprod_{n\in\Gamma}X^n[-n]$ by a Maurer--Cardan cochain produced
from the differentials of the complex~$X^\bu$.
 The coproduct totalization of a complex in $\bA$, if it exists, is
defined uniquely up to a unique closed isomorphism of degree~$0$.

 Dually, the \emph{product totalization} $\Tot^\sqcap(X^\bu)$ of
a complex $X^\bu$ in $\bA$ is an object $Y\in\bA$ such that
an isomorphism between the complex of abelian groups $\Hom_\bA^\bu(Z,Y)$
and the total complex of the bicomplex $\Hom_\bA^\bu(Z,X^\bu)$,
constructed by taking infinite products of abelian groups along
the diagonals, is specified for all objects $Z\in\bA$, in a way
functorial with respect to all morphisms in~$\bA$.
 In a DG\+category $\bA$ with twists, shifts, and products, the product
totalization $\Tot^\sqcap(X^\bu)$ can be constructed as a twist of
the product $\prod_{n\in\boZ}X^n[-n]$ or $\prod_{n\in\Gamma}X^n[-n]$
by a Maurer--Cardan cochain produced from the differentials of
the complex~$X^\bu$.
 For a bounded complex $X^\bu$ in $\bA$, the product and coproduct
totalizations agree, and can be constructed as an iterated cone
(with a possible shift).
 The cone of a closed morphism of degree~$0$ is the totalization of
the related two-term complex in~$\bA$.

 A (covariant) \emph{DG\+functor} $F\:\bA\rarrow\bB$ between two
DG\+categories $\bA$ and $\bB$ is a rule assigning to every object
$X\in\bA$ an object $F(X)\in\bB$ and to every pair of objects
$X$, $Y\in\bA$ a morphism of complexes of abelian groups
$\Hom_\bA^\bu(X,Y)\rarrow\Hom_\bB^\bu(F(X),F(Y))$ in a way compatible
with the compositions and identity morphisms.
 Any DG\+functor preserves finite direct sums, shifts, twists,
and cones (in other words, one can say that these operations are
examples of ``absolute weighted colimits'' in DG\+categories,
in the sense of~\cite[Section~5]{NST}).

 A DG\+functor $F\:\bA\rarrow\bB$ is said to be \emph{fully faithful}
if the morphism of complexes of abelian groups $\Hom_\bA^\bu(X,Y)
\rarrow\Hom_\bB^\bu(F(X),F(Y))$ is a (termwise) isomorphism for all
the objects $X$, $Y\in\bA$.
 A fully faithful DG\+functor is said to be an \emph{equivalence of
DG\+categories} if for any object $Z\in\bB$ there exists an object
$X\in\bA$ together with a closed isomorphism $F(X)\simeq Z$ of
degree~$0$ in~$\bB$.
 A \emph{full DG\+subcategory} $\bA\subset\bB$ is a subclass of
objects in $\bB$ endowed with the induced DG\+category structure
(so that the inclusion $\bA\rarrow\bB$ is a fully faithful
DG\+functor).
 A full DG\+subcategory $\bA\subset\bB$ is said to be \emph{closed
under direct summands} if the full subcategory $\sZ^0(\bA)$ is
closed under direct summands in the preadditive category~$\sZ^0(\bB)$.

 Any DG\+functor $F\:\bA\rarrow\bB$ induces graded functors
$F^*\:\bA^*\rarrow\bB^*$, \ $\cZ(F)\:\allowbreak\cZ(\bA)
\rarrow\cZ(\bB)$, and $\cH(F)\:\cH(\bA)\rarrow\cH(\bB)$.
 Hence additive functors $F^0\:\bA^0\rarrow\bB^0$, \
$\sZ^0(F)\:\sZ^0(\bA)\rarrow\sZ^0(\bB)$, and
$\sH^0(F)\:\sH^0(\bA)\rarrow\sH^0(\bB)$ are also induced.
 A DG\+functor $F\:\bA\rarrow\bB$ is said to be
a \emph{quasi-equivalence} if the induced graded functor
$\cH(F)\:\cH(\bA)\rarrow\cH(\bB)$ is an equivalence of graded
categories.

 For any additive DG\+category $\bA$ with shifts and cones,
the degree-zero cohomology category $\sH^0(\bA)$ with its induced
shift functor $X\longmapsto X[1]$ has a natural structure of
a triangulated category.
 For additive DG\+categories $\bA$ and $\bB$ with shifts and cones,
the functor $\sH^0(F)$ induced by any DG\+functor $F\:\bA\rarrow\bB$
is triangulated.

\Section{Thematic Examples}  \label{examples-secn}

\subsection{Complexes in an additive category} \label{complexes-subsecn}
 Let $\sA$ be a preadditive category.
 The preadditive category $\sG(\sA)$ of graded objects in $\sA$ (and
homogeneous morphisms of degree~$0$ between them) can be simply defined
as the Cartesian product $\sA^\Gamma$ of $\Gamma$ copies of~$\sA$.
 Explicitly, an object $X^*\in\sG(\sA)$ is a collection of objects
$X^*=(X^i\in\sA)_{i\in\Gamma}$; a morphism $f_*\:X^*\rarrow Y^*$ in
$\sG(\sA)$ is a collection of morphisms
$(f_i\:X^i\to Y^i)_{i\in\Gamma}$.
 The addition of morphisms, the composition of morphisms, and
the identity morphisms in $\sG(\sA)$ are defined in the obvious way.

 The graded category $\cG(\sA)$ of graded objects in $\sA$ (and
homogeneous morphisms of various degrees $n\in\Gamma$ between them)
has the same objects as the category $\sG(\sA)$.
 For any objects $X^*$ and $Y^*\in\cG(\sA)$, the graded abelian group
$\Hom_{\cG(\sA)}^*(X^*,Y^*)$ is defined by the rule
$\Hom_{\cG(\sA)}^n(X^*,Y^*)=\Hom_{\sG(\sA)}(X^*,Y^*[n])$, where
$Y^*[n]^i=Y^{n+i}$ for all $n$, $i\in\Gamma$.
 So $\cG(\sA)$ is a $\Gamma$\+graded category with shifts, and
$\sG(\sA)=(\cG(\sA))^0$.

 The categories $\sG(\sA)$ and $\cG(\sA)$ are additive whenever
the preadditive category $\sA$~is.
 Assuming $\sA$ is additive, the categories $\sG(\sA)$ and $\cG(\sA)$
are idempotent-complete whenever the category $\sA$~is.
 The categories $\sG(\sA)$ and $\cG(\sA)$ have infinite coproducts
or products whenever the category $\sA$~has.

 A \emph{complex} $X^\bu$ in $\sA$ is a $\Gamma$\+graded object
endowed with a homogeneous differential~$d$ of degree~$1$ satisfying
the equation $d^2=0$.
 This means for every $i\in\Gamma$ the component $d_i\:X^i\rarrow
X^{i+1}$ of the differential~$d$ is a morphism in $\sA$ and
$d_{i+1}d_i=0$ for all $i\in\Gamma$.
 For any two complexes $X^\bu$ and $Y^\bu$ in $\sA$, the graded
abelian group $\Hom_{\cG(\sA)}^*(X^*,Y^*)$ of morphisms between
the underlying graded objects of $X^\bu$ and $Y^\bu$ is endowed with
the differential~$d$ of degree~$1$ defined by the formula
$d(f)_i=d_{Y,n+i}f_i-(-1)^n f_{i+1}d_{X,i}\:\Hom_{\cG(\sA)}^n(X^*,Y^*)
\rarrow\Hom_{\cG(\sA)}^{n+1}(X^*,Y^*)$, where $d_{X,i}\:X^i\rarrow
X^{i+1}$ and $d_{Y,i}\:Y^i\rarrow Y^{i+1}$.
 The resulting complex of abelian groups is denoted by
$\Hom_{\bC(\sA)}^\bu(X^\bu,Y^\bu)$; it is the complex of morphisms
from $X^\bu$ to $Y^\bu$ in the \emph{DG\+category\/ $\bC(\sA)$
of complexes in\/~$\sA$}.

 Notice that \emph{every graded object in\/ $\sA$ admits a differential
making it a complex} (e.~g., the zero differential).
 Consequently, the graded category $\bC(\sA)^*$ is naturally equivalent
to the graded category $\cG(\sA)$ of graded objects in $\sA$, and
the preadditive category $\bC(\sA)^0$ is naturally equivalent
to~$\sG(\sA)$.

 The preadditive category $\sZ^0(\bC(\sA))$ of closed morphisms of
degree zero in $\bC(\sA)$ is called the \emph{category of complexes
in\/~$\sA$} and denoted by $\sC(\sA)$.
 One can also consider the graded category of complexes $\cC(\sA)=
\cZ(\bC(\sA))$.
 The preadditive category $\sK(\sA)=\sH^0(\bC(\sA))$ of closed
morphisms of degree zero in $\bC(\sA)$ viewed up to cochain homotopy
is called the \emph{homotopy category of} (\emph{complexes in})~$\sA$.
 One can also consider the graded homotopy category $\cK(\sA)=
\cH(\bC(\sA))$.

 The DG\+category $\bC(\sA)$ always has shifts and twists.
 So the graded categories $\cC(\sA)$ and $\cK(\sA)$ have shifts, too.
 When $\sA$ is an additive category, the DG\+category $\bC(\sA)$ is
additive as well; hence it also has cones.
 The categories $\cC(\sA)$, \ $\sC(\sA)$, \ $\cK(\sA)$, and $\sK(\sA)$
are additive, too, in this case (in fact, $\sK(\sA)$ is
a triangulated category).
 When $\sA$ is an abelian category, so is $\sC(\sA)$.

 When an additive category $\sA$ is idempotent-complete, so
are the DG\+category $\bC(\sA)$, the graded category $\cC(\sA)$,
and the additive category~$\sC(\sA)$.
 The categories $\bC(\sA)$, \ $\cC(\sA)$, \ $\sC(\sA)$, \ $\cK(\sA)$,
and $\sK(\sA)$ have infinite products or coproducts whenever
the category $\sA$~has.

\subsection{Curved DG-modules} \label{curved-dg-modules-subsecn}
 A \emph{graded ring} $R^*=\bigoplus_{n\in\Gamma}R^n$ is a graded
abelian group endowed with an associative, unital ring structure
given by a homogeneous multiplication map $R^*\ot_\boZ R^*\rarrow R^*$.
 The unit is an element of~$R^0$.

 A \emph{graded left $R^*$\+module} $M^*=\bigoplus_{n\in\Gamma}M^n$
is a graded abelian group endowed with an associative, unital left
$R^*$\+module structure given by a homogeneous action map
$R^*\ot_\boZ M^*\rarrow M^*$.
 \emph{Graded right $R^*$\+modules} are defined similarly.

 Let $L^*$ and $M^*$ be two graded left $R^*$\+modules.
 A \emph{homogeneous $R^*$\+module map $f\:L^*\rarrow M^*$ of degree
$n\in\Gamma$} is a homogeneous map of degree~$n$ between
underlying graded abelian groups of $L^*$ and $M^*$ satisfying
the equation $f(rx)=(-1)^{n|r|}rf(x)$ for all $r\in R^{|r|}$
and $x\in L^{|x|}$ (notice the sign rule!).
 In the graded category $R^*\cmodl$ of graded left $R^*$\+modules,
the objects are the graded left $R^*$\+modules, and the degree~$n$
component of the graded group of morphisms $\Hom_{R^*\cmodl}^*(L^*,M^*)
=\Hom^*_{R^*}(L^*,M^*)$ is the group of all homogeneous $R^*$\+module
maps $L^*\rarrow M^*$ of degree~$n$.

 The graded category $\cmodr R^*$ of graded right $R^*$\+modules is
defined similarly, except that there is no sign involved in
the definition of a homogeneous right $R^*$\+module map.
 In the additive categories $R^*\smodl =(R^*\cmodl)^0$ and
$\smodr R^*=(\cmodr R^*)^0$, the homogeneous $R^*$\+module maps of
degree~$0$ form the groups of morphisms.
 The graded categories $R^*\cmodl$ and $\cmodr R^*$ have shifts.
 The categories $R^*\cmodl$, \ $R^*\smodl$, \ $\cmodr R^*$, and
$\smodr R^*$ are idempotent-complete, and they have infinite
coproducts and products.
 The categories $R^*\smodl$ and $\smodr R^*$ are abelian.

 A \emph{DG\+ring} $\biR^\bu=(R^*,d)$ is a complex of abelian groups
endowed with an associative, unital graded ring structure given
a multiplication map $\biR^\bu\ot_\boZ\biR^\bu\rarrow\biR^\bu$ which is
a morphism of complexes.
 This means that $d\:R^*\rarrow R^*$ is a odd derivation of degree~$1$
on the graded ring $R^*$, i.~e., it is a homogenenous endomorphism
of degree~$1$ of the graded abelian group $R^*$ satisfying
the \emph{Leibniz rule with signs} and, in addition,
the equation $d^2=0$ is satisfied.

 A \emph{left DG\+module} $\biM^\bu=(M^*,d_M)$ over $\biR^\bu$ is
a complex of abelian groups endowed with an associative, unital graded
left $R^*$\+module structure given by a multiplication map $\biR^\bu
\ot_\boZ\biM^\bu\rarrow\biM^\bu$ which is a morphism of complexes.
 This means that $d_M\:M^*\rarrow M^*$ is an odd derivation of
degree~$1$ on the graded $R^*$\+module $M^*$ compatible with the odd
derivation~$d$ on the graded ring~$R^*$, i.~e., $d_M$~is a homogeneous
endomorphism of degree~$1$ of the graded abelian group $M^*$
satisfying the Leibniz rule with signs $d_M(rx)=d(r)x+(-1)^{|r|}rd_M(x)$
for all $r\in R^{|r|}$ and $x\in M^{|x|}$ and, in addition,
the equation $d_M^2=0$ is satisfied.
 \emph{Right DG\+modules} over $\biR^\bu$ are defined similarly.

 Left DG\+modules over a DG\+ring $\biR^\bu$ form a DG\+category
$\biR^\bu\bmodl$, and similarly right DG\+modules over $\biR^\bu$ form
a DG\+category $\bmodr\biR^\bu$ in a natural way.
 The context of DG\+rings is \emph{not} the maximal natural
generality for this construction of DG\+categories, however,
as the construction can be extended naturally to a wider setting
of \emph{curved} DG\+rings (\emph{CDG\+rings}) and curved
DG\+modules over them.
 Let us recall the related definitions.

 A \emph{CDG\+ring} $\biR^\cu=(R^*,d,h)$ is a graded ring endowed with
an odd derivation $d\:R^*\rarrow R^*$ of degree~$1$ (i.~e.,
$d(rs)=d(r)s+(-1)^{|r|}rd(s)$ for all $r\in R^{|r|}$ and $s\in R^{|s|}$)
and an element $h\in R^2$ for which the following two equations
are satisfied:
\begin{enumerate}
\renewcommand{\theenumi}{\roman{enumi}}
\item $d^2(r)=[h,r]$ (where $[h,r]=hr-rh$ is the commutator)
for all $r\in R^*$; and
\item $d(h)=0$.
\end{enumerate}
 The element $h\in R^2$ is called the \emph{curvature element}.
 A pair $(R^*,d)$ is a DG\+ring if and only if the triple $(R^*,d,0)$
is a CDG\+ring; so the DG\+rings are just the CDG\+rings with
zero curvature.
 We refer to~\cite[Section~3.1]{Pkoszul} or~\cite[Section~3.2]{Prel}
for the (nontrivial!) definition of a \emph{morphism of CDG\+rings};
we will recall it below in Section~\ref{quasi-cdg-subsecn}.

 A \emph{left CDG\+module} $\biM^\cu=(M^*,d_M)$ over $\biR^\cu$ is
a graded left $R^*$\+module endowed with an odd derivation
$d_M\:M^*\rarrow M^*$ of degree~$1$ compatible with the derivation~$d$
on $R^*$ (i.~e., $d_M(rx)=d(r)x+(-1)^{|r|}rd_M(x)$ for all
$r\in R^{|r|}$ and $x\in M^{|x|}$) such that the following
equation holds:
\begin{enumerate}
\renewcommand{\theenumi}{\roman{enumi}}
\setcounter{enumi}{2}
\item $d_M^2(x)=hx$ for all $x\in M^*$.
\end{enumerate}

 Similarly, a \emph{right CDG\+module} $\biN^\cu=(N^*,d_N)$
over $\biR^\cu$ is a graded right $R^*$\+module endowed with an odd
derivation $d_N\:N^*\rarrow N^*$ of degree~$1$ compatible with
the derivation~$d$ on $R^*$ (i.~e., $d_N(yr)=d_N(y)r+(-1)^{|y|}yd(r)$
for all $y\in N^{|y|}$ and $r\in R^{|r|}$) such that the following
equation holds:
\begin{enumerate}
\renewcommand{\theenumi}{\roman{enumi}}
\setcounter{enumi}{3}
\item $d_N^2(y)=-yh$ for all $y\in N^*$.
\end{enumerate}

 For any two left CDG\+modules $\biL^\cu=(L^*,d_L)$ and
$\biM^\cu=(M^*,d_M)$ over a CDG\+ring $\biR^\cu=(R^*,d,h)$,
the graded abelian group $\Hom_{R^*}^*(L^*,M^*)$ is endowed with
a natural differential~$d$ of degree~$1$ defined by the usual rule
$d(f)(x)=d_M(f(x))-(-1)^{|f|}f(d_L(x))$ for all $x\in L^*$, where
$f\in\Hom_{R^*}^{|f|}(L^*,M^*)$.
 One has $d^2(f)=0$, as the two curvature-related terms cancel each
other:
\begin{multline*}
 d(d(f))(x)=d_M(d(f)(x))-(-1)^{|f|+1}d(f)(d_L(x)) \\
 =d_M(d_M(f(x))-(-1)^{|f|}d_M(f(d_L(x)))
 -(-1)^{|f|+1}d_M(f(d_L(x)))-f(d_L(d_L(x))) \\
 =hf(x)-f(hx)=0.
\end{multline*}
 Hence we obtain a complex of abelian groups, denoted by
$\Hom_{\biR^\cu}^\bu(\biL^\cu,\biM^\cu)$.

 In the \emph{DG\+category $\biR^\cu\bmodl$ of left CDG\+modules
over~$\biR^\cu$}, the left CDG\+modules over $\biR^\cu$ are the objects,
and $\Hom_{\biR^\cu\bmodl}^\bu(\biL^\cu,\biM^\cu)=
\Hom_{\biR^\cu}^\bu(\biL^\cu,\biM^\cu)$
is the complex of morphisms from $\biL^\cu$ to~$\biM^\cu$.
 The \emph{DG\+category\/ $\bmodr\biR^\cu$ of right CDG\+modules
over~$\biR^\cu$} is defined similarly.

 The DG\+categories $\biR^\cu\bmodl$ and $\bmodr\biR^\cu$ are additive 
with shifts and twists, and with infinite coproducts and products.
 For example, if $\biM^\cu=(M^*,d_M)$ is a left CDG\+module over
$\biR^\cu$ and $a\in\Hom^1_{\biR^\cu}(\biM^\cu,\biM^\cu)$ is
a Maurer--Cartan cochain, then the rule $d_M'(x)=d_M(x)+ax$ for all
$x\in M^*$ defines a left CDG\+module ${}'\!\biM^\cu=(M^*,d_M')$
over $\biR^\cu$ such that ${}'\!\biM^\cu=\biM^\cu(a)$.

 The category $\sZ^0(\biR^\cu\bmodl)$ of closed morphisms of degree
zero in $\biR^\cu\bmodl$ is abelian.
 It is called the \emph{abelian category of left CDG\+modules
over~$\biR^\cu$}.

 A DG\+module over a DG\+ring $\biR^\bu=(R^*,d)$ is the same thing as
a CDG\+module over the CDG\+ring $\biR^\cu=(R^*,d,0)$.
 So the DG\+categories of CDG\+modules $\biR^\cu\bmodl$ and
$\bmodr\biR^\cu$ are equivalent (in fact, isomorphic) to
the DG\+categories of DG\+modules $\biR^\bu\bmodl$ and
$\bmodr\biR^\bu$, respectively, when $h=0$.

 The reader should be warned that the graded category
$(\biR^\cu\bmodl)^*$ is \emph{not} equivalent to the graded category
$R^*\cmodl$, generally speaking, and the additive category
$(\biR^\cu\bmodl)^0$ is likewise \emph{not} equivalent to the abelian
category $R^*\smodl$.
 Rather, $(\biR^\cu\bmodl)^*$ and $(\biR^\cu\bmodl)^0$ are equivalent
to certain (not well-behaved) full subcategories in $R^*\cmodl$ and
$R^*\smodl$, respectively.
 This happens because \emph{not every graded $R^*$\+module admits
a structure of CDG\+module over~$\biR^\cu$} (not even when $h=0$).
 We refer to Section~\ref{cdg-revisited-subsecn} below for
a substantial discussion.

\subsection{Quasi-coherent graded quasi-algebras}
\label{quasi-algebras-subsecn}
 This section presents the simplest version of the definitions of
quasi-modules and quasi-algebras.
 For a detailed discussion of a more general and advanced approach,
see~\cite[Sections~1.4\+-1.5 and~3.1]{Pdomc}; and for the comparison
between several approaches, \cite[Section~2]{Ptd}.
 Let us just mention that what we call \emph{quasi-modules} in this
paper are called ``strong quasi-modules'' in~\cite{Pdomc,Ptd}.

 Let $R$ be a commutative ring and $A$ be an $R$\+$R$\+bimodule.
 Define an increasing filtration $F_0A\subset F_1A\subset F_2A
\subset\dotsb$ on $A$ by the rule that $F_{-1}A=0$ and
an element $a\in A$ belongs to $F_nA$ if and only if
$ra-ar\in F_{n-1}A$ for all $r\in R$.
 Then $F_nA$ is an $R$\+$R$\+subbimodule in $A$ for every $n\ge0$.
 If $A$ is an associative ring with the $R$\+$R$\+bimodule structure
induced by a ring homomorphism $\rho\:R\rarrow A$, then $F$ is
a multiplicative filtration on $A$, i.~e., $1\in\rho(R)\subset F_0A$
and $F_nA\cdot F_mA\subset F_{n+m}A$ for all $n$, $m\ge0$.

 Let us say that $A$ is a \emph{quasi-module} (in a more common
language, a \emph{differential bimodule}~\cite[Section~1.1]{BB})
over $R$ if the filtration $F$ is exhaustive,
that is $A=\bigcup_{n\ge0}F_nA$.
 Notice that any $R$\+module $M$, viewed as an $R$\+$R$\+bimodule in
which the left and right actions of $R$ agree, becomes a quasi-module
over~$R$ (with $F_0M=M$).

 An associative ring $A$ endowed with a ring homomorphism $R\rarrow A$
making $A$ a quasi-module over $R$ is said to be a \emph{quasi-algebra}
over~$R$.
 For example, if $k\subset R$ is a subring in a commutative ring $R$,
then the ring $A$ of all $k$\+linear differential operators
$R\rarrow R$ \,\cite{Bern,Smi} is a quasi-algebra over~$R$.

\begin{lem} \label{tensor-product-quasi-module}
 Let $A$ and $B$ be two quasi-modules over~$R$.
 Then $A\ot_RB$ is also a quasi-module over~$R$.
\end{lem}

\begin{proof}
 The image of the map $F_iA\ot_RF_jB\rarrow A\ot_RB$ induced
by the inclusions $F_iA\rarrow A$ and $F_jB\rarrow B$ is
contained in the subbimodule $F_n(A\ot_RB)\subset A\ot_RB$
for $n=i+j$.
 Hence the filtration $F$ on $A\ot_RB$ is exhaustive whenever
the similar filtrations on $A$ and $B$ are exhaustive.
\end{proof}

 We denote by $R\rmodl$ the abelian category of (ungraded, left)
modules over a ring~$R$.
 A homomorphism of commutative rings $\sigma\:R\rarrow S$ is
an epimorphism in the category of commutative (equivalently,
associative) rings if and only if the natural maps
$S\rightrightarrows S\ot_RS\rarrow S$ are isomorphisms
(in this case, the two maps $S\rightrightarrows S\ot_RS$
are equal to each other).
 A ring homomorphism~$\sigma$ is an epimorphism if and only if
the related functor of restriction of scalars $\sigma_*\:
S\rmodl\rarrow R\rmodl$ is fully faithful~\cite[Section~XI.1]{Ste}.

 A commutative ring epimorphism $\sigma\:R\rarrow S$ is said
to be \emph{flat} if $S$ is a flat $R$\+module.
 In this case, the full subcategory $\sigma_*(S\rmodl)$ is
closed under extensions in $R\rmodl$ \,\cite[Theorem~4.4]{GL}.
 Since the functor~$\sigma_*$ preserves all limits and colimits,
it is also clear that, for any ring epimorphism~$\sigma$,
the full subcategory $\sigma_*(S\rmodl)$ is closed under
limits and colimits in $R\rmodl$.

 Let $U$ be an affine scheme and $V\subset U$ be an affine open
subscheme.
 Let $R=O(U)$ and $S=O(V)$ be the rings of functions on $U$ and $V$,
and let $\sigma\:R\rarrow S$ be the ring homomorphism corresponding
to the open immersion morphism $V\rarrow U$.
 Then $\sigma$~is a flat epimorphism of commutative rings.
 This is the main example of a flat epimorphism we are interested in.

\begin{lem} \label{quasi-module-localization-lemma}
 Let $\sigma\:R\rarrow S$ be a flat epimorphism of commutative rings,
and let $A$ be a quasi-module over~$R$.
 Then \par
\textup{(a)} the map~$\sigma$ induces isomorphisms of
$R$\+$R$\+bimodules $S\ot_RA\larrow S\ot_RA\ot_RS\rarrow A\ot_RS$; \par
\textup{(b)} the $S$\+$S$\+bimodule $S\ot_RA$ is a quasi-module
over~$S$.
\end{lem}

\begin{proof}
 Part~(a): it suffices to show that the $R$\+$S$\+bimodule $A\ot_RS$,
viewed as an $R$\+module with the left action of $R$, belongs to
the full subcategory $\sigma_*(S\rmodl)\subset R\rmodl$.
 Indeed, put $G_n(A\ot_RS)=F_nA\ot_RS\subset A\ot_RS$.
 Then $G$ is an exhaustive increasing filtration of the bimodule
$A\ot_RS$ by subbimodules such that the left and right actions of $R$
in the successive quotient bimodules $G_n(A\ot_RS)/G_{n-1}(A\ot_RS)
\simeq (F_nA/F_{n-1}A)\ot_RS$ agree.
 Hence, viewed as $R$\+modules with the left action of $R$,
the successive quotients admit an extension of their $R$\+module
structure to an $S$\+module structure.
 It remains to recall that the full subcategory $\sigma_*(S\rmodl)
\subset R\rmodl$ is closed under extensions and direct limits.

 Part~(b): Following the proof of part~(a), the $S$\+$S$\+bimodule
$S\ot_RA\simeq S\ot_RA\ot_RS\simeq A\ot_RS$ admits an exhaustive
increasing filtration by subbimodules such that the left and right
actions of $S$ in the successive quotient bimodules agree.
 Hence $S\ot_RA$ is a quasi-module over~$S$.
\end{proof}

\begin{cor} \label{quasi-module-left-right-tensor-adjunction-cor}
 Let $\sigma\:R\rarrow S$ be a flat epimorphism of commutative rings.
 Let $A$ be a quasi-module over $R$ and $B$ be an $S$\+$S$\+bimodule.
 Let $A\rarrow B$ be a morphism of $R$\+$R$\+bimodules.
 Then the induced morphism of $S$\+$R$\+bimodules $S\ot_RA\rarrow B$ is
an isomorphism if and only if the induced morphism of
$R$\+$S$\+bimodules $A\ot_RS\rarrow B$ is an isomorphism.
\end{cor}

\begin{proof}
 Follows from Lemma~\ref{quasi-module-localization-lemma}(a).
\end{proof}

\begin{prop} \label{left-right-quasi-coherent-prop}
 Let $X$ be a scheme and $A$ be a sheaf of bimodules over the structure
sheaf $O_X$ of~$X$.
 Assume that, for every affine open subscheme $U\subset X$,
the $O(U)$\+$O(U)$\+bimodule $A(U)$ is a quasi-module.
 Then the sheaf of $O_X$\+modules $A$ with the left $O_X$\+module
structure is quasi-coherent if and only if the sheaf of $O_X$\+modules
$A$ with the right $O_X$\+module structure is quasi-coherent.
\end{prop}

\begin{proof}
 A sheaf of $O_X$\+modules $M$ on a scheme $X$ is quasi-coherent if and
only if, for every pair of affine open subschemes $V\subset U\subset X$,
the map of $O(V)$\+modules $O(V)\ot_{O(U)}M(U)\rarrow M(V)$ induced
by the restriction map $M(U)\rarrow M(V)$ is an isomorphism.
 Hence the assertion of the proposition follows from
Corollary~\ref{quasi-module-left-right-tensor-adjunction-cor}.
\end{proof}

 Let $X$ be a scheme.
 A \emph{quasi-coherent quasi-module} on $X$ is a sheaf of bimodules
over $O_X$ satisfying the equivalent conditions of
Proposition~\ref{left-right-quasi-coherent-prop}.
 In order to produce a quasi-coherent quasi-module $A$ over $X$,
it suffices to specify a quasi-module $A(U)$ over $O(U)$ for every
affine open subscheme $U\subset X$ and a morphism of
$O(U)$\+$O(U)$\+bimodules $A(U)\rarrow A(V)$ for every pair of affine
open subscheme $V\subset U\subset X$ in such a way that the induced
map $O(V)\ot_{O(U)}A(U)\rarrow A(V)$ is an isomorphism and
the triangle diagram $A(U)\rarrow A(V)\rarrow A(W)$ is commutative
for every triple of affine open subschemes
$W\subset V\subset U\subset X$ (cf.\ the description of
quasi-coherent sheaves in~\cite[Section~2]{EE}).

 Quasi-coherent quasi-modules on $X$ can be thought of as
``quasi-coherent torsion sheaves on the formal completion of
the diagonal $X$ in $X\times_{\Spec\boZ}X$\,'' in the sense
of~\cite[Section~7.11]{BD2} and~\cite[Section~2]{Psemten}
(but one needs to explain what this means).

 \emph{Morphisms of quasi-coherent quasi-modules} on $X$ are defined
simply as their morphisms as sheaves of bimodules over~$O_X$.
 The category $X\qqcoh$ of quasi-coherent quasi-modules on $X$ is
a Grothendieck abelian category.
 It is endowed with two forgetful functors $X\qqcoh
\rightrightarrows X\qcoh$ (the left and the right one) to
the Grothendieck abelian category $X\qcoh$ of quasi-coherent sheaves
on~$X$.
 Both the forgetful funcators are faithful, exact, and preserve
colimits.

 Any quasi-coherent sheaf on $X$, viewed as a sheaf of bimodules over
$O_X$ in which the left and right actions of $O_X$ agree, becomes
a quasi-coherent quasi-module over~$X$.
 So the category $X\qcoh$ is also naturally a full subcategory in
$X\qqcoh$.

\begin{lem} \label{quasi-coherent-quasi-module-tensor-products}
\textup{(a)} For any quasi-coherent quasi-modules $A$ and $B$ on $X$,
the tensor product of sheaves of $O_X$\+$O_X$\+bimodules
$A\ot_{O_X}B$ is a quasi-coherent quasi-module on~$X$.
 For any affine open subscheme $U\subset X$, there is a natural
isomorphism $(A\ot_{O_X}\nobreak B)(U)\simeq A(U)\ot_{O(U)}B(U)$ of
quasi-modules over~$O(U)$. \par
\textup{(b)} For any quasi-coherent quasi-module $A$ and any
quasi-coherent sheaf $M$ on $X$, the tensor product $A\ot_{O_X}M$,
endowed with the $O_X$\+module structure induced by the left
$O_X$\+module structure on $A$, is a quasi-coherent sheaf on~$X$.
 For any affine open subscheme $U\subset X$, there is a natural
isomorphism of $O(U)$\+modules $(A\ot_{O_X}M)(U)\simeq
A(U)\ot_{O(U)}M(U)$.
\end{lem}

\begin{proof}
 Part~(a) follows from
Lemmas~\ref{tensor-product-quasi-module}
and~\ref{quasi-module-localization-lemma}.
 Part~(b) can be viewed as a particular case of part~(a).
\end{proof}

 The operation of tensor product over $O_X$, described in
Lemma~\ref{quasi-coherent-quasi-module-tensor-products}(a), makes
$X\qqcoh$ an (associative, noncommutative) tensor category.
 The structure sheaf $O_X\in X\qcoh\subset X\qqcoh$ of the scheme $X$
is the unit object of this tensor category structure.
  The tensor product operation described in
Lemma~\ref{quasi-coherent-quasi-module-tensor-products}(b) makes
$X\qcoh$ is a (left) module category over the tensor category
$X\qqcoh$.

 A \emph{quasi-coherent quasi-algebra} $A$ on $X$ is a sheaf of rings
endowed with a morphism of sheaves of rings $O_X\rarrow A$ such that
the resulting sheaf of $O_X$\+$O_X$\+bimodules $A$ is
a quasi-coherent quasi-module on~$X$.
 Equivalently, one can say that a quasi-coherent quasi-algebra on $X$
is an associative, unital algebra (monoid) object in the tensor
category $X\qqcoh$.
 This means that $A$ is a quasi-coherent quasi-module on $X$ endowed
with morphisms of quasi-coherent quasi-modules $O_X\rarrow A$ and
$A\ot_{O_X}A\rarrow A$ satisfying the conventional associativity
and unitality equations.

 A \emph{quasi-coherent graded quasi-algebra} $A^*$ on $X$ is
a sheaf of graded rings $A^*=(A^n)_{n\in\Gamma}$ on $X$ endowed
with a morphism of sheaves of rings $O_X\rarrow A^0$ such that
the resulting structure of a sheaf of $O_X$\+$O_X$\+bimodules on
the sheaf of abelian groups $A^n$ makes $A^n$ a quasi-coherent
quasi-module for every $n\in\Gamma$.
 Equivalently, a quasi-coherent graded quasi-algebra is
an associative, unital $\Gamma$\+graded algebra object in
the tensor category $X\qqcoh$.
 This means that $(A^n)_{n\in\Gamma}$ is a family of quasi-coherent
quasi-modules on $X$ endowed with morphisms of quasi-coherent
quasi-modules $O_X\rarrow A^0$ and $A^n\otimes_{O_X}A^m\rarrow
A^{n+m}$ for all $n$, $m\in\Gamma$ satisfying the conventional
associativity and unitality equations.

 A \emph{quasi-coherent left module} $M$ over a quasi-coherent
quasi-algebra $A$ on $X$ is a sheaf of left modules over the sheaf
of rings $A$ whose underlying sheaf of $O_X$\+modules $M$ is
quasi-coherent.
 Equivalently, one can say that a quasi-coherent left module $M$ over
is a module object in the module category $X\qcoh$ over the algebra
object $A$ in the tensor category $X\qqcoh$.
 This means that $M$ is a quasi-coherent sheaf on $X$ endowed
with a morphism of quasi-coherent sheaves $A\ot_{O_X}M\rarrow M$
satisfying the conventional associativity and unitality equations
together with the structure morphisms $O_X\rarrow A$ and $A\ot_{O_X}A
\rarrow A$ of the quasi-coherent quasi-algebra~$A$.

 The category $A\qcoh$ of quasi-coherent left modules over
a quasi-coherent quasi-algebra $A$ on $X$ is a Grothendieck
abelian category.
 The forgetful functor $A\qcoh\rarrow X\qcoh$ is faithful, exact,
and preserves colimits.

 A \emph{quasi-coherent graded left module} $M^*$ over a quasi-coherent
graded quasi-algebra $A^*$ on $X$ is a sheaf of graded left modules
over the sheaf of graded rings $A^*$ such that the sheaf of
$O_X$\+modules $M^n$ is quasi-coherent for every $n\in\Gamma$.
 Equivalently, one can say that a quasi-coherent graded left module
$M^*$ is a graded module object in the module category $X\qcoh$ over
the graded algebra object $A^*$ in the tensor category $X\qqcoh$.
 This means that $(M^m)_{m\in\Gamma}$ is a family of quasi-coherent
sheaves on $X$ endowed with morphisms of quasi-coherent sheaves
$A^n\ot_{O_X}M^m\rarrow M^{n+m}$, \ $n$,~$m\in\Gamma$, satisfying
the associativity and unitality equations.
 Here, as in Lemma~\ref{quasi-coherent-quasi-module-tensor-products}(b),
the $O_X$\+module structure on the tensor product $A^n\ot_{O_X}M^m$ is
induced by the left $O_X$\+module structure on~$A^n$.

 Given a quasi-coherent graded quasi-algebra $A^*$, the graded category
of quasi-coherent graded left modules $A^*\cqcoh$ and the abelian
category of quasi-coherent graded left modules $A^*\sqcoh$ are defined
similarly to Section~\ref{curved-dg-modules-subsecn}.
 The graded abelian group of morphisms in $A^*\cqcoh$ from
a quasi-coherent graded module $L^*$ to a quasi-coherent graded
module $M^*$ is denoted by $\Hom_{A^*\cqcoh}^*(L^*,M^*)=
\Hom_{A^*}^*(L^*,M^*)$.
 
\subsection{Quasi-coherent CDG-quasi-algebras}
\label{quasi-cdg-subsecn}
 The definition of a \emph{quasi-coherent CDG\+algebra} over a scheme
goes back to~\cite[Section~B.1]{Pkoszul} and~\cite[Section~1.2]{EP}.
 The extension to quasi-algebras is straightforward; we spell it out
below in this section for the sake of completeness of the exposition.
 A much more detailed discussion can be found
in~\cite[Sections~6.3\+-6.4]{Pdomc}.

 First we need to complete the discussion of
Section~\ref{curved-dg-modules-subsecn} by defining \emph{morphisms
of CDG\+rings}.
 Let $\biR^\cu=(R^*,d_R,h_R)$ and $\biS^\cu=(S^*,d_S,h_S)$ be
two CDG\+rings.
 A morphism of CDG\+rings $(f,a)\:\biR^\cu\rarrow\biS^\cu$ is a pair
$(f,a)$, where $f\:R^*\rarrow S^*$ is a morphism of graded rings and
$a\in S^1$ is an element for which the following two equations
are satisfied:
\begin{enumerate}
\renewcommand{\theenumi}{\roman{enumi}}
\setcounter{enumi}{4}
\item $f(d_R(r))=d_S(f(r))+[a,f(r)]$ for all $r\in R^{|r|}$
(where $[x,y]=xy-(-1)^{|x||y|}yx$ is the commutator with signs);
\item $f(h_R)=h_S+d_S(a)+a^2$.
\end{enumerate}
 The element $a\in S^1$ is called the \emph{change-of-connection
element}.
 The \emph{composition} of two morphisms of CDG\+rings
$(R^*,d_R,h_R)\xrightarrow{(f,a)}(S^*,d_S,h_S)\xrightarrow{(g,b)}
(T^*,d_T,h_T)$ is given by the rule $(g,b)\circ(f,a)=
(g\circ f,\>b+g(a))$.

 All CDG\+ring morphisms of the form $(\id,a)\:(S^*,d',h')
\rarrow(S^*,d,h)$ are isomorphisms.
 They are called \emph{change-of-connection isomorphisms};
while morphisms of the form $(f,0)\:(R^*,d_R,h_R)\rarrow(S^*,d_S,h_S)$
are called \emph{strict morphisms}.
 Any morphism of CDG\+rings can be uniquely decomposed into a strict
morphism followed by a change-of-connection isomorphism.

 The inclusion functor from the category of DG\+rings to the category
of CDG\+rings is faithful, but \emph{not} fully faithful.
 A morphism of DG\+rings is essentially the same thing as a strict
morphism between the related CDG\+rings.
 Still two quite different DG\+rings may be isomorphic as CDG\+rings,
connected by a change-of-connection isomorphism (otherwise known
as a \emph{Maurer--Cartan twist}).

 A rule for transforming CDG\+modules under morphisms of CDG\+rings is
needed for the discussion of sheaves of CDG\+modules over sheaves of
CDG\+rings, which is our aim in this section.
 Let $(f,a)\:\biR^\cu\rarrow\biS^\cu$ be a morphism of CDG\+rings, and
let $\biL^\cu=(L^*,d_L)$ and $\biM^\cu=(M^*,d_M)$ be left CDG\+modules
over $\biR^\cu$ and $\biS^\cu$, respectively.
 Then a \emph{map of CDG\+modules $g\:\biL^\cu\rarrow\biM^\cu$
compatible with the given morphism of CDG\+rings~$(f,a)$} is a morphism
of graded modules $L^*\rarrow M^*$ compatible with the morphism~$f$
of graded rings and satisfying the equation
\begin{enumerate}
\renewcommand{\theenumi}{\roman{enumi}}
\setcounter{enumi}{6}
\item $g(d_L(x))=d_M(g(x))+ag(x)$ for all $x\in L^*$.
\end{enumerate}

 Similarly, let $\biN^\cu=(N^*,d_N)$ and $\biQ^\cu=(Q^*,d_Q)$ be
right CDG\+modules over $\biR^\cu$ and $\biS^\cu$, respectively.
 Then a map of CDG\+modules $g\:\biN^\cu\rarrow\biQ^\cu$ compatible
with the given morphism of CDG\+rings~$(f,a)$ is a morphism of graded
modules $N^*\rarrow Q^*$ compatible with the morphism~$f$ of graded
rings and satisfying the equation
\begin{enumerate}
\renewcommand{\theenumi}{\roman{enumi}}
\setcounter{enumi}{7}
\item $g(d_N(y))=d_Q(g(y))-(-1)^{|y|}g(y)a$ for all $y\in N^{|y|}$.
\end{enumerate}

 Let $X$ be a scheme.
 A \emph{quasi-coherent CDG\+quasi-algebra} $\biB^\cu$ over $X$ is
the following set of data:
\begin{itemize}
\item a quasi-coherent graded quasi-algebra $B^*$ over~$X$ is given;
\item for each affine open subscheme $U\subset X$, an odd derivation
$d_U\:B^*(U)\rarrow B^*(U)$ of degree~$1$ and a curvature element
$h_U\in B^2(U)$ are given;
\item for each pair of affine open subschemes $V\subset U\subset X$,
a change-of-connection element $a_{VU}\in B^1(V)$ is given;
\end{itemize}
such that the following conditions are satisfied:
\begin{enumerate}
\renewcommand{\theenumi}{\roman{enumi}}
\setcounter{enumi}{8}
\item for each affine open subscheme $U\subset X$, the triple
$(B^*(U),d_U,h_U)$ is a CDG\+ring;
\item for each pair of affine open subschemes $V\subset U\subset X$,
the pair $(\rho_{VU},a_{VU})$ is a morphism of CDG\+rings
$(B^*(U),d_U,h_U)\rarrow(B^*(V),d_V,h_V)$, where $\rho_{VU}\:B^*(U)
\rarrow B^*(V)$ denotes the restriction morphism in the sheaf of
graded rings~$B^*$;
\item for each triple of affine open subschemes $W\subset V\subset U
\subset X$, one has $(\rho_{WV},a_{W,V})\circ(\rho_{VU},a_{VU})=
(\rho_{WU},a_{WU})\:(B^*(U),d_U,h_U)\rarrow(B^*(W),d_W,\allowbreak
h_W)$; in other words, the equation $a_{WU}=a_{WV}+\rho_{WV}(a_{VU})$
holds in~$B^1(W)$.
\end{enumerate}

\begin{rem}
 Let $X$ be a scheme endowed with an open covering $X=\bigcup_\alpha
Y_\alpha$, and let $B^*$ be a quasi-coherent graded quasi-algebra
over~$X$.
 Suppose that a CDG\+ring structure $(B^*(U),d_U,h_U)$ on the graded
ring $B^*(U)$ is given \emph{for all affine open subschemes $U\subset X$
for which there exists~$\alpha$ such that $U\subset Y_\alpha$}, and
change-of-connection elements $a_{VU}\in B^1(V)$ are chosen for all
$V\subset U$ with $U\subset Y_\alpha$ for some~$\alpha$ in such a way
that the conditions~(x\+-xi) are satisfied whenever $U\subset Y_\alpha$.
 Then this set of data can be (uniquely in a natural sense) extended
to a quasi-coherent CDG\+quasi-algebra structure $\biB^\cu$ on~$B^*$,
in the sense of the definition above.

 Indeed, it suffices to consider the case when $X=\bigcup_\alpha
Y_\alpha$ is an affine open covering of an affine scheme~$X$.
 In this case, the formula $c_{\alpha\beta}=a_{U_\alpha\cap U_\beta,
U_\alpha}-a_{U_\alpha\cap U_\beta,U_\beta}$ defines a \v Cech
$1$\+cocycle~$c$ with the coefficients in the sheaf of abelian groups
$B^1$ on the Zariski topology of the affine scheme $X$ with the chosen
affine open covering~$\{Y_\alpha\}$.
 The point is that, since $B^1$ is a quasi-coherent sheaf (e.~g.,
with the left action of~$O_X$, or alternatively with the right action
of~$O_X$), all such cocycles are coboundaries (in the same \v Cech
complex for the fixed covering $X=\bigcup_\alpha Y_\alpha$);
see, e.~g., \cite[Tag~01X8]{SP}.
 This allows to solve the problem of defining a CDG\+ring structure
on the graded ring $B^*(X)$ together with a compatible collection
of change-of-connection elements $a_{Y_\alpha,X}\in B^1(Y_\alpha)$
for all~$\alpha$.
\end{rem}

 Let $\biB^\cu$ be a quasi-coherent CDG\+quasi-algebra over~$X$.
 A \emph{quasi-coherent left CDG\+module} $\biM^\cu$ over $\biB^\cu$
is the following set of data:
\begin{itemize}
\item a quasi-coherent graded left module $M^*$ over
the quasi-coherent graded quasi-algebra $B^*$ is given;
\item for each affine open subscheme $U\subset X$, an odd derivation
$d_{U,M}\:M^*(U)\rarrow M^*(U)$ of degree~$1$ on the graded module
$M^*(U)$ compatible with the odd derivation $d_U$ on the graded ring 
$B^*(U)$ is given;
\end{itemize}
such that the following condition is satisfied:
\begin{enumerate}
\renewcommand{\theenumi}{\roman{enumi}}
\setcounter{enumi}{11}
\item for each pair of affine open subschemes $V\subset U\subset X$,
the restriction map $M^*(V)\rarrow M^*(U)$ in the sheaf of graded
modules $M^*$ is a map of CDG\+modules $(M^*(U),d_{U,M})\rarrow
(M^*(V),d_{V,M})$ compatible with the morphism of CDG\+rings
$(\rho_{VU},a_{VU})\:(B^*(U),d_U,h_U)\rarrow(B^*(V),d_V,h_V)$.
\end{enumerate}

 Let $\biL^\cu$ and $\biM^\cu$ be two quasi-coherent left CDG\+modules
over a quasi-coherent CDG\+quasi-algebra $\biB^\cu$ over $X$, and
let $L^*$ and $M^*$ be their underlying quasi-coherent graded modules
over the quasi-coherent graded quasi-algebra~$B^*$.
 Then the graded abelian group $\Hom_{B^*}^*(L^*,M^*)$ of homogeneous
maps of quasi-coherent graded modules $L^*\rarrow M^*$ is endowed
with a natural differential~$d$ of degree~$1$ defined as follows.

 Let $f\in\Hom_{B^*}^{|f|}(L^*,M^*)$ be a homogeneous map
of degree~$|f|$ between the two quasi-coherent graded modules.
 In order to define the homogeneous map
$d(f)\in\Hom_{B^*}^{|f|+1}(L^*,M^*)$, consider an affine open
subscheme $U\subset X$, and denote by $f_U\:L^*(U)\rarrow M^*(U)$
the action of~$f$ on the sections over~$U$.
 Following Section~\ref{curved-dg-modules-subsecn}, there is
a natural differential~$d$ on the graded abelian group
$\Hom_{B^*(U)}(L^*(U),M^*(U))$, defined in terms of the odd
derivations $d_{U,L}\:L^*(U)\rarrow L^*(U)$ and
$d_{U,M}\:M^*(U)\rarrow M^*(U)$.
 We put $d(f)_U=d(f_U)$.
 It is straightforward to check that the resulting homogeneous
maps of graded modules $d(f)_U\in\Hom_{B^*(U)}^{|f|+1}(L^*(U),M^*(U))$
form commutative square diagrams with the restriction maps
$L^*(U)\rarrow L^*(V)$ and $M^*(U)\rarrow M^*(V)$ for any pair
of affine open subschemes $V\subset U\subset X$.

 Hence we obtain a complex of abelian groups, denoted by
$\Hom_{\biB^\cu}^\bu(\biL^\cu,\biM^\cu)$.
 In the \emph{DG\+category $\biB^\cu\bqcoh$ of quasi-coherent left
CDG\+modules over $\biB^\cu$}, the quasi-coherent left CDG\+modules
over $\biB^\cu$ are the objects, and
$\Hom^\bu_{\biB^\cu\bqcoh}(\biL^\cu,\biM^\cu)
=\Hom_{\biB^\cu}^\bu(\biL^\cu,\biM^\cu)$ is the complex of morphisms
from $\biL^\cu$ to~$\biM^\cu$.

 The DG\+category $\biB^\cu\bqcoh$ is additive with shifts and twists,
and with infinite coproducts.
 The category $\sZ^0(\biB^\cu\bqcoh)$ of closed morphisms of degree~$0$
in $\biB^\cu\bqcoh$ is a Grothendieck abelian category.
 It is called the \emph{abelian category of quasi-coherent left
CDG\+modules over~$\biB^\cu$}.

\subsection{Factorization categories} \label{factorizations-subsecn}
 \emph{Factorizations}~\cite[Section~6 and Appendix~A]{Ef},
\cite{BDFIK} can be thought of as a category-theoretic generalization
of CDG\+modules over a CDG\+ring $\biR^\cu=(R^*,d,h)$ such that $R^n=0$
for $n$~odd (i.~e., $n\in\Gamma\setminus 2\Gamma$), the multiplication
maps $R^n\ot_{R^0}R^m\rarrow R^{n+m}$ are isomorphisms for all
$n$, $m\in2\Gamma$, and $d=0$, while $h\in R^2$ is a central
element in~$R^*$.
 In the exposition below in this paper, we take up a slightly
more general notational system for matrix factorizations than
in~\cite{BDFIK}, following an approach which has been used in
the algebro-geometric literature on matrix factorizations over stacks;
see, e.~g., \cite[Section~2.2]{Tod}.

 Let $\sA$ be a preadditive category and $\Lambda\:\sA\rarrow\sA$ be
an autoequivalence.
 We will consider two cases concerning the grading group~$\Gamma$:
either $\Gamma=\boZ$, or $\Lambda$ is an involutive autoequivalence
and $\Gamma=\boZ/2$.
 In the latter case, a natural isomorphism $\theta\:\Id_\sA\rarrow
\Lambda^2$ of functors $\sA\rarrow\sA$ is presumed to be given
such that the equation $\theta_{\Lambda(X)}=\Lambda(\theta_X)$
is satisfied for all objects $X\in\sA$.

 More generally, one could consider the $2p$\+periodic case of
$\Gamma=\boZ/2p$, where $p$~is a positive integer.
 Then a natural isomorphism $\theta\:\Id_\sA\overset\simeq\rarrow
\Lambda^{2p}$ of functors $\sA\rarrow\sA$ is needed satisfying
the equation $\theta_{\Lambda(X)}=\Lambda(\theta_X)$ for all objects
$X\in\sA$.
 In all these cases, the group $\Gamma$ acts on the category $\sA$
by powers of the functor~$\Lambda$: for every $n\in\Gamma$ there is
the functor $\Lambda^n\:\sA\rarrow\sA$, and the natural isomorphisms
$\Id_\sA\overset\simeq\rarrow\Lambda^0$ and $\Lambda^n\circ\Lambda^m
\overset\simeq\rarrow\Lambda^{n+m}$ for all $n$, $m\in\Gamma$ are
defined and satisfy all the natural compatibilities.

 A \emph{potential}~$w$ for $\Lambda$ is a natural transformation
$w\:\Id_\sA\rarrow\Lambda^2$ satifying the equation $w_{\Lambda(X)}=
\Lambda(w_X)$ for all $X\in\sA$.
 Our aim is to construct the DG\+category $\bF(\sA,\Lambda,w)$
of \emph{factorizations} of~$w$.
 On the way to factorizations, we start with constructing the graded
category $\cP(\sA,\Lambda)$ of \emph{$\Lambda$\+periodic objects}
in~$\sA$. 

 By definition, a $\Lambda$\+periodic object $X^\ci$ in $\sA$ is
a collection of objects $(X^i\in\sA)_{i\in\Gamma}$ endowed with
isomorphisms $\lambda_X^{i,j}\:\Lambda^{i-j}(X^j)\overset\simeq
\rarrow X^i$ defined for all $i$, $j\in\Gamma$.
 The equation $\lambda_X^{i,j}\circ\Lambda^{i-j}(\lambda^{j,k}_X)
=\lambda^{i,k}_X$ needs to be satisfied for all $i$, $j$, $k\in\Gamma$.
 The collection of all objects $(X^i\in\sA)_{i\in\Gamma}$ with
the periodicity isomorphisms $\lambda_X^{i,j}$ forgotten defines
the underlying graded object $X^*\in\cG(\sA)$ of
a $\Lambda$\+periodic object $X^\ci\in\cP(\sA,\Lambda)$.
 The shift $X^\ci[n]$ of a $\Lambda$\+periodic object $X^\ci$
by a degree $n\in\Gamma$ is a $\Lambda$\+periodic object defined
by the rules $X^\ci[n]^i=X^{n+i}$ and $\lambda_{X[n]}^{i,j}=
\lambda_X^{n+i,n+j}$ for all $n$, $i$, $j\in\Gamma$.

 Let us first define the preadditive category $\sP(\sA,\Lambda)$ of
$\Lambda$\+periodic objects in~$\sA$.
 In both $\cP(\sA,\Lambda)$ and $\sP(\sA,\Lambda)$, the objects are
the $\Lambda$\+periodic objects in~$\sA$.
 For any two objects $X^\ci$ and $Y^\ci\in\sP(\sA,\Lambda)$,
the abelian group $\Hom_{\sP(\sA,\Lambda)}(X^\ci,Y^\ci)$ is defined
as the subgroup in the group $\Hom_{\sG(\sA)}(X^*,Y^*)$ consisting of
all the morphisms $(f_i\:X^i\to Y^i)_{i\in\Gamma}$ satisfying
the equations $\lambda^{i,j}_Y\circ\Lambda^{i-j}(f_j)=
f_i\circ\lambda^{i,j}_X$ for all $i$, $j\in\Gamma$.
 The category $\sP(\sA,\Lambda)$ is naturally equivalent to
the category $\sA$, with the equivalence provided by the functor
taking a $\Lambda$\+periodic object $X^\ci\in\sP(\sA,\Lambda)$
to the object $X^0\in\sA$.

 The graded abelian group $\Hom_{\cP(\sA,\Lambda)}^*(X^\ci,Y^\ci)$
of morphisms in the graded category $\cP(\sA,\Lambda)$ is a homogeneous
subgroup of the graded abelian group $\Hom_{\cG(\sA)}(X^*,Y^*)$.
 Similarly to Section~\ref{complexes-subsecn}, it is defined by
the rule $\Hom^n_{\cP(\sA,\Lambda)}(X^\ci,Y^\ci)=
\Hom_{\sP(\sA,\Lambda)}(X^\ci,\allowbreak Y^\ci[n])$
for all $n\in\Gamma$.
 The composition of morphisms in the categories $\sP(\sA,\Lambda)$
and $\cP(\sA,\Lambda)$ agrees with the one in the categories
$\sG(\sA)$ and $\cG(\sA)$, respectively.
 \emph{Unlike} the preadditive category $\sP(\sA,\Lambda)$, the graded
category $\cP(\sA,\Lambda)$ cannot be recovered from the category $\sA$
alone; it depends on the autoequivalence $\Lambda$ in an essential way.

 The categories $\sP(\sA,\Lambda)$ and $\cP(\sA,\Lambda)$ are additive
whenever the preadditive category $\sA$ is.
 Assuming $\sA$ is additive, the categories $\sP(\sA,\Lambda)$ and
$\cP(\sA,\Lambda)$ are idempotent-complete whenever the category
$\sA$~is.
 The categories $\sP(\sA,\Lambda)$ and $\cP(\sA,\Lambda)$ have infinite
coproducts or products whenever the category $\sA$ has.

 A \emph{factorization} $\biM^\cu=(M^\ci,d_M)$ of the potential~$w$ is
a $\Lambda$\+periodic object endowed with a homogeneous endomorphism
of degree~$1$, i.~e., $d_M\in\Hom^1_{\cP(\sA,\Lambda)}(M^\ci,M^\ci)$,
such that, for every $n\in\Gamma$, the composition
$(\lambda_M^{n+2,n})^{-1}d_{M,n+1}d_{M,n}\:M^n\rarrow M^{n+2}\overset
\simeq\rarrow\Lambda^2(M^n)$ is equal to the morphism
$w_{M^n}\:M^n\rarrow\Lambda^2(M^n)$.
 For any two factorizations $\biL^\cu=(L^\ci,d_L)$ and
$\biM^\cu=(M^\ci,d_M)$ of the same potential~$w$, the graded abelian
group $\Hom^*_{\cP(\sA,\Lambda)}(L^\ci,M^\ci)$ is endowed with
a natural differential~$d$ of degree~$1$ defined by the usual rule
$d(f)=d_M\circ f-(-1)^{|f|}f\circ d_L$.
 Using the assumption that $w$~is a natural transformation, one
can check that $d^2=0$.
 Hence we obtain a complex of abelian groups, denoted by
$\Hom_{\bF(\sA,\Lambda,w)}^\bu(\biL^\cu,\biM^\cu)$.

 In the \emph{DG\+category\/ $\bF(\sA,\Lambda,w)$ of factorizations
of~$w$}, the factorizations are the objects, and
$\Hom_{\bF(\sA,\Lambda,w)}^\bu(\biL^\cu,\biM^\cu)$ is the complex
of morphisms from $\biL^\cu$ to~$\biM^\cu$.
 We denote by $\sF(\sA,\Lambda,w)=\sZ^0(\bF(\sA,\Lambda,w))$
the additive category of factorizations and closed morphisms
of degree~$0$ between them, and by $\sK(\sA,\Lambda,w)=
\sH^0(\bF(\sA,\Lambda,w))$ the homotopy category of factorizations
of~$w$.

 The DG\+category $\bF(\sA,\Lambda,w)$ always has shifts and twists.
 When $\sA$ is an additive category, the DG\+category
$\bF(\sA,\Lambda,w)$ is additive as well; hence it also has cones,
and the homotopy category $\sK(\sA,\Lambda,w)$ is triangulated in
this case.
 When an additive category $\sA$ is idempotent-complete, so is
the DG\+category $\bF(\sA,\Lambda,w)$.
 The DG\+category $\bF(\sA,\Lambda,w)$ has infinite products or
coproducts whenever the category $\sA$ has.
 When $\sA$ is an abelian category, so is $\sF(\sA,\Lambda,w)$.

\begin{rem} \label{usual-factorizations-remark}
 The more conventional notion of factorizations in the sense
of~\cite{BDFIK} can be recovered as a particular case of the definition
above in the following way.

 Let $\sB$ be a preadditive category and $\Delta\:\sB\rarrow\sB$
be an autoequivalence.
 Let $\Gamma=\boZ$ or $\boZ/2p$ be a grading group as above.
 In the latter case, assume that a natural isomorphism
$\theta\:\Id_\sB\rarrow\Delta^p$ is given and the equation
$\theta_{\Delta(A)}=\Delta(\theta_A)$ is satisfied for all $A\in\sB$.
 Then the subgroup $2\Gamma\subset\Gamma$ acts on the category $\sB$
by powers of the functor~$\Delta$: for every $n\in 2\Gamma$ one
can define the functor denoted formally by $\Delta^{n/2}\:\sB
\rarrow\sB$, together with the natural isomorphisms of unitality
and multiplicativity similar to the ones above, $\Id_\sB\overset\simeq
\rarrow\Delta^{0/2}$ and $\Delta^{n/2}\circ\Delta^{m/2}\overset\simeq
\rarrow\Delta^{(n+m)/2}$ for all $n$, $m\in2\Gamma$.
 A \emph{potential}~$v$ for $\Delta$ is a natural transformation
$v\:\Id_\sB\rarrow\Delta$ satifying the equation $v_{\Delta(A)}=
\Delta(v_A)$ for all $A\in\sB$.

 A \emph{$2$\+$\Delta$\+periodic object} $A^\ci$ in $\sB$ is
a collection of objects $(A^i\in\sB)_{i\in\Gamma}$ endowed with
isomorphisms $\delta_A^{i,j}\:\Delta^{(i-j)/2}(A^j)\overset\simeq
\rarrow A^i$ defined for all $i$, $j\in\Gamma$ with $i-j\in2\Gamma$.
 The equation $\delta_A^{i,j}\circ\Delta^{(i-j)/2}(\delta^{j,k}_A)
=\delta^{i,k}_A$ has to be satisfied for all $i$, $j$, $k\in\Gamma$
such that $i-j\in2\Gamma\ni j-k$.
 The shift $A^\ci[n]$ of a $2$\+$\Delta$\+periodic object $A^\ci$
by a degree $n\in\Gamma$ is a $2$\+$\Delta$\+periodic object defined
by the rules $A^\ci[n]^i=A^{n+i}$ and $\delta_{A[n]}^{i,j}=
\delta_A^{n+i,n+j}$ for all $n$, $i$, $j\in\Gamma$.
 The preadditive category $\sP'(\sB,\Delta)$ and the graded category
$\cP'(\sB,\Delta)$ of $2$\+$\Delta$\+periodic objects in $\sB$ are
defined similarly to the constructions above.

 In particular, the collection of all objects
$(A^i\in\sB)_{i\in\Gamma}$ with the periodicity isomorphisms
$\delta_A^{i,j}$ forgotten defines the underlying graded object
$A^*\in\cG(\sB)$ of a $2$\+$\Delta$\+periodic object
$A^\ci\in\cP'(\sB,\Delta)$.
 For any two objects $A^\ci$ and $B^\ci\in\cP'(\sB,\Delta)$,
the abelian group $\Hom_{\sP'(\sB,\Delta)}(A^\ci,B^\ci)$ is defined
as the subgroup in the group $\Hom_{\sG(\sB)}(A^*,B^*)$ consisting of
all the morphisms $(f_i\:A^i\to B^i)_{i\in\Gamma}$ satisfying
the equations $\delta^{i,j}_B\circ\Delta^{(i-j)/2}(f_j)=
f_i\circ\delta^{i,j}_A$ for all $i$, $j\in\Gamma$ with $i-j\in2\Gamma$.
 The category $\sP'(\sB,\Delta)$ is equivalent to the Cartesian
product $\sB^{\Gamma/2\Gamma}=\sB\times\sB$ of two copies of
the preadditive category $\sB$ indexed by representatives of
the cosets of the group $\Gamma$ modulo~$2\Gamma$.

 The graded abelian group $\Hom_{\cP'(\sB,\Delta)}^*(A^\ci,B^\ci)$
of morphisms in the graded category $\cP'(\sB,\Delta)$ is a homogeneous
subgroup of the graded abelian group $\Hom_{\cG(\sB)}(A^*,B^*)$
defined by the rule $\Hom^n_{\cP'(\sB,\Delta)}(A^\ci,B^\ci)=
\Hom_{\sP'(\sB,\Delta)}(A^\ci,B^\ci[n])$ for all $n\in\Gamma$.
 Unlike the preadditive category $\sP'(\sB,\Delta)$, the graded
category $\cP'(\sB,\Delta)$ cannot be recovered from the category $\sB$
alone; it depends on the autoequivalence~$\Delta$.

 A \emph{factorization} $\biN^\cu=(N^\ci,d_N)$ of the potential~$v$
(in the sense of~\cite{BDFIK}) is a $2$\+$\Delta$\+periodic object
endowed with a homogeneous endomorphism of degree~$1$, i.~e.,
$d_N\in\Hom^1_{\cP'(\sB,\Delta)}(N^\ci,N^\ci)$, such that, for
every $n\in\Gamma$, the composition $(\delta_N^{n+2,n})^{-1}d_{N,n+1}
d_{N,n}\:N^n\rarrow N^{n+2}\overset\simeq\rarrow\Delta(N^n)$
is equal to the morphism $v_{N^n}\:N^n\rarrow\Delta(N^n)$.
 The DG\+category $\bF'(\sB,\Delta,v)$ of factorizations of~$v$ in
the sense of~\cite{BDFIK} is defined similarly to the construction
above in the main body of this section. {\emergencystretch=1.5em\par}

 Now consider the Cartesian square $\sA=\sB\times\sB$ of
the preadditive category~$\sB$.
 Define the autoequivalence $\Lambda\:\sA\rarrow\sA$ by the rule
$\Lambda(B^0,B^1)=(B^1,\Delta(B^0))$ for all $B^0$, $B^1\in\sB$.
 Then one has $\Lambda^{2p}\simeq\Id_\sA$ whenever $\Delta^p\simeq
\Id_\sB$.
 Define the potential $w\:\Id_\sA\rarrow\Lambda^2$ by the rule
$w_{(B^0,B^1)}=(v_{B^0},v_{B^1})\:(B^0,B^1)\rarrow\Lambda^2(B^0,B^1)
=(\Delta(B^0),\Delta(B^1))$.

 Then the additive category $\sP'(\sB,\Delta)$ is naturally equivalent
to $\sP(\sA,\Lambda)$, the graded category $\cP'(\sB,\Delta)$ is
naturally equivalent to $\cP(\sA,\Lambda)$, and the DG\+category
$\bF'(\sB,\Delta,v)$ is naturally equivalent to the DG\+category
$\bF(\sA,\Lambda,w)$.
 These equivalences assign to a $2$\+$\Delta$\+periodic object
$B^\circ\in\cP'(\sB,\Delta)$ the $\Lambda$\+periodic object
$X^\circ\in\cP(\sA,\Lambda)$ with the components $X^i=(B^i,B^{i+1})$
for all $i\in\Gamma$.
\end{rem}

\Section{The Almost Involution on DG-Categories} \label{involution-secn}

 The construction assigning the DG\+category $\bA^\bec$ to
a DG\+category $\bA$ is presented below in
Section~\ref{main-construction-subsecn}.
 In the important example of the DG\+category of CDG\+modules over
a CDG\+ring, this construction is closely related to the construction
assigning the acyclic DG\+ring $\widehat\biR^\bu$ to a CDG\+ring
$\biR^\cu$.
 Before proceeding to introduce the main category-theoretic 
construction, we start with Section~\ref{cdg-revisited-subsecn}
discussing the related construction for curved DG\+rings.

 The DG\+category $\bA^\bec$ comes together with a ladder of adjunctions
connecting the additive categories $\sZ^0(\bA)$ and $\sZ^0(\bA^\bec)$.
 We construct these adjoint functors in the case of CDG\+modules over
CDG\+rings in Proposition~\ref{G-plus-minus-prop}.

 Furthermore, the counterexamples presented in
Examples~\ref{nonexistence-of-cdg-structure-examples}\+-%
\ref{nonexistence-of-derivation-examples} in
Section~\ref{cdg-revisited-subsecn} demonstrate that the full
subcategory $(\biR^\cu\bmodl)^0\subset R^*\smodl$ misses many objects
and is not well-behaved in general.
 This illustrates the need for the construction of the DG\+category
$(\biR^\cu\bmodl)^\bec$ as a way to eventually recover the whole
abelian category $R^*\smodl\simeq\sZ^0((\biR^\cu\bmodl)^\bec)$,
as per Example~\ref{CDG-ring-bec-example}
in Section~\ref{A-bec-examples-subsecn}.

\subsection{CDG-modules revisited} \label{cdg-revisited-subsecn}
 Let $\biR^\cu=(R^*,d,h)$ be a CDG\+ring.
 The following construction plays an important role.
 One adjoins an element $\delta$ of degree~$1$ to $R^*$ and imposes
the relations
\begin{itemize}
\item $\delta r-(-1)^{|r|}r\delta=d(r)$ for all $r\in R^{|r|}$;
\item $\delta^2=h$.
\end{itemize}
 The resulting graded ring, spanned by $R^*$ and~$\delta$ with
the relations above, is denoted by $R^*[\delta]$.
 Explicitly, the elements of degree~$n$ in $R^*[\delta]$ are formal
linear combinations $r+s\delta$ with $r\in R^n$ and $s\in R^{n-1}$.
 The addition in $R^*[\delta]$ is defined in the obvious way,
and the multiplication is defined in the way dictated by
the relations above, namely
$$
 (r+s\delta)(u+v\delta)=(ru+sd(u)+(-1)^{|v|}svh)
 +(rv+(-1)^{|u|}su+sd(v))\delta.
$$

 The graded ring $R^*[\delta]$ is endowed with an odd derivation
$\d=\d/\d\delta$ of degree~$-1$, defined by the rules $\d(r)=0$
for $r\in R^*\subset R^*[\delta]$ and $\d(\delta)=1$.
 Explicitly, $\d(r+s\delta)=(-1)^{|s|}s$.
 We will denote by $\widehat\biR^*$ the graded ring $R^*[\delta]$
with the odd derivation~$\d$ and the changed sign of the grading:
so $\widehat\biR^n=R^*[\delta]^{-n}$.

 One has $\d^2=0$, so $\widehat\biR^\bu=(\widehat\biR^*,\d)$ is
a DG\+ring.
 Moreover, the cohomology ring $H^*_\d(\widehat\biR^\bu)$ of
the DG\+ring $\widehat\biR^\bu$ vanishes (since the unit $1$ belongs
to the image of~$\d$).
 Both the kernel and the image of the odd derivation~$\d$ on
$R^*[\delta]$ coincide with the subring $R^*\subset R^*[\delta]$.
 So the map~$\d$ decomposes as $R^*[\delta]\twoheadrightarrow
R^*\rightarrowtail R^*[\delta]$.

 We denote the resulting surjective map by $\rho\:R^*[\delta]
\twoheadrightarrow R^*$.
 So $\rho$~is a surjective homogeneous map of degree~$-1$
respecting the graded $R^*$\+$R^*$\+bimodule structures.
 (Notice that there a sign rule involved in the notion of
a homogeneous map of graded left modules or the shift of a graded
left module; see Section~\ref{curved-dg-modules-subsecn}.)

 The correspondence assigning the DG\+ring $\widehat\biR^\bu$ to
a CDG\+ring $\biR^\cu=(R^*,d,h)$ is an equivalence of categories
between the category of CDG\+rings and the category of DG\+rings with
vanishing cohomology~\cite[Section~4.2]{Prel}.

 We recall the notation $R^*\smodl$ for the abelian category of
graded $R^*$\+modules (and homogeneous morphisms of degree~$0$
between them), as well as the notation $\sZ^0(\biR^\cu\bmodl)$ for
the abelian category of CDG\+modules over $\biR^\cu$ (and closed
morphisms of degree~$0$ between them); see
Sections~\ref{dg-categories-subsecn}
and~\ref{curved-dg-modules-subsecn}.
 There is an obvious exact, faithful forgetful functor
$\sZ^0(\biR^\cu\bmodl)\rarrow R^*\smodl$ assigning to a CDG\+module
$\biM^\cu=(M^*,d_M)$ its underlying graded module~$M^*$.
 We are interested in the left and right adjoint functors to
the forgetful functor.

 In the case when $\biR^\cu=\biR^\bu$ is a DG\+ring, the assertions
of the following proposition can be found in~\cite[proof of
Lemma~2.2]{Kel1}.

\begin{prop} \label{G-plus-minus-prop}
\textup{(a)} The forgetful functor\/ $\sZ^0(\biR^\cu\bmodl)\rarrow
R^*\smodl$ has an exact, faithful left adjoint functor
$G^+\:R^*\smodl\rarrow\sZ^0(\biR^\cu\bmodl)$ and an exact, faithful
right adjoint functor $G^-\:R^*\smodl\rarrow\sZ^0(\biR^\cu\bmodl)$. \par
\textup{(b)} For any graded left $R^*$\+module $M^*$, there are
natural short exact sequences of graded left $R^*$\+modules
\begin{align*}
 0\rarrow M^*&\rarrow  G^+(M^*)\rarrow M^*[-1]\rarrow0, \\
 0\rarrow M^*[1]&\rarrow G^-(M^*)\rarrow M^*\rarrow0.
\end{align*} \par
\textup{(c)} The functors $G^+$ and $G^-$ only differ from each
other by a shift: for any graded left $R^*$\+module $M^*$,
there is a natural closed isomorphism of degree\/~$0$ between
the two CDG\+modules
$$
 G^-(M^*)\simeq G^+(M^*)[1]
$$
transforming the two short exact sequences in~(b) into each other.
\end{prop}

\begin{proof}
 In order to describe the two adjoint functors, observe that
the abelian category $\sZ^0(\biR^\cu\bmodl)$ is naturally equivalent
(in fact, isomorphic) to the abelian category of graded modules
over the graded ring~$R^*[\delta]$, that is
$\sZ^0(\biR^\cu\bmodl)\simeq R^*[\delta]\smodl$.
 The equivalence assigns to a CDG\+module $(M^*,d_M)$ over
$(R^*,d,h)$ the graded module $M^*$ over $R^*[\delta]$ with
the action of $R^*$ in $M^*$ unchanged and the action of~$\delta$
in $M^*$ given by the obvious rule $\delta x=d_M(x)$ for
all $x\in M^*$.
 Taking this equivalence of abelian categories into account,
the forgetful functor $\sZ^0(\biR^\cu\bmodl)\rarrow R^*\smodl$ is
interpreted as the functor of restriction of scalars
$R^*[\delta]\smodl\rarrow R^*\smodl$ corresponding to
the inclusion morphism of graded rings $R^*\rarrow R^*[\delta]$.

 Now it is clear how to construct the adjoint functors.
 The functor $G^+\:R^*\smodl\rarrow\sZ^0(\biR^\cu\bmodl)$ left adjoint
to the forgetful functor takes a graded $R^*$\+module $M^*$ to
the CDG\+module corresponding to the graded $R^*[\delta]$\+module
$R^*[\delta]\ot_{R^*}M^*$.
 The functor $G^-\:R^*\smodl\rarrow\sZ^0(\biR^\cu\bmodl)$ right adjoint
to the forgetful functor takes a graded $R^*$\+module $M^*$ to
the CDG\+module corresponding to the graded $R^*[\delta]$\+module
$\Hom_{R^*}^*(R^*[\delta],M^*)$.
 A more explicit description of the functors $G^+$ and $G^-$ can be
found in~\cite[proof of Theorem~3.6]{Pkoszul}.

 There is a natural short exact sequence of graded
$R^*$\+$R^*$\+bimodules
\begin{equation} \label{Rdelta-sequence}
 0\lrarrow R^*\lrarrow R^*[\delta]\overset\rho\lrarrow R^*[-1]
 \lrarrow0.
\end{equation}
 All the terms of~\eqref{Rdelta-sequence} are finitely generated
projective (in fact, free) both as graded left $R^*$\+modules
and as graded right $R^*$\+modules.
 All the assertions of~(a\+-b) follow from that.
 In particular, the short exact sequences in~(b) are induced
by~\eqref{Rdelta-sequence}.
 The functors $G^+$ and $G^-$ are faithful, because any graded
left $R^*$\+module $M$ is naturally a submodule in $G^+(M)$ and
a quotient module of $G^-(M)$.

 The natural isomorphism in~(c) is constructed as the composition
$$
 \Hom_{R^*}(R^*[\delta],M^*)\simeq\Hom_{R^*}^*(R^*[\delta],R^*)
 \ot_{R^*}M^*\simeq R^*[\delta][1]\ot_{R^*}M^*,
$$
where the leftmost isomorphism holds because $R^*[\delta]$ is
a finitely generated projective graded left $R^*$\+module, while
the rightmost isomorphism is induced by the following natural
isomorphism of graded $R^*$\+$R^*$\+modules (in fact, of graded
$R^*[\delta]$\+$R^*$\+bimodules)
$\Hom_{R^*}^*(R^*[\delta],R^*)\simeq R^*[\delta][1]$.
 Consider the composition
\begin{equation} \label{Rdelta-pairing}
 R^*[\delta]\ot_{R^*}R^*[\delta]\lrarrow
 R^*[\delta]\overset\rho\lrarrow R^*[-1]
\end{equation}
of the multiplication map $R^*[\delta]\ot_{R^*}R^*[\delta]\rarrow
R^*[\delta]$ and the surjective map $\rho\:R^*[\delta]\rarrow R^*[-1]$.
 One easily checks that the $R^*$\+$R^*$\+bimodule map
$R^*[\delta]\ot_{R^*}R^*[\delta]\rarrow R^*[-1]$
\,\eqref{Rdelta-pairing} is a ``perfect pairing'', in the sense that
both the induced $R^*$\+$R^*$\+bimodule morphisms $R^*[\delta]\rarrow
\Hom_{R^*}^*(R^*[\delta],R^*)[-1]$ and $R^*[\delta]\rarrow
\Hom_{R^*{}^\rop}^*(R^*[\delta],R^*)[-1]$ are isomorphisms.
 Here $\Hom_{R^*}^*({-},{-})$ is the graded abelian group of morphisms
of graded left $R^*$\+modules, while $\Hom_{R^*{}^\rop}^*$ denotes
the graded abelian group of morphisms of graded right $R^*$\+modules.
\end{proof}

 Our next aim is to illustrate the fact, mentioned in
Section~\ref{curved-dg-modules-subsecn}, that not every graded
$R^*$\+module admits a structure of CDG\+module over $(R^*,d,h)$.
 For this purpose, we will present and discuss two classes
of examples.

\begin{exs} \label{nonexistence-of-cdg-structure-examples}
 Take $\Gamma=\boZ$, and let $R^*=k[x,x^{-1}]$ be the graded ring of
Laurent polynomials in one variable~$x$ of cohomological degree
$\deg x=2$ over a field~$k$.
 Endow $R^*$ with the zero differential $d=0$ and the curvature
element $h=x$.
 Then $\biR^\cu=(R^*,d,h)$ is a CDG\+ring.
 The free $R$\+module with one generator $M^*=R^*$ does \emph{not}
admit a CDG\+module structure over the CDG\+ring~$\biR^\cu$.
 Indeed, any homogeneous differential $d_M\:M^*\rarrow M^*$ of
degree~$1$ vanishes for the simple reason of $M^*$ being concentrated
in the even cohomological degrees.
 Hence the squared differential $d_M^2\:M^*\rarrow M^*$ also vanishes.
 This contradicts the equation~(iii) in the definition of a CDG\+module
in Section~\ref{curved-dg-modules-subsecn}, as the action of
the curvature element $h=x$ in $M^*$ is nonzero.
 This example admits some obvious variations: one can take
$R^*=k[x]$ instead of $k[x,x^{-1}]$, and $M^*=k[x]$ or $k[x,x^{-1}]$
or $k[x,x^{-1}]/k[x]$.

 Notice that $k[x,x^{-1}]$ is both a projective and an injective
module over the graded ring $k[x,x^{-1}]$.
 Over the graded ring $k[x]$, the graded module $k[x]$ is projective,
while the graded module $k[x,x^{-1}]/k[x]$ is injective and
the graded module $k[x,x^{-1}]$ is (flat and) injective.
 Thus in all these examples the graded $R^*$\+module $M^*$ is either
projective or injective.
 Hence both the short exact sequences of $R^*$\+modules in
Proposition~\ref{G-plus-minus-prop}(b) split for the graded
$R^*$\+module $M^*$, and $M^*$ is a direct summand of the graded
$R^*$\+modules $G^+(M^*)$ and $G^-(M^*)$ which come endowed with
structures of CDG\+modules over $(R^*,d,h)$.
 Thus the class of all graded $R^*$\+modules admitting a CDG\+module
structure over the given CDG\+ring $(R^*,d,h)$ is \emph{not} closed
under direct summands, generally speaking.
\end{exs}

 We will see below in Example~\ref{CDG-ring-bec-example} that for
any CDG\+ring $\biR^\cu$ there is an acyclic DG\+ring $\hathatRbu$
such that the DG\+category of (left) CDG\+modules over $\biR^\cu$
is equivalent to the DG\+category of DG\+modules over $\hathatRbu$,
that is $\biR^\cu\bmodl\simeq\hathatRbu\bmodl$.
 Consequently, Examples~\ref{nonexistence-of-cdg-structure-examples}
imply that, even for an acyclic DG\+ring $\biS^\bu$, the class of
all graded $S^*$\+modues admitting a DG\+module structure over
$\biS^\bu$ need not be closed under direct summands.

\begin{exs} \label{nonexistence-of-derivation-examples}
 Take again $\Gamma=\boZ$, and let $R^*=k[\epsilon,x]$ be the free
graded commutative ring generated by two elements $\epsilon\in R^{-1}$
and $x\in R^0$.
 So the relations $x\epsilon=\epsilon x$ and $\epsilon^2=0$ are
imposed in~$R^*$.
 Let $d\:R^*\rarrow R^*$ be the odd derivation of degree~$1$ defined
by the rules $d(\epsilon)=x$ and $d(x)=0$.
 Then we have $d^2=0$; so $\biR^\bu=(R^*,d)$ is a DG\+ring.
 One can consider it as a CDG\+ring with $h=0$, if one wishes.

 Let $M^*$ be a graded $R^*$\+module annihilated by~$\epsilon$,
that is $\epsilon M^*=0$.
 If $M^*$ is also annihilated by~$x$, that is $xM^*=0$, then one can
introduce the differential $d_M=0$ on $M^*$, making $(M^*,d_M)$
a DG\+module over $(R^*,d)$.
 Let us show that if $M^*$ is annihilated by~$\epsilon$ but not by~$x$
(e.~g., $M^*=R^*/R^*\epsilon$ or $M^*=R^*/(R^*\epsilon+R^*x^2)$),
then the graded $R^*$\+module $M^*$ is \emph{not} a direct summand
of any graded $R^*$\+module $N^*$ admitting a DG\+module structure
$(N^*,d_N)$ over $(R^*,d)$.

 Indeed, let $(N^*,d_N)$ be a DG\+module over $(R,d)$ such that
$N^*\simeq M^*\oplus L^*$ as a graded $R^*$\+module.
 Let $m\in M^*$ be a homogeneous element for which $xm\ne0$.
 Then $d_N(m,0)$ is some homogeneous element in~$N^*$;
denote it by $d_N(m,0)=(m',l')$, where $m'\in M^*$ and $l'\in L^*$.
 Now we have (since $\epsilon M^*=0$ by assumption):
\begin{align*}
 0 &=d_N(0,0)=d_N(\epsilon m,0)=d_N(\epsilon(m,0))=
 d(\epsilon)(m,0)-\epsilon d_N(m,0) \\ &=x(m,0)-\epsilon(m',l')=
 (xm,0)-(\epsilon m',\epsilon l')=(xm,0)-(0,\epsilon l')=
 (xm,-\epsilon l')\ne0.
\end{align*}
 The contradiction proves that no such DG\+module $(N^*,d_N)$ exists.
 
 Notice that $(R^*,d)$, as any DG\+ring, is a DG\+module over itself.
 Futhermore, the injective graded $R^*$\+module $\Hom_k(R^*,k)$,
endowed with the differential induced by~$d$, becomes a DG\+module
over $(R^*,d)$.
 Now these examples show that the class of all graded $R^*$\+modules
which can be obtained as direct summands of graded $R^*$\+modules
admitting a DG\+module structure over $(R^*,d)$ is \emph{not} closed
either under extensions, or under cokernels, or under kernels
in $R^*\smodl$.

 Furthermore, let us drop the equation $\epsilon^2=0$ from the list
of defining relations of the graded $k$\+algebra $R^*$ (keeping only
the relation $x\epsilon=\epsilon x$).
 Then the graded ring $R^*$ is no longer graded commutative, but
otherwise all the claims and arguments above in this example
remain applicable with this modification.
 In this new context, we have short exact sequences of graded left
$R^*$\+modules $0\rarrow R^*[1]\overset{*\epsilon}\rarrow R^*\rarrow
R^*/R^*\epsilon\rarrow0$ and $0\rarrow\Hom_k(R^*/\epsilon R^*,k)
\rarrow\Hom_k(R^*,k)\overset{*\epsilon}\rarrow\Hom_k(R^*,k)\rarrow0$.
 It follows that the class of all graded $R^*$\+modules admitting
a DG\+module structure over $(R^*,d)$ is not even closed under
the cokernels of monomorphisms or under the kernels of epimorphisms,
in this case.

 Finally, one can observe that the category $(\biR^\bu\bmodl)^0$
of all graded $R^*$\+modules admitting a DG\+module structure over
$(R^*,d)$ is not abelian, and in fact, the monomorphism
$R^*[1]\overset{*\epsilon}\rarrow R^*$ in the latter example has
\emph{no} cokernel in $(\biR^\bu\bmodl)^0$.
 Indeed, there are clearly enough injective objects in $R^*\smodl$
belonging to $(\biR^\bu\bmodl)^0$; in particular, one can find
an injective graded $R^*$\+module $J^*$ admitting a DG\+module
structure $(J^*,d_J^*)$ over $(R^*,d)$ such that $R^*/R^*\epsilon$
is a graded $R^*$\+submodule in~$J^*$.
 So there is a right exact sequence $0\rarrow R^*[1]
\overset{*\epsilon}\rarrow R^*\rarrow J^*$ in $R^*\smodl$.

 For the sake of contradiction, assume that the morphism
$R^*[1]\overset{*\epsilon}\rarrow R^*$ has a cokernel $L^*$ in
$(\biR^\bu\bmodl)^0$.
 Then the monomorphism $R^*/R^*\epsilon\rarrow J^*$ factorizes
through the object $L^*$ in $R^*\smodl$, and it follows that
the action of~$x$ in $L^*$ is nonzero.
 Consequently, the action of~$\epsilon$ in $L^*$ has to be nonzero
as well.
 Now the map $f\:L^*\rarrow L^*[-1]$ defined by the rule
$f(l)=(-1)^{|l|}\epsilon l$ is a morphism of graded $R^*$\+modules,
and this morphism is nonzero.
 But the composition $R^*\rarrow L^*\overset f\rarrow L^*[-1]$ vanishes, 
so the morphism $R^*\rarrow L^*$ is \emph{not} an epimorphism and
consequently not a cokernel in $(\biR^\bu\bmodl)^0$.
 Similarly one can show that the epimorphism $\Hom_k(R^*,k)
\overset{*\epsilon}\rarrow\Hom_k(R^*,k)$ has no kernel
in $(\biR^\bu\bmodl)^0$.

 For another somewhat similar (but simpler) example, see
Example~\ref{L-tilde-example} below.
\end{exs}

\subsection{Main Construction} \label{main-construction-subsecn}
 The aim of this section is to spell out the construction and results
of~\cite[Section~3.2]{Pkoszul} with sufficient amount of detail.
 This construction plays a fundamental role in the theory developed
in this paper.

 Let $\bA$ be a DG\+category.
 The DG\+category $\bA^\bec$ is constructed as follows.
 The objects of $\bA^\bec$ are pairs $X^\bec=(X,\sigma_X)$, where $X$
is an object of $\bA$ and $\sigma_X\in \Hom_\bA^{-1}(X,X)$ is a cochain
in the complex of endomorphisms such that $d(\sigma_X)=\id_X$
and $\sigma_X^2=0$.
 So $\sigma_X$~is a chosen contracting homotopy for $X$ with
zero square.

 Given two objects $X^\bec=(X,\sigma_X)$ and $Y^\bec=(Y,\sigma_Y)
\in\bA^\bec$, the complex of morphisms
$\Hom_{\bA^\bec}^\bu(X^\bec,Y^\bec)$ is constructed as follows.
 For every $i\in\Gamma$, the abelian group $\Hom^i_{\bA^\bec}
(X^\bec,Y^\bec)$ is the group of all closed morphisms
$X\rarrow Y$ of degree~$-i$ in $\bA$, that is
$\Hom^i_{\bA^\bec}(X^\bec,Y^\bec)=\Hom^{-i}_{\cZ(\bA)}(X,Y)$.
 The differential $d^\bec\:\Hom^i_{\bA^\bec}(X^\bec,Y^\bec)\rarrow
\Hom^{i+1}_{\bA^\bec}(X^\bec,Y^\bec)$ is defined as the graded
commutator with the contracting homotopies~$\sigma$; explicitly,
it is given by the formula $d^\bec(f)=\sigma_Y f-(-1)^i f\sigma_X$.
 One computes that
\begin{multline*}
 d(\sigma_Yf-(-1)^if\sigma_X)=
 d(\sigma_Y)f-\sigma_Yd(f) \\
 -(-1)^id(f)\sigma_X-fd(\sigma_X)
 =\id_Yf-f\id_X=0,
\end{multline*}
so $d^\bec(f)\in\Hom^{i+1}_{\bA^\bec}(X^\bec,Y^\bec)$ whenever
$f\in\Hom^i_{\bA^\bec} (X^\bec,Y^\bec)$ (as desired).
 Furthermore, $(d^\bec)^2=0$, since $\sigma_X^2=0=\sigma_Y^2$.
 The composition of morphisms in $\bA^\bec$ is induced by
the composition of morphisms in $\bA$, and the identity morphisms
in $\bA^\bec$ are the identity morphisms in~$\bA$.

 All twists always exist in the DG\+category~$\bA^\bec$.
 Indeed, let $a\in\Hom_{\bA^\bec}^1(X^\bec,X^\bec)$ be
a Maurer--Cartan cochain in the complex of endomorphisms of
an object $X^\bec=(X,\sigma_X)\in\bA^\bec$; this means that
$a\in\Hom_\bA^{-1}(X,X)$, \ $d(a)=0$, and
$d^\bec(a)+a^2=\sigma_Xa+a\sigma_X+a^2=0$ in $\Hom_\bA^{-2}(X,X)$.
 Then the object $X^\bec(a)\in\bA^\bec$ can be constructed as
$X^\bec(a)=(X,\sigma_X+a)$.

 The DG\+category $\bA^\bec$ has shifts whenever a DG\+category
$\bA$~has.
 The DG\+category $\bA^\bec$ is additive whenever a DG\+category
$\bA$~is.
 Assuming $\bA$ is additive, the DG\+cat\-e\-gory $\bA^\bec$ is
idempotent-complete whenever the DG\+category $\bA$~is.
 The DG\+category $\bA^\bec$ has infinite coproducts or products
whenever a DG\+category $\bA$~has.
 All these operations in the DG\+category $\bA^\bec$ are induced
by the respective operations in~$\bA$ (the only caveat is that
the shift~$[1]$ on $\bA^\bec$ is induced by the inverse
shift~$[-1]$ on~$\bA$).

\begin{lem} \label{G-functors-in-DG-categories}
 Let\/ $\bA$ be a DG\+category with shifts and cones.
 Then there is a faithful additive functor\/ $\Phi=\Phi_\bA\:\sZ^0(\bA)
\rarrow\sZ^0(\bA^\bec)$ given by the rule $A\longmapsto\cone(\id_A[-1])$
for all objects $A\in\bA$.
 The functor\/ $\Phi$ has faithful adjoint functors on both sides, with
the left adjoint\/ $\Psi^+=\Psi^+_\bA\:\sZ^0(\bA^\bec)\rarrow\sZ^0(\bA)$
and the right adjoint\/ $\Psi^-=\Psi^-_\bA\:\sZ^0(\bA^\bec)\rarrow
\sZ^0(\bA)$ given by\/ $\Psi^+(X,\sigma_X)=X$ and\/
$\Psi^-(X,\sigma_X)=X[1]$ (so the functors\/ $\Psi^+$ and\/ $\Psi^-$
only differ from each other by a shift).
\end{lem}

\begin{proof}
 Let us make our notation more explicit.
 Given an object $A\in\bA$, the object $E=\cone(\id_A[-1])$
is specified by the datum of closed morphisms
$\iota\in\Hom^1_\bA(A,E)$ and $\pi\in\Hom^0_\bA(E,A)$
and morphisms $\pi'\in\Hom^{-1}_\bA(E,A)$ and
$\iota'\in\Hom^0_\bA(A,E)$ satisfying the equations
\begin{gather*}
 \pi'\iota'=0=\pi\iota, \quad \pi'\iota=\id_A=\pi\iota', \quad
 \iota\pi'+\iota'\pi=\id_E, \\
 d(\pi')=\pi, \quad d(\iota')=\iota.
\end{gather*}

 The endomorphism $\sigma_E\in\Hom^{-1}_\bA(E,E)$ is given by
the formula $\sigma_E=\iota'\pi'$.
 Then $\sigma_E^2=0$ and $d(\sigma_E)=d(\iota')\pi'+\iota'd(\pi')
=\iota\pi'+\iota'\pi=\id_E$; so $E^\bec=(E,\sigma_E)$ is an object
of~$\bA^\bec$.
 We set $\Phi(A)=E^\bec$.
 The action of the functor $\Phi$ on morphisms in $\sZ^0(\bA)$
(i.~e., on the closed morphisms of degree~$0$ in~$\bA$) is defined
in the obvious way.

 The construction of the functors $\Psi^+$ and $\Psi^-$ needs
no additional explanations.
 It is also clear that the functors $\Psi^+$, $\Psi^-$, and $\Phi$
are faithful.

 Let us explain the adjunctions.
 Given two objects $X^\bec=(X,\sigma_X)\in\bA^\bec$ and $A\in\bA$,
the natural isomorphism between the abelian groups
$\Hom_{\sZ^0(\bA)}(X,A)$ and $\Hom_{\sZ^0(\bA^\bec)}(X^\bec,E^\bec)$
assigns to a morphism $f\:X\rarrow A$ the morphism $g\:X\rarrow E$
given by the formula $g=\iota'f+\iota f\sigma_X$.
 Conversely, to a morphism $g\:X\rarrow E$ the morphism
$f=\pi g$ is assigned.
 The natural isomorphism between the abelian groups
$\Hom_{\sZ^0(\bA)}(A,X[1])$ and $\Hom_{\sZ^0(\bA^\bec)}(E^\bec,X^\bec)$
assigns to a closed morphism $f\in\Hom_\bA^1(A,X)$ the closed
morphism $g=f\pi'+\sigma_Xf\pi\in\Hom_{\bA^\bec}^0(E^\bec,X^\bec)$.
 Conversely, to a morphism $g\:E\rarrow X$ the morphism
$f=g\iota$ is assigned.
\end{proof}

 Now we consider what happens when one applies the construction
$\bA\longmapsto\bA^\bec$ twice.

\begin{prop} \label{iterated-bec-construction}
 For any DG\+category\/ $\bA$ with shifts and cones, there is
a fully faithful DG\+functor $\bec\bec\:\bA\rarrow\bA^{\bec\bec}$
given by the rule $A\longmapsto\cone(\id_A[-1])$.
 The DG\+functor $\bec\bec\:\bA\rarrow\bA^{\bec\bec}$ is an equivalence
of DG\+categories whenever $\bA$ is an idempotent-complete additive
DG\+category with twists.
\end{prop}

\begin{proof}
 Let us first describe the DG\+category $\bA^{\bec\bec}$ more
explicitly.
 The objects of $\bA^{\bec\bec}$ are triples $W^{\bec\bec}=
(W,\sigma,\tau)$, where $W$ is an object of $\bA$, while $\sigma
\in\Hom^{-1}_\bA(W,W)$ and $\tau\in\Hom^1_\bA(W,W)$ are two
endomorphisms of the degrees~$-1$ and~$1$, respectively, satisfying
the following equations:
$$
 \sigma^2=0=\tau^2, \quad \sigma\tau+\tau\sigma=\id_W, \quad
 d(\sigma)=\id_W, \quad d(\tau)=0.
$$
 Here $W^\bec=(W,\sigma)$ is an object of $\bA^\bec$, and
$\tau\in\Hom_{\bA^\bec}^{-1}(W^\bec,W^\bec)$ is an endomorphism
with $d^\bec(\tau)=\id_W$ (and $\tau^2=0$).

 Given two objects $U^{\bec\bec}=(U,\sigma_U,\tau_U)$ and
$V^{\bec\bec}=(V,\sigma_V,\tau_V)\in\bA^{\bec\bec}$, the complex
of morphisms $\Hom_{\bA^{\bec\bec}}^\bu(U^{\bec\bec},V^{\bec\bec})$
has the term $\Hom_{\bA^{\bec\bec}}^i(U^{\bec\bec},V^{\bec\bec})$
which is a subgroup in $\Hom_\bA^i(U,V)$ consisting of all
the morphisms $f\:U\rarrow V$ of degree~$i$ such that
$d(f)=0$ and $d^\bec(f)=\sigma_Vf-(-1)^if\sigma_U=0$.
 The differential $d^{\bec\bec}$ on
$\Hom_{\bA^{\bec\bec}}^\bu(U^{\bec\bec},V^{\bec\bec})$ is given
by the formula $d^{\bec\bec}(f)=\tau_Vf-(-1)^{|f|}f\tau_U$.

 Let us explain the basic intuition behind the assertions of
the proposition.
 Leaving the differential structures aside, one can say that
the objects of $\bA^{\bec\bec}$ are representations of the graded
ring of $2\times 2$ matrices (with the entries in~$\boZ$) in
the graded category~$\bA^*$.
 The endomorphisms $\sigma$ and~$\tau$, subject to the equations
above, define an action of $2\times 2$ matrices in the object~$W$.
 The basic Morita theory tells that the category of representations
of a ring of matrices in a given additive category $\sA$ is
equivalent to $\sA$, under mild assumptions on~$\sA$.
 This is a bare-bones version of our argument below.

 The DG\+functor~$\bec\bec$ assigns to an object $A\in\bA$ the object
$E=\cone(\id_A[-1])$ endowed with the endomorphisms $\sigma_E=
\iota'\pi'\in\Hom_\bA^{-1}(E,E)$ and $\tau_E=\iota\pi\in\Hom_\bA^1(E,E)$
(in the notation from the previous proof).
 So $\bec\bec(A)=(E,\sigma_E,\tau_E)$.

 Now suppose that we are given two objects $A$ and $B\in\bA$; put
$E=\cone(\id_A[-1])$ and $F=\cone(\id_B[-1])$, and let
$\iota_A$, $\pi_A$, $\iota'_A$, $\pi'_A$ and $\iota_B$, $\pi_B$,
$\iota'_B$, $\pi'_B$ denote the related morphisms.
 Then the functor~$\bec\bec$ assigns to a morphism
$f\in\Hom_\bA^i(A,B)$ the morphism $g=\bec\bec(f)\in\Hom_\bA^i(E,F)$
given by the formula
\begin{equation} \label{morphism-induced-by-nonclosed-morphism}
 g=(-1)^i\iota'_Bf\pi_A+\iota_Bf\pi'_A+\iota'_Bd(f)\pi'_A.
\end{equation}
 One readily checks that $d(g)=0$, \ $d^\bec(g)=\sigma_Fg-
(-1)^ig\sigma_E=0$, and $d^{\bec\bec}(g)=\tau_Fg-
(-1)^ig\tau_E=\bec\bec(df)$.
 The equation of compatibility with the compositions,
$\bec\bec(f_1\circ f_2)=\bec\bec(f_1)\circ\bec\bec(f_2)$ for any
two composable morphisms $f_1$ and $f_2$ (of some degrees~$i_1$
and~$i_2$) in~$\bA$ also needs to be checked.

 Clearly, the map of abelian groups $\bec\bec\:\Hom^i_\bA(A,B)\rarrow
\Hom^i_{\bA^{\bec\bec}}(E,F)$ is injective.
 To check surjectivity, consider a morphism $g\in\Hom^i_\bA(E,F)$.
 Then one has $g=\iota'_Bg_{00}\pi_A+\iota'_Bg_{01}\pi'_A+
\iota_Bg_{10}\pi_B+\iota_Bg_{11}\pi'_B$ for certain (uniquely
defined) morphisms $g_{st}\in\Hom_\bA^{i-s+t}(A,B)$, where
$s$, $t\in\{0,1\}$.
 Then (computing the $00$\+component of) the equation $d^\bec(g)=0$
implies $g_{10}=0$, and (computing the $01$\+component of)
the same equation implies $g_{00}=(-1)^ig_{11}$.
 The equation $d(g)=0$ implies $(-1)^id(g_{00})=g_{01}=d(g_{11})$.
 Thus the DG\+functor~$\bec\bec$ induces an isomorphism of
the complexes of morphisms $\Hom^i_\bA(A,B)\simeq
\Hom^i_{\bA^{\bec\bec}}(\bec\bec(A),\bec\bec(B))$ for all
objects $A$, $B\in\bA$.

 Finally, assume that $\bA$ is an idempotent-complete additive
DG\+category with twists, and let $(W,\sigma,\tau)\in\bA^{\bec\bec}$
be an object.
 Then, since $d(\tau)=0=\tau^2$, the element $-\tau\in\Hom^1_\bA(W,W)$
is a Maurer--Cartan cochain.
 Consider the related twist $U=W(-\tau)$.
 Then the graded abelian group $\Hom_{\bA^*}^*(U,U)$ is naturally
identified with $\Hom_{\bA^*}^*(W,W)$, while the differentials on
the two complexes $\Hom_\bA^\bu(U,U)$ and $\Hom_\bA^\bu(W,W)$ are
connected by the formula $d_U=d_W-[\tau,{-}]$, or more explicitly
$d_U(f)=d_W(f)-\tau f+(-1)^{|f|}f\tau$.

 Now we have $d_U(\sigma)=d_U(\tau)=0$, so $\sigma\tau
\in\Hom^0_\bA(U,U)$ is a closed endomorphism of degree~$0$.
 Furthermore, $(\sigma\tau)^2=(\sigma\tau+\tau\sigma)\sigma\tau
=\sigma\tau$.
 So $\sigma\tau\:U\rarrow U$ is an idempotent endomorphism of
the object $U\in\sZ^0(\bA)$.
 The image $A$ of this idempotent endomorphism recovers an object
$A\in\bA$ such that $W=\bec\bec(A)$.
 In fact, the image of the complementary idempotent
$\tau\sigma=\id-\sigma\tau$ is isomorphic to $A[-1]$ as an object
of $\sZ^0(\bA)$; the endomorphisms~$\sigma$ and~$\tau$ provide
the requisite closed isomorphism.
 So we have $U=A\oplus A[-1]$ and $W=\cone(\id_A[-1])$ in~$\bA$.
\end{proof}

\subsection{Properties of the main construction}
\label{properties-of-main-construction-subsecn}
 For any DG\+category $\bA$ with shifts and cones, the DG\+category
$\bA^{\bec\bec}$ only differs from $\bA$ by adjoining all twists and
some of their direct summands.
 The aim of this section is to explain the meaning of this (somewhat
informal) statement.

\begin{lem} \label{obvious-composition-of-functors-observation}
 For any DG\+category\/ $\bA$ with shifts and cones, the composition
of the additive functor\/ $\sZ^0(\bec\bec)\:\sZ^0(\bA)\rarrow
\sZ^0(\bA^{\bec\bec})$ induced by the DG\+functor\/ $\bec\bec\:
\bA\rarrow\bA^{\bec\bec}$ with the additive functor\/
$\Psi^+_{\bA^\bec}\:\sZ^0(\bA^{\bec\bec})\rarrow\sZ^0(\bA^\bec)$
is naturally isomorphic to the additive functor\/
$\Phi_\bA\:\sZ^0(\bA)\rarrow\sZ^0(\bA^\bec)$,
$$
 \Psi^+_{\bA^\bec}\circ\sZ^0(\bec\bec)\simeq\Phi_\bA.
$$
\end{lem}

\begin{proof}
 Follows immediately from the constructions of the respective functors
in Lemma~\ref{G-functors-in-DG-categories}
and Proposition~\ref{iterated-bec-construction}.
\end{proof}

\begin{lem} \label{difficult-composition-of-functors-observation}
 For any DG\+category\/ $\bA$ with shifts and cones, the composition
of the additive functor\/ $\Psi^-_\bA\:\sZ^0(\bA^\bec)\rarrow\sZ^0(\bA)$
with the additive functor\/ $\sZ^0(\bec\bec)\:\sZ^0(\bA)\rarrow
\sZ^0(\bA^{\bec\bec})$ induced by the DG\+functor\/ $\bec\bec\:
\bA\rarrow\bA^{\bec\bec}$ is naturally isomorphic to the additive
functor\/ $\Phi_{\bA^\bec}\:\sZ^0(\bA^\bec)\rarrow
\sZ^0(\bA^{\bec\bec})$,
$$
 \sZ^0(\bec\bec)\circ\Psi^-_\bA\simeq\Phi_{\bA^\bec}.
$$
\end{lem}

\begin{proof}
 Let $(X,\sigma_X)$ be an object of the DG\+category~$\bA^\bec$
(so $\sigma_X\in\Hom^{-1}_\bA(X,X)$ and $d(\sigma_X)=\id_X$,
while $\sigma_X^2=0$).
 Then $\Psi^-(X,\sigma_X)=X[1]\in\bA$.
 By the definition (see the proof of
Proposition~\ref{iterated-bec-construction}), the object
$\bec\bec(X[1])\in\bA^{\bec\bec}$ is the triple $(E,\sigma_E,\tau_E)$,
where $E=\cone(\id_X)\in\bA$, while $\sigma_E=\iota'_E\pi'_E\in
\Hom_\bA^{-1}(E,E)$ and $\tau_E=\iota_E\pi_E\in\Hom_\bA^1(E,E)$.
 Here $\iota_E\in\Hom_\bA^0(X,E)$ and $\pi_E\in\Hom_\bA^1(E,X)$
are closed morphisms, while $\pi'_E\in\Hom_\bA^0(E,X)$ and
$\iota'_E\in\Hom_\bA^{-1}(X,E)$ are morphisms satisfying
the equations similar to the ones in the proof of
Lemma~\ref{G-functors-in-DG-categories}, except for one sign:
\begin{gather*}
 \pi'_E\iota'_E=0=\pi_E\iota_E, \quad
 \pi'_E\iota_E=\id_X=\pi_E\iota'_E, \quad
 \iota_E\pi'_E+\iota'_E\pi_E=\id_E, \\
 d(\pi'_E)=-\pi_E, \quad d(\iota'_E)=\iota_E.
\end{gather*}

 On the other hand, following the constructions in
Section~\ref{main-construction-subsecn}, the object
$\Phi_{\bA^\bec}(X,\sigma_X)$ is the triple $(S,\sigma_S,\tau_S)$,
where $S=X\oplus X[1]\in\bA$, while $\sigma_S=-\iota'_S\sigma_X\pi_S
+\iota_S\sigma_X\pi'_S+\iota'_S\pi'_S\in\Hom_\bA^{-1}(S,S)$ and
$\tau_S=\iota_S\pi_S\in\Hom_\bA^1(S,S)$.
 Here $\iota_S\in\Hom_\bA^0(X,S)$, \ $\pi_S\in\Hom_\bA^1(S,X)$,
$\pi'_S\in\Hom_\bA^0(S,X)$, and $\iota'_S\in\Hom_\bA^{-1}(X,S)$
are closed morphisms satisfying the equations
$$
 \pi'_S\iota'_S=0=\pi_S\iota_S, \quad
 \pi'_S\iota_S=\id_X=\pi_S\iota'_S, \quad
 \iota_S\pi'_S+\iota'_S\pi_S=\id_S.
$$

 Now $s=\iota'_S\pi_E+\iota_S\pi'_E+\iota_S\sigma_X\pi_E\:
E\rarrow S$ is a closed isomorphism of degree~$0$ in $\bA$
satisfying the equations $\sigma_Ss=s\sigma_E$ and $\tau_Ss=s\tau_E$.
 The inverse (closed) isomorphism is $e=\iota'_E\pi_S+\iota_E\pi'_S
-\iota_E\sigma_X\pi_S\:S\rarrow E$.
\end{proof}

 For any DG\+category $\bA$ with shifts and cones, denote by
$\Xi=\Xi_\bA\:\sZ^0(\bA)\rarrow\sZ^0(\bA)$ the additive functor taking
every object $A\in\bA$ to the object $\cone(\id_A[-1])$ (and acting
on the morphisms in the obvious way).

\begin{lem} \label{Xi-decomposed}
 For any DG\+category\/ $\bA$ with shifts and cones, the composition\/
$\Psi_\bA^+\circ\Phi_\bA\:\sZ^0(\bA)\rarrow\sZ^0(\bA)$ of the additive
functors\/ $\Phi_\bA\:\sZ^0(\bA)\rarrow\sZ^0(\bA^\bec)$ and\/
$\Psi^+_\bA\:\sZ^0(\bA^\bec)\rarrow\sZ^0(\bA)$ is naturally isomorphic
to the functor\/ $\Xi_\bA$,
$$
 \Psi^+_\bA\circ\Phi_\bA\simeq\Xi_\bA.
$$
 The composition\/ $\Phi_\bA\circ\Psi^-_\bA\:\sZ^0(\bA^\bec)\rarrow
\sZ^0(\bA^\bec)$ of the additive functor\/ $\Psi^-_\bA\:\sZ^0(\bA^\bec)
\rarrow\sZ^0(\bA)$ and the additive functor\/ $\Phi_\bA$ is naturally
isomorphic to the functor\/ $\Xi_{\bA^\bec}$,
$$
 \Phi_\bA\circ\Psi^-_\bA\simeq\Xi_{\bA^\bec}.
$$
\end{lem}

\begin{proof}
 The first assertion follows immediately from the constructions of
the functors in Lemma~\ref{G-functors-in-DG-categories}.
 A proof of the second assertion is contained in the proof of
Lemma~\ref{difficult-composition-of-functors-observation}.
 Alternatively, one can formally deduce the second assertion from
the first assertion for the DG\+category $\bA^\bec$ by applying
the functor $\Psi^+_{\bA^\bec}$ to both sides of the natural isomorphism
in Lemma~\ref{difficult-composition-of-functors-observation} and taking
into account Lemma~\ref{obvious-composition-of-functors-observation},
$$
 \Phi_\bA\circ\Psi^-_\bA\simeq
 \Psi^+_{\bA^\bec}\circ\sZ^0(\bec\bec)\circ\Psi_\bA^-\simeq
 \Psi^+_{\bA^\bec}\circ\Phi_{\bA^\bec}\simeq\Xi_{\bA^\bec}.
$$
\end{proof}

\begin{lem} \label{Phi-Psi-extended-to-nonclosed-morphisms}
 For any DG\+category\/ $\bA$ with shifts and cones, the additive
functor\/ $\Phi\:\sZ^0(\bA)\rarrow\sZ^0(\bA^\bec)$ can be extended
in a natural way to a fully faithful additive functor\/
$\widetilde\Phi=\widetilde\Phi_\bA\:\bA^0\rarrow\sZ^0(\bA^\bec)$.
 Similarly, the additive functors\/ $\Psi^+$ and\/
$\Psi^-\:\sZ^0(\bA^\bec)\rarrow\sZ^0(\bA)$ can be extended in a natural
way to fully faithful additive functors\/ $\widetilde\Psi^+=
\widetilde\Psi^+_\bA$ and\/ $\widetilde\Psi^-=\widetilde\Psi^-_\bA\:
(\bA^\bec)^0\rarrow\sZ^0(\bA)$ (still satisfying\/
$\widetilde\Psi^-=\widetilde\Psi^+[1]$).
\end{lem}

\begin{proof}
 The construction of the functor $\widetilde\Psi^+$ is essentially
obvious.
 By the definition, we have $\widetilde\Psi^+(X,\sigma_X)=
\Psi^+(X,\sigma_X)=X$ for all objects $(X,\sigma_X)\in\bA^\bec$;
and also by definition, $\Hom_{\bA^\bec}^0((X,\sigma_X),(Y,\sigma_Y))=
\Hom_{\cZ(\bA)}^0(X,Y)=\Hom_{\sZ^0(\bA)}(X,Y)$ for all
$(X,\sigma_X)$ and $(Y,\sigma_Y)\in\bA^\bec$.
 This natural isomorphism of the Hom groups provides the action of
the functor $\widetilde\Psi^+$ on morphisms.
 The functor $\widetilde\Psi^-$ is defined as $\widetilde\Psi^-=
\widetilde\Psi^+[1]$.

 Concerning the functor $\widetilde\Phi$, let $A$ and $B\in\bA$ be
two objects, and let $\Phi(A)=(E,\sigma_E)$ and $\Phi(B)=(F,\sigma_F)$
be their images in~$\bA^\bec$.
 Let $f\in\Hom^0_\bA(A,B)$ be a (not necessarily closed) morphism
of degree~$0$.
 Then we put $\widetilde\Phi(f)=g$, where $g\in\Hom^0_\bA(E,F)$ is
the morphism constructed in the proof of
Proposition~\ref{iterated-bec-construction} (take $i=0$ in
the formula~\eqref{morphism-induced-by-nonclosed-morphism}).
 Following the arguments in the proof of
Proposition~\ref{iterated-bec-construction}, this assignment is
a well-defined isomorphism of abelian groups $\Hom^0_\bA(A,B)\simeq
\Hom_{\sZ^0(\bA^\bec)}(E,F)$ compatible with the compositions of
morphisms; hence the fully faithful additive functor~$\widetilde\Phi$.
 Alternatively, one can refer to
Lemma~\ref{obvious-composition-of-functors-observation} and simply
put $\widetilde\Phi_\bA=\widetilde\Psi^+_{\bA^\bec}\circ(\bec\bec)^0$.
\end{proof}

\begin{rem}
 Both the isomorphisms of additive functors in
Lemmas~\ref{obvious-composition-of-functors-observation}
and~\ref{difficult-composition-of-functors-observation} remain
valid for the extended functors from
Lemma~\ref{Phi-Psi-extended-to-nonclosed-morphisms}, i.~e., one has
$$
 \widetilde\Psi^+_{\bA^\bec}\circ(\bec\bec)^0\simeq
 \widetilde\Phi_\bA
 \quad\text{and}\quad
 \sZ^0(\bec\bec)\circ\widetilde\Psi^-_\bA\simeq
 \widetilde\Phi_{\bA^\bec}.
$$
 The former isomorphism was essentially explained already in
the proof of Lemma~\ref{Phi-Psi-extended-to-nonclosed-morphisms}.
 The latter one is verified by a direct computation continuing
the computations in the proof of
Lemma~\ref{difficult-composition-of-functors-observation}.
 Let $(X,\sigma_X)$ and $(Y,\sigma_Y)$ be two objects of
the DG\+category $\bA^\bec$, and let $z\:(X,\sigma_X)\rarrow
(Y,\sigma_Y)$ be (not necessarily closed) morphism of degree~$0$
in~$\bA^\bec$; in other words, $z\:X\rarrow Y$ is a closed
morphism of degree~$0$ in~$\bA$.
 Put $(E,\sigma_E,\tau_E)=\bec\bec(X)$ and $(F,\sigma_F,\tau_F)
=\bec\bec(Y)$; then the morphism $g=\bec\bec(z)\:E\rarrow F$ is given
by the rule $g=\iota'_Fz\pi_E+\iota_Fz\pi'_E$.
 Furthermore, put $\Phi_{\bA^\bec}(X,\sigma_X)=(S,\sigma_S,\tau_S)$
and $\Phi_{\bA^\bec}(Y,\sigma_Y)=(T,\sigma_T,\tau_T)$; then
the morphism $u=\widetilde\Phi_{\bA^\bec}(z)\:S\rarrow T$ is given by
the formula $u=\iota'_Tz\pi_S+\iota_Tz\pi'_S+
\iota_T(\sigma_Yz-z\sigma_X)\pi_S$.
 Then one can readily check that $tg=us$, or equivalently, $ge=fu$,
where the mutually inverse closed isomorphisms $s$ and~$e$ were
defined in the proof of
Lemma~\ref{difficult-composition-of-functors-observation}, while
the mutually inverse closed isomorphisms $t\:F\rarrow T$ and
$f\:T\rarrow F$ are given by the similar formulas
$t=\iota'_T\pi_F+\iota_T\pi'_F+\iota_T\sigma_Y\pi_F$ and
$f=\iota'_F\pi_T+\iota_F\pi'_T-\iota_F\sigma_Y\pi_T$.
\end{rem}

\begin{lem} \label{twists-into-isomorphisms}
 For any DG\+category\/ $\bA$ with shifts and cones, the additive
functor\/ $\Phi\:\sZ^0(\bA)\rarrow\sZ^0(\bA^\bec)$ transforms
twists into isomorphisms, and so do the additive functors\/
$\Psi^+$ and\/ $\Psi^-\:\sZ^0(\bA^\bec)\rarrow\sZ^0(\bA)$.
 All these functors also transform the shifts\/~$[n]$ into the inverse
shifts\/~$[-n]$ in the respective categories.
\end{lem}

\begin{proof}
 The assertions about the shifts are obvious from the construction of
the complexes of morphisms in the DG\+category~$\bA^\bec$.
 The assertions about the twists, claiming esstentially that all
the three functors take objects connected by a not necessarily closed
isomorphism of degree~$0$ in the respective DG\+category to objects
connected by a closed isomorphism of degree~$0$, follow immediately
from Lemma~\ref{Phi-Psi-extended-to-nonclosed-morphisms}.
\end{proof}

\begin{lem} \label{Phi-Psi-conservative}
 For any DG\+category\/ $\bA$ with shifts and cones, the additive
functors\/ $\Phi\:\sZ^0(\bA)\rarrow\sZ^0(\bA^\bec)$ and \/
$\Psi^+$, $\Psi^-\:\sZ^0(\bA^\bec)\rarrow\sZ^0(\bA)$ are conservative
(i.~e., they take nonisomorphisms to nonisomorphisms).
\end{lem}

\begin{proof}
 By Lemma~\ref{Phi-Psi-extended-to-nonclosed-morphisms}, all the three
functors in question are compositions of the inclusion functors
$\sZ^0(\bA)\rarrow\bA^0$ or $\sZ^0(\bA^\bec)\rarrow(\bA^\bec)^0$
with fully faithful functors $\widetilde\Phi$, $\widetilde\Psi^+$,
or~$\widetilde\Psi^-$.
 So it suffices to observe that, for any DG\+category $\bA$,
the inclusion functor $\sZ^0(\bA)\rarrow\bA^0$ is conservative
(since the inverse morphism to a closed morphism, if it exists,
is also a closed morphism), any fully faithful functor is conservative,
and a composition of two conservative functors is conservative.
\end{proof}

 Let $\bA$ be an additive DG\+category.
 Then the \emph{idempotent completion} $\bA^\oplus$ of $\bA$ is
constructed as the DG\+category whose objects are pairs $(A,e_A)$,
where $A$ is an object of $\bA$ and $e_A\in\Hom^0_\bA(A,A)$ is
a closed idempotent endomorphism of degree~$0$.
 By the definition, the complex of morphisms $(A,e_A)\rarrow
(B,e_B)$ in $\bA^\oplus$ is the subcomplex
$$
 \Hom_{\bA^\oplus}^\bu((A,e_A),(B,e_B))=
 e_B\Hom_\bA^\bu(A,B)e_A\subset\Hom_\bA^\bu(A,B).
$$
 The composition of morphisms in $\bA^\oplus$ is induced by
that in~$\bA$.
 The DG\+category $\bA^\oplus$ is additive and idempotent-complete,
there is a fully faithful DG\+functor $\bA\rarrow\bA^\oplus$ taking
each object $A\in\bA$ to $(A,\id_A)\in\bA^\oplus$, and every object
of $\bA^\oplus$ is a direct summand of an object coming from~$\bA$.
 If the DG\+category $\bA$ has shifts, twists, cones, infinite
coproducts, or infinite products, then so has
the DG\+category~$\bA^\oplus$.

 Let $\bA$ be a DG\+category.
 Then the \emph{twist completion} (or ``Maurer--Cartan completion'')
$\bA\mc$ is constructed as the DG\+category whose objects are
pairs $(A,a_A)$, where $A$ is an object of $\bA$ and
$a\in\Hom^1_\bA(A,A)$ is a Maurer--Cartan cochain.
 By the definition, the underlying graded abelian group of the complex
of morphisms $(A,a_A)\rarrow(B,a_B)$ in $\bA\mc$ coincides with
the underlying graded abelan group of the complex of morphisms
$A\rarrow B$ in $\bA$, that is
$$
 \Hom_{\bA\mc^*}^*((A,a_A),(B,a_B))=
 \Hom_{\bA^*}^*(A,B).
$$
 The differential $d_{(a_A,a_B)}$ in the complex
$\Hom_{\bA\mc}^\bu((A,a_A),(B,a_B))$ is obtained from
the differential~$d$ in the complex $\Hom_\bA(A,B)$ by the rule
$$
 d_{(a_A,a_B)}(f)=d(f)+a_Bf-(-1)^{|f|}fa_A.
$$
 The composition of morphisms in $\bA\mc$ is induced by
that in~$\bA$.
 All twists exist in the DG\+category $\bA\mc$, there is a fully
faithful DG\+functor $\bA\rarrow\bA\mc$ taking each object $A\in\bA$
to the object $(A,0)\in\bA\mc$, and every object of $\bA\mc$ is
a twist of an object coming from~$\bA$.

 If the DG\+category $\bA$ is additive, then so is
the DG\+category~$\bA\mc$.
 If the DG\+category $\bA$ has shifts, cones, infinite coproducts,
or infinite products, then so has the DG\+category~$\bA\mc$
(in fact, if $\bA$ is additive with shifts, then it already
follows that $\bA\mc$ has cones).
 However, the passage from $\bA$ to $\bA\mc$ need not preserve
idempotent-completeness.

\begin{exs} \label{adding-twists-examples}
 (1)~Let $k$~be a field and $\bA$ be the DG\+category of acyclic
complexes of $k$\+vector spaces, i.~e., the full DG\+subcategory in
the DG\+category of complexes of $k$\+vector spaces $\bC(k\rmodl)$
whose objects are the acyclic complexes.
 Then $\bA$ is an idempotent-complete DG\+category, but $\bA\mc$
is not.
 The reason is that not every complex (of vector spaces) can be
endowed with a new differential making it an acyclic complex,
but every complex is a direct summand of a complex which can be so
endowed; hence $\bA\mc\varsubsetneq\bA\mc^\oplus= \bC(k\rmodl)$.

 (2)~For any idempotent-complete additive DG\+category $\bB$ with
twists, consider the full DG\+subcategory $\bA\subset\bB$
of contractible objects in~$\bB$.
 Then one has an equivalence of DG\+categories $\bA\mc^\oplus
\simeq\bB$, because every object $B\in\bB$ is a direct summand of
a twist $B\oplus B[1]\in\bB$ of the object $\cone(\id_B)\in\bA$.
\end{exs}

\begin{prop} \label{iterated-bec-construction-revisited}
 Let\/ $\bA$ be an additive DG\+category with shifts and cones.
 Then the DG\+category\/ $\bA^{\bec\bec}$ can be naturally viewed
as an intermediate full DG\+subcategory between\/ $\bA\mc$ and\/
$\bA\mc^\oplus$,
$$
 \bA\mc\subset\bA^{\bec\bec}\subset\bA\mc^\oplus.
$$
 If the DG\+category\/ $\bA$ is idempotent-complete, then\/
$\bA^{\bec\bec}=\bA\mc^\oplus$.
\end{prop}

\begin{proof}
 We recall from Section~\ref{main-construction-subsecn} that,
for any DG\+category $\bB$, the DG\+category $\bB^\bec$ has twists;
furthermore, $\bB^\bec$ is additive and idempotent-complete
whenever $\bB$~is.
 Now in the situation at hand, let us consider the commutative diagram
of fully faithful DG\+functors
\begin{equation}
\begin{gathered}
 \xymatrix{\bA \ar@{>->}[r] \ar@{>->}[d]^-{\bec\bec}
 & \bA\mc^\oplus \ar@{=}[d]^-{\bec\bec} \\
 \bA^{\bec\bec} \ar@{>->}[r] & \bA\mc^{\oplus\bec\bec}
 }
\end{gathered}
\end{equation}
 Here the rightmost vertical functor~$\bec\bec$ is an equivalence
of DG\+categories by Proposition~\ref{iterated-bec-construction},
as the DG\+category $\bA\mc^\oplus$ is additive and
idempotent-complete with twists.
 It follows that $\bA^{\bec\bec}$ is a full DG\+subcategory in
$\bA\mc^\oplus$.
 On the other hand, all twists exist in $\bA^{\bec\bec}$; so $\bA\mc
\subset\bA^{\bec\bec}$.
 Finally, if $\bA$ is idempotent-complete, then $\bA^{\bec\bec}$ is
idempotent-complete, too; hence $\bA^{\bec\bec}=\bA\mc^\oplus$.
\end{proof}

\begin{exs} \label{becbec-examples}
 (1)~The following example shows that the DG\+category $\bA^{\bec\bec}$
need not be idempotent-complete (for an additive DG\+category $\bA$
which is not idempotent-complete).
 Let $\bA$ be the following full DG\+subcategory in the DG\+category
of complexes of $k$\+vector spaces $\bC(k\rmodl)$, with the grading
group $\Gamma=\boZ$.
 The objects of $\bA$ are all the finite complexes of
finite-dimensional vector spaces whose terms are vector spaces of
dimension divisible by~$m$, where $m\ge2$ is a fixed integer.
 Then one has $\bA=\bA\mc=\bA^{\bec\bec}\varsubsetneq\bA\mc^\oplus$.
 The reason is that, for any finite complex of finite-dimensional
vector spaces $C^\bu$, if all the terms of the complex
$\cone(\id_{C^\bu}[-1])$ have dimensions divisible by~$m$,
then the complex $C^\bu$ has the same property.

 (2)~On the other hand, let $\bA$ be the full DG\+subcategory in
$\bC(k\rmodl)$ consisting of all the finite complexes of
finite-dimensional vector spaces whose Euler characteristic is
divisible by~$m$.
 Then $\bA=\bA\mc\varsubsetneq\bA^{\bec\bec}=\bA\mc^\oplus$.
 Here the reason is that the complex $\cone(\id_{C^\bu}[-1])$ has
zero Euler characteristic for any finite complex of finite-dimensional
vector spaces~$C^\bu$.
 In both the examples~(1) and~(2), \,$\bA^\oplus=\bA\mc^\oplus$ is,
of course, the full DG\+subcategory in $\bC(k\rmodl)$ consisting of
all finite complexes of finite-dimensional vector spaces.
\end{exs}

\subsection{Compatibility with DG\+functors}
\label{compatibility-with-dg-functors-subsecn}
 Let $F\:\bA\rarrow\bB$ be a DG\+functor.
 Then the DG\+functor $F^\bec\:\bA^\bec\rarrow\bB^\bec$ is defined by
the following (obvious) rules.
 To every object $X^\bec=(X,\sigma_X)\in\bA^\bec$, the functor $F^\bec$
assigns the object $F^\bec(X^\bec)=(F(X),F(\sigma_X))\in\bB^\bec$.
 For any two objects $X^\bec=(X,\sigma_X)$ and $Y^\bec=(Y,\sigma_Y)
\in\bA^\bec$ and any integer $i\in\boZ$, the map $F^\bec\:
\Hom^i_{\bA^\bec}(X^\bec,Y^\bec)\rarrow
\Hom^i_{\bB^\bec}(F(X^\bec),F(Y^\bec))$, by construction, coincides
with the map $\cZ(F)\:\Hom^{-i}_{\cZ^(\bA)}(X,Y)\rarrow
\Hom^{-i}_{\cZ(\bB)}(F(X),F(Y))$.  {\hfuzz=1.7pt\par}

 Assume that $\bA$ and $\bB$ are DG\+categories with shifts and cones.
 Then the following square diagrams of additive functors are
commutative:
$$
 \xymatrix{
  \sZ^0(\bA) \ar[rr]^{\sZ^0(F)} \ar[d]_{\Phi_\bA}
  && \sZ^0(\bB) \ar[d]^{\Phi_\bB} \\
  \sZ^0(\bA^\bec) \ar[rr]_{\sZ^0(F^\bec)} && \sZ^0(\bB^\bec)
 }
 \qquad
 \xymatrix{
  \bA^0 \ar[rr]^{F^0} \ar[d]_{\widetilde\Phi_\bA}
  && \bB^0 \ar[d]^{\widetilde\Phi_\bB} \\
  \sZ^0(\bA^\bec) \ar[rr]_{\sZ^0(F^\bec)} && \sZ^0(\bB^\bec)
 }
$$
$$
 \xymatrix{
  \sZ^0(\bA^\bec) \ar[rr]^{\sZ^0(F^\bec)} \ar[d]_{\Psi^\pm_\bA}
  && \sZ^0(\bB^\bec) \ar[d]^{\Psi^\pm_\bB} \\
  \sZ^0(\bA) \ar[rr]_{\sZ^0(F)} && \sZ^0(\bB)
 }
 \qquad
 \xymatrix{
  (\bA^\bec)^0 \ar[rr]^{(F^\bec)^0} \ar[d]_{\widetilde\Psi^\pm_\bA}
  && (\bB^\bec)^0 \ar[d]^{\widetilde\Psi^\pm_\bB} \\
  \sZ^0(\bA) \ar[rr]_{\sZ^0(F)} && \sZ^0(\bB)
 }
$$
 Furthermore, the following square diagram of DG\+functors is
commutative as well:
$$
 \xymatrix{
  \bA \ar[rr]^-F \ar[d]_-{\bec\bec}
  && \bB \ar[d]^-{\bec\bec} \\
  \bA^{\bec\bec} \ar[rr]^-{F^{\bec\bec}} && \bB^{\bec\bec}
 }
$$

 For a discussion of adjoint DG\+functors between DG\+categories
and the related adjunctions of additive functors between
preadditive categories, see~\cite[Section~7.5]{Pdomc}.

\subsection{Examples} \label{A-bec-examples-subsecn}
 In this section we compute the DG\+category $\bA^\bec$ and
the additive category $\sZ^0(\bA^\bec)$ for the concrete examples
of DG\+categories $\bA$ described in Section~\ref{examples-secn}.

\begin{ex} \label{complexes-bec-example}
 Let $\sA$ be a preadditive category and $\bC(\sA)$ be the DG\+category
of complexes in~$\sA$.
 We will use the notation from Section~\ref{complexes-subsecn}: in
particular, $\sG(\sA)$ is the preadditive category of graded objects
in~$\sA$ (and homogeneous morphisms of degree~$0$), while
$\sC(\sA)=\sZ^0(\bC(\sA))$ is the preadditive category of complexes
in~$\sA$ (and closed morphisms of degree~$0$ between them).

 To begin with, the following description of the DG\+category
$\bC(\sA)^\bec$ is obtained directly from the definitions.
 An object of $\bC(\sA)^\bec$ is a graded object $X^*$ in $\sA$ endowed
with a differential $d_X\in\Hom^1_{\cG(\sA)}(X^*,X^*)$ and
an endomorphism $\sigma_X\in\Hom^{-1}_{\cG(\sA)}(X^*,X^*)$ such that
$d_X^2=0=\sigma_X^2$ and $d_X\sigma_X+\sigma_Xd_X=\id_{X^*}$.
 The complex of morphisms from an object $X^\bec=(X^*,d_X,\sigma_X)$
to an object $Y^\bec=(Y^*,d_Y,\sigma_Y)$ in the DG\+category
$\bC(\sA)^\bec$ is constructed as follows: the group
$\Hom_{\bC(\sA)^\bec}^n(X^\bec,Y^\bec)$ is the subgroup in
$\Hom_{\cG(\sA)}^{-n}(X^*,Y^*)$ consisting of all homogeneous
morphisms~$f$ such that $d_Yf-(-1)^nfd_X=0$; the differential~$d^\bec$
in the complex $\Hom_{\bC(\sA)^\bec}^\bu(X^\bec,Y^\bec)$ is given
by the rule $d^\bec(f)=\sigma_Yf-(-1)^{|f|}f\sigma_X$.

 One can observe that all objects of the DG\+category
$\bC(\sA)^\bec$ are contractible.
 Indeed, for any object $X^\bec=(X^*,d_X,\sigma_X)\in\bC(\sA)^\bec$
the endomorphism~$d_X$ belongs to the subgroup
$\Hom_{\bC(\sA)^\bec}^{-1}(X^\bec,X^\bec)\subset
\Hom_{\cG(\sA)}^1(X^*,X^*)$, since $d_X^2=0$.
 Furthermore, one has $d^\bec(d_X)=\sigma_Xd_X+d_X\sigma_X=
\id_{X^\bec}$; so $d_X$~is a contracting homotopy for the object
$X^\bec\in\bC(\sA)^\bec$.
 In other words, this means that the DG\+category $\bC(\sA)^\bec$ is
quasi-equivalent to the zero DG\+category.

 On the other hand, if $\sA$ is an idempotent-complete additive
category, then $\bC(\sA)$ is an idempotent-complete additive
DG\+category with twists.
 So the DG\+category $\bC(\sA)^{\bec\bec}$ is equivalent to~$\bC(\sA)$
by Proposition~\ref{iterated-bec-construction}.
 This example demonstrates that the passage from a DG\+category $\bA$
to the DG\+category $\bA^\bec$ does \emph{not} preserve
quasi-equivalences of DG\+categories, in general.

 Let $\sA$ be an additive category.
 Our next aim is to compute the additive category
$\sZ^0(\bC(\sA)^\bec)$, at least, in the case when $\sA$ is
idempotent-complete.
 Similarly to Proposition~\ref{G-plus-minus-prop}, the faithful,
conservative forgetful functor $\sC(\sA)\rarrow\sG(\sA)$ has faithful,
conservative left and right adjoint functors $G^+$ and
$G^-\:\sG(\sA)\rarrow\sC(\sA)$.
 Specifically, the functor $G^+$ assigns to a graded object $A^*\in
\sG(\sA)$ the complex $G^+(A^*)$ with the components $G^+(A^*)^n=
A^n\oplus A^{n-1}$ and the differential $d_{G^+,n}\:A^n\oplus A^{n-1}
\rarrow A^{n+1}\oplus A^n$ whose only nonzero component is the identity
map $A^n\rarrow A^n$.
 Dually, the complex $G^-(A^*)$ has the components $G^-(A^*)^n=
A^n\oplus A^{n+1}$ and the differential $d_{G^-,n}\:A^n\oplus A^{n+1}
\rarrow A^{n+1}\oplus A^{n+2}$ whose only nonzero component is
the identity map $A^{n+1}\rarrow A^{n+1}$.
 So one has $G^-(A^*)\simeq G^+(A^*)[1]$.

 The additive functor $\Upsilon=\Upsilon_\sA\:\sG(\sA)\rarrow
\sZ^0(\bC(\sA)^\bec)$ is constructed as follows.
 For any object $A^*\in\sG(\sA)$, put $X^*=G^+(A^*)$ and $d_X=d_{G^+}
\in\Hom^1_{\cG(\sA)}(X^*,X^*)$.
 Let $\sigma_X\in\Hom^{-1}_{\cG(\sA)}(X^*,X^*)$ be the endomorphism
defined by the rule that the morphism $\sigma_{X,n}\:
X^n=A^n\oplus A^{n-1}\rarrow A^{n-1}\oplus A^{n-2}=X^{n-1}$ has
the identity map $A^{n-1}\rarrow A^{n-1}$ as its only nonzero component.
 Then one has $d_X^2=0=\sigma_X^2$ and $d_X\sigma_X+\sigma_Xd_X=
\id_{X^*}$, so $X^\bec=(X^*,d_X,\sigma_X)$ is an object of
$\sZ^0(\bC(\sA)^\bec)$.
 We put $\Upsilon(A^*)=X^\bec$.
 The action of the functor $\Upsilon$ on morphisms is defined in
the obvious way.

 There are commutative diagrams of additive functors
\begin{equation} \label{complexes-tilde-Phi-diagram}
\begin{gathered}
 \xymatrix{
  \bC(\sA)^0 \ar@<0.4ex>[rr] \ar@{>->}[rd]_-{\widetilde\Phi_{\bC(\sA)}}
  && \sG(\sA) \ar@<0.4ex>@{-}[ll] \ar@{>->}[ld]^-{\Upsilon_\sA} \\
  & \sZ^0(\bC(\sA)^\bec)
 }
\end{gathered}
\end{equation}
and
\begin{equation} \label{complexes-Psi-diagram}
\begin{gathered}
 \xymatrix{
  \sG(\sA) \ar@{>->}[rr]^-{\Upsilon_\sA} \ar[d]_{G^+}
  && \sZ^0(\bC(\sA)^\bec) \ar[d]^{\Psi_{\bC(\sA)}^+}
  \\ \sC(\sA)\ar@{=}[rr] && \sZ^0(\bC(\sA))
 }
\end{gathered}
\end{equation}
 Here the upper horizontal double arrow
in~\eqref{complexes-tilde-Phi-diagram} is the obvious equivalence
of additive categories mentioned in Section~\ref{complexes-subsecn},
while the lower horizontal double line
in~\eqref{complexes-Psi-diagram} is essentially the definition of
the additive category of complexes~$\sC(\sA)$.
 The leftmost diagonal arrow in~\eqref{complexes-tilde-Phi-diagram}
is the fully faithful functor from
Lemma~\ref{Phi-Psi-extended-to-nonclosed-morphisms}, while
the rightmost vertical arrow in~\eqref{complexes-Psi-diagram} is
the faithful functor from Lemma~\ref{G-functors-in-DG-categories}.

 The additive functor $\Upsilon_\sA$ is always fully faithful
(since so is the additive functor~$\widetilde\Phi_{\bC(\sA)}$).
 When the additive category $\sA$ is idempotent-complete, the functor
$\Upsilon$ is an equivalence of additive categories.
 Given an object $X^\bec=(X^*,d_X,\sigma_X)\in\sZ^0(\bC(\sA)^\bec)$,
the corresponding object $A^*\in\sG(\sA)$ can be recovered as
the image of the idempotent endomorphism $\sigma_Xd_X\in
\Hom_{\cG(\sA)}^0(X^*,X^*)$.
 Consequently, the additive functor $\widetilde\Phi_{\bC(\sA)}$ is
also an equivalence of categories in this case.
\end{ex}

\begin{ex} \label{CDG-ring-bec-example}
 Let $\biR^\cu=(R^*,d,h)$ be a CDG\+ring.
 We will use the notation from Sections~\ref{curved-dg-modules-subsecn}
and~\ref{cdg-revisited-subsecn}; in particular, recall the notation
$\widehat\biR^\bu=(\widehat\biR^*,\d)$ for the graded ring
$R^*[\delta]$ with the changed sign of the grading,
$\widehat\biR^n=R^*[\delta]^{-n}$, endowed with
the differential $\d=\d/\d\delta$.
 So $\widehat\biR^\bu$ is an acyclic DG\+ring (the cohomology ring of
the DG\+ring $\widehat\biR^\bu$ vanishes).

 Consider the DG\+category $\bA=\biR^\cu\bmodl$ of left CDG\+modules
over~$(R^*,d,h)$.
 We claim that the DG\+category $\bA^\bec$ is naturally equivalent
(in fact, isomorphic) to the DG\+category of left DG\+modules
over~$\widehat\biR^\bu$,
\begin{equation} \label{bec-of-cdg-modules}
 (\biR^\cu\bmodl)^\bec=\widehat\biR^\bu\bmodl.
\end{equation}

 Indeed, by the definition, the objects of the DG\+category
$(\biR^\cu\bmodl)^\bec$ are triples $(X^*,d_X,\sigma_X)$, where $X^*$
is a graded left $R^*$\+module and $d_X$, $\sigma_X\:X^*\rarrow X^*$
are homogeneous endomorphisms of the graded abelian group $X^*$
of the degrees~$1$ and~$-1$, respectively.
 In fact, $(X^*,d_X)$ has to be a left CDG\+module over $(R^*,d,h)$,
so the equations
\begin{itemize}
\item $d_X(rx)=d(r)x+(-1)^{|r|}rd_X(x)$ and
\item $d_X^2(x)=hx$
\end{itemize}
must be satisfied for all $r\in R^{|r|}$ and $x\in X^{|x|}$.
 Furthermore, $\sigma_X$ must be a contracting homotopy with zero
square for the CDG\+module $(X^*,d_X)$, which means the equations
\begin{itemize}
\item $\sigma_X(rx)=(-1)^{|r|}r\sigma_X(x)$,
\item $d_X\sigma_X(x)+\sigma_Xd_X(x)=x$, and
\item $\sigma_X^2(x)=0$
\end{itemize}
for all $r\in R^{|r|}$ and $x\in X^{|x|}$. 

 On the other hand, the objects of the DG\+category
$\widehat\biR^\bu\bmodl$ are pairs $(N^*,\d_N)$, where $N^*$ is
a graded left $R^*[\delta]$\+module (with the sign of the grading
changed) and $\d_N:N^*\rarrow N^*$ is an odd derivation compatible
with the odd derivation~$\d$ on~$\widehat\biR^*$.
 To establish the equivalence of
DG\+categories~\eqref{bec-of-cdg-modules} on the level of objects,
one assigns to a DG\+module $(N^*,\d_N)$ over $\widehat\biR^\bu$
the same graded abelian group $X^*=N^*$.
 The graded $R^*$\+module structure on $X^*$ agrees with the one
on $N^*$, the differential~$d_X$ on $X^*$ is given by the action
of the element $\delta\in R^*[\delta]$ on $N^*$, and
the contracting homotopy~$\sigma_X$ on $(X^*,d_X)$ is given
by the action of the differential~$\d_N$ on~$N^*$.
 Symbolically, we put
$$
 d_X=\delta \quad\text{and}\quad\sigma_X=\d_N.
$$
 We leave it to the reader to establish the equivalence of
DG\+categories~\eqref{bec-of-cdg-modules} on the level of
the complexes of morphisms.
 Similarly one constructs an isomorphism of DG\+categories
$(\bmodr\biR^\cu)^\bec=\bmodr\widehat\biR^\bu$ for right CDG\+modules
and DG\+modules.

 Notice that the DG\+functor $\bec\bec\:\biR^\cu\bmodl\rarrow
(\biR^\cu\bmodl)^{\bec\bec}$ is an equivalence of DG\+categories by
Proposition~\ref{iterated-bec-construction}.
 On the other hand, viewing $\widehat\biR^\bu$ as a CDG\+ring with
zero curvature, one can apply to it the same construction which
produced $\widehat\biR^\bu$ from $\biR^\cu$, thus obtaining an acyclic
DG\+ring $\hathatRbu$.

 The underlying graded ring $R^*[\delta][\epsilon]$ of $\hathatRbu$
is obtained by adjoining an element~$\epsilon$ of degree~$-1$ to
the graded ring $R^*[\delta]$, and imposing the relations
$\epsilon\hat r-(-1)^{|\hat r|}\hat r\epsilon=\d(\hat r)$ for all
$\hat r\in R^*[\delta]$ (where $\d=\d/\d\delta$) and $\epsilon^2=0$.
 One can introduce the notation $d^\bec=\d=\d/\d\delta$, and then
have $d^{\bec\bec}=\d/\d\epsilon$ denote the differential in
the DG\+ring $\hathatRbu$.

 Explicitly, one has $\deg\delta=1$ and $\deg\epsilon=-1$, and
the defining relations in the graded ring $R^*[\delta][\epsilon]$ are
$\delta r-(-1)^{|r|}r\delta=d(r)$, \ $\epsilon r-(-1)^{|r|}r\epsilon=0$,
\ $\delta^2=h$, \ $\epsilon^2=0$, and $\epsilon\delta+\delta\epsilon=1$,
for all $r\in R^{|r|}$.
 The resulting DG\+ring is $\hathatRbu=
(R^*[\delta][\epsilon],d^{\bec\bec})$, where the differential
$d^{\bec\bec}=\d/\d\epsilon\:R^*[\delta][\epsilon]\rarrow
R^*[\delta][\epsilon]$ is defined by the rules $d^{\bec\bec}(r)=0$
for all $r\in R^*$, \ $d^{\bec\bec}(\delta)=0$, and
$d^{\bec\bec}(\epsilon)=1$.

 Applying~\eqref{bec-of-cdg-modules} twice, we obtain an isomorphism of
DG\+categories $(\biR^\cu\bmodl)^{\bec\bec}=\hathatRbu$.
 Thus \emph{the DG\+category of CDG\+modules over any CDG\+ring}
(in particular, the DG\+category of DG\+modules over any DG\+ring)
\emph{is equivalent to the DG\+category of DG\+modules over
an acyclic DG\+ring}.

 Let us now compute the abelian category $\sZ^0((\biR^\cu\bmodl)^\bec)$.
 The additive functor $\Upsilon=\Upsilon_{\biR^\cu}\:R^*\smodl\rarrow
\sZ^0((\biR^\cu\bmodl)^\bec)$ is constructed as follows.
 Let $M^*$ be a graded left $R^*$\+module.
 Consider the CDG\+module $(X^*,d_X)=(G^+(M^*),d_{G^+})$ over
$\biR^\cu$, as constructed in Proposition~\ref{G-plus-minus-prop}(a),
and endow it with the endomorphism
$\sigma_X\in\Hom_{R^*\cmodl}^{-1}(M^*,M^*)$ constructed as
the composition $G^+(M^*)\rarrow M^*[-1]\rarrow G^+(M^*)[-1]$ of
the maps in the short exact sequence of
Proposition~\ref{G-plus-minus-prop}(b).

 Explicitly, the elements of the group $X^n=G^+(M^*)^n$,
\,$n\in\Gamma$, are formal expressions $m+\delta m'$,
with $m\in M^n$ and $m'\in M^{n-1}$.
 The action of $R^*$ in $X^*=G^+(M^*)$ is given by the formula
$r(m+\delta m')=rm-(-1)^{|r|}d(r)m'+(-1)^{|r|}\delta rm'$ for all
$r\in R^{|r|}$, while the differential on $G^+(M^*)$ is
$d_X(m+\delta m')=d_{G^+}(m+\delta m')=hm'+\delta m$.
 The endomorphism $\sigma_X\:X^*\rarrow X^*$ is given by the rule
$\sigma_X(m+\delta m')=m'+\delta0$.
 Then one has $\sigma_X(rx)=(-1)^{|r|}r\sigma_X(x)$ for all
$r\in R^{|x|}$ and $x\in X^*$, \ $d_X\sigma_X+\sigma_Xd_X=\id_{X^*}$,
and $\sigma_X^2=0$, so $X^\bec=(X^*,d_X,\sigma_X)$ is an object
of $\sZ^0((\biR^\cu\bmodl)^\bec)$.
 We put $\Upsilon(M^*)=X^\bec$.
 The action of the functor $\Upsilon$ on morphisms is defined in
the obvious way.

 The functor $\Upsilon_{\biR^\cu}$ is an equivalence of abelian categories.
 Given an object $X^\bec=(X^*,d_X,\sigma_X)\in
\sZ^0((\biR^\cu\bmodl)^\bec)$, the corresponding graded $R^*$\+module
$M^*$ can be recovered as the image of the homogeneous map
$\sigma_X\in\Hom_{R^*}^{-1}(X^*,X^*)$, or as the image of
the idempotent homogeneous map $\sigma_Xd_X\in\Hom_{R^*}^0(X^*,X^*)$.
 Indeed, put $\upsilon_X=d_X-h\sigma_X$; then the pair of homogeneous
maps of graded abelian groups $\sigma_X$ and $\upsilon_X\:X^*\rarrow
X^*$ of degrees~$-1$ and~$1$, respectively, satisfies the equations
$\sigma_X^2=0=\upsilon_X^2$ and $\sigma_X\upsilon_X+\upsilon_X\sigma_X
=\id_{X^*}$; so it defines a representation of a graded $2\times 2$
matrix ring in $X^*$, similarly to the discussion in the proof of
Proposition~\ref{iterated-bec-construction}.

 There are commutative diagrams of additive functors
\begin{equation} \label{cdg-modules-tilde-Phi-diagram}
\begin{gathered}
 \xymatrix{
  (\biR^\cu\bmodl)^0 \ar@{>->}[rr]
  \ar@{>->}[rd]_-{\widetilde\Phi_{\biR^\subcu\bmodl}}
  && R^*\smodl \ar@<0.4ex>[ld]^-{\Upsilon_{\biR^\subcu}} \\
  & \sZ^0((\biR^\cu\bmodl)^\bec) \ar@<0.4ex>@{-}[ru]
 }
\end{gathered}
\end{equation}
and
\begin{equation} \label{cdg-modules-Psi-diagram}
\begin{gathered}
 \xymatrix{
  R^*\smodl \ar@<0.4ex>[rr]^-{\Upsilon_{\biR^\subcu}} \ar[rd]_-{G^+}
  && \sZ^0((\biR^\cu\bmodl)^\bec) \ar@<0.4ex>@{-}[ll]
  \ar[ld]^-{\Psi_{\biR^\subcu\bmodl}^+} \\ & \sZ^0(\biR^\cu\bmodl)
 }
\end{gathered}
\end{equation}
 Here the upper horizontal arrow
in~\eqref{cdg-modules-tilde-Phi-diagram} is the inclusion of
the full subcategory mentioned at the end of
Section~\ref{curved-dg-modules-subsecn}, while the leftmost
diagonal arrow in~\eqref{cdg-modules-tilde-Phi-diagram} is
the fully faithful functor from
Lemma~\ref{Phi-Psi-extended-to-nonclosed-morphisms}.
 The functor~$\Upsilon_{\biR^\cu}$ is an equivalence of categories,
but the other two fully faithful functors
in~\eqref{cdg-modules-tilde-Phi-diagram} are \emph{not}
essentially surjective, as it was explained in
Examples~\ref{nonexistence-of-cdg-structure-examples}
and~\ref{nonexistence-of-derivation-examples}.
 The diagonal arrows in~\eqref{cdg-modules-Psi-diagram} are
the faithful functors from Proposition~\ref{G-plus-minus-prop}
and Lemma~\ref{G-functors-in-DG-categories}.
\end{ex}

\begin{ex} \label{qcoh-CDG-bec-example}
 We are building up on the theory developed in
Sections~\ref{quasi-algebras-subsecn}\+-\ref{quasi-cdg-subsecn}
and Example~\ref{CDG-ring-bec-example}.
 Let $\biR^\cu=(R^*,d_R,h_R)$ and $\biS^\cu=(S^*,d_S,h_S)$ be
two CDG\+rings, and let $(f,a)\:\biR^\cu\rarrow\biS^\cu$ be a morphism
between them (so $a\in S^1$).
 Consider the graded rings $R^*[\delta_R]$ and $S^*[\delta_S]$ with
their odd derivations $\d_R=\d/\d\delta_R$ and $\d_S=\d/\d\delta_S$
of degree~$-1$ with zero squares.
 Then, following~\cite[Section~4.2]{Prel}, the morphism of
CDG\+rings $(f,a)$ induces a morphism of graded rings
$\hat f\:R^*[\delta_R]\rarrow S^*[\delta_S]$ given by the rules
$\hat f(r)=f(r)$ for all $r\in R^*$ and $f(\delta_R)=\delta_S+a$.
 The map~$\hat f$ commutes with the differentials $\d_R$ and~$\d_S$
on $R^*[\delta]$ and $S^*[\delta]$, i.~e., $f\circ\d_R=\d_S\circ f$.

 Let $L^*$ and $M^*$ be left graded modules over the graded rings
$R^*$ and $S^*$, respectively; and let $g\:L^*\rarrow S^*$ be a map
of graded modules compatible with the morphism of graded rings
$f\:R^*\rarrow S^*$.
 Then the induced map of CDG\+modules $G^+(g)\:G^+(L^*)\rarrow
G^+(M^*)$ compatible with the morphism of CDG\+rings
$(f,a)\:\allowbreak\biR^\cu\rarrow\biS^\cu$ is defined by the formula
$G^+(g)^n(l+\delta_Rl')=g(l)+ag(l')+\delta_Sg(l')$ for all $l\in L^n$
and $l'\in L^{n-1}$.
 The map $G^+(g)$ commutes with the structural endomorphisms
$\sigma_X$ and~$\sigma_Y$ of the objects
$\Upsilon_{\biR^\cu}(L^*)=X^\bec=(X,d_X,\sigma_X)\in
\sZ^0((\biR^\cu\bmodl)^\bec)$ and $\Upsilon_{\biS^\cu}(M^*)=Y^\bec=
(Y,d_Y,\sigma_Y)\in\sZ^0((\biS^\cu\bmodl)^\bec)$, that is,
$\sigma_Y\circ G^+(g)=G^+(g)\circ\sigma_X$.

 Let $\biB^\cu$ be a quasi-coherent CDG\+quasi-algebra over
a scheme~$X$, and let $B^*$ be its underlying sheaf of graded rings.
 Then one can construct the sheaf of graded rings $B^*[\delta]$
assigning to every affine open subscheme $U\subset X$ the graded ring
$B^*[\delta](U)=B^*(U)[\delta]$ associated with the CDG\+ring
$(B^*(U),d_U,h_U)$.
 The restriction maps $B^*[\delta](U)\rarrow B^*[\delta](V)$ for pairs
of affine open subschemes $V\subset U\subset X$ are induced by
the restriction morphisms of CDG\+rings $(\rho_{VU},a_{VU})\:
(B^*(U),d_U,h_U)\rarrow(B^*(V),d_V,h_V)$ in the way explained above.
 The composition $O_X\rarrow B^0\rarrow B^*[\delta]^0$ defines
a morphism of sheaves of rings $O_X\rarrow B^*[\delta]^0$ making
$B^*[\delta]$ a quasi-coherent graded quasi-algebra over~$X$
(as one can see from the short exact sequence~\eqref{Rdelta-sequence}).
{\emergencystretch=1em\par}

 Furthermore, the differential $\d=\d/\d\delta$ is well-defined in
the sheaf of graded rings $B^*[\delta]$, making $B^*[\delta]$
a sheaf of DG\+rings over~$X$ (with the differential of degree~$-1$).
 Changing the sign of the grading, one puts
$\widehat\biB^n=B^*[\delta]^{-n}$ for all $n\in\Gamma$; then
$\widehat\biB^\bu=(\widehat\biB^*,\d)$ is a sheaf of DG\+rings on $X$
with the differential of degree~$1$ (as in our usual convention).
 The complex of sheaves (or quasi-modules) $\widehat\biB^\bu$ is
acyclic as a complex of sheaves, and in fact, its sections over
every affine open subscheme $U\subset X$ form an acyclic complex;
but the complex $\widehat\biB^\bu(X)$ of global sections of
$\widehat\biB^\bu$ over $X$ need \emph{not} be acyclic,
generally speaking.

 In fact, the DG\+ring $\widehat\biB^\bu(X)$ is acyclic if and only if
the quasi-coherent CDG\+quasi-algebra $\biB^\cu$ ``admits a global
connecton'', meaning that it is isomorphic (in a suitable sense, i.~e.,
change-of-connection isomorphic) to a quasi-coherent CDG\+quasi-algebra
${}'\!\biB^\cu$ with vanishing change-of-connection elements $a'_{VU}=0$
for all affine open subschemes $V\subset U\subset X$.
 In particular, one can view $\widehat\biB^\bu$ as a quasi-coherent
CDG\+quasi-algebra over $X$ and apply the same construction again,
producing the sheaf of DG\+rings $\hathatBbu$ over~$X$.
 Its DG\+ring of global sections $\hathatBbu(X)$ is acyclic (and so
are the DG\+rings of sections $\hathatBbu(Y)$ over all open subschemes
$Y\subset X$).
 
 Similarly to Example~\ref{CDG-ring-bec-example}, there is a natural
equivalence (in fact, isomorphism) of DG\+categories
\begin{equation} \label{bec-of-qcoh-cdg-modules}
 (\biB^\cu\bqcoh)^\bec=\widehat\biB^\bu\bqcoh.
\end{equation}
 We leave the details to the reader.

 Let us compute the abelian category $\sZ^0((\biB^\cu\bqcoh)^\bec)$.
 Similarly to the previous examples, we start with constructing
the adjoint functors to the exact, faithful forgetful functor
$\sZ^0(\biB^\cu\bqcoh)\rarrow B^*\sqcoh$.
 The left adjoint functor $G^+\:B^*\sqcoh\rarrow\sZ^0(\biB^\cu\bqcoh)$
assigns to a quasi-coherent graded left $B^*$\+module $M^*$
the quasi-coherent left CDG\+module $G^+(M^*)$ over $\biB^\cu$ defined
by applying the functor $G^+$ from Proposition~\ref{G-plus-minus-prop}
for the CDG\+ring $(B^*(U),d_U,h_U)$ to the graded $B^*(U)$\+module
$M^*(U)$ for every affine open subset $U\subset X$.
 In other words, we put $G^+(M^*)(U)=G^+(M^*(U))$.
 The restriction maps $G^+(M^*)(U)\rarrow G^+(M^*)(V)$ are induced
by the restriction morphisms of CDG\+rings $(\rho_{VU},a_{VU})\:
(B^*(U),d_U,h_U)\rarrow(B^*(V),d_V,h_V)$ in the way explained above.
 The underlying graded $B^*$\+module of $G^+(M^*)$ is quasi-coherent
(i.~e., quasi-coherent as an $O_X$\+module) in view of
the short exact sequence in Proposition~\ref{G-plus-minus-prop}(b),
because the full subcategory of quasi-coherent sheaves is closed
under extensions in the category of all sheaves of $O_X$\+modules.
 The right adjoint functor $G^-$ is constructed similarly, and
only differs by a shift, $G^-=G^+[1]$.
{\hbadness=1600\par}

 The additive functor $\Upsilon=\Upsilon_{\biB^\cu}\:B^*\sqcoh\rarrow
\sZ^0((\biB^\cu\bqcoh)^\bec)$ is constructed as follows.
 Let $M^*$ be a quasi-coherent graded $B^*$\+module.
 Put $\biQ^\cu=G^+(M^*)\in\sZ^0(\biB^\cu\bqcoh)$, and let
$Q^*\in B^*\sqcoh$ denote the underlying quasi-coherent graded
$B^*$\+module of $G^+(M^*)$.
 Then the homogeneous endomorphism $\sigma_Q\in\Hom_{B^*}^{-1}(Q^*,Q^*)$
is defined by the rule $\sigma_Q(U)=\sigma_{Q(U)}$ for all affine open 
subschemes $U\subset X$, where $(\biQ^\cu(U),\sigma_{Q(U)})=
(Q^*(U),d_{Q(U)},\sigma_{Q(U)})=\Upsilon_{(B^*(U),d_U,h_U)}(M^*(U))$.
 Finally, we put $\Upsilon_{\biB^\cu}(M^*)=(\biQ^\cu,\sigma_Q)$.
 The action of the functor $\Upsilon_{\biB^\cu}$ on morphisms is easy
to define.

 The functor $\Upsilon_{\biB^\cu}$ is an equivalence of abelian
categories.
 Given an object $Q^\bec=(\biQ^\cu,\sigma_Q)\in
\sZ^0((\biB^\cu\bqcoh)^\bec)$, the corresponding quasi-coherent
graded $B^*$\+module $M^*$ can be recovered as the image of
the homogeneous map of quasi-coherent graded $B^*$\+modules
$\sigma_Q\:Q^*[1]\rarrow Q^*$ (where $Q^*$ is the underlying
quasi-coherent graded $B^*$\+module of the quasi-coherent
CDG\+module~$\biQ^\cu$).
 Locally over an affine open subscheme $U\subset X$,
the $B^*(U)$\+module $M^*(U)$ can be also recovered as the image
of the idempotent homogeneous map $\sigma_{Q(U)}d_{Q(U)}\in
\Hom^0_{B^*(U)}(Q^*(U),Q^*(U))$.

 There are commutative diagrams of additive functors
\begin{equation} \label{cdg-qcoh-tilde-Phi-diagram}
\begin{gathered}
 \xymatrix{
  (\biB^\cu\bqcoh)^0 \ar@{>->}[rr]
  \ar@{>->}[rd]_-{\widetilde\Phi_{\biB^\subcu\bqcoh}}
  && B^*\sqcoh \ar@<0.4ex>[ld]^-{\Upsilon_{\biB^\subcu}} \\
  & \sZ^0((\biB^\cu\bqcoh)^\bec) \ar@<0.4ex>@{-}[ru]
 }
\end{gathered}
\end{equation}
and
\begin{equation} \label{cdg-qcoh-Psi-diagram}
\begin{gathered}
 \xymatrix{
  B^*\sqcoh \ar@<0.4ex>[rr]^-{\Upsilon_{\biB^\subcu}} \ar[rd]_-{G^+}
  && \sZ^0((\biB^\cu\bqcoh)^\bec) \ar@<0.4ex>@{-}[ll]
  \ar[ld]^-{\Psi_{\biB^\subcu\bqcoh}^+} \\ & \sZ^0(\biB^\cu\bqcoh)
 }
\end{gathered}
\end{equation}
 Here all the functors in~\eqref{cdg-qcoh-tilde-Phi-diagram} are
fully faithful, but only $\Upsilon_{\biB^\cu}$ is essentially surjective
(generally speaking).
  The upper horizontal arrow in~\eqref{cdg-qcoh-tilde-Phi-diagram} is
the obvious inclusion, while the leftmost diagonal arrow
in~\eqref{cdg-qcoh-tilde-Phi-diagram} is the fully faithful functor
from Lemma~\ref{Phi-Psi-extended-to-nonclosed-morphisms}.
 The diagonal arrows in~\eqref{cdg-qcoh-Psi-diagram} are faithful
functors; the leftmost one was defined above, while the rightmost one
is the one from Lemma~\ref{G-functors-in-DG-categories}.
\end{ex}

\begin{ex} \label{factorizations-bec-example}
 Let $\sA$ be a preadditive category and $\Lambda\:\sA\rarrow\sA$ be
an autoequivalence.
 We will use the notation from Section~\ref{factorizations-subsecn},
and assume that either $\Gamma=\boZ$, or $\Lambda$ is involutive and
$\Gamma=\boZ/2$.
 Let $w\:\Id_\sA\rarrow\Lambda^2$ be a potential.

 The following description of the DG\+category $\bF(\sA,\Lambda,w)^\bec$
is obtained directly from the definitions.
 An object of $\bF(\sA,\Lambda,w)^\bec$ is a $\Lambda$\+periodic
object $X^\ci$ in $\sA$ endowed with a differential $d_X\in
\Hom_{\cP(\sA,\Lambda)}^1(X^\ci,X^\ci)$ and an endomorphism
$\sigma_X\in\Hom_{\cP(\sA,\Lambda)}^{-1}(X^\ci,X^\ci)$
such that $\biX^\cu=(X^\ci,d_X)$ is a factorization of~$w$ and
the equations $\sigma_X^2=0$, \ $d_X\sigma_X+\sigma_Xd_X=\id_{X^\ci}$
are satisfied.
 The complex of morphisms from an object $X^\bec=(X^\ci,d_X,\sigma_X)$
to an object $Y^\bec=(Y^\ci,d_Y,\sigma_Y)$ in the DG\+category
$\bF(\sA,\Lambda,w)^\bec$ is constructed as follows: for any
$n\in\Gamma$, the group $\Hom^n_{\bF(\sA,\Lambda,w)^\bec}
(X^\bec,Y^\bec)$ is the subgroup in $\Hom_{\cP(\sA,\Lambda)}^{-n}
(X^\ci,Y^\ci)$ consisting of all homogeneous morphisms~$f$ of
degree~$-n$ (in the $\Gamma$\+graded category of $\Lambda$\+periodic
objects) for which the equation $d_Yf-(-1)^nfd_X=0$ holds.
 The differential~$d^\bec$ in the complex
$\Hom^\bu_{\bF(\sA,\Lambda,w)^\bec}(X^\bec,Y^\bec)$ is given by
the usual rule $d^\bec(f)=\sigma_Yf-(-1)^{|f|}f\sigma_X$.

 Assume that $\sA$ is an additive category.
 Our aim is to compute the additive category
$\sZ^0(\bF(\sA,\Lambda,w)^\bec)$, at least, in the case when $\sA$ is
idempotent-complete.
 Similarly to Proposition~\ref{G-plus-minus-prop} and
Example~\ref{complexes-bec-example}, the faithful, conservative
forgetful functor $\sF(\sA,\Lambda,w)\rarrow\sP(\sA,\Lambda)$
has faithful, conservative left and right adjoint functors
$G^+$ and $G^-\:\sP(\sA,\Lambda)\rarrow\sF(\sA,\Lambda,w)$.

 Specifically, the functor $G^+$ assigns to a $\Lambda$\+periodic
graded object $A^\ci\in\sP(\sA,\Lambda)$ the factorization
$G^+(A^\ci)$ whose underlying $\Lambda$\+periodic graded object
is $A^\ci\oplus A^\ci[-1]$; so the grading components of $G^+(A^\ci)$
are $G^+(A^\ci)^n=A^n\oplus A^{n-1}\simeq A^n\oplus\Lambda^{-1}(A^n)$.
 The differential $d_{G^+}^n\:A^n\oplus A^{n-1}\rarrow
A^{n+1}\oplus A^n$ has two nonzero components: the identity map
$A^n\rarrow A^n$ and the map $A^{n-1}\rarrow A^{n+1}$ whose
composition with the periodicity isomorphism
$(\lambda_A^{n+1,n-1})^{-1}\:A^{n+1}\rarrow\Lambda^2(A^{n-1})$ is
the map $w_{A^{n-1}}\:A^{n-1}\rarrow\Lambda^2(A^{n-1})$.
 Dually, the factorization $G^-(A^\ci)$ has the underlying
$\Lambda$\+periodic graded object $A^\ci\oplus A^\ci[1]$,
so its grading components are $G^-(A^\ci)^n=A^n\oplus A^{n+1}
\simeq A^n\oplus\Lambda(A^n)$.
 The differential $d_{G^-}^n\:A^n\oplus A^{n+1}\rarrow
A^{n+1}\oplus A^{n+2}$ has two nonzero components: the identity map
$A^{n+1}\rarrow A^{n+1}$ and the map $A^n\rarrow A^{n+2}$ whose
composition with the periodicity isomorphism
$(\lambda_A^{n+2,n})^{-1}\:A^{n+2}\rarrow\Lambda^2(A^n)$ is
the map $w_{A^n}\:A^n\rarrow\Lambda^2(A^n)$.
 One has $G^-(A^\ci)\simeq G^+(A^\ci)[1]$.

 The additive functor $\Upsilon=\Upsilon_{\sA,\Lambda,w}\:
\sP(\sA,\Lambda)\rarrow\sZ^0(\bF(\sA,\Lambda,w)^\bec)$ is constructed
as follows (cf.\ the similar construction in
Example~\ref{complexes-bec-example}).
 For any object $A^\ci\in\sP(\sA,\Lambda)$, put $X^\ci=G^+(A^\ci)$
and $d_X=d_{G^+}\in\Hom^1_{\cP(\sA,\Lambda)}(X^\ci,Y^\ci)$.
 Let $\sigma_X\in\Hom^{-1}_{\cP(\sA,\Lambda)}(X^\ci,Y^\ci)$ be
the endomorphism defined by the rule that the morphism
$\sigma_{X,n}\:A^n\oplus A^{n-1}\rarrow A^{n-1}\oplus A^{n-2}$
has the identity map $A^{n-1}\rarrow A^{n-1}$ as its only nonzero
component.
 Then one has $\sigma_X^2=0$ and $d_X\sigma_X+\sigma_Xd_X=\id_{X^\ci}$,
so $X^\bec=(X^\ci,d_X,\sigma_X)$ is an object of
$\sZ^0(\bF(\sA,\Lambda,w)^\bec)$.
 We put $\Upsilon(A^\ci)=X^\bec$.
 The action of the functor $\Upsilon$ on morphisms is defined in
the obvious way.

 There are commutative diagrams of additive functors
\begin{equation} \label{factorizations-tilde-Phi-diagram}
\begin{gathered}
 \xymatrix{
  \bF(\sA,\Lambda,w)^0 \ar@{>->}[rr]
  \ar@{>->}[rd]_-{\widetilde\Phi_{\bF(\sA,\Lambda,w)}}
  && \sP(\sA,\Lambda) \ar@{>->}[ld]^-{\Upsilon_{\sA,\Lambda,w}} \\
  & \sZ^0(\bF(\sA,\Lambda,w)^\bec)
 }
\end{gathered}
\end{equation}
and
\begin{equation} \label{factorizations-Psi-diagram}
\begin{gathered}
 \xymatrix{
  \sP(\sA,\Lambda) \ar@{>->}[rr]^-{\Upsilon_{\sA,\Lambda,w}}
  \ar[d]_{G^+} && \sZ^0(\bF(\sA,\Lambda,w)^\bec)
  \ar[d]^{\Psi_{\bF(\sA,\Lambda,w)}^+}
  \\ \sF(\sA,\Lambda,w)\ar@{=}[rr] && \sZ^0(\bF(\sA,\Lambda,w))
 }
\end{gathered}
\end{equation}
 Here all the functors in~\eqref{factorizations-tilde-Phi-diagram}
are fully faithful.
 The upper horizontal arrow in~\eqref{factorizations-tilde-Phi-diagram}
is the obvious inclusion, while the lower horizontal double line
in~\eqref{factorizations-Psi-diagram} is the definition of the additive
category of factorizations $\sF(\sA,\Lambda,w)$.
 The leftmost diagonal arrow in~\eqref{factorizations-tilde-Phi-diagram}
is the fully faithful functor from
Lemma~\ref{Phi-Psi-extended-to-nonclosed-morphisms}, while the rightmost
vertical arrow in~\eqref{factorizations-Psi-diagram} is the faithful
functor from Lemma~\ref{G-functors-in-DG-categories}.

 The inclusion of $\sF(\sA,\Lambda,w)^0$ into $\sP(\sA,\Lambda)$ is
\emph{not} essentially surjective, generally speaking, as \emph{not
every $\Lambda$\+periodic object admits a differential making
it a factorization of~$w$}.
 This was explained in
Examples~\ref{nonexistence-of-cdg-structure-examples} (take
$R^*=k[x,x^{-1}]$ and $h=x$, or equivalently, the category
$\sA=k\rmodl\times k\rmodl$ with its natural autoequivalence $\Lambda$
switching the factors $k\rmodl$ and the potential $w=\id$;
cf.\ Remark~\ref{usual-factorizations-remark}).
 However, every object of $\sP(\sA,\Lambda)$ is a direct summand of
an object coming from $\sF(\sA,\Lambda,w)^0$; in fact,
any $\Lambda$\+periodic object $A^\ci\in\sP(\sA,\Lambda)$ is
a direct summand of the underlying $\Lambda$\+periodic object $X^\ci$
of the factorization $\biX^\cu=G^+(A^\ci)$.
 So the counterexample in
Examples~\ref{nonexistence-of-derivation-examples} has no counterpart
for factorizations.

 The additive functor $\Upsilon_{\sA,\Lambda,w}$ is always fully
faithful (one can deduce this, e.~g., from full-and-faithfulness of
the functor~$\widetilde\Phi_{\bF(\sA,\Lambda,w)}$ from
Lemma~\ref{Phi-Psi-extended-to-nonclosed-morphisms}).
 When the additive category $\sA$ is idempotent-complete,
the functor $\Upsilon$ is an equivalence of additive categories.
 Given an object $X^\bec=(X^\ci,d_X,\sigma_X)\in
\sZ^0(\bF(\sA,\Lambda,w)^\bec)$, the corresponding object
$A^\ci\in\sP(\sA,\Lambda)$ can be recovered as the image of
the idempotent endomorphism $\sigma_Xd_X\in
\Hom_{\cP(\sA,\Lambda)}^0(X^\ci,X^\ci)$.

 Indeed, one can define an endomorphism $\upsilon_X\in
\Hom_{\cP(\sA,\Lambda)}^1(X^\ci,X^\ci)$ by subtracting from~$d_X$
the composition of the action of the natural transformation
$w_{X^\ci}\:X^\ci\rarrow\Lambda^2(X^\ci)$ with
the periodicity isomorphism $\lambda_X^{*,*+2}\:\Lambda^2(X^\ci)
\simeq X^\ci[2]$ and with the endomorphism
$\sigma_X\in\Hom_{\cP(\sA,\Lambda)}^{-1}(X^\ci,X^\ci)$.
 Then the pair of homogeneous endomorphisms $\sigma_X$
and~$\upsilon_X$ of degrees $-1$ and~$1$, respectively, satisfies
the equations $\sigma_X^2=0=\upsilon_X^2$ and $\sigma_X\upsilon_X
+\upsilon_X\sigma_X=\id_{X^\ci}$; so it defines an action of
of the graded $2\times 2$ matrix ring in the graded category
object $X^\ci\in\cP(\sA,\Lambda)$, similarly to the discussion
in the proof of Proposition~\ref{iterated-bec-construction}
(cf.\ Example~\ref{CDG-ring-bec-example}).
\end{ex}

\Section{Abelian and Exact DG-Categories} \label{exact-dg-cats-secn}

\subsection{The functor of cone of identity endomorphism}
\label{Xi-functor-subsecn}
 Let $\bA$ be an additive DG\+category with shifts and cones.
 Recall the notation $\Xi=\Xi_\bA\:\sZ^0(\bA)\rarrow\sZ^0(\bA)$
introduced in Section~\ref{properties-of-main-construction-subsecn}
for the functor $\Xi(A)=\cone(\id_A[-1])$.

\begin{lem} \label{Xi-adjoints}
 The endofunctor\/ $\Xi\:\sZ^0(\bA)\rarrow\sZ^0(\bA)$ is left and
right adjoint to its own shift, viz., to the endofunctor\/
$\Xi[1]\:\sZ^0(\bA)\rarrow\sZ^0(\bA)$.
 Consequently, the endofunctor\/ $\Xi$ preserves all limits and
colimits (and in particular, all kernels and cokernels) that exist
in the additive category\/ $\sZ^0(\bA)$.
\end{lem}

\begin{proof}
 For any objects $A$ and $B\in\bA$, there are natural isomorphisms
of complexes of abelian groups
$$
 \Hom_\bA^\bu(\cone(\id_A[-1]),B)\simeq
 \cone(\id_{\Hom_\bA^\bu(A,B)})\simeq
 \Hom_\bA^\bu(A,\cone(\id_B))
$$
and
$$
 \Hom_\bA^\bu(B,\cone(\id_A[-1]))\simeq
 \cone(\id_{\Hom_\bA^\bu(B,A)}[-1])\simeq
 \Hom_\bA^\bu(\cone(\id_B),A).
$$
 The desired adjunction is obtained by passing to the groups of
degree~$0$ cocycles in these isomorphisms of complexes.
\end{proof}

 There is an obvious pair of natural transformations of additive
endofunctors $\Id_{\sZ^0(\bA)}[-1]\rarrow\Xi_\bA\rarrow\Id_{\sZ^0(\bA)}$
on the category $\sZ^0(\bA)$ assigning to an object $A\in\sA$
the natural closed morphisms $A[-1]\rarrow\Xi(A)\rarrow A$
of degree~$0$.

\begin{lem} \label{cone-kernel-cokernel}
 For any object $A\in\sZ^0(\bA)$, in the pair of natural morphisms
$A[-1]\rarrow\Xi(A)\rarrow A$, the former morphism is a kernel
of the latter one and the latter morphism is a cokernel of
the former one in the category\/ $\sZ^0(\bA)$.
\end{lem}

\begin{proof}
 More generally, for any closed morphism $f\:A\rarrow B$ of
degree~$0$ in the DG\+category $\bA$, in the pair of natural morphisms
$B[-1]\rarrow\cone(f[-1])\rarrow A$, the former morphism is a kernel of
the latter one and the latter morphism is a cokernel of the former one
in $\sZ^0(\bA)$.
 Let us check the first assertion.
 For any object $X\in\bA$, we have a short exact sequence of
complexes of abelian groups
$$
 0\rarrow\Hom_\bA^\bu(X,B[-1])\rarrow\Hom_\bA^\bu(X,\cone(f[-1]))
 \rarrow\Hom_\bA^\bu(X,A)\rarrow0.
$$
 Passing to the groups of degree~$0$ cocycles and using the fact
that the kernel functor is left exact in the category of abelian
groups, we obtain a left exact sequence
$$
 0\rarrow\Hom_{\sZ^0(\bA)}(X,B[-1])\rarrow\Hom_{\sZ^0(\bA)}
 (X,\cone(f[-1]))\rarrow\Hom_{\sZ^0(\bA)}(X,A),
$$
as desired.
 The second assertion is dual.
\end{proof}

\begin{lem} \label{Xi-faithful-and-conservative}
 The endofunctor\/ $\Xi\:\sZ^0(\bA)\rarrow\sZ^0(\bA)$ is faithful
and conservative.
\end{lem}

\begin{proof}
 The assertions follow from Lemma~\ref{Xi-decomposed},
telling that $\Xi_\bA\simeq\Psi^+_\bA\circ\Phi_\bA$, together with
Lemmas~\ref{G-functors-in-DG-categories} and~\ref{Phi-Psi-conservative},
claiming that the functors $\Psi^+_\bA$ and $\Phi_\bA$ are faithful
and conservative.
 To give a direct proof, the assertion that $\Xi$ is faithful follows
immediately from Lemma~\ref{cone-kernel-cokernel}.
 To show that $\Xi$ is conservative, one observes that the inclusion
functor $\sZ^0(\bA)\rightarrowtail\bA^0$ is conservative and
the composition of $\sZ^0(\bA)\overset\Xi\rarrow\sZ^0(\bA)
\rightarrowtail\bA^0$ has the inclusion functor as a direct summand.
\end{proof}

\begin{lem} \label{Xi-reflects-co-kernels}
 Let $A\overset f\rarrow B\overset g\rarrow C$ be a composable pair
of morphisms in the additive category\/ $\sZ^0(\bA)$.
 In this context \par
\textup{(a)} if the morphism\/ $\Xi(f)$ is a kernel of the morphism\/
$\Xi(g)$, then the morphism~$f$ is a kernel of the morphism~$g$; \par
\textup{(b)} if the morphism\/ $\Xi(g)$ is a cokernel of
the morphism\/ $\Xi(f)$, then the morphism~$g$ is a cokernel of
the morphism~$f$.
\end{lem}

\begin{proof}
 The functor $\Xi$ is faithful by
Lemma~\ref{Xi-faithful-and-conservative}; so the equation
$\Xi(g)\circ\Xi(f)=0$ implies $g\circ f=0$.
 To prove part~(a) (part~(b) is dual), consider the commutative diagram
\begin{equation} \label{3x3-Xi-diagram-with-shift}
\begin{gathered}
 \xymatrix{
  A \ar[rr]^f \ar@{>->}[d] && B \ar[rr]^g \ar@{>->}[d]
  && C \ar@{>->}[d] \\
  \Xi(A)[1] \ar[rr]^{\Xi(f)[1]} \ar@{->>}[d] && \Xi(B)[1]
  \ar[rr]^{\Xi(g)[1]} \ar@{->>}[d] && \Xi(C)[1] \ar@{->>}[d] \\
  A[1] \ar[rr]^{f[1]} && B[1] \ar[rr]^{g[1]} && C[1]
 }
\end{gathered}
\end{equation}
 If the morphism $\Xi(f)$ is a monomorphism, then so is the morphism
$\Xi(f)[1]$; since $A\rightarrowtail\Xi(A)[1]$ is a monomorphism
as well, it follows from commutativity of the upper leftmost square
that $f$~is a monomorphism.
 Now if $\Xi(f)$ is a kernel of $\Xi(g)$, then $\Xi(f)[1]$ is
a kernel of $\Xi(g)[1]$.
 Given an object $X\in\sZ^0(\bA)$ and a morphism $X\rarrow B$
for which the composition $X\rarrow B\overset g\rarrow C$ vanishes,
consider the composition $X\rarrow B\rarrow\Xi(B)[1]$ and notice
that the composition $X\rarrow B\rarrow\Xi(B)[1]\rarrow\Xi(C)[1]$
vanishes.
 So the composition $X\rarrow B\rarrow\Xi(B)[1]$ factorizes through
$\Xi(f)[1]$, and we obtain a morphism $X\rarrow\Xi(A)[1]$.
 Continuing the diagram chase and using
Lemma~\ref{cone-kernel-cokernel}, one proves that the morphism
$X\rarrow B$ factorizes through~$f$.
\end{proof}

\begin{lem} \label{tilde-Xi}
 For any DG\+category\/ $\bA$ with shifts and cones, the additive
functor\/ $\Xi_\bA\:\sZ^0(\bA)\rarrow\sZ^0(\bA)$ can be extended in
a natural way to a faithful additive functor\/ $\widetilde\Xi=
\widetilde\Xi_\bA\:\bA^0\rarrow\sZ^0(\bA)$.
\end{lem}

\begin{proof}
 The assertion follows immediately from Lemmas~\ref{Xi-decomposed}
and~\ref{Phi-Psi-extended-to-nonclosed-morphisms}.
 It suffices to put $\widetilde\Xi_\bA=\Psi^+_\bA\circ
\widetilde\Phi_\bA$.
\end{proof}

\begin{lem} \label{Xi-creates-co-kernels}
 Let $\bA$ be an idempotent-complete additive DG\+category with shifts 
and twists.
 In this setting \par
\textup{(a)} if $g$~is a morphism in\/ $\sZ^0(\bA)$ and the morphism\/
$\Xi(g)$ has a kernel in\/ $\sZ^0(\bA)$, then the morphism~$g$ has
a kernel in\/ $\sZ^0(\bA)$; \par
\textup{(b)} if $f$~is a morphism in\/ $\sZ^0(\bA)$ and the morphism\/
$\Xi(f)$ has a cokernel in\/ $\sZ^0(\bA)$, then the morphism~$f$ has
a cokernel in\/ $\sZ^0(\bA)$.
\end{lem}

\begin{proof}
 Let $g\:B\rarrow C$ be a morphism in $\sZ^0(\bA)$ such that
the morphism $\Xi(g)\:\Xi(B)\allowbreak\rarrow\Xi(C)$ has a kernel
$h\:E\rarrow\Xi(B)$ in~$\sZ^0(\bA)$.

 Let $A\in\bA$ be an object.
 Recall the notation $\iota_A$, $\pi_A$, $\iota'_A$, $\pi'_A$ from
the proofs of Lemma~\ref{G-functors-in-DG-categories}
and Proposition~\ref{iterated-bec-construction}.
 In particular, $\iota_A\in\Hom^1_\bA(A,\Xi(A))$ and
$\pi_A\in\Hom^0_\bA(\Xi(A),A)$ are closed morphisms,
$d(\iota_A)=0=d(\pi_A)$, while $\iota'_A\in\Hom^0_\bA(A,\Xi(A))$ and
$\pi'_A\in\Hom^{-1}_\bA(A,\Xi(A))$ are morphisms satisfying
$d(\pi'_A)=\pi_A$ and $d(\iota'_A)=\iota_A$. {\emergencystretch
=0.5em\par}

 Introduce the notation $s_A\in\Hom^1_\bA(A,A[-1])$ for the canonical
shift isomorphism (by the definition of a shift, one has $d(s_A)=0$).
 Put $\phi_A=(\iota'_A,\iota_As_A^{-1})\:A\oplus A[-1]\rarrow\Xi(A)$
and $\psi_A=(\pi_A,s_A\pi'_A)\:\Xi(A)\rarrow A\oplus A[-1]$; then
$\phi_A\in\Hom^0_\bA(A\oplus A[-1],\>\Xi(A))$ and
$\psi_A\in\Hom^0_\bA(\Xi(A),\>A\oplus A[-1])$ are mutually inverse
isomorphisms in the additive category~$\bA^0$.
 As usually, we put $\tau_A=\iota_A\pi_A\in\Hom^1_\bA(\Xi(A),\Xi(A))$;
then $d(\phi_A)=(\iota_A,0)=\tau_A\phi_A$ and $d(\psi_A)=(0,-s_A\pi_A)
=-\psi_A\tau_A$.
 So $-\tau_A$ is a Maurer--Cartan cochain (indeed, $d(\tau_A)=0=
\tau_A^2$), and the morphisms $\phi_A$ and~$\psi_A$ establish
an isomorphism $A\oplus A^{-1}=\Xi(A)(-\tau_A)$ in
the DG\+category~$\bA$.

 Let $X\in\bA$ be an object.
 Introduce further notation $t_X\in\Hom^{-1}_\bA(X,X[1])$
for the canonical shift isomorphism in $\cZ(\bA)$ connecting
the objects $X$ and $X[1]$ (so $d(t_X)=0$).
 Then $t_{\Xi(A)}\tau_A\in\Hom^0_\bA(\Xi(A),\Xi(A)[1])$ is a morphism
in the additive category $\sZ^0(\bA)$.

 Returning to the situation at hand, we have a commutative diagram of
solid arrows in~$\sZ^0(\bA)$
$$
 \xymatrix{
  0 \ar[r] & E \ar[rr]^-h \ar@{-->}[d]^a
  && \Xi(B) \ar[rrr]^-{\Xi(g)} \ar[d]^{t_{\Xi(B)}\tau_B}
  &&& \Xi(C) \ar[d]^{t_{\Xi(C)}\tau_C} \\
  0 \ar[r] & E[1] \ar[rr]^-{h[1]} \ar@{-->}[d]^{a[1]}
  && \Xi(B)[1] \ar[rrr]^-{\Xi(g)[1]} \ar[d]^{(t_{\Xi(B)}\tau_B)[1]}
  &&& \Xi(C)[1] \ar[d]^{(t_{\Xi(C)}\tau_C)[1]} \\
  0 \ar[r] & E[2] \ar[rr]^-{h[2]} && \Xi(B)[2]
  \ar[rrr]^-{\Xi(g)[2]} &&& \Xi(C)[2]
 }
$$
 Since the morphism~$h$ is a kernel of $\Xi(g)$ and the morphism~$h[1]$
is a kernel of $\Xi(g)[1]$, there exists a unique morphism
$a\:E\rarrow E[1]$ in $\sZ^0(\bA)$ making the leftmost square
commutative.
 Since the composition $(t_{\Xi(B)}\tau_B)[1]\circ(t_{\Xi(B)}\tau_B)
=t_{\Xi(B)}[1]t_{\Xi(B)}\tau_B^2$ vanishes in $\sZ^0(\bA)$, so
does the composition $a[1]\circ a$.

 Put $\alpha=t_E^{-1}a\in\Hom_\bA^1(E,E)$.
 Then $d(\alpha)=0=\alpha^2$, so $-\alpha$~is a Maurer--Cartan cochain.
 By assumption, there exists an object $S=E(-\alpha)\in\bA$
together with mutually inverse isomorphisms
$\tilde\phi\in\Hom_\bA^0(S,E)$ and $\tilde\psi\in\Hom_\bA^0(E,S)$
such that $d(\tilde\phi)=\alpha\tilde\phi$
and $d(\tilde\psi)=-\tilde\psi\alpha$.
 Put $u=\psi_Bh\tilde\phi\in\Hom^0_\bA(S,\>B\oplus B[-1])$.
 Then $d(u)=-\psi_B\tau_Bh\tilde\phi+\psi_Bh\alpha\tilde\phi=0$, 
so $u\:S\rarrow B\oplus B[-1]$ is a morphism in
the category~$\sZ^0(\bA)$.
 We have a composable pair of morphisms
\begin{equation} \label{kernel-of-Xi-untwisted}
 \xymatrix{
  S \ar[r]^-u & B\oplus B[-1]\ar[rr]^{g\oplus g[-1]}
  && C\oplus C[-1]
 } 
\end{equation}
in the additive category~$\sZ^0(\bA)$.

 Our next aim is to prove that the morphism~$u$ is a kernel of
the morphism $g\oplus g[-1]$ in the additive category~$\sZ^0(\bA)$.
 For this purpose, we observe that, in the ambient additive
category $\bA^0$, the composable pair of
morphisms~\eqref{kernel-of-Xi-untwisted} is isomorphic to
the composable pair of morphisms
\begin{equation} \label{kernel-of-Xi-original}
 \xymatrix{
  E \ar[r]^-h & \Xi(B) \ar[rr]^{\Xi(g)} && \Xi(C),
 } 
\end{equation}
where $h$~is a kernel of $\Xi(g)$ by our original assumption.
 Indeed, the morphisms~$\tilde\psi$, \,$\psi_B$, and~$\psi_C$,
together with their inverse morphisms~$\tilde\phi$, \,$\phi_B$,
and~$\phi_C$, provide an isomorphism
between~\eqref{kernel-of-Xi-untwisted} and~\eqref{kernel-of-Xi-original}
in~$\bA^0$.

 By Lemma~\ref{tilde-Xi}, it follows that the pair of morphisms
obtained by applying $\Xi$ to~\eqref{kernel-of-Xi-untwisted} is
isomorphic to the pair of morphisms obtained by applying~$\Xi$
to~\eqref{kernel-of-Xi-original} in the category~$\sZ^0(\bA)$.
 Since the functor $\Xi$ preserves kernels by Lemma~\ref{Xi-adjoints},
the morphism $\Xi(h)$ is a kernel of the morphism $\Xi(\Xi(g))$,
and therefore the morphism $\Xi(u)$ is a kernel of the morphism
$\Xi(g\oplus g[-1])$ in~$\sZ^0(\bA)$.
 According to Lemma~\ref{Xi-reflects-co-kernels}(a), we can conclude
that the morphism~$u$ is a kernel of $g\oplus g[-1]$ in~$\sZ^0(\bA)$.

 The rest of the proof is now straightforward.
 Given an object $A\in\bA$, denote by $p_A\:A\oplus A[-1]\rarrow
A\oplus A[-1]$ the idempotent projector on the direct summand $A$
in $A\oplus A[-1]$.
 Then we have $p_C\circ(g\oplus g[-1])=(g\oplus g[-1])\circ p_B$.
 An argument similar to the one above, using the fact that $u$~is
a kernel of $g\oplus g[-1]$, shows that there exists a unique
morphism $q\:S\rarrow S$ in $\sZ^0(\bA)$ such that $p_Bu=uq$
in $\sZ^0(\bA)$.
 Moreover, one similarly shows that the equation $p_B^2=p_B$
implies $q^2=q$.

 Since the additive category $\sZ^0(\bA)$ is idempotent-complete
by assumption, the idempotent endomorphism $q\:S\rarrow S$ has
an image $A\in\sZ^0(\bA)$.
 There exists a unique morphism $f\:A\rarrow B$ in $\sZ^0(\bA)$
such that the composition $A\rightarrowtail S\overset u\rarrow
B\oplus B[-1]$ is equal to the composition $A\overset f\rarrow B
\rightarrowtail B\oplus B[-1]$.
 One easily shows that the morphism~$f$ is a kernel of~$g$
in~$\sZ^0(\bA)$.
 This proves part~(a); part~(b) is dual.
\end{proof}

\subsection{DG-compatible exact structures}
\label{DG-compatible-subsecn}
 An \emph{exact category} $\sE$ (in the sense of Quillen) is an additive
category endowed with a class of (\emph{admissible}) \emph{short exact
sequences} (or ``conflations'') $0\rarrow E'\rarrow E\rarrow E''
\rarrow0$ satisfying the axioms for which we refer the reader
to~\cite[Appendix~A]{Kel}, \cite{Bueh}, or~\cite[Appendix~A]{Partin}.
 If $0\rarrow E'\rarrow E\rarrow E''\rarrow0$ is an admissible short
exact sequence, then the morphism $E'\rarrow E$ is said to be
an \emph{admissible monomorphism} (or ``inflation''), while
the morphism $E\rarrow E''$ is said to be an \emph{admissible
epimorphism} (or ``deflation'').

 Given two exact categories $\sE$ and $\sF$, an additive functor
$\Theta\:\sE\rarrow\sF$ is said to be \emph{exact} if it preserves
admissible short exact sequences, i.~e., takes admissible short exact
sequences to admissible short exact sequences.
 Obviously, an exact functor $\Theta$ also takes admissible
monomorphisms to admissible monomorphisms and admissible epimorphisms
to admissible epimorphisms.

 We will say that an additive functor $\Theta\:\sE\rarrow\sF$ acting
between two exact categories $\sE$ and $\sF$ \emph{reflects admissible
short exact sequences} if, for any composable pair of morphisms
$E'\rarrow E\rarrow E''$ in $\sE$, the condition that $0\rarrow
\Theta(E')\rarrow\Theta(E)\rarrow\Theta(E'')\rarrow0$ is an admissible
short exact sequence in $\sF$ implies that $0\rarrow E'\rarrow E\rarrow
E''\rarrow0$ is an admissble short exact sequence in~$\sE$.
 An additive functor $\Theta$ is said to \emph{reflect admissible
monomorphisms} if, given a morphism~$f$ in $\sE$, the condition that
$\Theta(f)$ is an admissible monomorphism in $\sF$ implies that $f$~is
an admissible monomorphism in~$\sE$.
 Dually, $\Theta$ \emph{reflects admissible epimorphisms} if
a morphism~$g$ in $\sE$ is an admissible epimorphism whenever
$\Theta(g)$ is an admissible epimorphism in~$\sF$.
 
 Let $\bE$ be an additive DG\+category with shifts and cones.
 We will say that an exact category structure on the additive category
$\sZ^0(\bE)$ is \emph{DG\+compatible} (with the DG\+category~$\bE$)
if the following two conditions hold:
\begin{itemize}
\item the exact structure on $\sZ^0(\bE)$ is preserved by
the shift functors~$[1]$ and~$[-1]$;
\item for any object $E\in\sZ^0(\bE)$, the natural sequence
$0\rarrow E[-1]\rarrow\Xi_\bE(E)\rarrow E\rarrow0$ from
Section~\ref{Xi-functor-subsecn} is an admissible short exact
sequence in~$\sZ^0(\bE)$.
\end{itemize}

\begin{prop} \label{DG-compatible-prop}
 Let\/ $\bE$ be an additive DG\+category with shifts and cones.
 Then an exact category structure on\/ $\sZ^0(\bE)$ that is preserved
by the shift functors\/ $[1]$ and\/~$[-1]$ is DG\+compatible
if and only if the additive functor\/ $\Xi_\bE\:\sZ^0(\bE)\rarrow
\sZ^0(\bE)$ preserves and reflects admissible short exact sequences.
\end{prop}

\begin{proof}
 ``If'': the point is that the sequence $0\rarrow E[-1]\rarrow\Xi(E)
\rarrow E\rarrow0$ becomes split exact after the functor $\Xi$ is
applied.
 Indeed, this sequence is split exact in the additive category
$\bE^0$, so it remains to use Lemma~\ref{tilde-Xi}
(cf.\ Lemma~\ref{twists-into-isomorphisms}).
 All split short exact sequences are always admissible; so if $\Xi$
reflects admissible short exact sequences, then $0\rarrow E[-1]
\rarrow\Xi(E)\rarrow E\rarrow0$ is admissible.

 ``Only if'': both the preservation and reflection of admissible
short exact sequences can be seen from
the commutative diagram~\eqref{3x3-Xi-diagram-with-shift}.
 Let us redraw the diagram, omitting the unnecessary shift:
\begin{equation} \label{3x3-Xi-diagram}
\begin{gathered}
 \xymatrix{
  A[-1] \ar[rr]^{f[-1]} \ar@{>->}[d] && B[-1]
  \ar[rr]^{g[-1]} \ar@{>->}[d] && C[-1] \ar@{>->}[d] \\
  \Xi(A) \ar[rr]^{\Xi(f)} \ar@{->>}[d] && \Xi(B)
  \ar[rr]^{\Xi(g)} \ar@{->>}[d] && \Xi(C) \ar@{->>}[d] \\
  A \ar[rr]^f && B \ar[rr]^g && C
 }
\end{gathered}
\end{equation}
 Here $A\overset f\rarrow B\overset g\rarrow C$ is an arbitrary
composable pair of morphisms in~$\bE$.
 By assumption, the columns in~\eqref{3x3-Xi-diagram}
are admissible short exact sequences.

 Now if $0\rarrow A\rarrow B\rarrow C\rarrow0$ is an admissible
short exact sequence, then so is $0\rarrow A[-1]\rarrow B[-1]
\rarrow C[-1]\rarrow0$.
 Furthermore, notice that $\Xi(g)\circ\Xi(f)=\Xi(g\circ f)=0$,
since $\Xi$ is an additive functor.
 It remains to recall that the class of all admissible short
exact sequences is closed under (admissible) extensions within
the class of all pairs of morphisms with zero composition in
any exact category~\cite[Corollary~3.6(ii)]{Bueh}.

 Conversely, assume that the short sequence $0\rarrow\Xi(A)\rarrow
\Xi(B)\rarrow\Xi(C)\rarrow0$ is admissible.
 The point is that it then follows
from~\eqref{3x3-Xi-diagram} that the upper and lower rows are
admissible short exact sequences as well, because they are each
other's shifts, so their exactness properties are the same.
 Specifically, applying the following lemma to $\sE=\sZ^0(\bE)$
allows to finish the proof.
\end{proof}

\begin{lem}
 Let\/ $\sE$ be an exact category endowed with an autoequivalence\/
$[1]\:\sE\rarrow\sE$ such that both\/~$[1]$ and its inverse
functor\/~$[-1]$ preserve admissible short exact sequences.
 Suppose that we are given a commutative diagram in\/~$\sE$
\begin{equation} \label{3x3-abstract-diagram}
\begin{gathered}
 \xymatrix{
  A[-1] \ar[rr]^{f[-1]} \ar@{>->}[d] && B[-1]
  \ar[rr]^{g[-1]} \ar@{>->}[d] && C[-1] \ar@{>->}[d] \\
  X \ar[rr] \ar@{->>}[d] && Y \ar[rr] \ar@{->>}[d]
  && Z \ar@{->>}[d] \\ A \ar[rr]^f && B \ar[rr]^g && C
 }
\end{gathered}
\end{equation}
where the columns and the middle row are admissible short exact 
sequences.
 Then the upper and lower rows are admissible short exact sequences
as well.
\end{lem}

\begin{proof}
 Let us be lazy and use the embedding
theorem~\cite[Proposition~A.2]{Kel}, \cite[Appendix~A]{Bueh},
\cite[Section~A.6]{Partin}.
 It is convenient to use~\cite[Lemma~A.3]{Kel} (or rather its obvious
extension to exact categories with a shift), which allows to restrict
the consideration to a small subcategory, thus avoiding
set-theoretical issues.
 So we can assume that there exists an abelian category $\sA$ and
a fully faithful functor $\theta\:\sE\rarrow\sA$ representing $\sE$ as
a fully exact subcategory in~$\sA$.
 Notice that the construction in the embedding theorem is natural
(with respect to equivalences of exact categories), and therefore
the shift autoequivalence $[1]\:\sE\rarrow\sE$ induces a similar
autoequivalence $[1]\:\sA\rarrow\sA$ such that the embedding
functor~$\theta$ commutes with the shifts in $\sE$ and in~$\sA$.

 Applying $\theta$ to~\eqref{3x3-abstract-diagram}, we get a commutative
$3\times 3$ diagram in $\sA$ with exact columns and an exact middle row.
 Now it follows that the upper leftmost morphism $\theta(A[-1])\rarrow
\theta(B[-1])$ is a monomorphism; hence so is the lower leftmost
morphism $\theta(A)\rarrow\theta(B)$.
 Dually, it follows that the lower rightmost morphism $\theta(B)\rarrow
\theta(C)$ is an epimorphism; hence so is the upper rightmost morphism
$\theta(B[-1])\rarrow\theta(C[-1])$.
 Using the cohomology long exact sequence for our short exact sequence
of complexes in $\sA$, one concludes that all the rows are exact in
$\sA$, hence also in~$\sE$.
\end{proof}

\begin{cor} \label{cone-sequences-universally-admissible-exact}
 Let\/ $\bE$ be an additive DG\+category with shifts and cones.
 Then, for any closed morphism $f\:A\rarrow B$ of degree\/~$0$ in\/
$\bE$, the natural pair $B\rarrow\cone(f)\rarrow A[1]$ of closed
morphisms of degree\/~$0$ in\/ $\bE$ is an admissible short exact
sequence\/ $0\rarrow B\rarrow\cone(f)\rarrow A[1]\rarrow0$ in any
DG\+compatible exact category structure on\/~$\sZ^0(\bE)$.
\end{cor}

\begin{proof}
 The sequence $0\rarrow B\rarrow\cone(f)\rarrow A[1]\rarrow0$ is split
exact in the additive category~$\bE^0$.
 By Lemma~\ref{tilde-Xi} (cf.\ Lemma~\ref{twists-into-isomorphisms}),
it follows that the functor $\Xi_\bE$ takes this sequence to a split
short exact sequence in\/~$\sZ^0(\bE)$.
 In view of the ``only if'' assertion of
Proposition~\ref{DG-compatible-prop}, we can conclude that $0\rarrow B
\rarrow\cone(f)\rarrow A[1]\rarrow0$ is an admissible short exact
sequence in any DG\+compatible exact structure on\/~$\sZ^0(\bE)$.

 Alternatively, one can construct the desired admissible short exact
sequence $0\rarrow B\rarrow\cone(f)\rarrow A[1]\rarrow0$ as
the pushout of the admissible short exact sequence $0\rarrow A\rarrow
\cone(\id_A)\rarrow A[1]\rarrow0$ by the morphism~$f$, or as
the pullback of the admissible short exact sequence $0\rarrow B\rarrow
\cone(\id_B)\rarrow B[1]\rarrow0$ by the morphism $f[1]\:A[1]
\rarrow B[1]$.
\end{proof}

 Let us say that an additive DG\+category $\bE$ is \emph{weakly
idempotent-complete} if the additive category $\sZ^0(\bE)$ is weakly
idempotent-complete in the sense of~\cite[Section~7]{Bueh}.
 One can easily see that if an additive DG\+category $\bE$ is
weakly idempotent-complete, then so is the DG\+category~$\bE^\bec$.

\begin{lem} \label{DG-compatible-reflecting-admissible-morphisms}
 Let\/ $\bE$ be a weakly idempotent-complete additive DG\+category
with shifts and cones.
 Then, for any DG\+compatible exact category structure on\/
$\sZ^0(\bE)$, the functor\/ $\Xi_\bE\:\sZ^0(\bE)\rarrow
\sZ^0(\bE)$ reflects admissible monomorphisms and admissible
epimorphisms.
\end{lem}

\begin{proof}
 Let $f\:A\rarrow B$ be a morphism in $\sZ^0(\bE)$ such that $\Xi(f)$
is an admissible monomorphism.
 Then the morphism $\Xi(f)[1]$ is an admissible monomorphism as well.
 The natural morphism $A\rarrow\Xi(A)[1]$ is an admissible monomorphism
by the definition of a DG\+compatible exact structure.
 Hence the composition $A\rarrow\Xi(A)[1]\rarrow\Xi(B)[1]$ is
an admissible monomorphism.
 This composition factorizes as $A\rarrow B\rarrow\Xi(B)[1]$.
 Since the additive category $\sZ^0(\bE)$ is weakly idempotent-complete
by assumption, it follows that the morphism $A\rarrow B$ is
an admissible monomorphism (use the result dual
to~\cite[Proposition~7.6]{Bueh}).
 Thus $\Xi$ reflects admissible monomorphisms; the argument for
admissible epimorphisms is dual.
\end{proof}

\subsection{$\bec$-compatible exact structures}
\label{bec-compatible-subsecn}
 Let\/ $\bE$ be an additive DG\+category with shifts and cones.
 Assume that we are given two exact category structures, one on
the additive category $\sZ^0(\bE)$ and the other one on
the additive category $\sZ^0(\bE^\bec)$.
 We will say that the two exact structures are
\emph{$\bec$\+compatible} (with each other) if
\begin{itemize}
\item both the exact structures are preserved by the shifts~$[1]$
and~$[-1]$;
\item both the additive functors
$$
 \Phi_\bE\:\sZ^0(\bE)\lrarrow\sZ^0(\bE^\bec)
 \quad\text{and}\quad
 \Psi^+_\bE\:\sZ^0(\bE^\bec)\lrarrow\sZ^0(\bE)
$$
from Lemma~\ref{G-functors-in-DG-categories}
preserve and reflect admissible short exact sequences.
\end{itemize}
 Clearly, in a pair of $\bec$\+compatible exact structures on
$\sZ^0(\bE)$ and $\sZ^0(\bE^\bec)$ any one of the two exact
structures determines uniquely the other one.

 We define an \emph{exact DG\+category} $\bE$ as an additive
DG\+category with shifts and cones endowed with a $\bec$\+compatible
pair of exact structures on the additive categories $\sZ^0(\bE)$ and
$\sZ^0(\bE^\bec)$.

\begin{lem} \label{DG-compatible-vs-bec-compatible}
 A pair of exact structures on the additive categories\/
$\sZ^0(\bE)$ and\/ $\sZ^0(\bE^\bec)$ is\/ $\bec$\+compatible if and
only if the following two conditions hold: both the exact structures
are DG\+compatible (in the sense of the definition in
Section~\ref{DG-compatible-subsecn}) with the respective
DG\+categories\/ $\bE$ and\/ $\bE^\bec$, and both the additive
functors\/ $\Phi_\bE$ and\/ $\Psi^+_\bE$ are exact with respect to
these exact category structures.
\end{lem}

\begin{proof}
 Assuming that both the exact category structures are preserved by
the shifts and both the functors $\Phi_\bE$ and $\Psi^+_\bE$ preserve
admissible short exact sequences, we have to prove that both these
functors reflect admissible short exact sequences if and only if
both the exact structures are DG\+compatible.
 By Proposition~\ref{DG-compatible-prop}, the condition that both
the exact structures are DG\+compatible means that both the functors
$\Xi_\bE$ and $\Xi_{\bE^\bec}$ preserve and reflect admissible
short exact sequences.
 By Lemma~\ref{Xi-decomposed}, we have
$\Xi_\bE\simeq\Psi^+_\bE\circ\Phi_\bE$ and
$\Xi_{\bE^\bec}\simeq\Phi_\bE\circ\Psi^-_\bE$
(where $\Psi^-_\bE=\Psi^+_\bE[1]$).

 Now if both the functors $\Phi_\bE$ and $\Psi^+_\bE$ preserve
(respectively, reflect) admissible short exact sequences, then
their compositions also preserve (resp., reflect) them.
 Conversely, if the functor $\Psi^+_\bE$ preserves admissible
short exact sequences, while the functor $\Psi^+_\bE\circ\Phi_\bE$
reflects them, then the functor $\Phi_\bE$ also reflects them.
 Similarly, if the functor $\Phi_\bE$ preserves admissible short
exact sequences, while the functor $\Phi_\bE\circ\Psi^-_\bE$ reflects
them, then the functor $\Psi^-_\bE$ also reflects them.
\end{proof}

\begin{lem} \label{exact-DG-reflecting-admissible-morphisms}
 Let\/ $\bE$ be a weakly idempotent-complete additive DG\+category
with shifts and cones.
 Then, for any exact DG\+category structure on\/ $\bE$, both
the additive functors\/ $\Phi_\bE$ and\/ $\Psi^+_\bE$ reflect
admissible monomorphisms and admissible epimorphisms.
\end{lem}

\begin{proof}
 Follows from Lemmas~\ref{DG-compatible-reflecting-admissible-morphisms}
and~\ref{DG-compatible-vs-bec-compatible}.
 To wit, if the functor $\Psi^+_\bE$ preserves admissible
short exact sequences (hence also admissible monomorphisms),
while the functor $\Xi_\bE\simeq\Psi^+_\bE\circ\Phi_\bE$
reflects admissible monomorphisms, then the functor $\Phi_\bE$
also reflects admissible monomorphisms.
 If the functor $\Phi_\bE$ preserves admissible short
exact sequences (hence also admissible monomorphisms), while
the functor $\Xi_{\bE^\bec}\simeq\Phi_\bE\circ\Psi^-_\bE$
reflects admissible monomorphisms, then the functor $\Psi^-_\bE$
also reflects admissible monomorphisms.
 The same argument for admissible epimorphisms.
\end{proof}

\begin{lem} \label{faithful-conservative-lifting-exact-structure}
 Let\/ $\sE$ be an additive category, $\sF$ be an exact category,
and\/ $\Theta\:\sE\rarrow\sF$ be a faithful, conservative additive
functor satisfying the following conditions:
\begin{itemize}
\item if $g\:B\rarrow C$ is a morphism in\/ $\sE$ such that\/
$\Theta(g)$ is an admissible epimorphism in\/ $\sF$, then
the morphism~$g$ has a kernel $f\:A\rarrow B$ in\/ $\sE$ and
the morphism\/ $\Theta(f)$ is a kernel of\/ $\Theta(g)$ in\/~$\sF$;
\item if $f\:A\rarrow B$ is a morphism in\/ $\sE$ such that\/
$\Theta(f)$ is an admissible monomorphism in\/ $\sF$, then
the morphism~$f$ has a cokernel $g\:B\rarrow C$ in\/ $\sE$ and
the morphism\/ $\Theta(g)$ is a cokernel of\/ $\Theta(f)$ in\/~$\sF$.
\end{itemize}
 Then the class of all composable pairs of morphisms $A\rarrow B
\rarrow C$ in\/ $\sE$ whose image under\/ $\Theta$ is an admissible
short exact sequence\/ $0\rarrow\Theta(A)\rarrow\Theta(B)\rarrow
\Theta(C)\rarrow0$ in\/ $\sF$ defines an exact category structure
on\/~$\sE$.
 The functor\/ $\Theta$ preserves and reflects admissible short
exact sequences, admissible monomorphisms, and admissible epimorphisms
in this exact category structure on\/~$\sE$.
\end{lem}

\begin{proof}
 Let $A\overset f\rarrow B\overset g\rarrow C$ be a composable pair
of morphisms in $\sE$ such that $0\rarrow\Theta(A)
\xrightarrow{\Theta(f)}\Theta(B)\xrightarrow{\Theta(g)}
\Theta(C)\rarrow0$ is an admissible short exact sequence in~$\sF$.
 Then the composition $gf$ vanishes, since the composition
$\Theta(g)\Theta(f)$ vanishes and the functor $\Theta$ is faithful.
 Furthermore, the morphism $\Theta(f)$ being a monomorphism in $\sF$
implies that the morphism~$f$ is a monomorphism in $\sE$ (also
since $\Theta$ is faithful); and dually, the morphism $\Theta(g)$
being an epimorphism in $\sF$ implies that the morphism~$g$ is
an epimorphism in~$\sE$.

 Let us show that the morphism~$f$ is a kernel of the morphism~$g$.
 By assumption, since $\Theta(g)$ is an admissible epimorphism,
the morphism~$g$ has a kernel $f'\:A'\rarrow B$.
 Then the morphism~$f$ factorizes as $A\overset a\rarrow A'\overset{f'}
\rarrow B$.
 By assumption, the morphism $\Theta(f')$ is a kernel of
the morphism $\Theta(g)$ in~$\sF$.
 But the morphism $\Theta(f)$ is also a kernel of $\Theta(g)$, since
$0\rarrow\Theta(A)\rarrow\Theta(B)\rarrow\Theta(C)\rarrow0$ is
an admissible short exact sequence.
 It follows that the morphism $\Theta(a)\:\Theta(A)\rarrow\Theta(A')$
is an isomorphism in~$\sF$.
 Since the functor $\Theta$ is conservative, we can conclude that
$a$~is an isomorphism in~$\sE$.
 Dually one shows that the morphism~$g$ is a cokernel of
the morphism~$f$.

 By construction, the functor $\Theta$ preserves and reflects
admissible short exact sequences; hence it also preserves admissible
monomorphisms and admissible epimorphisms.
 Let us show that $\Theta$ reflects admissible epimorphisms.
 Let $g\:B\rarrow C$ be a morphism in $\sE$ such that $\Theta(g)$
is an admissible epimorphism in~$\sF$.
 By assumption, it follows that the morphism~$g$ has a kernel
$f\:A\rarrow B$ in $\sE$ and $0\rarrow\Theta(A)
\xrightarrow{\Theta(f)}\Theta(B)\xrightarrow{\Theta(g)}
\Theta(C)\rarrow0$ is an admissible short exact sequence in~$\sF$.
 Hence $0\rarrow A\overset f\rarrow B\overset g\rarrow C\rarrow0$ is
an admissible short exact sequence in $\sE$, and $g$~is an admissible
epimorphism.
 Dually, $\Theta$ reflects admissible monomorphisms.

 It follows immediately that the classes of admissible monomorphisms
and admissible epimorphisms in $\sE$ are closed under compositions.

 Let us check that the class of admissible epimorphisms in $\sE$
is stable under pullbacks.
 Let $g\:B\rarrow C$ be an admissible epimorphism and $c\:C'\rarrow C$
be a morphism in~$\sE$.
 Then the admissible epimorphism $\Theta(g)\:\Theta(B)\rarrow\Theta(C)$
in $\sF$ has a pullback by the morphism $\Theta(c)\:\Theta(C')\rarrow
\Theta(C)$.
 Furthermore, the morphism $(\Theta(g),\Theta(c))\:\Theta(B)\oplus
\Theta(C')\rarrow\Theta(C)$ is also an admissible epimorphism in $\sF$
(use~\cite[1st~step of the proof in Section~A.1]{Kel} or
the assertion dual to~\cite[Proposition~2.12]{Bueh}).

 By assumption, this implies that the morphism $(g,c)\:B\oplus C'
\rarrow C$ has a kernel $(b,-g')\:B'\rarrow B\oplus C'$ in $\sE$
whose image $(\Theta(b),-\Theta(g'))\:\Theta(B')\rarrow\Theta(B)
\oplus\Theta(C')$ is a kernel of the morphism $(\Theta(g),\Theta(c))$
in~$\sF$.
 Then $g'\:B'\rarrow C'$ is a pullback of the morphism $g\:B\rarrow C$
by the morphism $c\:C'\rarrow C$ in $\sE$, and $\Theta(g')\:\Theta(B')
\rarrow\Theta(C')$ is a pullback of the morphism $\Theta(g)\:\Theta(B)
\rarrow\Theta(C)$ by the morphism $\Theta(c)\:\Theta(C')\rarrow
\Theta(C)$ in~$\sF$.
 It follows that $\Theta(g')$ is an admissible epimorphism in $\sF$,
and therefore $g'$~is an admissible epimorphism in~$\sE$.
 Dually one produces the pushouts of admissible monomorphisms in~$\sE$.
\end{proof}

\begin{thm} \label{transferring-exact-structure-from-E-bec-to-E-thm}
 Let\/ $\bE$ be a weakly idempotent-complete additive DG\+category
with shifts and cones.
 Then a given DG\+compatible exact category structure on\/
$\sZ^0(\bE^\bec)$ admits a (necessarily unique)\/ $\bec$\+compatible
exact category structure on\/ $\sZ^0(\bE)$ if and only if
the following two conditions hold:
\begin{itemize}
\item any morphism~$g$ in $\sZ^0(\bE)$ for which\/ $\Phi(g)$ is
an admissible epimorphism in\/ $\sZ^0(\bE^\bec)$ has a kernel
in\/~$\sZ^0(\bE)$;
\item any morphism~$f$ in $\sZ^0(\bE)$ for which\/ $\Phi(f)$ is
an admissible monomorphism in\/ $\sZ^0(\bE^\bec)$ has a cokernel
in\/~$\sZ^0(\bE)$.
\end{itemize}
\end{thm}

\begin{proof}
 The ``only if'' assertion follows immediately from
Lemma~\ref{exact-DG-reflecting-admissible-morphisms}.
 To prove the ``if'', we need to construct an exact category
structure on $\sZ^0(\bE)$.
 The definition is obvious: let us say that $0\rarrow A\rarrow B
\rarrow C\rarrow0$ is an admissible short exact sequence in $\sZ^0(\bE)$
if $0\rarrow\Phi(A)\rarrow\Phi(B)\rarrow\Phi(C)\rarrow0$ is
an admissible short exact sequence in~$\sZ^0(\bE^\bec)$.
 Then the functor $\Phi_\bE$ preserves and reflects admissible
short exact sequences by the definition.
 Since the functor $\Xi_{\bE^\bec}\simeq\Phi_\bE\circ\Psi^-_\bE$
preserves and reflects admissible short exact sequences by assumption
and Proposition~\ref{DG-compatible-prop}, it then follows that
the functor $\Psi^-_\bE$ also preserves and reflects admissible short
exact sequences.

 It remains to check that the class of admissible short exact
sequences in $\sZ^0(\bE)$ that we have defined satisfies the axioms
of an exact category structure.
 For this purpose, we apply
Lemma~\ref{faithful-conservative-lifting-exact-structure} to
the additive functor $\Phi_\bE$ (which is faithful and conservative
by Lemmas~\ref{G-functors-in-DG-categories}
and~\ref{Phi-Psi-conservative}).
 We only have to check the assumptions of
Lemma~\ref{faithful-conservative-lifting-exact-structure}.

 Let $g\:B\rarrow C$ be a morphism in $\sZ^0(\bE)$ for which
$\Phi(g)$ is an admissible epimorphism in~$\sZ^0(\bE^\bec)$.
 Then, by the assumption of the theorem, the morphism~$g$
has a kernel $f\:A\rarrow B$ in $\sZ^0(\bE)$.
 The morphism $\Phi(f)$ is a kernel of the morphism $\Phi(g)$
because the functor $\Phi_\bE$, being a right adjoint by
Lemma~\ref{G-functors-in-DG-categories}, preserves all kernels.
 The other assumption of 
Lemma~\ref{faithful-conservative-lifting-exact-structure} is dual.
\end{proof}

\begin{cor} \label{transferring-from-E-bec-to-E-under-all-twists}
 Let\/ $\bE$ be an idempotent-complete additive DG\+category
with shifts and twists.
 Then any DG\+compatible exact category structure on\/ $\sZ^0(\bE^\bec)$
admits a unique\/ $\bec$\+compatible exact category structure
on\/~$\sZ^0(\bE)$.
\end{cor}

\begin{proof}
 This is a corollary of
Theorem~\ref{transferring-exact-structure-from-E-bec-to-E-thm},
whose assumptions we have to check.
 Let $g\:B\rarrow C$ be a morphism in $\sZ^0(\bE)$ for which
$\Phi_\bE(g)$ is an admissible epimorphism in~$\sZ^0(\bE^\bec)$.
 Then the morphism $\Phi_\bE(g)$ has a kernel in~$\sZ^0(\bE^\bec)$.
 The functor $\Psi^+_\bE$ preserves kernels, being a right adjoint
(to $\Phi_\bE[-1]$) by Lemma~\ref{G-functors-in-DG-categories};
so it follows that the morphism $\Psi^+_\bE\Phi_\bE(g)$ has a kernel
in~$\sZ^0(\bE)$.
 By Lemma~\ref{Xi-decomposed}, we have $\Psi^+_\bE\Phi_\bE(g)\simeq
\Xi_\bE(g)$; hence the morphism $\Xi_\bE(g)$ has a kernel.
 It remains to apply Lemma~\ref{Xi-creates-co-kernels}(a) in order to
conclude that the morphism~$g$ has a kernel in $\sZ^0(\bE)$ under
the assumptions of the corollary.
 The other assumption of
Theorem~\ref{transferring-exact-structure-from-E-bec-to-E-thm} is dual.
\end{proof}

\begin{lem} \label{Psi-creates-co-kernels}
 Let\/ $\bA$ be an additive DG\+category with shifts and cones.
 In this setting \par
\textup{(a)} if $g$~is a morphism in\/ $\sZ^0(\bA^\bec)$ and
the morphism\/ $\Psi^+(g)$ has a kernel in\/ $\sZ^0(\bA)$, then
the morphism~$g$ has a kernel in\/~$\sZ^0(\bA^\bec)$; \par
\textup{(b)} if $f$~is a morphism in\/ $\sZ^0(\bA^\bec)$ and
the morphism\/ $\Psi^+(f)$ has a cokernel in\/ $\sZ^0(\bA)$, then
the morphism~$f$ has a cokernel in\/~$\sZ^0(\bA^\bec)$.
\end{lem}

\begin{proof}
 Let us prove part~(a); part~(b) is dual.
 Let $Y^\bec=(Y,\sigma_X)$ and $Z^\bec=(Z,\sigma_Z)$ be two objects
in~$\sZ^0(\bA^\bec)$.
 Then $\Psi^+(Y^\bec)=Y$ and $\Psi^+(Z^\bec)=Z$.
 Let $g\:Y^\bec\rarrow Z^\bec$ be a morphism in $\sZ^0(\bA^\bec)$;
this means that $g\:Y\rarrow Z$ is a morphism in $\sZ^0(\bA)$ and
$g\sigma_Y=\sigma_Zg$.
 Assume that the morphism $g\:Y\rarrow Z$ has a kernel $f\:X\rarrow Y$
in~$\sZ^0(\bA)$.
 Let us construct a kernel of the morphism $g\:Y^\bec\rarrow Z^\bec$
in~$\sZ^0(\bA^\bec)$.

 Given an object $A\in\bA$, we will use the notation $\iota_A$,
$\pi_A$, $\iota'_A$, $\pi'_A$ from the proofs of
Lemmas~\ref{G-functors-in-DG-categories}
and~\ref{Xi-creates-co-kernels}.
 Given an object $T^\bec=(T,\sigma_T)\in\bA^\bec$, consider
the morphism $\beta_T=\iota'_T+\iota_T\sigma_T\in\Hom_\bA^0(T,\Xi(T))$.
 We have $d(\beta_T)=\iota_T-\iota_T d(\sigma_T)=0$, so $\beta_T$~is
a morphism $T\rarrow\Xi(T)$ in the category~$\sZ^0(\bA)$.
 Furthermore, $\pi_T\:\Xi(T)\rarrow T$ is a morphism in $\sZ^0(\bA)$,
and the equation $\pi_T\beta_T=\id_T$ holds in~$\sZ^0(\bA)$.

 We need to express the equation $\sigma_T^2=0$, which is a part of
the definition of an object of the DG\+category $\bA^\bec$, as
an equation on morphisms in the category~$\sZ^0(\bA)$.
 Notice that $d(\sigma_T)=\id_T\ne0$, so $\sigma_T$ cannot be
directly viewed as a morphism in $\sZ^0(\bA)$ (even after the shift
is taken care of).
 Let us start with the shift: recall the notation
$s_A\in\Hom_\bA^1(A,A[-1])$ from the proof of
Lemma~\ref{Xi-creates-co-kernels}.
 So $s_T\sigma_T\in\Hom_\bA^0(T,T[-1])$ is a morphism in
the category~$\bA^0$.

 Applying the functor $\widetilde\Xi$ from Lemma~\ref{tilde-Xi},
we obtain a morphism $\widetilde\Xi(s_T\sigma_T)\:\Xi(T)\rarrow
\Xi(T[-1])$ in the category~$\sZ^0(\bA)$.
 Now the equation $\sigma_T^2=0$ can be expressed as
$\widetilde\Xi(s_T\sigma_T)[-1]\circ\widetilde\Xi(s_T\sigma_T)=
\widetilde\Xi((s_T\sigma_T)[-1]\circ s_T\sigma_T)=
\widetilde\Xi(s_{T[-1]}s_T\sigma_T^2)=0$.
 Conversely, if $U\in\bA$ is an object and $\sigma\in
\Hom_\bA^{-1}(U,U)$ is an endomorphism of degree~$-1$, then
$s_U\sigma_U$ is a morphism in $\bA^0$ and $\widetilde\Xi(s_U\sigma)$
is a morphism in $\sZ^0(\bA)$; the equation
$\widetilde\Xi(s_U\sigma)[-1]\circ\widetilde\Xi(s_U\sigma)=0$ holds
in $\sZ^0(\bA)$ if and only if the equation $\sigma^2=0$ holds
in~$\bA^*$ (because the functor $\widetilde\Xi$ is faithful).

 Returning to the situation at hand, consider the commutative diagram
of solid arrows in the category~$\sZ^0(\bA)$
$$
 \xymatrix{
  0 \ar[r] & X \ar[rr]^-f \ar@{-->}[d]^\beta
  && Y \ar[rr]^-{g} \ar[d]^{\beta_Y}
  && Z \ar[d]^{\beta_Z} \\
  0 \ar[r] & \Xi(X) \ar[rr]^-{\Xi(f)} \ar[d]^{\pi_X}
  && \Xi(Y) \ar[rr]^-{\Xi(g)} \ar[d]^{\pi_Y}
  && \Xi(Z) \ar[d]^{\pi_Z} \\
  0 \ar[r] & X \ar[rr]^-f && Y \ar[rr]^-g && Z
 }
$$
 Since the morphism~$f$ is a kernel of~$g$ and the morphism~$\Xi(f)$
is a kernel of~$\Xi(g)$ (by Lemma~\ref{Xi-adjoints}),
there exists a unique morphism $\beta\:X\rarrow\Xi(X)$ in~$\sZ^0(\bA)$
making the leftmost square commutative.
 The equation $\pi_Y\beta_Y=\id_Y$ implies $\pi_X\beta=\id_X$.

 Put $\sigma=\pi'\beta\in\Hom_\bA^{-1}(X,X)$.
 Then we have $\beta=(\iota'\pi+\iota\pi')\beta=\iota'+\iota\sigma$,
hence $0=d(\beta)=\iota-\iota d(\sigma)$.
 Multiplying with~$\pi'$ on the left, we deduce that $d(\sigma)=\id_X$.
 Furthermore, we have a commutative diagram in the category~$\bA^0$
\begin{equation} \label{recovering-sigmas-from-betas}
\begin{gathered}
 \xymatrix{
  0 \ar[r] & X \ar[rr]^-f \ar[d]^\beta
  && Y \ar[rr]^-{g} \ar[d]^{\beta_Y}
  && Z \ar[d]^{\beta_Z} \\
  0 \ar[r] & \Xi(X) \ar[rr]^-{\Xi(f)} \ar[d]^{s_X\pi'_X}
  && \Xi(Y) \ar[rr]^-{\Xi(g)} \ar[d]^{s_Y\pi'_Y}
  && \Xi(Z) \ar[d]^{s_Z\pi'_Z} \\
  0 \ar[r] & X[-1] \ar[rr]^-{f[-1]}
  && Y[-1] \ar[rr]^-{g[-1]} && Z[-1]
 }
\end{gathered}
\end{equation}
with the vertical compositions equal to~$s_X\sigma$, \ $s_Y\sigma_Y$,
and~$s_Z\sigma_Z$, respectively.
 Collapsing the middle row, applying $\widetilde\Xi$ and building up
the composition with a shift of the same diagram, we obtain
a commutative diagram in~$\sZ^0(\bA)$
$$
 \xymatrix{
  0 \ar[r] & \Xi(X) \ar[rrr]^-{\Xi(f)}
  \ar[d]^{\widetilde\Xi(s_X\sigma)}
  &&& \Xi(Y) \ar[rrr]^-{\Xi(g)} \ar[d]^{\widetilde\Xi(s_Y\sigma_Y)}
  &&& \Xi(Z) \ar[d]^{\widetilde\Xi(s_Z\sigma_Z)} \\
  0 \ar[r] & \Xi(X[-1]) \ar[rrr]^-{\Xi(f[-1])}
  \ar[d]^{\widetilde\Xi(s_X\sigma)[-1]}
  &&& \Xi(Y[-1]) \ar[rrr]^-{\Xi(g[-1])}
  \ar[d]^{\widetilde\Xi(s_Y\sigma_Y)[-1]}
  &&& \Xi(Z[-1]) \ar[d]^{\widetilde\Xi(s_Z\sigma_Z)[-1]} \\
  0 \ar[r] & \Xi(X[-2]) \ar[rrr]^-{\Xi(f[-2])} 
  &&& \Xi(Y[-2]) \ar[rrr]^-{\Xi(g[-2])} &&& \Xi(Z[-2])
 }
$$
 Now the equation $\widetilde\Xi(s_Y\sigma_Y)[-1]\circ
\widetilde\Xi(s_Y\sigma_Y)=0$ implies
$\widetilde\Xi(s_X\sigma)[-1]\circ\widetilde\Xi(s_X\sigma)=0$,
because $\Xi(f[-2])$ is a monomorphism in~$\sZ^0(\bA)$.
 Thus~$\sigma^2=0$.

 Finally, we can put $\sigma_X=\sigma$ and $X^\bec=(X,\sigma_X)$,
producing an object of the DG\+cat\-e\-gory~$\bA^\bec$.
 The equation $f\sigma_X=\sigma_Yf$ follows easily from
the commutativity of~\eqref{recovering-sigmas-from-betas};
so the morphism $f\:X\rarrow Y$ in $\sZ^0(\bA)$ defines
a morphism $f\:X^\bec\rarrow Y^\bec$ in~$\sZ^0(\bA^\bec)$.

 In order to check that $f$~is a kernel of~$g$ in $\sZ^0(\bA^\bec)$,
it suffices to observe that $\Xi_{\bA^\bec}(f)=\Phi_\bA\Psi^-_\bA(f)$
is a kernel of $\Xi_{\bA^\bec}(g)=\Phi_\bA\Psi^-_\bA(g)$
in $\sZ^0(\bA^\bec)$ (since $\Psi^-_\bA(f)=f[1]$ is a kernel of
$\Psi^-_\bA(g)=g[1]$ in the category $\sZ^0(\bA)$, and the functor
$\Phi_\bA$ preserves kernels), and apply
Lemma~\ref{Xi-reflects-co-kernels}(a) for the DG\+category~$\bA^\bec$.
\end{proof}

\begin{thm} \label{transferring-from-E-to-E-bec}
 Let\/ $\bE$ be an additive DG\+category with shifts and cones.
 Then any DG\+compatible exact category structure on\/ $\sZ^0(\bE)$
admits a unique\/ $\bec$\+compatible exact category structure
on\/~$\sZ^0(\bE^\bec)$.
\end{thm}

\begin{proof}
 The definition of the exact category structure on $\sZ^0(\bE^\bec)$ is
obvious: we say that $0\rarrow X\rarrow Y\rarrow Z\rarrow0$ is
an admissible short exact sequence in $\sZ^0(\bE^\bec)$ if
$0\rarrow\Psi^+(X)\rarrow\Psi^+(Y)\rarrow\Psi^+(Z)\rarrow0$ is
an admissible short exact sequence in~$\sZ^0(\bE)$.
 Then the functor $\Psi^+_\bE$ preserves and reflects admissible
short exact sequences by the definition.
 Since the functor $\Xi_\bE\simeq\Psi^+_\bE\circ\Phi_\bE$ preserves
and reflects admissible short exact sequences by assumption, it then
follows that the functor $\Phi_\bE$ also preserves and reflects
admissible short exact sequences.

 To check that the class of admissible short exact sequences in
$\sZ^0(\bE^\bec)$ that we have defined satisfies the axioms of
an exact category structure, we apply
Lemma~\ref{faithful-conservative-lifting-exact-structure} to
the additive functor~$\Psi^+_\bE$.
 This functor is faithful and conservative by
Lemmas~\ref{G-functors-in-DG-categories} and~\ref{Phi-Psi-conservative}.
 We have to check the assumptions of
Lemma~\ref{faithful-conservative-lifting-exact-structure}.

 Let $g\:Y\rarrow Z$ be a morphism in $\sZ^0(\bE^\bec)$ for which
$\Psi^+(g)$ is an admissible epimorphism in~$\sZ^0(\bE)$.
 Then the morphism $\Psi^+(g)$ has a kernel in~$\sZ^0(\bE)$.
 By Lemma~\ref{Psi-creates-co-kernels}, it follows that the morphism~$g$
has a kernel $f\:X\rarrow Y$ in $\sZ^0(\bE^\bec)$.
 The morphism $\Psi^+(f)$ is a kernel of the morphism $\Psi^+(g)$
in $\sZ^0(\bE)$, because the functor $\Psi^+_\bE$, being a right adjoint
(to $\Phi_\bE[-1]$) by Lemma~\ref{G-functors-in-DG-categories},
preserves all kernels.
 The other assumption of
Lemma~\ref{faithful-conservative-lifting-exact-structure} is dual.

 Alternatively, one could assume that $\bE$ is idempotent-complete,
notice that the DG\+category $\bE^\bec$ always has all twists,
and argue similarly to the proof of
Theorem~\ref{transferring-exact-structure-from-E-bec-to-E-thm} combined
with Corollary~\ref{transferring-from-E-bec-to-E-under-all-twists},
using Lemma~\ref{Xi-creates-co-kernels}.
\end{proof}

\begin{rem} \label{long-journey-rem}
 In view of Theorem~\ref{transferring-from-E-to-E-bec}, we could have
simply defined an exact DG\+category $\bE$ as an additive DG\+category
with shifts and cones endowed with a DG\+compatible exact category
structure on $\sZ^0(\bE)$, in the sense of
Section~\ref{DG-compatible-subsecn}.
 This would be a simple and straightforward definition.

 Instead, we have made a long journey through the construction and
study of the DG\+category $\bA^\bec$ in Section~\ref{involution-secn},
the study of the functor $\Xi_\bA$ in Section~\ref{Xi-functor-subsecn},
the proof of Lemma~\ref{Psi-creates-co-kernels}, etc., only to come to
the conclusion that having a DG\+compatible exact structure on
$\sZ^0(\bE)$ is always sufficient to uniquely recover what we call
the exact DG\+category structure.

 The upside of the long journey is that we have also produced
an exact category $\sZ^0(\bE^\bec)$ together with the exact functors
$\Phi_\bE$ and $\Psi^+_\bE$ along the way.
 For naturally occurring DG\+categories $\bE$, having an exact structure
on ``the category of underlying graded objects'' (whose role is played
by the additive category $\sZ^0(\bE^\bec)$ in our formalism;
cf.\ the examples in Section~\ref{A-bec-examples-subsecn}) is no less
important than having an exact structure on ``the category of
DG\+objects''~$\sZ^0(\bE)$.

 Arguably one could say that, e.~g., the abelian exact structure on
the category of graded modules conceptually and logically precedes
the abelian exact structure on the category of DG\+modules and
closed morphisms of degree~$0$ between them.
 Having an exact structure on the category of underlying graded objects
is also important for the definitions and study of the derived
categories of the second kind for exact DG\+categories, which is our
main aim.

 The point is that complexes of projective or injective objects in exact
categories and graded-projective or graded-injective CDG\+modules are
crucial for the definitions of derived categories of the second kind
in the sense of Becker~\cite[Section~1.3]{Bec}, \cite{PS5} and for
the most important properties of derived categories of the second kind
in the sense of~\cite{Psemi,Pkoszul} and the present paper.
 This means objects whose underlying graded objects are projective
or injective.
 For an exact DG\+category $\bE$, the definitions of graded-projective
and graded-injective objects in $\bE$ refer to the projective and
injective objects in the exact category $\sZ^0(\bE^\bec)$
(see Section~\ref{derived-second-kind-secn}, where these definitions
and the mentioned important properties are collected).
 The exact category $\sZ^0(\bE^\bec)$ (or its exact subcategory $\sK$,
in the context of an exact DG\+pair $(\bE,\sK)$) also appears in
the formulations of the main results of
Sections~\ref{finite-resol-dim-secn}\+-\ref{finite-homol-dim-secn}.
\end{rem}

\begin{ex} \label{minimal-exact-dg}
 Let $\bE$ be an additive DG\+category with shifts and cones.
 Then, in any exact DG\+category structure on $\bE$, all the composable
pairs of morphisms in $\sZ^0(\bE)$ taken to split short exact sequences
by the functor $\Phi_\bE$ must be admissible short exact sequences,
and all the composable pairs of morphisms in $\sZ^0(\bE^\bec)$ taken to
split short exact sequences by the functor $\Psi^+_\bE$ must be
admissible short exact sequences.

 Notice that iterating the functors $\Phi$ and $\Psi$ any further
does \emph{not} expand the classes of ``necessary admissible short
exact sequences''.
 Specifically, if $A\rarrow B\rarrow C$ is a composable pair of
morphisms in $\sZ^0(\bE)$ such that $0\rarrow\Psi_\bE^+\Phi_\bE(A)
\rarrow\Psi_\bE^+\Phi_\bE(B)\rarrow\Psi_\bE^+\Phi_\bE(C)\rarrow0$ is
a split short exact sequence in $\sZ^0(\bE)$, then $0\rarrow\Phi_\bE(A)
\rarrow\Phi_\bE(B)\rarrow\Phi_\bE(C)\rarrow0$ is a split short exact
sequence in $\sZ^0(\bE^\bec)$.

 Indeed, additive functors take split short exact sequences to split
short exact sequences; so if $0\rarrow\Psi_\bE^+\Phi_\bE(A)
\rarrow\Psi_\bE^+\Phi_\bE(B)\rarrow\Psi_\bE^+\Phi_\bE(C)\rarrow0$ is
a split short exact sequence in $\sZ^0(\bE)$, then
$0\rarrow\Phi_\bE\Psi_\bE^+\Phi_\bE(A)\rarrow
\Phi_\bE\Psi_\bE^+\Phi_\bE(B)\rarrow
\Phi_\bE\Psi_\bE^+\Phi_\bE(C)\rarrow0$ is a split short exact
sequence in $\sZ^0(\bE^\bec)$.
 The functor $\Phi_\bE\circ\Psi_\bE^+\circ\Phi_\bE\simeq
\Phi_\bE\circ\Xi_\bE$ is isomorphic to the direct sum
$\Phi_\bE\oplus\Phi_\bE[1]$ by Lemma~\ref{twists-into-isomorphisms}
(as $\Xi_\bE(E)$ is naturally a twist of the direct sum
$E\oplus E[-1]$ for every object $E\in\bE$).
 It remains to point out that a direct summand of a split short
exact sequence is always a split short exact sequence in order to
conclude that $0\rarrow\Phi_\bE(A)\rarrow\Phi_\bE(B)\rarrow\Phi_\bE(C)
\rarrow0$ is a split short exact sequence in $\sZ^0(\bE^\bec)$.

 Similarly, if $X^\bec\rarrow Y^\bec\rarrow Z^\bec$ is a composable
pair of morphisms in $\sZ^0(\bE^\bec)$ such that $0\rarrow\Phi_\bE
\Psi^-_\bE(X^\bec)\rarrow\Phi_\bE\Psi^-_\bE(Y^\bec)\rarrow
\Phi_\bE\Psi^-_\bE(Z^\bec)\rarrow0$ is a split short exact sequence
in $\sZ^0(\bE^\bec)$, then $0\rarrow\Psi^-_\bE(X^\bec)\rarrow
\Psi^-_\bE(Y^\bec)\rarrow\Psi^-_\bE(Z^\bec)\rarrow0$ is a split short
exact sequence in $\sZ^0(\bE)$.
 This is provable using the fact that the functor
$\Psi^+_\bE\circ\Phi_\bE\circ\Psi^-_\bE
\simeq\Psi^+_\bE\circ\Xi_{\bE^\bec}$ is isomorphic to
$\Psi_\bE^+\oplus\Psi_\bE^-$.

 One can now define the \emph{minimal DG\+compatible exact structure
on\/~$\sZ^0(\bE^\bec)$} as the exact structure given by the class of
all composable pairs of morphisms whose image under $\Psi^+_\bE$
is a split short exact sequence.
 Arguing similarly to the proof of
Theorem~\ref{transferring-from-E-to-E-bec}, one shows that
Lemma~\ref{faithful-conservative-lifting-exact-structure} is
applicable to the functor $\Psi_\bE^+\:\sE=\sZ^0(\bE^\bec)\rarrow
\sZ^0(\bE)=\sF$, with the split exact structure on~$\sF$.
 Hence what we have defined is indeed an exact category structure
on~$\sZ^0(\bE^\bec)$.

 Assuming that the DG\+category $\bE$ is idempotent-complete and
has twists, one can similarly define the \emph{minimal DG\+compatible
exact structure on\/~$\sZ^0(\bE)$} as the exact structure given by
the class of all composable pairs of morphisms whose image under
$\Phi_\bE$ is a split short exact sequence in~$\sZ^0(\bE^\bec)$.
 Arguing similarly to the proofs of
Theorem~\ref{transferring-exact-structure-from-E-bec-to-E-thm}
and Corollary~\ref{transferring-from-E-bec-to-E-under-all-twists}, one
checks that Lemma~\ref{faithful-conservative-lifting-exact-structure}
is applicable to the functor $\Phi_\bE\:\sE=\sZ^0(\bE)\rarrow
\sZ^0(\bE^\bec)=\sF$, with the split exact structure on~$\sF$.
 Hence the class of all composable pairs of morphisms turned into
split short exact sequences by $\Phi_\bE$ is indeed an exact
category structure on~$\sZ^0(\bE)$.

 Finally, it is clear from the discussion of compositions of
the functors $\Phi$ and $\Psi$ above that the two minimal DG\+compatible
exact structures on the additive categories $\sZ^0(\bE)$ and
$\sZ^0(\bE^\bec)$ are $\bec$\+compatible with each other.
 Thus we have described the \emph{minimal exact DG\+category structure}
on an idempotent-complete additive DG\+category $\bE$ with shifts
and twists.
\end{ex}

\subsection{Exact DG\+functors}
 The following lemma characterizes exact DG\+functors between exact
DG\+categories.

\begin{lem} \label{exact-dg-functors-lemma}
 Let\/ $\bF$ and\/ $\bE$ be exact DG\+categories, and let
$G\:\bF\rarrow\bE$ be a DG\+functor.
 Then the following conditions are equivalent:
\begin{enumerate}
\item the functor\/ $\sZ^0(G)\:\sZ^0(\bF)\rarrow\sZ^0(\bE)$ is exact
(as a functor between exact categories);
\item the functor\/ $\sZ^0(G^\bec)\:\sZ^0(\bF^\bec)\rarrow
\sZ^0(\bE^\bec)$ is exact (as a functor between exact categories).
\end{enumerate}
\end{lem}

\begin{proof}
 The assertions follow from the assumption that the functors $\Phi_\bE$,
$\Phi_\bF$, $\Psi^+_\bE$, $\Psi^+_\bF$ preserve and reflect admissible
short exact sequences, together with the commutativity of the diagrams
of additive functors in
Section~\ref{compatibility-with-dg-functors-subsecn}.
\end{proof}

 A DG\+functor $G\:\bF\rarrow\bE$ between exact DG\+categories $\bF$
and $\bE$ is said to be \emph{exact} (as a functor between exact
DG\+categories) if it satisfies any one of the equivalent conditions
of Lemma~\ref{exact-dg-functors-lemma}.
 It follows that the DG\+functor $G\:\bF\rarrow\bE$ is exact if and
only if the DG\+functor $G^\bec\:\bF^\bec\rarrow\bE^\bec$ is exact.

\subsection{Inheriting exact DG-category structure}
\label{inheriting-exact-dg-category-structure-subsecn}
 Let $\sE$ be an exact category and $\sF\subset\sE$ be a full additive
subcategory.
 We will say that $\sF$ \emph{inherits an exact category structure}
from $\sE$ if the class of all admissible short exact sequences
in $\sE$ whose terms belong to $\sF$ defines an exact category
structure on~$\sF$.

 In particular, if the full subcategory $\sF$ is closed under extensions
in $\sE$, then it inherits an exact category structure.
 In this case, the category $\sF$ endowed with the inherited exact
category structure is called a \emph{fully exact subcategory} in~$\sE$
\,\cite[Sections~10.5 and~13.3]{Bueh}.
 Another particular case occurs when $\sF$ is closed under both
the kernels of admissible epimorphisms and the cokernels of admissible
monomorphisms in~$\sE$.
 In this case, $\sF$ also inherits an exact category structure
from~$\sE$ \,\cite[Section~A.5(3)(b)]{Partin}.

 Quite generally, full subcategories inheriting an exact category
structure are characterized by the following lemma (which can be
also found in~\cite[Theorem~2.6]{DS}).

\begin{lem} \label{inheriting-exact-structure-lemma}
 Let\/ $\sE$ be an exact category and\/ $\sF\subset\sE$ be a full
additive subcategory.
 Then\/ $\sF$ inherits an exact category structure from\/ $\sE$ if
and only if the following two conditions hold:
\begin{enumerate}
\renewcommand{\theenumi}{\roman{enumi}}
\item for any commutative diagram in $\sE$
\begin{equation} \label{pullback-diagram}
\begin{gathered}
 \xymatrix{
  A \ar@{>->}[r] \ar[rd] & B \ar@{->>}[r] & C \\
  & B' \ar[u] \ar[r] & C' \ar[u]
 }
\end{gathered}
\end{equation}
with an admissible short exact sequence $0\rarrow A\rarrow B\rarrow C
\rarrow0$ and a pullback square $B'\rarrow B\rarrow C$, \ $B'
\rarrow C'\rarrow C$ such that the objects $A$, $B$, $C$, and $C'$
belong to\/ $\sF$, the object $B'$ also belongs to\/~$\sF$;
\item for any commutative diagram in $\sE$
\begin{equation} \label{pushout-diagram}
\begin{gathered}
 \xymatrix{
  A \ar@{>->}[r] \ar[d] & B \ar@{->>}[r] \ar[d] & C \\
  A' \ar[r] & B' \ar[ur]
 }
\end{gathered}
\end{equation}
with an admissible short exact sequence $0\rarrow A\rarrow B\rarrow C
\rarrow0$ and a pushout square $A\rarrow B\rarrow B'$, \ $A\rarrow
A'\rarrow B'$ such that the objects $A$, $B$, $C$, and $A'$ belong
to\/ $\sF$, the object $B'$ also belongs to\/~$\sF$.
\end{enumerate}
\end{lem}

\begin{proof}
 ``Only if'': assuming that $\sF$ inherits an exact category structure
from $\sE$, let us prove~(i).
 Since $\sF$ is an exact category, the admissible short exact
sequence $0\rarrow A\rarrow B\rarrow C\rarrow0$ has a pullback
by the morphism $C'\rarrow C$ in the category $\sF$; denote this
pullback by $0\rarrow A\rarrow B''\rarrow C'\rarrow0$.
 Then $0\rarrow A\rarrow B''\rarrow C'\rarrow0$ is an admissible
short exact sequence in $\sF$, and consequently in $\sE$.
 It follows that the square $B''\rarrow B\rarrow C$, \ $B''\rarrow C'
\rarrow C$ is a pullback square in~$\sE$
\,\cite[Proposition~A.2 and Corollary~A.3]{Partin}, hence
$B'\simeq B''\in\sF$.
 The proof of~(ii) is dual.
 
 ``If'': the only nonobvious property to check is that the classes
of admissible monomorphisms and admissible epimorphisms in $\sF$ are
closed under compositions.
 Here one observes that the kernel object of the composition of two
admissible epimorphisms $g=g'g''$ (in an exact category~$\sE$)
is a pullback of~$g''$ by the kernel of~$g'$, and similarly,
the cokernel object of the composition of two admissible monomorphisms
$f=f'f''$ is a pushout of~$f'$ by the cokernel of~$f''$.
\end{proof}

\begin{lem} \label{lifting-exact-subcategory-lemma}
 Let\/ $\Theta\:\sE\rarrow\sG$ be an exact functor between exact
categories, and let\/ $\sH\subset\sG$ be a full additive subcategory
inheriting an exact category structure.
 Then the full preimage\/ $\sF=\Theta^{-1}(\sH)\subset\sE$ is a full
additive subcategory inheriting an exact category structure
from\/~$\sE$.
\end{lem}

\begin{proof}
 It suffices to check the conditions~(i) and~(ii) of
Lemma~\ref{inheriting-exact-structure-lemma} for the full
subcategory $\sF\subset\sE$.
 Suppose that we are given a pullback diagram~\eqref{pullback-diagram}
in $\sE$ with $A$, $B$, $C$, $C'\in\sF$ and $B'\in\sE$.
 Then $0\rarrow B'\rarrow B\oplus C'\rarrow C\rarrow0$ is
an admissible short exact sequence in $\sE$ (see~\cite[1st~step
of the proof in Section~A.1]{Kel} or the assertion dual
to~\cite[Proposition~2.12]{Bueh}).
 Applying the functor $\Theta$ to~\eqref{pullback-diagram}, we obtain
a similar pullback diagram in the exact category $\sG$ (since
the functor $\Theta$ preserves admissible short exact sequences).
 Since the full subcategory $\sH$ inherits an exact category structure
from $\sG$, by the other implication of
Lemma~\ref{inheriting-exact-structure-lemma} we have $\Theta(B')\in
\sH$, hence $B'\in\sF$.
 The condition~(ii) is dual.
 Alternatively, one could refer to the more
general~\cite[Section~A.5(4)]{Partin}.
\end{proof}

 Let $\bE$ be an additive DG\+category with shifts and cones, and let
$\bF\subset\bE$ be a full additive DG\+subcategory closed under
shifts and cones.
 Then the DG\+category $\bF^\bec$ is naturally a full DG\+subcategory
in~$\bE^\bec$.
 The functors $\Phi_\bF\:\sZ^0(\bF)\rarrow\sZ^0(\bF^\bec)$ and
$\Psi^+_\bF\:\sZ^0(\bF^\bec)\rarrow\sZ^0(\bF)$ form commutative
square diagrams with the functors $\Phi_\bE$ and $\Psi^+_\bE$ and
the additive inclusion functors $\sZ^0(\bF)\rarrow\sZ^0(\bE)$ and
$\sZ^0(\bF^\bec)\rarrow\sZ^0(\bE^\bec)$.
 If the full DG\+subcategory $\bF$ is closed under direct summands
in $\bE$, then the full DG\+subcategory $\bF^\bec$ is closed under
direct summands in~$\bE^\bec$.

 Assume that a pair of $\bec$\+compatible exact category structures is
defined on the additive categories $\sZ^0(\bE)$ and $\sZ^0(\bE^\bec)$,
making $\bE$ an exact DG\+category.
 We will say that the full DG\+subcategory $\bF$ \emph{inherits an exact
DG\+category structure} from $\bE$ if both the full additive 
subcategories $\sZ^0(\bF)\subset\sZ^0(\bE)$ and $\sZ^0(\bF^\bec)\subset
\sZ^0(\bE^\bec)$ inherit exact category structures.
 Obviously, in this case the inherited pair of exact structures
on $\sZ^0(\bF)$ and $\sZ^0(\bF^\bec)$ is $\bec$\+compatible,
so $\bF$ becomes an exact DG\+category.

\begin{prop} \label{inheriting-from-E-bec-may-imply-from-E}
 Let\/ $\bE$ be a weakly idempotent-complete exact DG\+category,
and let\/ $\bF\subset\bE$ be a full additive DG\+subcategory closed
under shifts, cones, and direct summands.
 Assume that the full subcategory\/ $\sZ^0(\bF^\bec)$ inherits
an exact category structure from\/~$\sZ^0(\bE^\bec)$.
 Then the full subcategory\/ $\sZ^0(\bF)$ inherits an exact category
structure from\/ $\sZ^0(\bE)$ if and only if the following two
conditions hold:
\begin{itemize}
\item for any admissible short exact sequence\/ $0\rarrow A\rarrow B
\rarrow C\rarrow0$ in\/ $\sZ^0(\bE)$ such that the objects $B$ and $C$
belong to\/ $\sZ^0(\bF)$ and the object\/ $\Xi_\bE(A)$ belongs to\/
$\sZ^0(\bF)$, the object $A$ also belongs to\/~$\sZ^0(\bF)$;
\item for any admissible short exact sequence\/ $0\rarrow A\rarrow B
\rarrow C\rarrow0$ in\/ $\sZ^0(\bE)$ such that the objects $A$ and $B$
belong to\/ $\sZ^0(\bF)$ and the object\/ $\Xi_\bE(C)$ belongs to\/
$\sZ^0(\bF)$, the object $C$ also belongs to\/~$\sZ^0(\bF)$.
\end{itemize}
\end{prop}

\begin{proof}
 ``Only if'': the assumptions of the proposition imply that
the DG\+category $\bF$ is weakly idempotent-complete.
 If the additive subcategory $\sZ^0(\bF)$ inherits an exact category
structure from $\sZ^0(\bE)$, then the inherited exact category
structure is DG\+compatible.
 By Lemma~\ref{DG-compatible-reflecting-admissible-morphisms}, it
follows that the functor $\Xi_\bF\:\sZ^0(\bF)\rarrow\sZ^0(\bF)$
reflects admissible monomorphisms and admissible epimorphisms.
 Now if $0\rarrow A\rarrow B\rarrow C\rarrow0$ is an admissible
short exact sequence in $\sZ^0(\bE)$ with $B$, $C\in\sZ^0(\bF)$ and
$\Xi_\bE(A)\in\sZ^0(\bF)$, then $0\rarrow\Xi_\bE(A)\rarrow
\Xi_\bE(B)\rarrow\Xi_\bE(C)\rarrow0$ is an admissible short exact
sequence in $\sZ^0(\bE)$ with the terms in $\sZ^0(\bF)$, so it is
an admissible short exact sequence in~$\sZ^0(\bF)$.
 It follows that $\Xi_\bF(B)\rarrow\Xi_\bF(C)$ is an admissible
epimorphism in $\sZ^0(\bF)$, hence $B\rarrow C$ is an admissible
epimorphism in~$\sZ^0(\bF)$.
 So there exists an admissible short exact sequence $0\rarrow K
\rarrow B\rarrow C\rarrow0$ in~$\sZ^0(\bF)$.
 Since the exact structure on $\sZ^0(\bF)$ is inherited from
$\sZ^0(\bE)$, the latter sequence must be an admissible short
exact sequence in $\sZ^0(\bE)$ as well, and we can conclude
that $A\simeq K\in\sZ^0(\bF)$.
 The second property is dual.

 ``If'': it suffices to check the conditions~(i) and~(ii) of
Lemma~\ref{inheriting-exact-structure-lemma} for the full
subcategory $\sZ^0(\bF)\subset\sZ^0(\bE)$.
 Suppose that we are given a pullback diagram~\eqref{pullback-diagram}
in $\sZ^0(\bE)$ with $A$, $B$, $C$, $C'\in\sZ^0(\bF)$ and
$B'\in\sZ^0(\bE)$.
 Following the proof of Lemma~\ref{lifting-exact-subcategory-lemma}
applied to the exact functor $\Phi_\bE\:\sZ^0(\bE)\rarrow
\sZ^0(\bE^\bec)$ and the full subcategory $\sZ^0(\bF^\bec)\subset
\sZ^0(\bE^\bec)$, we see that $0\rarrow B'\rarrow B\oplus C'\rarrow C
\rarrow0$ is an admissible short exact sequence in $\sZ^0(\bE)$ and
$\Phi_\bE(B')\in\sZ^0(\bF^\bec)$.
 Consequently, $\Xi_\bE(B')\simeq\Psi^+_\bE\Phi_\bE(B')\in\sZ^0(\bF)$.
 By assumption of the proposition, it follows that $B'\in\sZ^0(\bF)$.
 The condition~(ii) is dual.
 Alternatively, one could refer to
Theorem~\ref{transferring-exact-structure-from-E-bec-to-E-thm}, but
that would be a more indirect, roundabout argument using
the weak idempotent-completeness assumption on~$\bE$.
\end{proof}

\begin{prop} \label{inheriting-from-E-implies-from-E-bec}
 Let\/ $\bE$ be an exact DG\+category, and let\/ $\bF\subset\bE$ be
a full additive DG\+subcategory closed under shifts and cones.
 Assume that the full subcategory\/ $\sZ^0(\bF)$ inherits an exact
category structure from\/~$\sZ^0(\bE)$.
 Then the full subcategory\/ $\sZ^0(\bF^\bec)$ inherits an exact
category structure from\/~$\sZ^0(\bE^\bec)$.
 So the full DG\+subcategory\/ $\bF$ inherits an exact DG\+category
structure from\/~$\bE$.
\end{prop}

\begin{proof}
 For any DG\+category $\bE$ with shifts and cones, and any full
DG\+subcategory $\bF\subset\bE$ closed under shifts and cones,
the full subcategory $\sZ^0(\bF^\bec)\subset\sZ^0(\bE^\bec)$ is
the full preimage of the full subcategory $\sZ^0(\bF)\subset\sZ^0(\bE)$
under the functor $\Psi^+\:\sZ^0(\bE^\bec)\rarrow\sZ^0(\bE)$.
 Hence it suffices to apply Lemma~\ref{lifting-exact-subcategory-lemma}
to the exact functor $\Theta=\Psi^+_\bE\:\sZ^0(\bE^\bec)\rarrow
\sZ^0(\bE)$ and the full subcategory $\sH=\sZ^0(\bF)\subset\sZ^0(\bE)$.

 Alternatively, one could notice that a short sequence $0\rarrow X
\rarrow Y\rarrow Z\rarrow0$ in $\sZ^0(\bF^\bec)$ is an admissible short
exact sequence in $\sZ^0(\bE^\bec)$ if and only if $0\rarrow
\Psi_\bF^+(X)\rarrow\Psi_\bF^+(Y)\rarrow\Psi_\bF^+(Z)\rarrow0$
is an admissible short exact sequence in~$\sZ^0(\bF)$, and apply
Theorem~\ref{transferring-from-E-to-E-bec} to the DG\+category~$\bF$.
\end{proof}

\begin{cor} \label{inheriting-from-E-or-E-bec-under-all-twists}
 Let\/ $\bE$ be an exact DG\+category, and let\/ $\bF\subset\bE$ be
a full additive DG\+subcategory closed under shifts, twists, and
direct summands.
 Then the full subcategory\/ $\sZ^0(\bF)$ inherits an exact category
structure from\/ $\sZ^0(\bE)$ \emph{if and only if} the full
subcategory\/ $\sZ^0(\bF^\bec)$ inherits an exact category structure
from\/~$\sZ^0(\bE^\bec)$.
 Hence the full DG\+subcategory\/ $\bF$ inherits an exact DG\+category
structure from\/ $\bE$ if any one of the two conditions holds.
\end{cor}

\begin{proof}
 The ``only if'' assertion is provided by
Proposition~\ref{inheriting-from-E-implies-from-E-bec}.
 To prove the ``if'', it suffices to check that the assumptions
of Proposition~\ref{inheriting-from-E-bec-may-imply-from-E} are
satisfied (notice that the weak idempotent-completeness assumption
was not used in the proof of the ``if'' implication in
Proposition~\ref{inheriting-from-E-bec-may-imply-from-E}).
 Indeed, if $\bF\subset\bE$ is a full DG\+subcategory closed under
twists and direct summands, and the object $\Xi_\bE(G)$ belongs to
$\sZ^0(\bF)$ for some $G\in\sZ^0(\bE)$, then $G$ also belongs to
$\sZ^0(\bF)$, since $G$ is a direct summand of a twist $G\oplus G[-1]$
of the object~$\Xi_\bE(G)$.
\end{proof}

\begin{lem} \label{closed-under-extensions-F-and-F-bec}
 Let\/ $\bE$ be an exact DG\+category, and let\/ $\bF\subset\bE$ be
a full additive DG\+subcategory closed under shifts and cones.
 Assume that the full subcategory\/ $\sZ^0(\bF)\subset\sZ^0(\bE)$
is closed under extensions.
 Then the full subcategory\/ $\sZ^0(\bF^\bec)\subset\sZ^0(\bE^\bec)$
is closed under extensions.
 The converse implication holds assuming additionally that\/ $\bF$
is closed under twists and direct summands in\/~$\bE$.
\end{lem}

\begin{proof}
 To prove the direct implication, let $0\rarrow X\rarrow Y\rarrow Z
\rarrow0$ be an admissible short exact sequence in $\sZ^0(\bE^\bec)$
with $X$, $Z\in \sZ^0(\bF^\bec)$.
 Then $0\rarrow\Psi^+_\bE(X)\rarrow\Psi^+_\bE(Y)\rarrow\Psi^+_\bE(Z)
\rarrow0$ is an admissible short exact sequence in $\sZ^0(\bE)$ with
$\Psi^+_\bE(X)$, $\Psi^+_\bE(Z)\in\sZ^0(\bF)$.
 By assumption, it follows that $\Psi^+_\bE(Y)\in\sZ^0(\bF)$, hence
$Y\in\sZ^0(\bF^\bec)$.

 To prove the converse, let $0\rarrow A\rarrow B\rarrow C\rarrow0$ be
an admissible short exact sequence in $\sZ^0(\bE)$
with $A$, $C\in \sZ^0(\bF)$.
 Then $0\rarrow\Phi_\bE(A)\rarrow\Phi_\bE(B)\rarrow\Phi_\bE(C)
\rarrow0$ is an admissible short exact sequence in $\sZ^0(\bE^\bec)$
with $\Phi_\bE(A)$, $\Phi_\bE(C)\in\sZ^0(\bF^\bec)$.
 By assumption, it follows that $\Phi_\bE(B)\in\sZ^0(\bF^\bec)$, hence
$\Xi_\bE(B)\simeq\Psi^+_\bE\Phi_\bE(B)\in\sZ^0(\bF)$.
 As $B$ is a direct summand of a twist $B\oplus B[-1]$ of the object
$\Xi_\bE(B)$, we can conclude that $B\in\sZ^0(\bF)$.
\end{proof}

 The obvious analogue of
Lemma~\ref{closed-under-extensions-F-and-F-bec}, provable in
the same way, holds for the closedness under the kernels of admissible
epimorphisms or the cokernels of admissible monomorphisms in lieu of
the closedness under extensions.

\begin{ex} \label{bad-DG-subcategory-in-abelian-DG}
 The following counterexample shows that the technical additional
assumptions are indeed necessary in
Theorem~\ref{transferring-exact-structure-from-E-bec-to-E-thm}
and Proposition~\ref{inheriting-from-E-bec-may-imply-from-E}, and
also that the assumptions of existence of twists or closedness
under twists are necessary in
Corollaries~\ref{transferring-from-E-bec-to-E-under-all-twists}
and~\ref{inheriting-from-E-or-E-bec-under-all-twists}, as well as
in Lemma~\ref{closed-under-extensions-F-and-F-bec} and in the previous
paragraph (concerning the kernels of admissible epimorphisms and 
the cokernels of admissible monomorphisms).
 It is an example of a well-behaved (in fact, abelian) DG\+category
$\bA$ with a badly behaved full DG\+subcategory $\bB\subset\bA$
whose induced full DG\+subcategory $\bB^\bec\subset\bA^\bec$ is
well-behaved.

 Choose a field~$k$, and let $\sA$ denote the abelian category of
morphisms of $k$\+vector spaces (i.~e., representations of
the quiver $\bullet\to\bullet$).
 Let $\bA=\bC(\sA)$ be the DG\+category of complexes in~$\sA$.
 Denote by $\sE\subset\sA$ the fully exact subcategory of
injective morphisms of $k$\+vector spaces (this is the full
subcategory of projective objects in~$\sA$).

 Let $\bB$ be the full additive DG\+subcategory in $\bA$ consisting
of all the direct summands (in $\sZ^0(\bA)$) of those complexes
in $\sA$ that can be represented as a direct sum $E^\bu\oplus X^\bu$,
where $E^\bu\in\bC(\sE)$ is a complex in $\sE$ and $X^\bu$ is
a contractible complex in~$\sA$.
 The DG\+subcategory $\bB\subset\bA$ is obviously closed under shifts;
let us show that it is also closed under cones.
 Indeed, let $E^\bu$ and $F^\bu$ be two complexes from $\bC(\sE)$, and
$X^\bu$ and $Y^\bu$ be two contractible complexes in~$\sA$.
 Then any morphism $E^\bu\rarrow Y^\bu$ or $X^\bu\rarrow F^\bu$ in
$\bA$ is homotopic to zero.
 In any DG\+category $\bA$, the cones of homotopic closed morphisms
are isomorphic as objects of~$\sZ^0(\bA)$.
 Consequently, in the situation at hand the cone of any closed morphism
$E^\bu\oplus X^\bu\rarrow F^\bu\oplus Y^\bu$ in $\bA$ has the form
$G^\bu\oplus Z^\bu$, where $G^\bu\in\bC(\sE)$ and $Z^\bu$ is
a contractible complex in~$\sA$.

 The functor $\Upsilon_\sA$ from Example~\ref{complexes-bec-example}
provides an equivalence of additive categories $\sG(\sA)\simeq
\sZ^0(\bA^\bec)$ (where $\sG(\sA)$ is the category of graded objects
in~$\sA$) forming a commutative triangle diagram with the forgetful
functor $\sZ^0(\bA)=\sC(\sA)\rarrow\sG(\sA)$ and the functor
$\Phi_\bA\:\sZ^0(\bA)\rarrow\sZ^0(\bA^\bec)$.
 Hence both the additive categories $\sZ^0(\bA)$ and $\sZ^0(\bA^\bec)$
are abelian, so $\bA$ is indeed an abelian DG\+category in the sense
of the next Section~\ref{abelian-dg-categories-subsecn}
(cf.\ Example~\ref{exact-dg-category-of-complexes} below).
 Furthermore, one can easily see from the constructions that
$\bB^\bec=\bA^\bec$ (since all the contractible objects of $\bA$ belong
to~$\bB$).
 Consequently, we have $\sZ^0(\bB^\bec)=\sZ^0(\bA)\simeq\sG(\sA)$.

 Let us introduce notation for the indecomposable objects
$I_{01}=(0\to k)$, \ $I_{10}=(k\to0)$, and $I_{11}=(k\to k)$
of the abelian category~$\sA$.
 So we have $I_{01}\in\sE$ and $I_{11}\in\sE$, but $I_{10}\notin\sE$.
 Let us show that the one-term complex $I_{10}\in\sA\subset
\sC(\sA)$ does not belong to $\sZ^0(\bB)\subset\sC(\sA)$.
 Indeed, suppose that $I_{10}$ is a direct summand in a complex of
the form $E^\bu\oplus X^\bu$, where $E^\bu\in\sC(\sE)$ and $X^\bu$
is a contractible complex.
 Since any morphism from $I_{10}$ to an object of the full
subcategory $\sE\subset\sA$ vanishes, it follows that any morphism
$I_{10}\rarrow E^\bu$ vanishes in~$\sC(\sA)$.
 Hence $I_{10}$ is a direct summand of a contractible complex in
$\sC(\sA)$, an obvious contradiction.

 The additive category $\sZ^0(\bB)$ is \emph{not} abelian, because
the monomorphism $i\:I_{01}\rarrow I_{11}$ has no cokernel in it.
 Indeed, if such a cokernel existed, it would have to be preserved
by the functor $\Phi_\bB\:\sZ^0(\bB)\rarrow\sZ^0(\bB^\bec)$,
which means just the forgetful functor $\sC(\sA)\rarrow\sG(\sA)$
restricted to $\sZ^0(\bB)\subset\sC(\sA)$.
 This would imply that the cokernel of~$i$ in $\sZ^0(\bB)$ agrees
with the cokernel of~$i$ in $\sC(\sA)$; but the cokernel $I_{10}$
of the morphism~$i$ in $\sC(\sA)$ does not belong to~$\sZ^0(\bB)$.

 Moreover, the full subcategory $\sZ^0(\bB)$ does \emph{not} inherit
an exact category structure from $\sZ^0(\bA)$ and does \emph{not}
admit an exact category structure $\bec$\+compatible with the abelian
exact structure on $\sZ^0(\bB^\bec)$.
 Indeed, if it did, then such an exact structure on $\sZ^0(\bB)$
would be DG\+compatible, and by
Lemma~\ref{DG-compatible-reflecting-admissible-morphisms} it would
follow that the functor $\Xi_\bB\:\sZ^0(\bB)\rarrow\sZ^0(\bB)$
reflects admissible monomorphisms (since the DG\+category $\bB$ is
additive and idempotent-complete with shifts and cones).

 On the other hand, the short exact sequence $0\rarrow\Xi_\bA(I_{01})
\rarrow\Xi_\bA(I_{11})\rarrow\Xi_\bA(I_{10})\rarrow0$ in $\sC(\sA)$
obtained by applying $\Xi_\bA$ to the short exact sequence $0\rarrow
I_{01}\rarrow I_{11}\rarrow I_{10}\rarrow0$ in $\sA$ would be
admissible in $\sZ^0(\bB)$, since it is admissible in $\sZ^0(\bA)$
and its terms belong to~$\sZ^0(\bB)$.
 But the monomorphism $i\:I_{01}\rarrow I_{11}$ cannot be admissible
in $\sZ^0(\bB)$, as it has no cokernel in this category.
\end{ex}

 In the preceding results in this
Section~\ref{inheriting-exact-dg-category-structure-subsecn},
we were starting with a full DG\+sub\-cat\-e\-gory $\bF\subset\bE$.
 The next proposition starts with an additive subcategory
$\sL\subset\sZ^0(\bE^\bec)$.

\begin{prop} \label{subcategory-L-prop}
\textup{(a)} Let $\bE$ be an additive DG\+category with shifts and
cones, and let\/ $\sL\subset\sZ^0(\bE^\bec)$ be a full additive
subcategory preserved by the shift functors\/~$[n]$, \,$n\in\Gamma$.
 Let\/ $\bF\subset\bE$ be the full DG\+subcategory consisting of all
objects $F\in\bE$ such that $\Phi_\bE(F)\in\sL$, and let\/
$\widetilde\sL=\Xi_{\bE^\bec}^{-1}(\sL)\subset\sZ^0(\bE^\bec)$ be
the full subcategory consisting of all objects $\widetilde L\in
\sZ^0(\bE^\bec)$ such that\/ $\Xi_{\bE^\bec}(\widetilde L)\in\sL$.
 Then\/ $\bF\subset\bE$ is a full additive DG\+subcategory closed
under shifts and twists, and\/ $\sZ^0(\bF^\bec)=\widetilde\sL$ as
a full subcategory in\/~$\sZ^0(\bE^\bec)$. \par
\textup{(b)} In the context of~(a), assume that the full subcategory\/
$\sL\subset\sZ^0(\bE^\bec)$ is closed under direct summands.
 Then so is the full DG\+subcategory\/ $\bF\subset\bE$.
 Furthermore, given an object $F\in\bE$, one has $F\in\bF$ if and
only if\/ $\Phi_\bE(F)\in\widetilde\sL$. \par
\textup{(c)} In the context of~(a), assume that\/ $\bE$ is an exact
DG\+category and the full additive subcategory\/ $\sL\subset
\sZ^0(\bE^\bec)$ inherits an exact category structure.
 Then the full additive DG\+subcategory\/ $\bF\subset\bE$ inherits
an exact DG\+category structure. \par
\textup{(d)} In the context of~(c), assume that\/ the full
subcategory\/ $\sL\subset\sZ^0(\bE^\bec)$ is closed under extensions.
 Then the inclusion $\sL\subset\sZ^0(\bF^\bec)$ holds.
\end{prop}

\begin{proof}
 Part~(a): the full DG\+subcategory $\bF\subset\bE$ is additive and
preserved by the shifts, since the full subcategory $\sL\subset
\sZ^0(\bE^\bec)$~is.
 Furthermore, $\bF$ is closed under twists in $\bE$ because
the functor $\Phi_\bE$ takes twists to isomorphisms
(see Lemma~\ref{twists-into-isomorphisms}).
 Finally, an object $X\in\sZ^0(\bE^\bec)$ belongs to $\sZ^0(\bF^\bec)$
if and only the object $\Psi^+_\bE(X)$ belongs to $\sZ^0(\bF)$, which
in our context means that the object $\Xi_{\bE^\bec}(X)\simeq
\Phi_\bE\Psi_\bE^+(X)[1]$ belongs to~$\sL$.

 Part~(b): the first assertion is obvious.
 To prove the second one, notice that the functor $\Xi_{\bE^\bec}\circ
\Phi_\bE\:\sZ^0(\bE)\rarrow\sZ^0(\bE^\bec)$ is isomorphic to
$\Phi_\bE\oplus\Phi_\bE[-1]$ (cf.\ Example~\ref{minimal-exact-dg}).
 Hence, for any given object $F\in\sZ^0(\bE)$, one has $\Phi_\bE(F)
\in\widetilde\sL$ if and only if $\Phi_\bE(F)\in\sL$.

 Part~(c): in view of
Proposition~\ref{inheriting-from-E-implies-from-E-bec}
or Corollary~\ref{inheriting-from-E-or-E-bec-under-all-twists},
it suffices to check that the full additive subcategory
$\sZ^0(\bF)\subset\sZ^0(\bE)$ inherits an exact category structure.
 For this purpose, one can apply
Lemma~\ref{lifting-exact-subcategory-lemma} to the exact functor
$\Phi_\bE\:\sZ^0(\bE)\rarrow\sZ^0(\bE^\bec)$ and
the full subcategory $\sL\subset\sZ^0(\bE^\bec)$.
 In part~(d), the inclusion $\sL\subset\widetilde\sL$ follows from
the short exact sequence $0\rarrow L[1]\rarrow\Xi_{\bE^\bec}(L)\rarrow
L\rarrow0$, which is admissible in $\sZ^0(\bE^\bec)$ for any $L\in\sL$
by the definition of a DG\+compatible exact structure.
\end{proof}

\begin{ex} \label{L-tilde-example}
 The construction of Proposition~\ref{subcategory-L-prop} does
provide full DG\+sub\-cat\-e\-gories inheriting exact DG\+category
structures, but it does not always work quite as one would expect.
 The following example is thematic.

 Let $\biR^\cu=(R^*,d,h)$ be a CDG\+ring and $\bA=\biR^\cu\bmodl$ be
the DG\+category of left CDG\+modules over~$\biR^\cu$.
 Then the functor $\Upsilon_{\biR^\cu}$ from
Example~\ref{CDG-ring-bec-example} provides an equivalence of
additive categories $R^*\smodl\simeq\sZ^0(\bA^\bec)$ (where
$R^*\smodl$ is the category of graded left $R^*$\+modules) forming
a commutative triangle diagram with the forgetful functor
$\sZ^0(\biR^\cu\bmodl)\rarrow R^*\smodl$ and the functor
$\Phi_\bA\:\sZ^0(\bA)\rarrow\sZ^0(\bA^\bec)$.
 The equivalence $\Upsilon_{\biR^\cu}$ also forms a commutative
triangle diagram with the functor $G^+\:R^*\smodl\rarrow
\sZ^0(\biR^\cu\bmodl)$ and the functor $\Psi_\bA^+\:\sZ^0(\bA^\bec)
\rarrow\sZ^0(\bA)$.
 Both the additive categories $\sZ^0(\bA)$ and $\sZ^0(\bA^\bec)$ are
abelian, so $\bA$ is an abelian DG\+category in the sense of the next
Section~\ref{abelian-dg-categories-subsecn} (cf.\
Example~\ref{abelian-dg-category-of-cdg-modules}).

 Let $\sL\subset R^*\smodl$ be the full subcategory of projective
(or flat) graded $R^*$\+modules.
 Then one certainly has $\Xi_{\bA^\bec}(\sL)\subset\sL$, as
the full subcategory $\sL$ is closed under extensions in
$R^*\smodl$.
 For the same reason, $\sL$ also inherits an exact category structure
from the abelian exact structure on $R^*\smodl$.
  The full DG\+subcategory $\bF\subset\bA$ assigned to the class $\sL$
by the construction of Proposition~\ref{subcategory-L-prop}
consists of all left CDG\+modules over $\biR^\cu$ whose underling graded
$R^*$\+modules are projective (respectively, flat).
 However, the full subcategory $\sZ^0(\bF^\bec)=\widetilde\sL\subset
R^*\smodl$ can be \emph{wider} than~$\sL$.
 In fact, $\widetilde\sL$ is the full subcategory consisting of all
graded $R^*$\+modules $\widetilde L$ for which the underlying graded
$R^*$\+module of the CDG\+module $G^+(\widetilde L)$ is projective
(resp., flat).

 Here is a specific counterexample.
 Consider the following DG\+ring $\biR^\bu=(R^*,d)$.
 The graded ring $R^*=k[\epsilon]$, where $\epsilon\in R^{-1}$
and $\epsilon^2=0$, is the exterior algebra in one variable over
a field~$k$.
 The differential~$d$ is given by $d(\epsilon)=1$; so $\biR^\bu$
is an acyclic DG\+algebra over~$k$.

 Consider the graded $R^*$\+module $L^*=k=R^*/\epsilon R^*$,
concentrated in the cohomological degree~$0$.
 By construction, the elements of cohomological degree $n\in\boZ$
in the graded $R^*$\+module $G^+(L^*)$ are formal expressions
$l+\delta l'$ with $l\in L^n$ and $l'\in L^{n-1}$.
 The action of $R^*$ in $G^+(L^*)$ is given by the rule
$r(l+\delta l')=rl-(-1)^{|r|}d(r)l'+(-1)^{|r|}\delta rl'$ for all
$r\in R^{|r|}$.
 Now, denoting by $l_0=1\in L^0$ a nonzero vector in the one-dimensional
$k$\+vector space $L^*$, we have $\epsilon (0+\delta l_0)=l_0+\delta0$.
 Thus $G^+(L^*)$ is a free graded $R^*$\+module with one
generator~$0+\delta l_0\in G^+(L^*)$ of cohomological degree~$1$.
 But $L^*$ is \emph{not} a flat graded $R^*$\+module.
 
 In fact, in this example one has $\widetilde\sL=R^*\smodl$, as
the graded $R^*$\+module $G^+(L^*)$ is projective for \emph{any} graded
$R^*$\+module~$L^*$.
 Moreover, the graded $R^*$\+module $M^*$ is projective for any
DG\+module $(M^*,d_M)$ over~$\biR^\bu$; so one has $\bF=\bA$.
 Indeed, viewed as a complex with the differential given by
the action of~$\epsilon$, the underlying graded $R^*$\+module $M^*$
of any DG\+module $(M^*,d_M)$ is a contractible complex, with
a contracting homotopy~$d_M$.
\end{ex}

\subsection{Abelian DG-categories}
\label{abelian-dg-categories-subsecn}
 Let $\bA$ be an additive DG\+category with shifts and cones.
 We will say that the DG\+category $\bA$ is \emph{abelian} if both
the additive categories $\sZ^0(\bA)$ and $\sZ^0(\bA^\bec)$ are
abelian.
 (As we will see in Proposition~\ref{A-abelian-implies-A-bec-abelian}
below, it suffices that the additive category $\sZ^0(\bA)$ be abelian.)

\begin{lem} \label{abelian-exact-DG-structure}
 Let\/ $\bA$ be an abelian DG\+category.
 Then the abelian exact category structure on\/ $\sZ^0(\bA)$ is\/
$\bec$\+compatible with the abelian exact category structure
on\/~$\sZ^0(\bA^\bec)$.
 So the pair of abelian exact category structures defines
an exact DG\+category structure on\/~$\bA$.
\end{lem}

 The exact DG\+category structure defined in the lemma is called
the \emph{abelian exact DG\+category structure} on an abelian
DG\+category\/~$\bA$.

\begin{proof}
 Both the functors $\Phi_\bA$ and $\Psi_\bA$ have left and right
adjoints by Lemma~\ref{G-functors-in-DG-categories}; so they preserve
kernels and cokernels.
 Therefore, viewed as functors between the abelian (by assumption)
categories $\sZ^0(\bA)$ and $\sZ^0(\bA^\bec)$, they are exact.
 Furthermore, both the functors $\Phi_\bA$ and $\Psi_\bA$ are faithful
by Lemma~\ref{G-functors-in-DG-categories}.
 A faithful exact functor between abelian categories reflects short
exact sequences.
\end{proof}

\begin{ex}
 The following example shows that, in an idempotent-complete exact
DG\+category $\bE$, the exact category $\sZ^0(\bE^\bec)$ can be abelian
with the abelian exact structure, while the category $\sZ^0(\bE)$
is not abelian.

 Let $\bA=\bC(k\rmodl)$ be the DG\+category of complexes of vector
spaces over a field~$k$.
 Then, following Example~\ref{complexes-bec-example}, the additive
category $\sZ^0(\bA^\bec)$ is equivalent to the category of graded
vector spaces $\sG(k\rmodl)$.
 So both the additive categories $\sZ^0(\bA)=\sC(k\rmodl)$ and
$\sZ^0(\bA^\bec)$ are abelian, and $\bA$ is an abelian DG\+category.
 By Lemma~\ref{abelian-exact-DG-structure}, the DG\+category $\bA$
carries the abelian exact DG\+category structure.

 Consider the full DG\+subcategory $\bE\subset\bA$ whose objects are
the acyclic complexes of $k$\+vector spaces (cf.\
Examples~\ref{adding-twists-examples}).
 Then the full DG\+subcategory $\bE^\bec$ coincides with the whole
ambient DG\+category $\bA^\bec$, that is $\bE^\bec=\bA^\bec$.
 The full subcategory $\sZ^0(\bE)\subset\sZ^0(\bA)$ is closed under
extensions and direct summands (as well as the kernels of epimorphisms
and the cokernels of monomorphisms); so it inherits an exact
category structure.
 Thus the full additive DG\+subcategory $\bE$ inherits an exact
DG\+category structure from~$\bA$.

 So $\bE$ becomes an idempotent-complete exact DG\+category in which
the exact category $\sZ^0(\bE^\bec)$ is abelian with the abelian
exact category structure.
 But the additive category $\sZ^0(\bE)$ is far from being abelian:
it does not even have kernels or cokernels of some morphisms.
\end{ex}

\begin{lem} \label{conservative-exact-reflects-abelianity}
\textup{(a)} Let\/ $\sA$ be an additive category with kernels and
cokernels, $\sB$ be an abelian category, and\/ $\Theta\:\sA\rarrow\sB$
be a conservative additive functor preserving kernels and cokernels.
 Then\/ $\sA$ is also an abelian category. \par
\textup{(b)} Let\/ $\sA$ and\/ $\sB$ be abelian categories endowed
with exact category structures, and let\/ $\Theta\:\sA\rarrow\sB$ be
an additive functor preserving kernels and cokernels and reflecting
admissible short exact sequences.
 Assume that the exact structure on\/ $\sB$ is the abelian exact
structure.
 Then the exact structure on\/ $\sA$ is also the abelian exact
structure.
\end{lem}

\begin{proof}
 Both the assertions follow straightforwardly from the definitions.
\end{proof}

\begin{prop} \label{A-bec-abelian-may-imply-A-abelian}
\textup{(a)} Let\/ $\bA$ be an additive DG\+category with shifts
and cones.
 Assume that all kernels and cokernels exist in the additive
category\/ $\sZ^0(\bA)$, while the additive category\/
$\sZ^0(\bA^\bec)$ is abelian.
 Then the category\/ $\sZ^0(\bA)$ is abelian, too. \par
\textup{(b)} Let\/ $\bA$ be an exact DG\+category.
 Assume that all kernels and cokernels exist in the additive
category\/ $\sZ^0(\bA)$, while the exact category\/ $\sZ^0(\bA^\bec)$
is abelian with the abelian exact category structure.
 Then the exact category\/ $\sZ^0(\bA)$ is also abelian with
the abelian exact category structure.
\end{prop}

\begin{proof}
 Part~(a): the additive functor $\Phi_\bA\:\sZ^0(\bA)\rarrow
\sZ^0(\bA^\bec)$ preserves kernels and cokernels
(because it has adjoints on both sides by
Lemma~\ref{G-functors-in-DG-categories}), and it is also
conservative (by Lemma~\ref{Phi-Psi-conservative}).
 Hence the assertion follows from
Lemma~\ref{conservative-exact-reflects-abelianity}(a) applied
to the functor~$\Phi_\bA$.
 Part~(b) follows from
Lemma~\ref{conservative-exact-reflects-abelianity}(b).
\end{proof}

\begin{cor} \label{with-twists-A-bec-abelian-implies-A-abelian}
\textup{(a)} Let\/ $\bA$ be an idempotent-complete additive
DG\+category with shifts and twists.
 Assume that the additive category\/ $\sZ^0(\bA^\bec)$ is abelian.
 Then the additive category\/ $\sZ^0(\bA)$ is abelian, too. \par
\textup{(b)} Let\/ $\bA$ be an idempotent-complete exact DG\+category
with twists.
 Assume that the exact category\/ $\sZ^0(\bA^\bec)$ is abelian with
the abelian exact category structure.
 Then the exact category\/ $\sZ^0(\bA)$ is also abelian with
the abelian exact category structure.
\end{cor}

\begin{proof}
 In view of Proposition~\ref{A-bec-abelian-may-imply-A-abelian},
it suffices to show that all kernels and cokernels exist in
the additive category~$\sZ^0(\bA)$.
 Indeed, let $g\:B\rarrow C$ be a morphism in $\sZ^0(\bA)$.
 Then the morphism $\Phi(g)\:\Phi(B)\rarrow\Phi(C)$ has a kernel
in the abelian category~$\sZ^0(\bA^\bec)$.
 Since the functor $\Psi^+\:\sZ^0(\bA^\bec)\rarrow\sZ^0(\bA)$ is
a right adjoint (to~$\Phi[-1]$) by
Lemma~\ref{G-functors-in-DG-categories}, it follows that
the morphism $\Psi^+\Phi(g)$ has a kernel in~$\sZ^0(\bA)$.
 By Lemma~\ref{Xi-decomposed}, we have $\Xi(g)\simeq\Psi^+\Phi(g)$;
so the morphism $\Xi(g)$ has a kernel.
 It remains to apply Lemma~\ref{Xi-creates-co-kernels}(a) in order
to establish that the morphism~$g$ has a kernel in~$\sZ^0(\bA)$.
 The dual argument proves the existence of cokernels.
\end{proof}

\begin{prop} \label{A-abelian-implies-A-bec-abelian}
\textup{(a)} Let\/ $\bA$ be an additive DG\+category with shifts
and cones.
 Assume that the additive category\/ $\sZ^0(\bA)$ is abelian.
 Then the category\/ $\sZ^0(\bA^\bec)$ is abelian, too. \par
\textup{(b)} Let\/ $\bA$ be an exact DG\+category.
 Assume that the exact category\/ $\sZ^0(\bA)$ is abelian with
the abelian exact category structure.
 Then the exact category\/ $\sZ^0(\bA^\bec)$ is also abelian with
the abelian exact category structure.
\end{prop}

\begin{proof}
 Since all kernels and cokernels exist in the abelian category
$\sZ^0(\bA)$, \,Lemma~\ref{Psi-creates-co-kernels} tells that
they also exist in the additive category~$\sZ^0(\bA^\bec)$.
 The additive functor $\Psi_\bA^+\:\sZ^0(\bA^\bec)\rarrow\sZ^0(\bA)$
preserves kernels and cokernels (because it has adjoints on both sides
by Lemma~\ref{G-functors-in-DG-categories}), and it is also
conservative (by Lemma~\ref{Phi-Psi-conservative}).
 Hence the assertion follows from
Lemma~\ref{conservative-exact-reflects-abelianity}(a) applied
to the functor~$\Psi^+_\bA$.
 Part~(b) follows from part~(a) and
Lemma~\ref{conservative-exact-reflects-abelianity}(b).
\hbadness=1950
\end{proof}

\begin{prop} \label{abelian-DG-categories-have-twists}
 Any abelian DG\+category is idempotent-complete, and all twists
exist in it.
\end{prop}

\begin{proof}
 The first assertion holds by the definition, as any abelian category
is idempotent-complete.
 The nontrivial part is the existence of twists.
 The discussion in~\cite[Section~7]{NST} sheds some additional light
on the argument below.

 Let $\bA$ be an abelian DG\+category.
 Then the DG\+category $\bA^\bec$ is also abelian by
Proposition~\ref{A-abelian-implies-A-bec-abelian}(a), since
the additive category $\sZ^0(\bA^\bec)$ is abelian.
 By Proposition~\ref{iterated-bec-construction}, we have a fully
faithful DG\+functor $\bec\bec\:\bA\rarrow\bA^{\bec\bec}$.
 Hence the induced additive functor $\sZ^0(\bec\bec)\:\sZ^0(\bA)
\rarrow\sZ^0(\bA^{\bec\bec})$ is a fully faithful functor between
abelian categories.

 Let us check that the functor $\sZ^0(\bec\bec)$ is exact.
 Indeed, the functor $\Phi_\bA\:\sZ^0(\bA)\rarrow\sZ^0(\bA^\bec)$
preserves and reflects short exact sequences, and so does the
functor $\Psi^+_{\bA^\bec}\:\sZ^0(\bA^{\bec\bec})\rarrow
\sZ^0(\bA^\bec)$.
 It remains to recall that the three functors $\sZ^0(\bec\bec)$,
\,$\Psi^+_{\bA^\bec}$, and $\Phi_\bA$ form a commutative triangle
diagram of additive functors by
Lemma~\ref{obvious-composition-of-functors-observation}.

 We have shown that the essential image of the functor $\sZ^0(\bec\bec)$
is a full subcategory closed under kernels and cokernels in the abelian
category $\sZ^0(\bA^{\bec\bec})$.
 In order to prove that this full subcategory coincides with the whole
category $\sZ^0(\bA^{\bec\bec})$, it remains to check that any object
of $\sZ^0(\bA^{\bec\bec})$ is the cokernel of a morphism between two
objects coming from $\sZ^0(\bA)$.
 For this purpose, it is sufficient to see that every object of
$\sZ^0(\bA^{\bec\bec})$ is the target of an epimorphism with
the source belonging to $\sZ^0(\bA)$.

 Indeed, for any object $X\in\bA^{\bec\bec}$ we have a natural
short exact sequence $0\rarrow X[-1]\rarrow\Xi_{\bA^{\bec\bec}}(X)
\rarrow X\rarrow0$ in $\sZ^0(\bA^{\bec\bec})$ by
Lemma~\ref{cone-kernel-cokernel}, and
$$
 \Xi_{\bA^{\bec\bec}}(X)\simeq\Phi_{\bA^\bec}\Psi^-_{\bA^\bec}(X)
 \simeq\sZ^0(\bec\bec)\Psi^-_\bA\Psi^-_{\bA^\bec}(X)
$$
by Lemmas~\ref{difficult-composition-of-functors-observation}
and~\ref{Xi-decomposed}.
 So $X$ is the cokernel of a natural morphism
$\Xi_{\bA^{\bec\bec}}(X)[-1]\rarrow\Xi_{\bA^{\bec\bec}}(X)$ in
$\sZ^0(\bA^{\bec\bec})$ between two objects coming from $\sZ^0(\bA)$.

 We can conclude that the functor $\sZ^0(\bec\bec)\:\sZ^0(\bA)\rarrow
\sZ^0(\bA^{\bec\bec})$ is an equivalence of (abelian) categories,
and it follows that the fully faithful DG\+functor $\bec\bec\:
\bA\rarrow\bA^{\bec\bec}$ is an equivalence of DG\+categories.
 Since the DG\+category $\bB^\bec$ has twists for any DG\+category
$\bB$, it follows that our DG\+category $\bA$ has twists.

 For a further discussion suggesting a direct constuction of twists
of objects of $\bA$ as cokernels in $\sZ^0(\bA)$,
see~\cite[proof of Proposition~3.9]{PS5}.
\end{proof}

 Let us formulate the main point of the proof above as a separate
assertion.

\begin{cor}
 For any abelian DG\+category\/ $\bA$, the DG\+functor\/
$\bec\bec\:\bA\rarrow\bA^{\bec\bec}$ from
Proposition~\ref{iterated-bec-construction} is an equivalence of
DG\+categories.
\end{cor}

\begin{proof}
 Follows from Propositions~\ref{iterated-bec-construction}
and~\ref{abelian-DG-categories-have-twists}.
\end{proof}

 The following corollaries, summarizing the previous results, list
several equivalent characterizations of abelian DG\+categories.

\begin{cor}
 For any DG\+category\/ $\bA$ with shifts, the following conditions
are equivalent:
\begin{enumerate}
\item the DG\+category\/ $\bA$ is abelian;
\item the DG\+category\/ $\bA$ has cones, and the preadditive
category\/ $\sZ^0(\bA)$ is abelian;
\item the DG\+category\/ $\bA$ is additive, has cones, all kernels
and cokernels exists in the additive category $\sZ^0(\bA)$, and
the additive category\/ $\sZ^0(\bA^\bec)$ is abelian;
\item the DG\+category\/ $\bA$ is additive and idempotent-complete,
has all twists, and the additive category\/ $\sZ^0(\bA^\bec)$ is
abelian. 
\end{enumerate}
\end{cor}

\begin{proof}
 Follows from Propositions~\ref{A-bec-abelian-may-imply-A-abelian}(a),
\ref{A-abelian-implies-A-bec-abelian}(a),
\ref{abelian-DG-categories-have-twists}, and
Corollary~\ref{with-twists-A-bec-abelian-implies-A-abelian}(a).
\end{proof}

\begin{cor}
 For any exact DG\+category\/ $\bA$, the following conditions
are equivalent:
\begin{enumerate}
\item the exact DG\+category\/ $\bA$ is abelian with the abelian
exact DG\+category structure;
\item the exact category\/ $\sZ^0(\bA)$ is abelian with the abelian
exact structure;
\item all kernels and cokernels exist in the additive category\/
$\sZ^0(\bA)$, and the exact category\/ $\sZ^0(\bA^\bec)$ is abelian
with the abelian exact structure;
\item the DG\+category\/ $\bA$ is idempotent-complete, has all twists,
and the exact category\/ $\sZ^0(\bA^\bec)$ is abelian with the abelian
exact structure.
\end{enumerate}
\end{cor}

\begin{proof}
 Follows from Propositions~\ref{A-bec-abelian-may-imply-A-abelian}(b),
\ref{A-abelian-implies-A-bec-abelian}(b),
\ref{abelian-DG-categories-have-twists}, and
Corollary~\ref{with-twists-A-bec-abelian-implies-A-abelian}(b).
\end{proof}

\subsection{Examples} \label{exact-dg-examples-subsecn}
 In this section we briefly discuss the natural exact (or even abelian)
DG\+category structures for our thematic examples of DG\+categories
from Sections~\ref{examples-secn} and~\ref{A-bec-examples-subsecn}.

\begin{ex} \label{exact-dg-category-of-complexes}
 Let $\sA$ be an additive category and $\bC(\sA)$ be the DG\+category
of complexes in~$\sA$.
 Building up on the discussion in Example~\ref{complexes-bec-example},
we consider the additive category $\sG(\sA)$ of graded objects in $\sA$
and the additive category $\sC(\sA)$ of complexes in~$\sA$.
 The forgetful functor $\Theta\:\sC(\sA)\rarrow\sG(\sA)$ has adjoint
functors on both sides, $G^+$ and $G^-\:\sG(\sA)\rarrow\sC(\sA)$,
which were constructed in Example~\ref{complexes-bec-example}.

 Let $\sE$ be an exact category.
 Then the additive category $\sG(\sE)$ has a natural exact category
structure in which a pair of composable morphisms is an admissible
short exact sequence if and only if it is an admissible short exact
sequence in every degree.
 Applying Lemma~\ref{faithful-conservative-lifting-exact-structure} to
the faithful, conservative forgetful functor
$\Theta\:\sC(\sE)\rarrow\sG(\sE)$, one produces the (standard and
well-known) exact category structure on the category of
complexes~$\sC(\sE)$.
 It is easy to check that the assumptions of the lemma hold: in fact,
the functor $\Theta$ preserves kernels and cokernels, and any morphism
in $\sC(\sE)$ whose image under $\Theta$ has a (co)kernel in
$\sG(\sE)$ has a (co)kernel in~$\sC(\sE)$.

 One can easily check that the exact category structure on
$\sC(\sE)=\sZ^0(\bC(\sE))$ is DG\+admissible.
 Applying Theorem~\ref{transferring-from-E-to-E-bec}, one produces
an exact DG\+category structure on the DG\+category
of complexes~$\bC(\sE)$.

 The composition $\Theta\circ G^+\:\sG(\sE)\rarrow\sG(\sE)$ of
the functors $G^+$ and $\Theta$ is isomorphic to the direct sum
$\Id\oplus\Id[-1]$ of the identity endofunctor and the shift functor.
 Therefore, the functor $\Theta\circ G^+$ preserves and reflects
admissible short exact sequences.
 By construction, so does the functor~$\Theta$.
 It follows that the functor $G^+$ also preserves and reflects
admissible short exact sequences.

 Looking on the commutative diagram~\eqref{complexes-Psi-diagram} and
keeping in mind that the functor $\Psi^+_{\bC(\sE)}$ preserves and
reflects admissible short exact sequences by the definition of
an exact DG\+category, one concludes that the functor $\Upsilon_\sE$
preserves and reflects admissible short exact sequences.
 Thus the exact category structure on $\sG(\sE)$ is inherited from
the exact structure on $\sZ^0(\bC(\sE)^\bec)$ under the inclusion
$\Upsilon_\sE\:\sG(\sE)\rightarrowtail\sZ^0(\bC(\sE)^\bec)$.

 The composition $\Phi_{\bC(\sE)}\circ\Psi^+_{\bC(\sE)}\:
\sZ^0(\bC(\sE)^\bec)\rarrow\sZ^0(\bC(\sE)^\bec)$ of the functors
$\Psi^+_{\bC(\sE)}$ and $\Phi_{\bC(\sE)}$ is also isomorphic to
the direct sum $\Id\oplus\Id[1]$ of the identity endofunctor and
the shift functor.
 Therefore, any short exact sequence in the exact category
$\sZ^0(\bC(\sE)^\bec)$ is isomorphic to a direct summand of the image
of a short exact sequence from the exact category $\sZ^0(\bC(\sE))$
under the functor~$\Phi_{\bC(\sE)}$.
 In view of the commutative diagram~\eqref{complexes-tilde-Phi-diagram},
it follows that any short exact sequence in $\sZ^0(\bC(\sE)^\bec)$
is a direct summand of a short exact sequence in
$\Upsilon_\sE(\sG(\sE))$.

 Let $\sF\subset\sE$ be a full additive subcategory inheriting
an exact category structure.
 Then the exact category structure on the category of graded objects
$\sG(\sF)$ is inherited from the exact structure on~$\sG(\sE)$.
 It follows that the exact structure on the category of complexes
$\sC(\sF)$ is inherited from the exact structure on~$\sC(\sE)$.
 Hence the exact DG\+category structure on the DG\+category of
complexes $\bC(\sF)$ is inherited from the exact DG\+category
structure on~$\bC(\sE)$.

 Finally, let $\sA$ be an abelian category.
 In this case, the functor $\Upsilon_\sA\:\sG(\sA)\rarrow
\sZ^0(\bC(\sA)^\bec)$ is an equivalence of categories (since $\sA$
is idempotent-complete; see Example~\ref{complexes-bec-example}).
 Both the category of graded objects $\sG(\sA)$ and the category
of complexes $\sC(\sA)=\sZ^0(\bC(\sA))$ are abelian when $\sA$ is
abelian.
 Thus the DG\+category of complexes $\bC(\sA)$ is abelian.
 The construction of the exact DG\+category structure on the category
of complexes above produces the abelian exact DG\+category structure
on $\bC(\sA)$ from the abelian exact category structure on~$\sA$.
\end{ex}

\begin{ex} \label{abelian-dg-category-of-cdg-modules}
 Let $\biR^\cu=(R^*,d,h)$ be a CDG\+ring and $\bA=\biR^\cu\bmodl$ be
the DG\+category of left CDG\+modules over~$\biR^\cu$.

 Then $\sZ^0(\bA)$ is the category of CDG\+modules and closed
morphisms of degree~$0$ between them; one can easily see that it is
an abelian category.
 In fact, $\sZ^0(\bA)$ is equivalent to the abelian category
$R^*[\delta]\smodl$ of left graded modules over the graded ring
$R^*[\delta]$ from Section~\ref{cdg-revisited-subsecn} (as explained
in the first paragraph of the proof of
Proposition~\ref{G-plus-minus-prop}).

 Furthermore, according to Example~\ref{CDG-ring-bec-example},
the functor $\Upsilon_{\biR^\cu}\:R^*\smodl\rarrow\sZ^0(\bA^\bec)$
is an equivalence between the category $\sZ^0(\bA^\bec)$
and the abelian category of graded left $R^*$\+modules.
 So the category $\sZ^0(\bA^\bec)$ is abelian as well.

 Thus the DG\+category $\bA=\biR^\cu\bmodl$ of left CDG\+modules
over a CDG\+ring $\biR^\cu$ is an abelian DG\+category.
\end{ex}

\begin{ex} \label{abelian-dg-category-of-quasi-coh-cdg-modules}
 Let $X$ be a scheme and $\biB^\cu=(B^*,d,h)$ be a quasi-coherent
CDG\+quasi-algebra over~$X$.
 Let $\bA=\biB^\cu\bqcoh$ be the DG\+category of quasi-coherent left
CDG\+modules over $\biB^\cu$ (as defined in
Section~\ref{quasi-cdg-subsecn}).

 Then $\sZ^0(\bA)$ is the category of quasi-coherent CDG\+modules
and closed morphisms between them.
 Similarly to Example~\ref{abelian-dg-category-of-cdg-modules},
this category is abelian (this fact was mentioned at the end of
Section~\ref{quasi-cdg-subsecn}).
 In fact, $\sZ^0(\bA)$ is equivalent to the category of
quasi-coherent graded left modules over the quasi-coherent
graded quasi-algebra $B^*[\delta]$ on $X$ constructed in
Example~\ref{qcoh-CDG-bec-example}.

 Furthermore, as explained in Example~\ref{qcoh-CDG-bec-example},
the functor $\Upsilon_{\biB^\cu}\:B^*\sqcoh\rarrow\sZ^0(\bA^\bec)$
is an equivalence between the category $\sZ^0(\bA^\bec)$
and the abelian category of quasi-coherent graded left $B^*$\+modules.
 Hence the category $\sZ^0(\bA^\bec)$ is abelian as well.

 Therefore, the DG\+category $\bA=\biB^\cu\bqcoh$ of quasi-coherent
left CDG\+modules over a quasi-coherent CDG\+quasi-algebra $\biB^\cu$
is an abelian DG\+category.
\end{ex}

\begin{ex} \label{exact-dg-category-of-factorizations}
 Let $\sA$ be an additive category and $\Lambda\:\sA\rarrow\sA$ be
an autoequivalence.
 We continue the discussion of factorization categories started in
Section~\ref{factorizations-subsecn} and
Example~\ref{factorizations-bec-example}, and assume that either
$\Gamma=\boZ$, or $\Lambda$ is involutive and $\Gamma=\boZ/2$.
 Let $w\:\Id_\sA\rarrow\Lambda^2$ be a potential.

 We consider the additive category $\sP(\sA,\Lambda)$ of
$\Lambda$\+periodic objects in $\sA$ (and homogeneous morphisms
of degree~$0$ between them) and the additive category
$\sF(\sA,\Lambda,w)$ of factorizations of~$w$ (and closed morphisms
of degree~$0$ between them).
 Following Example~\ref{factorizations-bec-example}, the forgetful
functor $\Theta\:\sF(\sA,\Lambda,w)\rarrow\sP(\sA,\Lambda)$ has
adjoint functors on both sides, $G^+$ and $G^-\:\sP(\sA,\Lambda)
\rarrow\sF(\sA,\Lambda,w)$.

 Let $\sE$ be an exact category and $\Lambda\:\sE\rarrow\sE$ be
an autoequivalence preserving and reflecting admissible short exact
sequences.
 Then the additive category of graded objects $\sG(\sE)$ has
the induced exact structure discussed in
Example~\ref{exact-dg-category-of-complexes}.
 Applying Lemma~\ref{faithful-conservative-lifting-exact-structure}
to the faithful, conservative forgetful functor $\sP(\sE,\Lambda)
\rarrow\sG(\sE)$, one obtains the induced exact category structure
on the category of $\Lambda$\+periodic objects $\sP(\sE,\Lambda)$
(in fact, the same exact structure can be simply obtained from
the category equivalence $\sP(\sE,\Lambda)\simeq\sE$).
 Applying the same lemma again to the faithful, conservative
forgetful functor $\Theta\:\sF(\sE,\Lambda,w)\rarrow\sP(\sE,\Lambda)$,
one produces an exact category structure on the additive category
of factorizations $\sF(\sE,\Lambda,w)$.
 In both cases, the forgetful functor preserves kernels and cokernels,
and any morphism in the source category whose image in the target
category has a (co)kernel has a (co)kernel in the source category;
so the assumptions of lemma hold.

 One can easily check that the resulting exact category structure on
$\sF(\sE,\Lambda,w)=\sZ^0(\bF(\sE,\Lambda,w))$ is DG\+admissible.
 In fact, let $\bE$ denote the DG\+category of factorizations
$\bF(\sE,\Lambda,w)$.
 Then one observes that, for any object $\biX^\cu\in\bE$, the forgetful
functor $\Theta\:\sZ^0(\bE)\rarrow\sP(\sE,\Lambda)$ takes
the canonical sequence $0\rarrow\biX^\cu[-1]\rarrow\Xi_\bE(\biX^\cu)
\rarrow\biX^\cu\rarrow0$ to a split short exact sequence in
$\sP(\sE,\Lambda)$.
 Applying Theorem~\ref{transferring-from-E-to-E-bec}, one produces
an exact DG\+category structure on the DG\+category
of factorizations $\bE=\bF(\sE,\Lambda,w)$.

 Similarly to Example~\ref{exact-dg-category-of-complexes},
the composition $\Theta\circ G^+\:\sP(\sE,\Lambda)\rarrow
\sP(\sE,\Lambda)$ of the functors $G^+$ and $\Theta$ is isomorphic
to the direct sum $\Id\oplus\Id[-1]$ of the identity endofunctor
and the shift functor.
 (Basically, the reason for this splitting behavior is that
factorizations are similar to CDG\+modules over a CDG\+ring with
zero differential; cf.\ the formula for the action of $R^*$
in $G^+(M^*)$ in Example~\ref{CDG-ring-bec-example}.)
 It follows that the functor $\Theta\circ G^+$ preserves and
reflects admissible short exact sequences, and one can conclude
that so does the functor~$G^+$.

 Looking on the commutative diagram~\eqref{factorizations-Psi-diagram}
and remembering that the functor $\Psi_\bE^+$ preserves and reflects
admissible short exact sequences by the definition of an exact
DG\+category, one concludes that the functor $\Upsilon_{\sE,\Lambda,w}$
preserves and reflects admissible short exact sequences.
 Thus the exact category structure on $\sP(\sE,\Lambda)$ is inherited
from the exact structure on $\sZ^0(\bE^\bec)$ under the inclusion
$\Upsilon_{\sE,\Lambda,w}\:\sP(\sE,\Lambda)\rarrow
\sZ^0(\bF(\sE,\Lambda,w)^\bec)$.

 The composition $\Phi_\bE\circ\Psi_\bE^+\:\sZ(\bE^\bec)\rarrow
\sZ^0(\bE^\bec)$ of the functors $\Psi_\bE^+$ and $\Phi_\bE$ is
also isomorphic to the direct sum $\Id\oplus\Id[1]$ of the identity
endofunctor and the shift functor.
 Therefore, any short exact sequence in the exact category
$\sZ^0(\bE^\bec)$ is isomorphic to a direct summand of the image
of a short exact sequence from the exact category $\sZ^0(\bE)$
under the functor~$\Phi_\bE$.
 In view of the commutative
diagram~\eqref{factorizations-tilde-Phi-diagram}, it follows that
any short exact sequence in $\sZ^0(\bF(\sE,\Lambda,w)^\bec)$ is
a direct summand of a short exact sequence in
$\Upsilon_{\sE,\Lambda,w}(\sP(\sE,\Lambda))$.

 Let $\sH\subset\sE$ be a full additive subcategory preserved by
the autoequivalences $\Lambda$ and $\Lambda^{-1}$ and inheriting
an exact category structure from~$\sE$.
 (We will denote the restrictions of $\Lambda$ and~$w$ onto $\sH$
simply by $\Lambda$ and~$w$.)
 Then the exact category structure on the category of
$\Lambda$\+periodic objects $\sP(\sH,\Lambda)$ is inherited from
the exact structure on $\sP(\sE,\Lambda)$.
 It follows that the exact structure on the category of factorizations
$\sF(\sH,\Lambda,w)$ is inherited from the exact structure on
$\sF(\sE,\Lambda,w)$.
 Therefore, the exact DG\+category structure on the DG\+category of
factorizations $\bF(\sH,\Lambda,w)$ is inherited from the exact
DG\+category structure on~$\bF(\sE,\Lambda,w)$.

 Finally, assume that $\sA$ is an abelian category.
 Then the functor $\Upsilon_{\sA,\Lambda,w}\:\sP(\sA,\Lambda)\rarrow
\sF(\sA,\Lambda,w)^\bec)$ is an equivalence of categories (since $\sA$
is idempotent-complete; see Example~\ref{factorizations-bec-example}).
 Both the category of $\Lambda$\+periodic objects $\sP(\sA,\Lambda)$
and the category of factorizations $\sF(\sA,\Lambda,w)=
\sZ^0(\bF(\sA,\Lambda,w))$ are abelian.
 Thus the DG\+category of factorizations $\bF(\sA,\Lambda,w)$ is
abelian.
 The construction of the exact DG\+category structure on the category
of factorizations above produces the abelian exact DG\+category
structure on $\bF(\sA,\Lambda,w)$ from the abelian exact category
structure on~$\sA$.
\end{ex}

\Section{Derived Categories of the Second Kind}
\label{derived-second-kind-secn}

\subsection{Definitions of derived categories of the second kind}
\label{derived-second-kind-definitions-subsecn}
 The idea of defining derived categories of the second kind for exact
DG\+categories was suggested in~\cite[Remarks~3.5\+-3.7]{Pkoszul}.
 This section provides the implementation.
 
 Let $\bE$ be an exact DG\+category (see the definition in
Section~\ref{bec-compatible-subsecn}).
 In particular, $\bE$ is an additive DG\+category with shifts and
cones, so its homotopy category $\sH^0(\bE)$ is triangulated.
 We refer to Section~\ref{dg-categories-subsecn}
and~\cite[Section~1.2]{Pkoszul} for the definitions of totalizations
of complexes in DG\+categories.

 Any short exact sequence in the exact category $\sZ^0(\bE)$ can be
viewed as a three-term complex in the DG\+category~$\bE$.
 Denote by $\Ac^\abs(\bE)$ the minimal thick subcategory in
the homotopy category $\sH^0(\bE)$ containing the totalizations of
(admissible) short exact sequences in $\sZ^0(\bE)$.
 The objects of $\Ac^\abs(\bE)$ are said to be \emph{absolutely
acyclic} (with respect to the given exact DG\+category structure
on~$\bE$).
 The triangulated Verdier quotient category
$$
 \sD^\abs(\bE)=\sH^0(\bE)/\Ac^\abs(\bE)
$$
is called the \emph{absolute derived category} of an exact
DG\+category~$\bE$.

 Assume that infinite coproducts exist in a DG\+category~$\bE$.
 We recall that in this case such coproducts also exist in
the DG\+category $\bE^\bec$ and in the additive categories
$\sZ^0(\bE)$ and $\sZ^0(\bE^\bec)$, as well as in the triangulated
category $\sH^0(\bE)$.
 Moreover, the coproducts are preserved by the additive functors
$\Phi_\bE\:\sZ^0(\bE)\rarrow\sZ^0(\bE^\bec)$ and
$\Psi^+_\bE$, $\Psi^-_\bE\:\sZ^0(\bE^\bec)\rarrow\sZ^0(\bE)$,
since all of them are left (as well as right) adjoint functors.

 Dually, assume that infinite products exist in a DG\+category~$\bE$.
 Then such products also exist in the DG\+category $\bE^\bec$ and
in the additive categories $\sZ^0(\bE)$ and $\sZ^0(\bE^\bec)$, as well
as in the triangulated category $\sH^0(\bE)$.
 Moreover, the products are preserved by the additive functors
$\Phi_\bE\:\sZ^0(\bE)\rarrow\sZ^0(\bE^\bec)$ and
$\Psi^+_\bE$, $\Psi^-_\bE\:\sZ^0(\bE^\bec)\rarrow\sZ^0(\bE)$.

 Let $\sE$ be an exact category with infinite coproducts.
 Then we will say that \emph{the coproducts are exact} in $\sE$ if
the coproduct of any family of (admissible) short exact sequences
is a short exact sequence in~$\sE$.
 Dually, if $\sE$ is an exact category with infinite products,
then we will say that \emph{the products are exact} in $\sE$ whenever
the product of an arbitrary family of short exact sequences is
a short exact sequence.

\begin{lem} \label{exact-co-products-lemma}
 Let\/ $\bE$ be an exact DG\+category. \par
\textup{(a)} Assume that infinite coproducts exist in
the DG\+category\/~$\bE$.
 Then the coproducts are exact in the exact category\/ $\sZ(\bE)$ if
and only if they are exact in the exact category\/ $\sZ(\bE^\bec)$. \par
\textup{(b)} Assume that infinite products exist in
the DG\+category\/~$\bE$.
 Then the products are exact in the exact category\/ $\sZ(\bE)$ if
and only if they are exact in the exact category\/ $\sZ(\bE^\bec)$.
\end{lem}

\begin{proof}
 By the definition of an exact DG\+category, short exact sequences
are preserved and reflected by the additive functors $\Phi$
and~$\Psi^\pm$.
 Since the products and coproducts are also preserved by these
functors, the assertions follow. 
\end{proof}

 We will say that an exact DG\+category $\bE$ \emph{has exact
coproducts} if infinite coproducts exist in the DG\+category $\bE$
and the equivalent conditions of
Lemma~\ref{exact-co-products-lemma}(a) hold.
 Dually, an exact DG\+category $\bE$ is said to \emph{have exact
products} if infinite products exist in the DG\+category $\bE$
and the equivalent conditions of
Lemma~\ref{exact-co-products-lemma}(b) hold.

 Let $\bE$ be an exact DG\+category with exact coproducts.
 Denote by $\Ac^\co(\bE)$ the minimal triangulated subcategory in
the homotopy category $\sH^0(\bE)$ containing the totalizations
of short exact sequences in $\sZ^0(\bE)$ and closed under
coproducts in $\sH^0(\bE)$.
 The objects of $\Ac^\co(\bE)$ are called \emph{coacyclic}
(with respect to the given exact DG\+category structure on~$\bE$).
 The triangulated Verdier quotient category
$$
 \sD^\co(\bE)=\sH^0(\bE)/\Ac^\co(\bE)
$$
is called the \emph{coderived category} of an exact DG\+category~$\bE$.

 Dually, let $\bE$ be an exact DG\+category with exact products.
 Denote by $\Ac^\ctr(\bE)$ the minimal triangulated subcategory in
the homotopy category $\sH^0(\bE)$ containing the totalizations of
short exact sequences in $\sZ^0(\bE)$ and closed under products
in $\sH^0(\bE)$.
 The objects of $\Ac^\ctr(\bE)$ are called \emph{contraacyclic}
(with respect to the given exact DG\+category structure on~$\bE$).
 The triangulated quotient category
$$
 \sD^\ctr(\bE)=\sH^0(\bE)/\Ac^\ctr(\bE)
$$
is called the \emph{contraderived category} of an exact
DG\+category~$\bE$.

 Notice that any triangulated category with countable (co)products
is idempotent-complete~\cite[Proposition~3.2]{BN}.
 Consequently, any triangulated subcategory closed under countable
(co)products in a triangulated category with countable (co)products
is a thick subcategory~\cite[Criterion~1.3]{Neem}.
 Hence all absolutely acyclic objects are coacyclic and contraacyclic
(under the respective assumptions).

\subsection{Injective and projective resolutions: semiorthogonality}
\label{injective-projective-semiorthogonality-subsecn}
 Let $\sE$ be an exact category.
 An object $P\in\sE$ is said to be \emph{projective} if the functor
$\Hom_\sE(P,{-})\:\allowbreak\sE\rarrow\boZ\rmodl$ is exact, i.~e.,
it takes (admissible) short exact sequences in $\sE$ to short exact
sequences of abelian groups.
 Dually, an object $J\in\sE$ is said to be \emph{injective} if
the functor $\Hom_\sE({-},J)\:\sE^\sop\rarrow\boZ\rmodl$ takes short
exact sequences in $\sE$ to short exact sequences of abelian groups.
 We will denote the full subcategory of projective objects by
$\sE_\proj\subset\sE$ and the full subcategory of injective objects
by $\sE_\inj\subset\sE$.

 One says that an exact category $\sE$ has \emph{enough projectives}
if every object of $\sE$ is the target of an admissible epimorphism
from a projective object.
 Dually, $\sE$ has \emph{enough injectives} if every object of $\sE$
is the source of an admissible monomorphism into an injective
object~\cite[Section~11]{Bueh}.

\begin{rem} \label{enough-projectives-implies-exact-products-remark}
 In any exact category with infinite products and enough projectives,
the product functors are exact.
 Dually, in any exact category with infinite coproducts and enough
injectives, the coproduct functors are exact.

 Indeed, let $0\rarrow X_\alpha\rarrow Y_\alpha\rarrow Z_\alpha
\rarrow0$ be a family of short exact sequences in an exact
category~$\sE$.
 Assume that the products $\prod_\alpha X_\alpha$, \ $\prod_\alpha
Y_\alpha$, and $\prod_\alpha Z_\alpha$ exist in~$\sE$, and consider
the sequence $0\rarrow\prod_\alpha X_\alpha\rarrow\prod_\alpha
Y_\alpha\rarrow\prod_\alpha Z_\alpha\rarrow0$.
 Since infinite products always preserve kernels, the morphism
$\prod_\alpha X_\alpha\rarrow\prod_\alpha Y_\alpha$ is a kernel
of the morphism $\prod_\alpha Y_\alpha\rarrow\prod_\alpha Z_\alpha$.

 Let $P$ be a projective object in $\sE$ such that there is
an admissible epimorphism $P\rarrow\prod_\alpha Z_\alpha$.
 Then we have a family of morphisms $P\rarrow Z_\alpha$, each of
which can be lifted to a morphism $P\rarrow Y_\alpha$.
 Consequently, there is a morphism $P\rarrow\prod_\alpha Y_\alpha$
making the triangle diagram $P\rarrow\prod_\alpha Y_\alpha\rarrow
\prod_\alpha Z_\alpha$ commutative.
 By the ``obscure axiom'' (the dual assertion
to~\cite[Proposition~2.16]{Bueh}), it follows that the morphism
$\prod_\alpha Y_\alpha\rarrow\prod_\alpha Z_\alpha$ is
an admissible epimorphism.
 Thus the short sequence of products $0\rarrow\prod_\alpha X_\alpha
\rarrow\prod_\alpha Y_\alpha\rarrow\prod_\alpha Z_\alpha\rarrow0$
is exact in~$\sE$.
\end{rem}

 Let $\bE$ be an exact DG\+category.
 An object $P\in\bE$ is said to be \emph{graded-projective} if
$\Phi_\bE(P)\in\sZ^0(\bE^\bec)$ is a projective object of the exact
category $\sZ^0(\bE^\bec)$.
 Dually, an object $J\in\bE$ is said to be \emph{graded-injective}
if $\Phi_\bE(J)\in\sZ^0(\bE^\bec)$ is an injective object of
the exact category $\sZ^0(\bE^\bec)$.
 We will denote the full DG\+subcategory of graded-projective
objects by $\bE_\bproj\subset\bE$ and the full DG\+subcategory of
graded-injective objects by $\bE_\binj\subset\bE$.
 The following examples explain the terminology.

\begin{exs}
 (1)~Let $\sA$ be an exact category.
 Consider the exact DG\+category structure on the DG\+category of
complexes $\bE=\bC(\sA)$ constructed in
Example~\ref{exact-dg-category-of-complexes}.

 According to diagram~\eqref{complexes-tilde-Phi-diagram},
the forgetful functor $\sZ^0(\bC(\sA))\rarrow\sG(\sA)$ forms
a commutative diagram with the functor $\Phi_\bE\:\sZ^0(\bC(\sA))
\rarrow\sZ^0(\bC(\sA)^\bec)$ and the fully faithful functor
$\Upsilon_\sA\:\sG(\sA)\rarrow\sZ^0(\bC(\sA)^\bec)$.
 Moreover, according to the discussion in
Example~\ref{exact-dg-category-of-complexes}, the exact category
structure on $\sG(\sA)$ is inherited from the exact category
structure on $\sZ^0(\bC(\sA)^\bec)$ via the embedding $\Upsilon_\sA$,
and any short exact sequence in $\sZ^0(\bC(\sA)^\bec)$ is a direct
summand of a short exact sequence coming from $\sG(\sA)$.

 Consequently, an object of $\bC(\sA)$ is graded-projective
(respectively, graded-injective) if and only if its image under
the forgetful functor $\sZ^0(\bC(\sA))\rarrow\sG(\sA)$ is
projective (resp., injective).
 In other words, the graded-projective objects in $\bC(\sA)$
are the complexes of projective objects in $\sA$ and
the graded-injective objects in $\bC(\sA)$ are the complexes of
injective objects in $\sA$, that is $\bC(\sA)_\bproj=\bC(\sA_\proj)$
and $\bC(\sA)_\binj=\bC(\sA_\inj)$.

\smallskip
 (2)~Let $\biR^\cu=(R^*,d,h)$ be a CDG\+ring and $\bA=\biR^\cu\bmodl$
be the DG\+category of left CDG\+modules over~$\biR^\cu$.
 We endow $\bA$ with the abelian exact DG\+category structure,
as per Example~\ref{abelian-dg-category-of-cdg-modules}.

 The functor $\Upsilon_{\biR^\cu}\:R^*\smodl\rarrow
\sZ^0(\biR^\cu\bmodl)$ establishes an equivalence between the abelian
category $\sZ^0(\biR^\cu\bmodl)$ and the abelian category of graded
left $R^*$\+modules.
 According to diagram~\eqref{cdg-modules-tilde-Phi-diagram},
the forgetful functor $\sZ^0(\biR^\cu\bmodl)\rarrow R^*\smodl$
forms a commutative diagram with the functor
$\Phi_\bA\:\sZ^0(\biR^\cu\bmodl)\rarrow\sZ^0((\biR^\cu\bmodl)^\bec)$
and the category equivalence~$\Upsilon_{\biR^\cu}$.

 Consequently, a CDG\+module over $\biR^\cu$ is graded-projective
(resp., graded-injective) in the sense of the definition above
if and only if its underlying graded $R^*$\+module is projective
(resp., injective) as an object of the abelian category of
graded $R^*$\+modules.

\smallskip
 (3)~Let $X$ be a scheme and $\biB^\cu=(B^*,d,h)$ be a quasi-coherent
CDG\+quasi-algebra over~$X$.
 Let $\bA=\biB^\cu\bqcoh$ be the DG\+category of quasi-coherent left
CDG\+modules over~$\biB^\cu$.
 We endow $\bA$ with the abelian exact DG\+category structure,
as per Example~\ref{abelian-dg-category-of-quasi-coh-cdg-modules}.

 The functor $\Upsilon_{\biB^\cu}\:B^*\sqcoh\rarrow
\sZ^0(\biB^\cu\bqcoh)$ establishes an equivalence between the abelian
category $\sZ^0(\biB^\cu\bqcoh)$ and the abelian category of
quasi-coherent left $B^*$\+modules.
 According to diagram~\eqref{cdg-qcoh-tilde-Phi-diagram},
the forgetful functor $\sZ^0(\biB^\cu\bqcoh)\rarrow B^*\sqcoh$
forms a commutative diagram with the functor
$\Phi_\bA\:\sZ^0(\biB^\cu\bqcoh)\rarrow\sZ^0((\biB^\cu\bqcoh)^\bec)$
and the category equivalence~$\Upsilon_{\biB^\cu}$.
{\emergencystretch=1em\par}

 Consequently, a quasi-coherent CDG\+module over $\biB^\cu$ is
graded-injective in the sense of the definition above if and only if
its underlying quasi-coherent graded $B^*$\+module is injective
as an object of the abelian category of quasi-coherent graded
$B^*$\+modules.

\smallskip
 (4)~Let $\sA$ be an exact category and $\Lambda\:\sA\rarrow\sA$ be
its autoequivalence preserving and reflecting admissible short
exact sequences.
 We assume that either $\Gamma=\boZ$, or $\Lambda$ is involutive
and $\Gamma=\boZ/2$.
 Let $w\:\Id_\sA\:\Id_\sA\rarrow\Lambda^2$ be a potential.
 Consider the exact DG\+category structure on the DG\+category of
factorizations $\bE=\bF(\sA,\Lambda,w)$ constructed in
Example~\ref{exact-dg-category-of-factorizations}.

 According to diagram~\eqref{factorizations-tilde-Phi-diagram},
the forgetful functor $\sZ^0(\bE)\rarrow\sP(\sA,\Lambda)$ forms
a commutative diagram with the functor $\Phi_\bE\:\sZ^0(\bE)
\rarrow\sZ^0(\bE^\bec)$ and the fully faithful functor
$\Upsilon_{\sA,\Lambda,w}\:\sP(\sA,\Lambda)\rarrow\sZ^0(\bE^\bec)$.
 Moreover, according to the discussion in
Example~\ref{exact-dg-category-of-factorizations}, the exact category
structure on $\sP(\sA,\Lambda)$ is inherited from the exact category
structure on $\sZ^0(\bE^\bec)$ via the embedding 
$\Upsilon_{\sA,\Lambda,w}$, and any short exact sequence in
$\sZ^0(\bE^\bec)$ is a direct summand of a short exact sequence
coming from $\sP(\sA,\Lambda)$.

 Consequently, an object of $\bF(\sA,\Lambda,w)$ is graded-projective
(respectively, graded-injective) if and only if its image under
the forgetful functor $\sZ^0(\bF(\sA,\Lambda,w))\rarrow
\sP(\sA,\Lambda)$ is projective (resp., injective).
 In other words, the graded-projective objects in $\bF(\sA,\Lambda,w)$
are the factorizations with projective components in $\sA$ and
the graded-injective objects in $\bF(\sA,\Lambda,w)$ are
the factorizations with injective components in $\sA$, that is
$\bF(\sA,\Lambda,w)_\bproj=\bF(\sA_\proj,\Lambda,w)$ and
$\bF(\sA,\Lambda,w)_\binj=\bF(\sA_\inj,\Lambda,w)$.
 (Here the restrictions of $\Lambda$ and~$w$ to $\sA_\proj$ and
$\sA_\inj$ are denoted simply by $\Lambda$ and~$w$ for brevity.)
\end{exs}

 The reader can find a detailed discussion of projective and injective
objects in abelian DG\+categories in~\cite[Sections~6.1 and~7.1]{PS5}.
 For a further discussion of projective and injective objects in exact
DG\+categories and exact DG\+pairs, see~\cite[Section~15.4]{Pdomc}.

\begin{lem} \label{graded-proj-inj-lemma}
 Let\/ $\bE$ be an exact DG\+category. \par
\textup{(a)} The full DG\+subcategory of graded-projective objects\/
$\bE_\bproj\subset\bE$ is additive, and closed under shifts and twists
(hence also under cones) and direct summands.
 The full DG\+subcategory\/ $\bE_\bproj\subset\bE$ inherits an exact
DG\+category structure.
 The inclusion\/ $\sZ^0(\bE^\bec)_\proj\subset\sZ^0((\bE_\bproj)^\bec)$
holds in\/ $\sZ^0(\bE^\bec)$.
 If\/ $\bE$ has infinite coproducts, then\/ $\bE_\bproj$ is closed
under infinite coproducts in\/~$\bE$. \par
\textup{(b)} The full DG\+subcategory of graded-injective objects\/
$\bE_\binj\subset\bE$ is additive, and closed under shifts and twists
(hence also under cones) and direct summands.
 The full DG\+subcategory\/ $\bE_\binj\subset\bE$ inherits an exact
DG\+category structure.
 The inclusion\/ $\sZ^0(\bE^\bec)_\inj\subset\sZ^0((\bE_\binj)^\bec)$
holds in\/ $\sZ^0(\bE^\bec)$.
 If\/ $\bE$ has infinite products, then\/ $\bE_\binj$ is closed under
infinite products in\/~$\bE$.
\end{lem}

\begin{proof}
 It suffices to prove part~(a), as part~(b) is dual.
 Let $\sL=\sZ^0(\bE^\bec)_\proj$ be the full subcategory of projective
objects in the exact category $\sZ^0(\bE^\bec)$.
 The full subcategory $\sL$ is additive, closed under extensions and
direct summands, and it is preserved by the shift functors acting
naturally on $\sZ^0(\bE^\bec)$.
 Hence Proposition~\ref{subcategory-L-prop}(a\+-d) is applicable,
implying the first three assertions of part~(a).
 The full DG\+subcategory $\bE_\bproj\subset\bE$ is closed under
infinite coproducts because the full subcategory of projective objects
$\sZ^0(\bE^\bec)_\proj\subset\sZ^0(\bE^\bec)$ is closed under
coproducts and the functor $\Phi_\bE$ preserves coproducts.
\end{proof}

 Notice that the inclusions $\sZ^0(\bE^\bec)_\proj\subset
\sZ^0((\bE_\bproj)^\bec)$ and $\sZ^0(\bE^\bec)_\inj\subset
\sZ^0((\bE_\binj)^\bec)$ in Lemma~\ref{graded-proj-inj-lemma}
can be strict, as Example~\ref{L-tilde-example} illustrates.

\begin{thm} \label{inj-proj-resolutions-semiorthogonality-theorem}
\textup{(a)} Let\/ $\bE$ be an exact DG\+category with exact coproducts.
 Then for any objects $J\in\sH^0(\bE_\binj)$ and $X\in\Ac^\co(\bE)$
one has\/ $\Hom_{\sH^0(\bE)}(X,J)=0$.
 Consequently, the composition of the triangulated inclusion functor\/
$\sH^0(\bE_\binj)\rarrow\sH^0(\bE)$ and the Verdier quotient functor\/
$\sH^0(\bE)\rarrow\sD^\co(\bE)$ is a fully faithful triangulated
functor\/ $\sH^0(\bE_\binj)\rarrow\sD^\co(\bE)$. \par
\textup{(b)} Let\/ $\bE$ be an exact DG\+category with exact products.
 Then for any objects $P\in\sH^0(\bE_\bproj)$ and $Y\in\Ac^\ctr(\bE)$
one has\/ $\Hom_{\sH^0(\bE)}(P,Y)=0$.
 Consequently, the composition of the triangulated inclusion functor\/
$\sH^0(\bE_\bproj)\rarrow\sH^0(\bE)$ and the Verdier quotient functor\/
$\sH^0(\bE)\rarrow\sD^\ctr(\bE)$ is a fully faithful triangulated
functor\/ $\sH^0(\bE_\bproj)\rarrow\sD^\ctr(\bE)$.
\end{thm}

\begin{proof}
 This is a generalization of~\cite[Theorem in Section~3.5]{Pkoszul}
suggested in~\cite[Remark in Section~3.5]{Pkoszul}.
 Let us prove part~(b) (part~(a) is dual).
 The second assertion in~(b) follows from the first one in view of
the well-known general properties of semiorthogonal triangulated
subcategories.
 Furthermore, for any object $P\in\sH^0(\bE)$ the class of all
objects $Y\in\sH^0(\bE)$ such that $\Hom_{\sH^0(\bE)}(P,Y[*])=0$ is
a full triangulated subcategory closed under infinite products
in $\sH^0(\bE)$.

 It remains to consider a short exact sequence $0\rarrow A\rarrow B
\rarrow C\rarrow0$ in $\sZ^0(\bE)$ and its totalization $T\in\bE$.
 Let $P\in\sH^0(\bE_\bproj)$ be a graded-projective object.
 In order to show that the complex of abelian groups
$\Hom^\bu_\bE(P,T)$ is acyclic, we observe that, by the definition
of the totalization, $\Hom^\bu_\bE(P,T)$ is the totalization of
the bicomplex with three rows $\Hom^\bu_\bE(P,A)\rarrow
\Hom^\bu_\bE(P,B)\rarrow\Hom^\bu_\bE(P,C)$.
 Clearly, the totalization of any finite acyclic complex of complexes
of abelian groups is acyclic; so it suffices to check that
$0\rarrow\Hom^\bu_\bE(P,A)\rarrow\Hom^\bu_\bE(P,B)\rarrow
\Hom^\bu_\bE(P,C)\rarrow0$ is a short exact sequence of complexes
of abelian groups.
 For this purpose, we only need to show that $0\rarrow
\Hom^0_\bE(P,A)\rarrow\Hom^0_\bE(P,B)\rarrow\Hom^0_\bE(P,C)\rarrow0$
is a short exact sequence of abelian groups (as all the classes of
objects involved are obviously closed under shifts in~$\bE$).

 Now we recall that the additive functor $\Phi_\bE\:\sZ^0(\bE)\rarrow
\sZ^0(\bE^\bec)$ can be naturally extended to a fully faithful
additive functor $\widetilde\Phi_\bE\:\bE^0\rarrow\sZ^0(\bE^\bec)$,
by Lemma~\ref{Phi-Psi-extended-to-nonclosed-morphisms}.
 So the desired assertion can be rephrased by saying that the short
sequence of abelian groups $0\rarrow\Hom_{\sZ^0(\bE^\bec)}(\Phi_\bE(P),
\Phi_\bE(A))\rarrow\Hom_{\sZ^0(\bE^\bec)}(\Phi_\bE(P),\Phi_\bE(B))
\rarrow\Hom_{\sZ^0(\bE^\bec)}(\Phi_\bE(P),\Phi_\bE(C))\rarrow0$ is
exact.
 Finally, it remains to use the assumption that $\Phi_\bE(P)$ is
a projective object in the exact category $\sZ^0(\bE^\bec)$, together
with the fact that the short sequence $0\rarrow\Phi_\bE(A)\rarrow
\Phi_\bE(B)\rarrow\Phi_\bE(C)\rarrow0$ is exact in $\sZ^0(\bE^\bec)$
(as the functor $\Phi_\bE\:\sZ^0(\bE)\rarrow\sZ^0(\bE^\bec)$ is
exact by the definition of an exact DG\+category).
\end{proof}

\subsection{Injective and projective resolutions: finite homological
dimension case}
 Let $\sE$ be an exact category.
 Given two objects $X$ and $Y\in\sE$, the Ext groups
$\Ext_\sE^n(X,Y)$, \,$n\ge0$, can be defined using the Yoneda
Ext construction.
 Equivalently, put $\Ext_\sE^n(X,Y)=\Hom_{\sD^\bb(\sE)}(X,Y[n])$,
where $\sD^\bb(\sE)$ denotes the bounded derived category of
the exact category~$\sE$ (see, e.~g., \cite{Neem}
and~\cite[Section~A.7]{Partin}).
 We refer to Section~\ref{prelim-Yoneda-Ext-subsecn} below
for a further discussion.

 An object $A\in\sE$ is said to have \emph{injective
dimension~$\le n$} if $\Ext_\sE^i(X,A)=0$ for all $X\in\sE$ and $i>n$.
 If this is the case, we write $\idi_\sE A\le n$.
 Dually, an object $B\in\sE$ is said to have \emph{projective
dimension~$\le n$} if $\Ext_\sE^i(B,Y)=0$ for all $Y\in\sE$ and $i>n$.
 If this is the case, we write $\pdi_\sE B\le n$.
 An exact category $\sE$ is said to have \emph{homological
dimension~$\le n$} if $\Ext_\sE^i(X,Y)=0$ for all $X$, $Y\in\sE$
and $i>n$.

 Assume that the exact category $\sE$ has enough projective objects.
 Then an object $B\in\sE$ has projective dimension~$\le n$ if and
only if it admits a projective resolution $0\rarrow P_n\rarrow\dotsb
\rarrow P_0\rarrow B\rarrow0$ of length~$n$.
 In fact, if $\pdi_\sE B\le n$ and $0\rarrow Q\rarrow P_{n-1}\rarrow
\dotsb\rarrow P_0\rarrow B\rarrow0$ is an exact complex with
projective objects $P_i\in\sE$, then $Q\in\sE$ is also a projective
object.

 Dually, assume that $\sE$ has enough injective objects.
 Then an object $A\in\sE$ has injective dimension~$\le n$ if and
only if it admits an injective coresolution $0\rarrow A\rarrow J^0
\rarrow\dotsb\rarrow J^n\rarrow0$ of length~$n$.
 In fact, if $\idi_\sE A\le n$ and $0\rarrow A\rarrow J^0\rarrow
\dotsb\rarrow J^{n-1}\rarrow K\rarrow0$ is an exact complex with
injective objects $J^i\in\sE$, then $K\in\sE$ is also an injective
object.

\begin{thm} \label{inj-proj-resolutions-finite-homol-dim-theorem}
\textup{(a)} Let\/ $\bE$ be an exact DG\+category with infinite
coproducts.
 Assume that the exact category\/ $\sZ^0(\bE^\bec)$ has finite 
homological dimension and enough injective objects.
 Then the two thick subcategories\/ $\Ac^\abs(\bE)$ and\/
$\Ac^\co(\bE)$ in the homotopy category\/ $\sH^0(\bE)$ coincide.
 Furthermore, the composition of the triangulated inclusion functor\/
$\sH^0(\bE_\binj)\rarrow\sH^0(\bE)$ and the Verdier quotient functor\/
$\sH^0(\bE)\rarrow\sD^\co(\bE)$ is a triangulated equivalence.
 So one has\/ $\sD^\abs(\bE)\simeq\sD^\co(\bE)\simeq\sH^0(\bE_\binj)$.
\par
\textup{(b)} Let\/ $\bE$ be an exact DG\+category with infinite
products.
 Assume that the exact category\/ $\sZ^0(\bE^\bec)$ has finite
homological dimension and enough projective objects.
 Then the two thick subcategories\/ $\Ac^\abs(\bE)$ and\/
$\Ac^\ctr(\bE)$ in the homotopy category\/ $\sH^0(\bE)$ coincide.
 Furthermore, the composition of the triangulated inclusion functor\/
$\sH^0(\bE_\bproj)\rarrow\sH^0(\bE)$ and the Verdier quotient functor\/
$\sH^0(\bE)\rarrow\sD^\ctr(\bE)$ is a triangulated equivalence.
 So one has\/ $\sD^\abs(\bE)\simeq\sD^\ctr(\bE)\simeq\sH^0(\bE_\bproj)$.
\end{thm}

\begin{proof}
 This is the generalization of~\cite[Theorem in Section~3.6]{Pkoszul}
suggested in~\cite[Remark in Section~3.6]{Pkoszul}.
 We start with several lemmas.

\begin{lem} \label{mono-epi-exact-DG-adjunction}
\textup{(a)} Let\/ $\bE$ be an exact DG\+category, $A\in\bE$ and
$K\in\bE^\bec$ be two objects, and\/ $\Phi_\bE(A)\rarrow K$ be
an admissible monomorphism in the exact category\/ $\sZ^0(\bE^\bec)$.
 Then the corresponding morphism $A\rarrow\Psi^-_\bE(K)$
(by adjunction) is an admissible monomorphism in the exact category\/
$\sZ^0(\bE)$. \par
\textup{(b)} Let\/ $\bE$ be an exact DG\+category, $B\in\bE$ and
$L\in\bE^\bec$ be two objects, and $L\rarrow\Phi_\bE(B)$ be
an admissible epimorphism in the exact category\/ $\sZ^0(\bE^\bec)$.
 Then the corresponding morphism\/ $\Psi^+_\bE(L)\rarrow B$
(by adjunction) is an admissible epimorphism in the exact category\/
$\sZ^0(\bE)$.
\end{lem}

\begin{proof}
 Let us prove part~(b) (part~(a) is dual).
 The functor $\Psi^+_\bE\:\sZ^0(\bE^\bec)\rarrow\sZ^0(\bE)$ is exact
by the definition of an exact DG\+category, so it takes admissible
epimorphisms to admissible epimorphisms.
 Hence $\Psi^+_\bE(L)\rarrow\Psi^+_\bE\Phi_\bE(B)$ is an admissible
epimorphism.
 On the other hand, by Lemma~\ref{Xi-decomposed} and by the definition
of a DG\+compatible exact structure, we have a natural short exact
sequence $0\rarrow B[-1]\rarrow\Psi^+_\bE\Phi_\bE(B)\rarrow B\rarrow0$
in the exact category $\sZ^0(\bE)$.
 In particular, the adjunction morphism $\Psi^+_\bE\Phi_\bE(B)\rarrow B$
is an admissible epimorphism.
 Thus the composition $\Psi^+_\bE(L)\rarrow\Psi^+_\bE\Phi_\bE(B)
\rarrow B$ is an admissible epimorphism.
\end{proof}

\begin{lem} \label{one-step-inj-proj-resolution-lemma}
\textup{(a)} Let\/ $\bE$ be an exact DG\+category for which
the exact category\/ $\sZ^0(\bE^\bec)$ has enough injective objects.
 Then for any object $A\in\bE$ there exists an admissible monomorphism
$A\rarrow J$ in the exact category\/ $\sZ^0(\bE)$ such that the object
$\Phi_\bE(J)$ is injective in the exact category\/ $\sZ^0(\bE^\bec)$.
\par
\textup{(b)} Let\/ $\bE$ be an exact DG\+category for which
the exact category\/ $\sZ^0(\bE^\bec)$ has enough projective objects.
 Then for any object $B\in\bE$ there exists an admissible epimorphism
$P\rarrow B$ in the exact category\/ $\sZ^0(\bE)$ such that the object
$\Phi_\bE(P)$ is projective in the exact category\/ $\sZ^0(\bE^\bec)$.
\end{lem}

\begin{proof}
 Let us prove part~(b) (part~(a) is dual).
 Given an object $B\in\bE$, consider the object $\Phi_\bE(B)\in
\sZ^0(\bE^\bec)$ and choose a projective object
$Q\in\sZ^0(\bE^\bec)_\proj$ together with an admissible epimorphism
$Q\rarrow\Phi_\bE(B)$ in $\sZ^0(\bE^\bec)$.
 By Lemma~\ref{mono-epi-exact-DG-adjunction}(b), the related morphism
$\Psi^+_\bE(Q)\rarrow B$ is an admissible epimorphism in $\sZ^0(\bE)$.

 By Lemma~\ref{Xi-decomposed} and by the definition of
a DG\+compatible exact structure, we have a natural short exact
sequence $0\rarrow Q\rarrow\Phi_\bE\Psi^+_\bE(Q)\rarrow Q[1]\rarrow0$
in the exact category $\sZ^0(\bE^\bec)$.
 Hence the object $\Phi_\bE\Psi^+_\bE(Q)\in\sZ^0(\bE^\bec)$ is
projective as an extension of two projective objects, and we can put
$P=\Psi^+_\bE(Q)$.
\end{proof}

\begin{lem} \label{totalization-of-finite-absolutely-acyclic}
 Let\/ $\bE$ be an exact DG\+category, and let\/
$0\rarrow X^0\rarrow X^1\rarrow\dotsb\rarrow X^n\rarrow0$
be a finite exact complex in the exact category\/ $\sZ^0(\bE)$.
 Then the totalization\/ $\Tot(X^\bu)$ is an absolutely acyclic
object of\/~$\bE$.
\end{lem}

\begin{proof}
 Let $0\rarrow K^i\rarrow X^i\rarrow K^{i+1}\rarrow0$ be the short
exact sequences in $\sZ^0(\bE)$ from which the exact complex $X^\bu$
is obtained by splicing.
 Then the totalization $\Tot(X^\bu)$ is homotopy equivalent to
an object obtainable from the totalizations
$\Tot(K^i\to X^i\to K^{i+1})$ as an iterated cone of natural
closed morphisms.
\end{proof}

 For a converse assertion to
Lemma~\ref{totalization-of-finite-absolutely-acyclic},
see Proposition~\ref{abs-acyclic-as-totalizations} below.

 Let us prove part~(b) of the theorem (part~(a) is dual).
 Our aim is to show that any object $B\in\sH^0(\bE)$ can be included
into a distinguished triangle $P\rarrow B\rarrow Y\rarrow P[1]$
in $\sH^0(\bE)$ with $P\in\sH^0(\bE_\bproj)$ and $Y\in\Ac^\abs(\bE)$.
 Then all the assertions of~(b) follow in view of
Theorem~\ref{inj-proj-resolutions-semiorthogonality-theorem}(b).
 In this argument, we use the inclusion $\Ac^\abs(\bE)\subset
\Ac^\ctr(\bE)$.
 Notice that the infinite products are exact in $\bE$ by
Remark~\ref{enough-projectives-implies-exact-products-remark};
so Theorem~\ref{inj-proj-resolutions-semiorthogonality-theorem}(b)
is applicable.

 Let $B\in\bE$ be an object.
 By Lemma~\ref{one-step-inj-proj-resolution-lemma}(b), there is
a short exact sequence $0\rarrow B_1\rarrow P_0\rarrow B\rarrow0$
in $\sZ^0(\bE)$ with $\Phi_\bE(P_0)\in\sZ^0(\bE^\bec)_\proj$.
 Applying the same lemma to the object $B_1\in\bE$, we obtain
a short exact sequence $0\rarrow B_2\rarrow P_1\rarrow B_1\rarrow0$
in $\sZ^0(\bE)$ with $\Phi_\bE(P_1)\in\sZ^0(\bE^\bec)_\proj$, etc.
 Proceeding in this way, we construct an exact complex
$0\rarrow Q\rarrow P_{n-1}\rarrow\dotsb\rarrow P_0\rarrow B\rarrow0$
in $\sZ^0(\bE)$ with $\Phi_\bE(P_i)\in\sZ^0(\bE^\bec)_\proj$,
where $n$~is the homological dimension of the exact category
$\sZ^0(\bE^\bec)$.
 Then the complex $0\rarrow\Phi_\bE(Q)\rarrow\Phi_\bE(P_{n-1})
\rarrow\dotsb\rarrow\Phi_\bE(P_0)\rarrow\Phi_\bE(B)\rarrow0$ is
exact in $\sZ^0(\bE^\bec)$, and it follows that
$\Phi_\bE(Q)\in\sZ^0(\bE^\bec)_\proj$.

 Put $P_n=Q$, and denote by $P$ the totalization of the finite complex
$P_n\rarrow P_{n-1}\rarrow\dotsb\rarrow P_0$ in
the DG\+category~$\bE$.
 By Lemma~\ref{twists-into-isomorphisms}, the functor
$\Phi_\bE\:\sZ^0(\bE)\rarrow\sZ^0(\bE^\bec)$ transforms twists into
isomorphisms, and it follows that the object $\Phi_\bE(P)$ is
isomorphic to the direct sum $\bigoplus_{i=0}^n\Phi_\bE(P_i)[-i]$ in
the additive category $\sZ^0(\bE^\bec)$.
 Hence the object $\Phi_\bE(P)$ is projective in $\sZ^0(\bE^\bec)$.
 Finally, the cone $Y$ of the natural closed morphism $P\rarrow B$
in $\bE$ is the totalization of the finite exact complex
$0\rarrow P_n\rarrow P_{n-1}\rarrow\dotsb\rarrow P_1\rarrow P_0
\rarrow B\rarrow0$ in $\sZ^0(\bE)$; so it is absolutely acyclic by
Lemma~\ref{totalization-of-finite-absolutely-acyclic}.
\end{proof}

 One can notice that the triangulated equivalences
$\sD^\abs(\bE)\simeq\sH^0(\bE_\binj)$ in 
Theorem~\ref{inj-proj-resolutions-finite-homol-dim-theorem}(a)
and $\sD^\abs(\bE)\simeq\sH^0(\bE_\bproj)$
in Theorem~\ref{inj-proj-resolutions-finite-homol-dim-theorem}(b)
do not depend on the assumption of existence of (co)products in $\bE$,
as it is clear from the proofs of
Theorems~\ref{inj-proj-resolutions-semiorthogonality-theorem}
and~\ref{inj-proj-resolutions-finite-homol-dim-theorem} above.

 For a far-reaching generalization of the first assertions of
both parts~(a) and~(b) of
Theorem~\ref{inj-proj-resolutions-finite-homol-dim-theorem},
see Theorem~\ref{finite-homological-dimension-main-theorem} below.
 For a generalization of the second assertions of both parts of
Theorem~\ref{inj-proj-resolutions-finite-homol-dim-theorem}, see
Theorem~\ref{inj-proj-resolutions-star-conditions-theorem}.

\subsection{Injective and projective resolutions: semiorthogonal
decompositions}
\label{injective-projective-semiorthogonal-decompositions-subsecn}
 The following two technical conditions on an exact category (going
back to~\cite[Sections~3.7 and~3.8]{Pkoszul}) will appear in
the main result of this section.

 Let $\sE$ be an exact category with infinite coproducts and enough
injective objects.
 We will say that $\sE$ satisfies~($*$) if
\begin{itemize}
\item[($*$)] Any countable coproduct of injective objects in $\sE$
has finite injective dimension.
\end{itemize}
 One can easily see that ($*$)~implies existence of a finite integer
$n\ge0$ such that any countable coproduct of injective objects
in $\sE$ has injective dimension~$\le n$.

 Dually, let $\sE$ be an exact category with infinite products and
enough projective objects.
 We will say that $\sE$ satisfies~($**$) if
\begin{itemize}
\item[($**$)] Any countable product of projective objects in $\sE$
has finite projective dimension.
\end{itemize}
 Existence of a finite integer $n\ge0$ such that any countable
product of projective objects in $\sE$ has projective dimension~$\le n$
follows easily from~($**$).

\begin{thm} \label{inj-proj-resolutions-star-conditions-theorem}
\textup{(a)} Let\/ $\bE$ be an exact DG\+category with twists
and infinite coproducts.
 Assume that the exact category\/ $\sZ^0(\bE^\bec)$ has enough
injective objects and satisfies~\textup{($*$)}.
 Then the composition of the triangulated inclusion functor\/
$\sH^0(\bE_\binj)\rarrow\sH^0(\bE)$ and the Verdier quotient functor\/
$\sH^0(\bE)\rarrow\sD^\co(\bE)$ is a triangulated equivalence\/
$\sH^0(\bE_\binj)\simeq\sD^\co(\bE)$. \par
\textup{(b)} Let\/ $\bE$ be an exact DG\+category with twists
and infinite products.
 Assume that the exact category\/ $\sZ^0(\bE^\bec)$ has enough
projective objects and satisfies~\textup{($**$)}.
 Then the composition of the triangulated inclusion functor\/
$\sH^0(\bE_\bproj)\rarrow\sH^0(\bE)$ and the Verdier quotient functor\/
$\sH^0(\bE)\rarrow\sD^\ctr(\bE)$ is a triangulated equivalence\/
$\sH^0(\bE_\bproj)\rarrow\sD^\ctr(\bE)$.
\end{thm}

 For a far-reaching generalization of
Theorem~\ref{inj-proj-resolutions-star-conditions-theorem}, see
Theorem~\ref{finite-homol-dim-star-condition-theorem} below.

\begin{proof}
 Notice that existence of infinite coproducts in $\bE$ and enough
injective objects in $\sZ^0(\bE^\bec)$ implies exactness of
coproducts in $\bE$ by Lemma~\ref{exact-co-products-lemma}
and Remark~\ref{enough-projectives-implies-exact-products-remark}.
 Dually, existence of infinite products in $\bE$ and enough
projective objects in $\sZ^0(\bE^\bec)$ implies exactness of
products in~$\bE$.

 The two dual parts of the theorem are the generalizations
of~\cite[Theorems in Sections~3.7 and~3.8]{Pkoszul} hinted at
in~\cite[Remark in Section~3.7]{Pkoszul}.
 We start with a lemma generalizing~\cite[Lemmas~2.1 and~4.1]{Psemi}.

\begin{lem} \label{totalization-of-one-sided-bounded-co-contra-acyclic}
\textup{(a)} Let\/ $\bE$ be an exact DG\+category with twists and
exact coproducts, and let\/ $0\rarrow X^0\rarrow X^1\rarrow X^2\rarrow
\dotsb$ be a bounded below exact complex in the exact category\/
$\sZ^0(\bE)$.
 Then the coproduct totalization\/ $\Tot^\sqcup(X^\bu)$ is
a coacyclic object of\/~$\bE$. \par
\textup{(b)} Let\/ $\bE$ be an exact DG\+category with twists and
exact products, and let\/ $\dotsb\rarrow Y_2\rarrow Y_1\rarrow Y_0
\rarrow0$ be a bounded above exact complex in the exact category\/
$\sZ^0(\bE)$.
 Then the product totalization\/ $\Tot^\sqcap(Y_\bu)$ is
a contraacyclic object of\/~$\bE$.
\end{lem}

\begin{proof}
 Let us prove part~(b) (part~(a) is dual).
 Let $0\rarrow K_i\rarrow Y_i\rarrow K_{i-1}\rarrow0$ be the short
exact sequences in $\sZ^0(\bE)$ from which the exact complex $Y_\bu$
is obtained by splicing.
 Then the finite complex $0\rarrow K_n\rarrow Y_n\rarrow Y_{n-1}\rarrow
\dotsb\rarrow Y_1\rarrow Y_0\rarrow0$ is exact in $\sZ^0(\bE)$ for
every $n\ge1$.
 By Lemma~\ref{totalization-of-finite-absolutely-acyclic},
the objects $L_n=\Tot(K_n\to Y_n\to\dotsb\to Y_1\to Y_0)\in\bE$
are absolutely acyclic.

 Consider the telescope sequence
\begin{equation} \label{contraacyclicity-telescope-sequence}
 0\lrarrow\Tot^\sqcap(Y_\bu)\lrarrow\prod\nolimits_{n\ge1}L_n
 \lrarrow\prod\nolimits_{n\ge1}L_n\lrarrow0.
\end{equation}
 In view of Lemma~\ref{twists-into-isomorphisms}, the functor
$\Phi_\bE$ transforms~\eqref{contraacyclicity-telescope-sequence}
into a split short exact sequence in $\sZ^0(\bE^\bec)$ (because
the projective system of finite quotient complexes of canonical
filtration of the complex $Y_\bu$ is termwise stabilizing).
 Since, by the definition of an exact DG\+category, the functor
$\Phi_\bE$ reflects short exact sequences, it follows that
the short sequence~\eqref{contraacyclicity-telescope-sequence}
is exact in $\sZ^0(\bE)$.
 Thus the totalization of~\eqref{contraacyclicity-telescope-sequence}
is absolutely acyclic in~$\bE$.
 As the middle and rightmost terms
of~\eqref{contraacyclicity-telescope-sequence} are contraacyclic
in~$\bE$, it follows that the leftmost term is contraacyclic.
\end{proof}

 Let us prove part~(b) of the theorem (part~(a) is dual).
 Our aim is to show that for any object $B\in\sH^0(\bE)$ there
exists a graded-projective object $P\in\sH^0(\bE_\bproj)$ together
with a morphism $P\rarrow B$ in $\sH^0(\bE)$ whose cone belongs
to $\Ac^\ctr(\bE)$.

 Let $B\in\bE$ be an object.
 By Lemma~\ref{one-step-inj-proj-resolution-lemma}(b), there is
a short exact sequence $0\rarrow B_1\rarrow Q_0\rarrow B\rarrow0$
in $\sZ^0(\bE)$ with $\Phi_\bE(Q_0)\in\sZ^0(\bE^\bec)_\proj$.
 Applying the same lemma to the object $B_1\in\bE$, we obtain
a short exact sequence $0\rarrow B_2\rarrow Q_1\rarrow B_1\rarrow0$
in $\sZ^0(\bE)$ with $\Phi_\bE(Q_1)\in\sZ^0(\bE^\bec)_\proj$, etc.
 Proceeding in this way, we construct an exact complex
$\dotsb\rarrow Q_2\rarrow Q_1\rarrow Q_0\rarrow B\rarrow0$
in $\sZ^0(\bE)$ with $\Phi_\bE(Q_i)\in\sZ^0(\bE^\bec)_\proj$.

 Denote by $T=\Tot^\sqcap(Q_\bu)$ the product totalization of
the complex $\dotsb\rarrow Q_2\rarrow Q_1\rarrow Q_0\rarrow0$
in the DG\+category~$\bE$.
 By Lemma~\ref{twists-into-isomorphisms}, the object $\Phi_\bE(T)$
is isomorphic to the product $\prod_{i=0}^\infty\Phi_\bE(Q_i)[-i]$
in the additive category $\sZ^0(\bE^\bec)$.
 Hence condition~($**$) tells that the object $\Phi_\bE(T)$ has
finite projective dimension in $\sZ^0(\bE^\bec)$.
 Let us denote this projective dimension by~$n\ge0$.
 The cone of the natural morphism $T\rarrow B$ in $\bE$ is
the product totalization of the bounded above exact complex
$\dotsb\rarrow Q_2\rarrow Q_1\rarrow Q_0\rarrow B\rarrow0$
in $\sZ^0(\bE)$, so it is contraacyclic by
Lemma~\ref{totalization-of-one-sided-bounded-co-contra-acyclic}(b).

 Now we restart the resolution procedure following the construction
from the proof of
Theorem~\ref{inj-proj-resolutions-finite-homol-dim-theorem}.
 By Lemma~\ref{one-step-inj-proj-resolution-lemma}(b), there is
a short exact sequence $0\rarrow T_1\rarrow P_0\rarrow T\rarrow0$
in $\sZ^0(\bE)$ with $\Phi_\bE(P_0)\in\sZ^0(\bE^\bec)_\proj$.
 Applying the same lemma to the object $T_1\in\bE$, we obtain
a short exact sequence $0\rarrow T_2\rarrow P_1\rarrow T_1\rarrow0$
in $\sZ^0(\bE)$ with $\Phi_\bE(P_1)\in\sZ^0(\bE^\bec)_\proj$, etc.
 Proceeding in this way, we construct an exact complex
$0\rarrow Q\rarrow P_{n-1}\rarrow\dotsb\rarrow P_0\rarrow T\rarrow0$
in $\sZ^0(\bE)$ with $\Phi_\bE(P_i)\in\sZ^0(\bE^\bec)_\proj$.
 Then the complex $0\rarrow\Phi_\bE(Q)\rarrow\Phi_\bE(P_{n-1})
\rarrow\dotsb\rarrow\Phi_\bE(P_0)\rarrow\Phi_\bE(T)\rarrow0$ is
exact in $\sZ^0(\bE^\bec)$, and it follows that
$\Phi_\bE(Q)\in\sZ^0(\bE^\bec)_\proj$.

 Put $P_n=Q$ and $P=\Tot(P_n\to P_{n-1}\to\dotsb\to P_1\to P_0)
\in\bE$.
 Once again, by Lemma~\ref{twists-into-isomorphisms}, the object
$\Phi_\bE(P)$ is isomorphic to $\bigoplus_{i=0}^\infty
\Phi_\bE(P_i)[-i]$, hence it is projective in $\sZ^0(\bE^\bec)$.
 The cone of the natural closed morphism $P\rarrow T$ in $\bE$ is
the totalization of the finite exact complex $0\rarrow P_n\rarrow\dotsb
\rarrow P_0\rarrow T\rarrow0$ in $\sZ^0(\bE)$, so it is absolutely
acyclic by Lemma~\ref{totalization-of-finite-absolutely-acyclic}.

 Finally, the cone of the composition $P\rarrow T\rarrow B$ is
contraacyclic since both the cones of the morphisms $P\rarrow T$
and $T\rarrow B$ are contraacyclic in~$\bE$.
\end{proof}

\begin{rem} \label{exact-DG-pairs-work-in-proj-inj-section-remark}
 All the results of
Sections~\ref{injective-projective-semiorthogonality-subsecn}\+-%
\ref{injective-projective-semiorthogonal-decompositions-subsecn}
can be extended rather straightforwardly to the setting of
\emph{exact DG\+pairs} $(\bE,\sK)$ introduced below in
Section~\ref{exact-DG-pairs-subsecn}.
 In this context, one can say that an object $P\in\bE$ is
graded-projective with respect to $(\bE,\sK)$ if the object
$\Phi_\bE^\sK(P)$ is projective in the exact category~$\sK$, which is
in fact equivalent to $P$ being graded-projective in $\bE$ in the sense
of the definition above~\cite[Corollary~15.15]{Pdomc} (and similarly
for graded-injective objects with respect to $(\bE,\sK)$).
 In particular, a version of
Theorem~\ref{inj-proj-resolutions-finite-homol-dim-theorem} for
exact DG\+pairs can be found in~\cite[Proposition~15.17]{Pdomc}.
\end{rem}

\Section{Finite Resolution Dimension Theorem}
\label{finite-resol-dim-secn}

\subsection{Exact DG-pairs} \label{exact-DG-pairs-subsecn}
 The natural generality for the main results of
Sections~\ref{finite-resol-dim-secn}\+-\ref{finite-homol-dim-secn}
seems to be slightly wider than that of exact DG\+categories.
 In this connection, we introduce the following definition.

 An \emph{exact DG\+pair} $(\bE,\sK)$ is a pair consisting of
an exact DG\+category $\bE$ and a full subcategory $\sK\subset
\sZ^0(\bE^\bec)$ satisfying the following conditions:
\begin{itemize}
\item the full subcategory $\sK\subset\sZ^0(\bE^\bec)$ is preserved
by the shift functors~$[1]$ and~$[-1]$;
\item the image of the exact functor $\Phi_\bE\:\sZ^0(\bE)\rarrow
\sZ^0(\bE^\bec)$ is contained in~$\sK$;
\item the full subcategory $\sK$ inherits an exact category
structure from the exact category $\sZ^0(\bE^\bec)$.
\end{itemize}

 We will denote the functor $\Phi_\bE$, viewed as taking values
in~$\sK$, by $\Phi_\bE^\sK\:\sZ^0(\bE)\rarrow\sK$.
 The restrictions of the functors $\Psi_\bE^+$ and $\Psi_\bE^-\:
\sZ^0(\bE^\bec)\rarrow\sZ^0(\bE)$ to the full subcategory
$\sK\subset\sZ^0(\bE^\bec)$ will be denoted by $\Psi_\bE^{\sK,+}$
and $\Psi_\bE^{\sK,-}\:\sK\rarrow\sZ^0(\bE)$.

 Let $(\bE,\sK)$ be an exact DG\+pair.
 An \emph{exact DG\+subpair} $(\bF,\sL)\subset(\bE,\sK)$ consists
of a full DG\+subcategory $\bF\subset\bE$ and a full subcategory
$\sL\subset\sK$ such that
\begin{itemize}
\item the full subcategory $\sL\subset\sK$ is preserved by
the shift functors~$[1]$ and~$[-1$];
\item the full subcategory $\sL\subset\sK\subset\sZ^0(\bE^\bec)$
is contained in the full subcategory $\sZ^0(\bF^\bec)\subset
\sZ^0(\bE^\bec)$, that is $\sL\subset\sZ^0(\bF^\bec)$;
\item the image of the composition of the inclusion functor
$\sZ^0(\bF)\rightarrowtail\sZ^0(\bE)$ with the functor
$\Phi_\bE^\sK\:\sZ^0(\bE)\rarrow\sK$ is contained in~$\sL$;
\item the full DG\+subcategory $\bF$ inherits an exact DG\+category
structure from the exact DG\+category~$\bE$;
\item the full subcategory $\sL$ inherits an exact category
structure from the exact category~$\sK$.
\end{itemize}
 Clearly, if $(\bF,\sL)$ is an exact DG\+subpair in $(\bE,\sK)$,
then $(\bF,\sL)$ is an exact DG\+pair.

 An exact DG\+subpair $(\bF,\sL)\subset(\bE,\sK)$ is said to be
\emph{strict} if ($\sL$~is a strictly full subcategory in $\sK$ and)
$\Phi_\bE^\sK(F)\in\sL$ implies $F\in\bF$ for any given
object $F\in\bE$.

\begin{exs} \label{exact-DG-pairs-examples}
 (0)~If $\bE$ is an exact DG\+category, then $(\bE,\sZ^0(\bE^\bec))$
is an exact DG\+pair.
 If a DG\+subcategory $\bF\subset\bE$ inherits an exact DG\+category
structure from an exact DG\+category $\bE$, then $(\bF,\sZ^0(\bF^\bec))$
is an exact DG\+subpair in $(\bE,\sZ^0(\bE^\bec))$.

 Any exact DG\+pair $(\bE,\sK)$ is a strict exact DG\+subpair in
$(\bE,\sZ^0(\bE^\bec))$.

\smallskip
 (1)~Let $\bE$ be an exact DG\+category and $\sL\subset\sZ^0(\bE^\bec)$
be a full subcategory preserved by the shift functors~$[n]$,
\,$n\in\Gamma$, and inheriting an exact category structure.
 Consider the full DG\+subcategory $\bF\subset\bE$ whose objects are
all $F\in\bE$ for which $\Phi_\bE(F)\in\sL$, as in
Proposition~\ref{subcategory-L-prop}.
 Then, by part~(c) of the proposition, $\bF$ inherits an exact
DG\+category structure from $\bE$; in particular, the full subcategory
$\widetilde\sL=\sZ^0(\bF^\bec)\subset\sZ^0(\bE^\bec)$ inherits
an exact category structure.
 Hence the intersection $\sL\cap\widetilde\sL\subset\sZ^0(\bE^\bec)$
inherits an exact category structure as well.
 Thus $(\bF,\sL\cap\widetilde\sL)$ is a strict exact DG\+subpair in
$(\bE,\sZ^0(\bE^\bec))$.
 In particular, if $\sL\subset\widetilde\sL$ (e.~g., by part~(d) of
the proposition, this holds whenever the full subcategory $\sL$ is
closed under extensions in $\sZ^0(\bE^\bec)$), then
$(\bF,\sL)$ is a strict exact DG\+subpair in $(\bE,\sZ^0(\bE^\bec))$.

\smallskip
 (2)~More generally, let $(\bE,\sK)$ be an exact DG\+pair, and let
$\sL\subset\sK$ be a full subcategory preserved by the shift
functors and inheriting an exact category structure.
 Let $\bF\subset\bE$ be the full DG\+subcategory whose objects are
all $F\in\bE$ for which $\Phi_\bE^\sK(F)\in\sL$.
 Then (1)~implies that $(\bF,\sL\cap\widetilde\sL)$ is a strict
exact DG\+subpair in $(\bE,\sK)$.
 Here $\widetilde\sL=\sZ^0(\bF^\bec)\subset\sZ^0(\bE^\bec)$ is
the full subcategory consisting of all objects $\widetilde L\in
\sZ^0(\bE^\bec)$ such that $\Xi_{\bE^\bec}(\widetilde L)\in\sL$,
as in Proposition~\ref{subcategory-L-prop}.

 Assume that the full subcategory $\sL$ is closed under extensions
in~$\sK$.
 Then, for any $L\in\sL$, the short sequence $0\rarrow L[1]
\rarrow\Xi_{\bE^\bec}(L)\rarrow L\rarrow0$ is admissible exact in
$\sZ^0(\bE^\bec)$ with all the three terms belonging to~$\sK$.
 It follows that $\Xi_{\bE^\bec}(L)\in\sL$ and therefore
$\sL\subset\widetilde\sL$.
 Thus $(\bF,\sL)$ is a strict exact DG\+subpair in $(\bE,\sK)$.

\smallskip
 (3)~This example illustrates how the notion of an exact DG\+pair
can be a nontrivial generalization of that of an exact DG\+category.
 Let $\biR^\bu=(R^*,d)$ be the DG\+ring from
Example~\ref{L-tilde-example}, that is $R^*=k[\epsilon]$ with
$\deg\epsilon=-1$, \,$\epsilon^2=0$, and $d(\epsilon)=1$.
 Let $\bA=\biR^\bu\bmodl$ be the abelian DG\+category of DG\+modules
over~$\biR^\bu$.
 Following Example~\ref{CDG-ring-bec-example}, \,$\sZ^0(\bA^\bec)
\simeq R^*\smodl$ is the abelian category of graded $R^*$\+modules.
 Let $\sK\subset R^*\smodl$ be the full subcategory of free
(equivalently, flat or projective) graded $R^*$\+modules.
 Following the discussion in Example~\ref{L-tilde-example},
the underlying graded $R^*$\+module of any DG\+module over $\biR^\bu$
belongs to~$\sK$.
 Hence $(\bA,\sK)$ is an exact DG\+pair; in fact, it is a strict exact
DG\+subpair in the exact DG\+pair $(\bA,\sZ^0(\bA^\bec))$.

\smallskip
 (4)~Let $\sE$ be an exact category.
 Consider the DG\+category of complexes $\bC(\sE)$, and endow it
with the exact DG\+category structure constructed in
Example~\ref{exact-dg-category-of-complexes}.

 Let the category of graded objects $\sG(\sE)$ be endowed with
the exact category structure constructed in the same example.
 The category $\sG(\sE)$ can be viewed as a full subcategory in
$\sZ^0(\bC(\sE)^\bec)$ embedded by the functor~$\Upsilon_\sE$.
 Then $(\bC(\sE),\sG(\sE))$ is an exact DG\+pair.

 Let $\sF\subset\sE$ be a full subcategory inheriting an exact
category structure.
 Then $(\bC(\sF),\sG(\sF))$ is a strict exact DG\+subpair in
$(\bC(\sE),\sG(\sE))$.

\smallskip
 (5)~Let $\sE$ be an exact category and $\Lambda\:\sE\rarrow\sE$
be an autoequivalence preserving and reflecting short exact
sequences.
 Assume that either $\Gamma=\boZ$, or $\Lambda$ is involutive and
$\Gamma=\boZ/2$.
 Let $w\:\Id_\sE\rarrow\Lambda^2$ be a potential.
 Consider the DG\+category of factorizations $\bF(\sE,\Lambda,w)$
and endow it with the exact DG\+category structure constructed
in Example~\ref{exact-dg-category-of-factorizations}.

 Let the category of $\Lambda$\+periodic objects $\sP(\sE,\Lambda)$
be endowed with the exact category structure constructed in
the same example.
 The category $\sP(\sE,\Lambda)$ can be viewed as a full subcategory
in $\sZ^0(\bF(\sE,\Lambda,w)^\bec)$ embedded by the functor
$\Upsilon_{\sE,\Lambda,w}$.
 Then $(\bF(\sE,\Lambda,w),\sP(\sE,\Lambda))$ is an exact DG\+pair.

 Let $\sH\subset\sE$ be a full subcategory preserved by
the autoequivalences $\Lambda$ and $\Lambda^{-1}$ and inheriting
an exact category structure.
 Then $(\bF(\sH,\Lambda,w),\sP(\sH,\Lambda))$ is a strict exact
DG\+subpair in $(\bF(\sE,\Lambda,w),\sP(\sE,\Lambda))$.
\end{exs}

\subsection{Resolving subcategories and resolution dimension}
 Let $\sE$ be an exact category.
 A full subcategory $\sF\subset\sE$ is said to be \emph{resolving}
\cite[Section~3]{AuRe}, \cite[Section~2]{Sto} if the following
conditions hold:
\begin{itemize}
\item $\sF$ is closed under extensions in~$\sE$;
\item $\sF$ is closed under the kernels of admissible epimorphisms
in~$\sE$;
\item for every object $E\in\sE$ there exists an admissible
epimorphism $F\rarrow E$ in $\sE$ with $F\in\sF$.
\end{itemize}

 Dually, a full subcategory $\sG\subset\sE$ is said to be
\emph{coresolving} if
\begin{itemize}
\item $\sG$ is closed under extensions in\/~$\sE$;
\item $\sG$ is closed under the cokernels of admissible monomorphisms
in~$\sE$;
\item for every object $E\in\sE$ there exists an admissible
monomorphism $E\rarrow G$ in $\sE$ with $G\in\sG$.
\end{itemize}

 Obviously, any resolving or coresolving subcategory in an exact
category inherits an exact category structure (since it is closed
under extensions).

 Assume that the exact category $\sE$ is weakly idempotent-complete
(in the sense of~\cite[Section~7]{Bueh}).
 Clearly, then any resolving or coresolving subcategory in $\sE$ is
weakly idempotent-complete as well.

 Let $\sF\subset\sE$ be a resolving subcategory and $\sG\subset\sE$
be a coresolving subcategory.
 One says that the \emph{resolution dimension}
\cite[Section~1]{AB}, \cite[Section~2]{Sto} of an object $E\in\sE$
with respect to a resolving subcategory $\sF\subset\sE$ is~$\le n$
if there exists an exact sequence
$0\rarrow F_n\rarrow F_{n-1}\rarrow\dotsb\rarrow F_1\rarrow F_0
\rarrow E\rarrow0$ in $\sE$ with the terms $F_i\in\sF$.
 Dually, the \emph{coresolution dimension} of an object $E\in\sE$
with respect to a coresolving subcategory $\sG\subset\sE$ is~$\le n$
if there exists an exact sequence
$0\rarrow E\rarrow G^0\rarrow G^1\rarrow\dotsb\rarrow G^{n-1}\rarrow
G^n\rarrow0$ in $\sE$ with the terms $G^i\in\sG$.

 Given an integer $n\ge0$, denote by $\sF_n\subset\sE$ the full
subcategory of all objects of resolution dimension~$\le n$ with
respect to~$\sF$, and by $\sG^n\subset\sE$ the full subcategory
of all objects of coresolution dimension~$\le n$ with respect to~$\sG$.
 In particular, for $n=0$ we have $\sF_0=\sF$ and $\sG^0=\sG$.

\begin{prop} \label{resol-dim-well-defined}
 Let\/ $\sF\subset\sE$ be a resolving subcategory and\/ $\sG\subset\sE$
be a coresolving subcategory. \par
\textup{(a)} Let\/ $0\rarrow D\rarrow F_{n-1}\rarrow\dotsb\rarrow F_0
\rarrow E\rarrow0$ be an exact complex in\/ $\sE$ with $E\in\sF_n$
and $F_i\in\sF$.
 Then $D\in\sF$. \par
\textup{(b)} Let\/ $0\rarrow E\rarrow G^0\rarrow\dotsb\rarrow G^{n-1}
\rarrow H\rarrow0$ be an exact complex in\/ $\sE$ with $E\in\sG^n$
and $G^i\in\sG$.
 Then $H\in\sG$.
\end{prop}

\begin{proof}
 Part~(a) is~\cite[Proposition~2.3(1)]{Sto}
or a particular case of~\cite[Corollary~A.5.2]{Pcosh}.
 Part~(b) is dual.
 Alternatively, the assertions can be deduced by induction
from Proposition~\ref{resol-dim-in-exact-sequence} below.
\end{proof}

\begin{prop} \label{resol-dim-defines-resolving-subcat}
\textup{(a)} The full subcategory\/ $\sF_n$ is resolving in\/~$\sE$.
\par
\textup{(b)} The full subcategory\/ $\sG^n$ is coresolving in\/~$\sE$.
\end{prop}

\begin{proof}
 Part~(a) is~\cite[Proposition~2.3(2)]{Sto}
or~\cite[Lemma~A.5.6(a\+-b)]{Pcosh}.
 Part~(b) is dual.
\end{proof}

\begin{prop} \label{resol-dim-in-exact-sequence}
 Let\/ $0\rarrow E'\rarrow E\rarrow E''\rarrow0$ be a short exact
sequence in\/~$\sE$ and $n\ge0$ be an integer. \par
\textup{(a)} If $E\in\sF_n$ and $E''\in\sF_{n+1}$, then $E'\in\sF_n$.
\par
\textup{(b)} If $E\in\sG^n$ and $E'\in\sG^{n+1}$, then $E''\in\sG^n$.
\end{prop}

\begin{proof}
 Part~(a) is~\cite[Lemma~A.5.6(b)]{Pcosh}.
 Part~(b) is dual.
\end{proof}

\subsection{Finite resolution dimension theorem}
\label{finite-resol-dim-subsecn}
 We start with a generalization of
Lemma~\ref{one-step-inj-proj-resolution-lemma}.

\begin{lem} \label{one-step-exact-subcategory-resolution-lemma}
 Let $(\bE,\sK)$ be an exact DG\+pair, and let $(\bF,\sL)$,
$(\bG,\sM)\subset(\bE,\sK)$ be two exact DG\+subpairs. \par
\textup{(a)} Assume that for every object $K\in\sK$ there exists
an admissible epimorphism $L\rarrow K$ in the exact category\/ $\sK$
with an object $L\in\sL$.
 Then for any object $E\in\bE$ there exists an admissible epimorphism
$F\rarrow E$ in the exact category\/ $\sZ^0(\bE)$ with
an object $F\in\bF\subset\bE$. \par
\textup{(b)} Assume that for every object $K\in\sK$ there exists
an admissible monomorphism $K\rarrow M$ in the exact category\/ $\sK$
with an object $M\in\sM$.
 Then for any object $E\in\bE$ there exists an admissible monomorphism
$E\rarrow G$ in the exact category\/ $\sZ^0(\bE)$ with
an object $G\in\bG\subset\bE$.
\end{lem}

\begin{proof}
 Let us prove part~(a) (part~(b) is dual).
 Given an object $E\in\bE$, consider the object $\Phi(E)\in\sK$
and choose an object $L\in\sL$ together with an admissible
epimorphism $L\rarrow\Phi(E)$ in~$\sK$.
 By Lemma~\ref{mono-epi-exact-DG-adjunction}(b), the related morphism
$\Psi^+(L)\rarrow E$ in an admissible epimorphism in $\sZ^0(\bE)$.
 Finally, we have $\Psi^+(L)\in\sZ^0(\bF)$, because $L\in\sL\subset
\sZ^0(\bF^\bec)$.
 So it remains to put $F=\Psi^+(L)$.
\end{proof}

 The idea of the following theorem, which is the main result of
Section~\ref{finite-resol-dim-secn}, goes back
to~\cite[Theorem~7.2.2(a)]{Psemi}, \cite[Theorem~3.2(a)]{PP2},
\cite[Theorem~1.4]{EP}, and~\cite[Proposition~A.5.8]{Pcosh}
(see also~\cite[Theorem~5.5]{PS1}).
 We use this opportunity to correct a small mistake in the proof
of~\cite[Theorem~1.4]{EP} related to the difference between
the thick and the full triangulated subcategory in $\sH^0(\bE)$
generated by the totalizations of short exact sequences
(i.~e., the adjoining of direct summands).
 The proof of this theorem occupies the rest of
Section~\ref{finite-resol-dim-subsecn}.

\begin{thm} \label{finite-resolution-dimension-main-theorem}
 Let $(\bE,\sK)$ be an exact DG\+pair with a weakly idempotent-complete
additive category\/~$\sK$.
 Let $(\bF,\sL)$, $(\bG,\sM)\subset(\bE,\sK)$ be two strict exact
DG\+subpairs.
\par
\textup{(a)} Assume that the full subcategory\/ $\sL$ is resolving in
the exact category\/ $\sK$ and all the objects of\/ $\sK$ have finite
resolution dimensions with respect to\/~$\sL$.
 Then the triangulated functor
$$
 \sD^\abs(\bF)\lrarrow\sD^\abs(\bE)
$$
induced by the inclusion of exact DG\+categories\/ $\bF\rightarrowtail
\bE$ is an equivalence of triangulated categories.

 Furthermore, if the exact DG\+category\/ $\bE$ has exact coproducts,
the full DG\+sub\-cat\-e\-gory\/ $\bF\subset\bE$ is closed under
coproducts, and the additive subcategories\/ $\sL$, $\sK\subset
\sZ^0(\bE^\bec)$ are closed under coproducts, then
the triangulated functor
$$
 \sD^\co(\bF)\lrarrow\sD^\co(\bE)
$$
induced by the inclusion of exact DG\+categories\/ $\bF\rightarrowtail
\bE$ is a triangulated equivalence.
 If the exact DG\+category\/ $\bE$ has exact products,
the full DG\+subcategory\/ $\bF\subset\bE$ is closed under products,
and the additive subcategories\/ $\sL$, $\sK\subset\sZ^0(\bE^\bec)$ are
closed under products, then the triangulated functor
$$
 \sD^\ctr(\bF)\lrarrow\sD^\ctr(\bE)
$$
induced by the inclusion\/ $\bF\rightarrowtail\bE$ is a triangulated
equivalence. \par
\textup{(b)} Assume that the full subcategory\/ $\sM$ is coresolving in
the exact category\/ $\sK$ and all the objects of\/ $\sK$ have finite
coresolution dimensions with respect to\/~$\sM$.
 Then the triangulated functor
$$
 \sD^\abs(\bG)\lrarrow\sD^\abs(\bE)
$$
induced by the inclusion of exact DG\+categories\/ $\bG\rightarrowtail
\bE$ is an equivalence of triangulated categories.

 Furthermore, if the exact DG\+category\/ $\bE$ has exact coproducts,
the full DG\+sub\-cat\-e\-gory\/ $\bG\subset\bE$ is closed under 
coproducts, and the additive subcategories\/ $\sM$, $\sK\subset
\sZ^0(\bE^\bec)$ are closed under coproducts, then
the triangulated functor
$$
 \sD^\co(\bG)\lrarrow\sD^\co(\bE)
$$
induced by the inclusion of exact DG\+categories\/ $\bG\rightarrowtail
\bE$ is a triangulated equivalence.
 If the exact DG\+category\/ $\bE$ has exact products,
the full DG\+subcategory\/ $\bG\subset\bE$ is closed under products,
and the additive subcategories\/ $\sM$, $\sK\subset\sZ^0(\bE^\bec)$ are
closed under products, then the triangulated functor\/
$$
 \sD^\ctr(\bG)\lrarrow\sD^\ctr(\bE)
$$
induced by the inclusion\/ $\bG\rightarrowtail\bE$ is a triangulated
equivalence.
\end{thm}

\begin{proof}
 Let us prove part~(a) (part~(b) is dual).
 First of all, we show that for any object $E\in\bE$ there exists
an object $F\in\bF$ together with a morphism $F\rarrow E$ in
$\sZ^0(\bE)$ whose cone belongs to $\Ac^\abs(\bE)$.

 By Lemma~\ref{one-step-exact-subcategory-resolution-lemma}(a), there
is a short exact sequence $0\rarrow E_1\rarrow F_0\rarrow E\rarrow0$
in $\sZ^0(\bE)$ with $F_0\in\bF$.
 Applying the same lemma to the object $E_1\in\bE$, we obtain
a short exact sequence $0\rarrow E_2\rarrow F_1\rarrow E_1\rarrow0$
in $\sZ^0(\bE)$ with $F_1\in\bF$.
 Proceeding in this way, we construct an exact complex $0\rarrow D
\rarrow F_{n-1}\rarrow\dotsb\rarrow F_0\rarrow E\rarrow0$ in
$\sZ^0(\bE)$ with $F_i\in\bF$, where $n$~is the resolution dimension
of the object $\Phi(E)\in\sK$ with respect to the subcategory
$\sL\subset\sK$.
 Then the complex $0\rarrow\Phi(D)\rarrow\Phi(F_{n-1})\rarrow\dotsb
\rarrow\Phi(F_0)\rarrow\Phi(E)\rarrow0$ is exact in $\sK$ with
$\Phi(F_i)\in\sL$.
 By Proposition~\ref{resol-dim-well-defined}(a), it follows that
$\Phi(D)\in\sL$.

 By the assumption that $(\bF,\sL)$ is a strict exact DG\+subpair
in $(\bE,\sK)$, we can conclude that $D\in\bF$.
 Put $F_n=D$, and denote by $F$ the totalization of the finite
complex $F_n\rarrow F_{n-1}\rarrow\dotsb\rarrow F_0$ in
the DG\+category~$\bF$.
 Then the cone of the natural closed morphism $F\rarrow E$ in
the DG\+category $\bE$ is the totalization of the finite exact
complex $0\rarrow F_n\rarrow F_{n-1}\rarrow\dotsb\rarrow F_0\rarrow E
\rarrow0$ in $\sZ^0(\bE)$; so it is absolutely acyclic in $\bE$
by Lemma~\ref{totalization-of-finite-absolutely-acyclic}.

 Now we apply the following lemma.

\begin{lem} \label{Verdier-quotient-by-intersection-lemma}
 Let\/ $\sT$ be a triangulated category and\/ $\sS$, $\sX\subset\sT$
be (strictly) full triangulated subcategories.
 Assume that for every object $T\in\sT$ there exists an object $S\in\sS$
together with a morphism $S\rarrow T$ in\/ $\sT$ with a cone
belonging to\/~$\sX$.
 Then the triangulated functor between the Verdier quotient categories
$$
 \sS/(\sX\cap\sS)\lrarrow\sT/\sX
$$
induced by the inclusion of triangulated categories\/
$\sS\rightarrowtail\sT$ is an equivalence of triangulated categories.
\end{lem}

\begin{proof}
 See~\cite[Lemma~1.6(a)]{Pkoszul}.
\end{proof}

 In view of Lemma~\ref{Verdier-quotient-by-intersection-lemma} and
the construction in the preceding paragraphs, in order
to prove the equivalence $\sD^\abs(\bF)\simeq\sD^\abs(\bE)$, it
suffices to check that
\begin{equation} \label{Ac-abs-intersection}
 \Ac^\abs(\bE)\cap\sH^0(\bF)=\Ac^\abs(\bF)
 \quad\text{in $\sH^0(\bE)$.}
\end{equation}
 Similarly, in order to prove the equivalence $\sD^\co(\bF)\simeq
\sD^\co(\bE)$, it suffices to check that
\begin{equation} \label{Ac-co-intersection}
 \Ac^\co(\bE)\cap\sH^0(\bF)=\Ac^\co(\bF)
 \quad\text{in $\sH^0(\bE)$,}
\end{equation}
while in order to prove the equivalence $\sD^\ctr(\bF)\simeq
\sD^\ctr(\bE)$, it suffices to check that
\begin{equation} \label{Ac-ctr-intersection}
 \Ac^\ctr(\bE)\cap\sH^0(\bF)=\Ac^\ctr(\bF)
 \quad\text{in $\sH^0(\bE)$}
\end{equation}
(under the respective assumptions).

 For any integer $n\ge1$, let $\sL_n\subset\sK$ denote the full 
subcategory of objects of resolution dimension~$\le n$ with respect
to the resolving subcategory $\sL$ in the exact category~$\sK$.
 By Proposition~\ref{resol-dim-defines-resolving-subcat}, \,$\sL_n$
is a resolving subcategory in~$\sK$; in particular, the full
subcategory $\sL_n\subset\sK$ is closed under extensions.
 Following Example~\ref{exact-DG-pairs-examples}(2), consider
the full DG\+subcategory $\bF_n\subset\bE$ consisting of all
the objects $E\in\bE$ such that $\Phi(E)\in\sL_n$.
 Then $(\bF_n,\sL_n)$ is a strict exact DG\+subpair in $(\bE,\sK)$.

 Recall that $\sL_n$ is a resolving subcategory in $\sK$ by
Proposition~\ref{resol-dim-defines-resolving-subcat}(a).
 Taking into account
Lemma~\ref{one-step-exact-subcategory-resolution-lemma}(a),
we can conclude that $\sZ^0(\bF_n)$ is a resolving subcategory
in the exact category $\sZ^0(\bE)$ (since $(\bF_n,\sL_n)$ is
a strict exact DG\+subpair).
 Moreover, in view of Proposition~\ref{resol-dim-in-exact-sequence}(a),
the kernel of any admissible epimorphism in $\sZ^0(\bE)$ from
an object of $\sZ^0(\bF_{n-1})$ to an object of $\sZ^0(\bF_n)$
belongs to $\sZ^0(\bF_{n-1})$.

 In particular, for $n=0$ we have $\sL_0=\sL$, hence $\bF_0=\bF$.
 By assumption, $\bigcup_{n\ge0}\sL_n=\sK$, hence $\bigcup_{n\ge0}
\bF_n=\bE$.
 It follows that $\Ac^\abs(\bE)=\bigcup_{n\ge0}\Ac^\abs(\bF_n)
\subset\sH^0(\bE)$.

 In order to prove~\eqref{Ac-abs-intersection}, we will show that
\begin{equation} \label{one-step-Ac-abs-intersection}
 \Ac^\abs(\bF_n)\cap\sH^0(\bF_{n-1})=\Ac^\abs(\bF_{n-1})
\end{equation}
for every $n\ge1$.
 Then it will follow by induction in~$n$ that $\Ac^\abs(\bF_n)\cap
\sH^0(\bF)=\Ac^\abs(\bF)$.

 Let us say that an object $C\in\bF_n$ is
\emph{absolutely $(n-1)$\+resolvable} if there exists a short exact
sequence $0\rarrow A\rarrow B\rarrow C\rarrow0$ in $\sZ^0(\bF_n)$
with $A$, $B\in\Ac^\abs(\bF_{n-1})$.
 Let us emphasize that it is presumed here that $A$ and $B$ are
objects of $\bF_{n-1}$ (and \emph{not} only homotopy equivalent to
such objects).
 The following lemma is intended to express the idea of our approach.
 
\begin{lem} \label{absolutely-resolvable-is-sufficient}
 If an object $C\in\bF_{n-1}$ is absolutely $(n-1)$\+resolvable,
then it belongs to $\Ac^\abs(\bF_{n-1})$.
\end{lem}

\begin{proof}
 By assumptions, all the three terms of the short exact sequence
$0\rarrow A\rarrow B\rarrow C\rarrow0$ belong to $\bF_{n-1}$.
 So we have $\Tot(A\to B\to C)\in\Ac^\abs(\bF_{n-1})$ and $A$,
$B\in\Ac^\abs(\bF_{n-1})$, hence also $C\in\Ac^\abs(\bF_{n-1})$.
\end{proof}

 Our aim is to show that, for any object $D\in\Ac^\abs(\bF_n)$,
the object $D\oplus D[1]\in\Ac^\abs(\bF_n)$ is
absolutely $(n-1)$\+resolvable.
 In view of Lemma~\ref{absolutely-resolvable-is-sufficient} and
the discussion above, this will
imply~\eqref{one-step-Ac-abs-intersection}, hence
also~\eqref{Ac-abs-intersection}.

 The full subcategory $\Ac^\abs(\bF_n)$ of absolutely acyclic objects
in $\sH^0(\bF_n)$ is defined by a generation procedure, and
the argument proceeds along the steps of this procedure.
 We present it as a sequence of five lemmas.

\begin{lem} \label{totalizations-abs-resolvable}
 The total object of any short exact sequence in\/ $\sZ^0(\bF_n)$
is absolutely $(n-1)$\+resolvable.
\end{lem}

\begin{proof}
 This is our version of~\cite[Lemma~3.2.A]{PP2}.
 Let $0\rarrow E'\rarrow E\rarrow E''\rarrow0$ be a short exact
sequence in $\sZ^0(\bF_n)$.
 Choose two objects $F'$, $F''\in\bF$ together with admissible
epimorphisms $F'\rarrow E'$ and $F''\rarrow E$ in the exact
category $\sZ^0(\bF_n)$ (or equivalently, in the exact category
$\sZ^0(\bE)$).
 Then there is a termwise admissible epimorphism from the split short
exact sequence $0\rarrow F'\rarrow F'\oplus F''\rarrow F''\rarrow0$
in $\sZ^0(\bF)$  to the short exact sequence $0\rarrow E'\rarrow E
\rarrow E''\rarrow0$ in $\sZ^0(\bF_n)$.
 The sequence of kernels of this termwise admissible epimorphism of
short exact sequences in $\sZ^0(\bF_n)$ is a short exact sequence
$0\rarrow G'\rarrow G \rarrow G''\rarrow0$ in $\sZ^0(\bF_{n-1})$.
 Finally, we have a short exact sequence $0\rarrow\Tot(G'\to G\to G'')
\rarrow\Tot(F'\to F'\oplus F''\to F'')\rarrow\Tot(E'\to E\to E'')
\rarrow0$ in $\sZ^0(\bF_n)$ with $\Tot(G'\to G\to G'')\in
\Ac^\abs(\bF_{n-1})$ and $\Tot(F'\to F'\oplus F''\to F'')
\in\Ac^\abs(\bF)$.
\end{proof}

\begin{lem} \label{cones-cokernels-abs-resolvable}
\textup{(a)} The class of all absolutely $(n-1)$\+resolvable objects
is preserved by the cones of closed morphisms in\/~$\bF_n$. \par
\textup{(b)} The class of all absolutely $(n-1)$\+resolvable objects
is preserved by the cokernels of admissible monomorphisms in\/
$\sZ^0(\bF_n)$.
\end{lem}

\begin{proof}
 We follow the argument from~\cite[Lemma~3.2.B]{PP2}.
 Let $0\rarrow A'\rarrow B'\rarrow C'\rarrow0$ and $0\rarrow A''
\rarrow B''\rarrow C''\rarrow0$ be two short exact sequences in
$\sZ^0(\bF_n)$ with $A'$, $B'$, $A''$, $B''\in\Ac^\abs(\bF_{n-1})$,
and let $C'\rarrow C''$ be a morphism in $\sZ^0(\bF_n)$.
 Put $B'''=B'\oplus B''$, and consider two commutative diagrams of
morphisms of short exact sequences in $\sZ^0(\bF_n)$,
$$
 \xymatrix{
  0 \ar[r] & A' \ar[r] \ar@{>->}[d] & B' \ar[r] \ar@{>->}[d]^{(1,0)}
  & C' \ar[r] \ar[d] & 0 \\
  0 \ar[r] & A''' \ar[r] & B''' \ar[r] & C'' \ar[r] & 0
 }
$$
and
$$
 \xymatrix{
  0 \ar[r] & A'' \ar[r] \ar@{>->}[d] & B'' \ar[r] \ar@{>->}[d]^{(0,1)}
  & C'' \ar[r] \ar@{=}[d] & 0 \\
  0 \ar[r] & A''' \ar[r] & B''' \ar[r] & C'' \ar[r] & 0
 }
$$
where, on both the diagrams, $B'''\rarrow C''$ is the morphism
whose components are the composition $B'\rarrow C'\rarrow C''$
and the admissible epimorphism $B''\rarrow C''$.
 The composition $B''\rarrow B'''\rarrow C''$ is an admissible
epimorphism, while the direct summand inclusion $B''\rarrow B'''$
is admissible monomorphism; hence the morphism $B'''\rarrow C''$ is
an admissible epimorphism in $\sZ^0(\bF_n)$ by the dual version
of~\cite[first assertion of Exercise~3.11(i)]{Bueh}, and we denote
by $A'''$ its kernel.
 Then there is an admissible short exact sequence $0\rarrow A''
\rarrow A'''\rarrow B'\rarrow0$ in $\sZ^0(\bF_n)$ with $A''$, $B'\in
\bF_{n-1}$.
 Therefore, we have $A'''\in\bF_{n-1}$.
 Moreover, since $A''$, $B'\in\Ac^\abs(\bF_{n-1})$ and
$\Tot(A''\to A'''\to B')\in\Ac^\abs(\bF_{n-1})$, it follows that
$A'''\in\Ac^\abs(\bF_{n-1})$.

 To prove part~(a), denote by $A$ and $B$ the cones of the closed
morphisms $A'\rarrow A'''$ and $B'\rarrow B'''$, respectively.
 Then we have a short exact sequence $0\rarrow A\rarrow B\rarrow
\cone(C'\to C'')\to0$ in $\sZ^0(\bF_n)$ with $A$,
$B\in\Ac^\abs(\bF_{n-1})$, showing that $\cone(C'\to C'')$ is
absolutely $(n-1)$\+resolvable.

 To prove part~(b), assume that $C'\rarrow C''$ is an admissible
monomorphism in $\sZ^0(\bF_n)$ with a cokernel $C_0\in\bF_n$.
 Then we have a morphism of (vertical) short exact sequences
in $\sZ^0(\bF_n)$ depicted in the right-hand side of
the commutative diagram
$$
 \xymatrix{
  0 \ar[r] & A' \ar[r] \ar@{>->}[d] & B' \ar[r] \ar@{>->}[d]^{(1,0)}
  & C' \ar[r] \ar[d] & 0 \\
  0 \ar[r] & A''' \ar[r] \ar@{->>}[d] & B''' \ar[r] \ar@{->>}[d]^{(0,1)}
  & C'' \ar[r] \ar@{->>}[d] & 0 \\
  0 \ar[r] & A_0 \ar[r] & B'' \ar[r] & C_0 \ar[r] & 0
 }
$$
 Here the lower rightmost horizontal morphism $B''\rarrow C_0$ is
the composition of admissible epimorphisms $B''\rarrow C''\rarrow C_0$,
so it is an admissible epimorphism as well.
 Thus the morphism from the middle to the rightmost column is
a termwise admissible epimorphism of short exact sequences, and
therefore its kernel $0\rarrow A'\rarrow A'''\rarrow A_0\rarrow0$
is a short exact sequence in $\sZ^0(\bF_n)$, too.

 Now we know that $B_0=B''\in\sZ^0(\bF_{n-1})$ and $C_0\in\sZ^0(\bF_n)$,
hence $A_0\in\sZ^0(\bF_{n-1})$.
 Consequently, $0\rarrow A'\rarrow A'''\rarrow A_0\rarrow0$ is
a short exact sequence in $\sZ^0(\bF_{n-1})$.
 In this short exact sequence, we have $A'$, $A'''\in
\Ac^\abs(\bF_{n-1})$ and $\Tot(A'\to A'''\to A_0)\in
\Ac^\abs(\bF_{n-1})$, hence $A_0\in\Ac^\abs(\bF_{n-1})$.
 Finally, the short exact sequence $0\rarrow A_0\rarrow B_0\rarrow C_0
\rarrow0$ in $\sZ^0(\bF_n)$ with $A_0$, $B_0\in\Ac^\abs(\bF_{n-1})$
shows that the object $C_0\in\bF_n$ is absolutely $(n-1)$\+resolvable.
\end{proof}

\begin{lem} \label{contractibles-abs-resolvable}
 All contractible objects in\/ $\bF_n$ are absolutely
$(n-1)$\+resolvable.
\end{lem}

\begin{proof}
 We follow the argument from~\cite[Lemma~1.4.C]{EP}.
 Let $C$ be a contractible object in the DG\+category $\bF_n$ with
a contracting homotopy $\sigma\in\Hom^{-1}_\bE(C,C)$,
\,$d(\sigma)=\id_C$.
 Following (the proof of)
Lemma~\ref{one-step-exact-subcategory-resolution-lemma}(a),
for any object $L\in\sL$ endowed with an admissible epimorphism
$p\:L\rarrow\Phi(C)$ in the exact category~$\sL_n$, the corresponding
(by adjunction) morphism $\tilde p\:\Psi^+(L)\rarrow C$ is
an admissible epimorphism in the exact category $\sZ^0(\bF_n)$.
 Then we have $\Psi^+(L)\in\bF$.
 Our aim is to show that the kernel $A$ of the admissible
epimorphism~$\tilde p$ in $\sZ^0(\bF_n)$ belongs to
$\Ac^\abs(\bF_{n-1}$).
 This will immediately imply the assertion of the lemma.

 By the definitions of the DG\+category $\bE^\bec$ and
the additive functor $\Psi^+$, for any object $X^\bec\in\bE^\bec$,
the object $\Psi^+(X^\bec)\in\bE$ comes endowed with a contracting
homotopy $\sigma_X\in\Hom_\bE^{-1}(\Psi^+(X^\bec),\Psi^+(X^\bec))$.
 We are interested in the element~$\sigma_L$, which is
the natural contracting homotopy for the object $\Psi^+(L)\in\bF$.
 Then $\tilde u=\tilde p\sigma_L-\sigma\tilde p\in
\Hom^{-1}_\bE(\Psi^+(L),C)$ is a closed morphism of degree~$-1$
in the DG\+category~$\bF_n$.

 Denote by $u\:L[-1]\rarrow\Phi(C)$ the morphism in the additive
category $\sL_n$ corresponding by adjunction to the morphism
$\tilde u\:\Psi^+(L)[1]\rarrow C$ in the additive category
$\sZ^0(\bF_n)$ (cf.\ Lemma~\ref{twists-into-isomorphisms}).
 Consider the pullback of the admissible epimorphism $p\:L\rarrow
\Phi(C)$ along the morphism $u\:L[-1]\rarrow\Phi(C)$ in the exact
category~$\sL_n$, and choose an admissible epimorphism onto
the resulting object from an object $M[-1]\in\sL$.
 We obtain a commutative square diagram
$$
 \xymatrix{
  M[-1] \ar@{->>}[r]^{q[-1]} \ar[d]_v & L[-1] \ar[d]^u \\
  L \ar@{->>}[r]^p & \Phi(C)
 }
$$
with an admissible epimorphism $q\:M\rarrow L$ and some morphism
$v\:M[-1]\rarrow L$ in the exact category~$\sL_n$.
 Then it follows that $q$~is an admissible epimorphism in~$\sL$.

 The induced morphism $\tilde q=\Psi^+(q)\:\Psi^+(M)\rarrow
\Psi^+(L)$ is an admissible epimorphism in $\sZ^0(\bF)$.
 The morphism~$\tilde q$ is homotopic to zero in the DG\+category
$\bF$ with the natural choice of the homotopy $\tilde q\sigma_{M}
=\sigma_L\tilde q$.
 Applying the functor $\Psi^+$ to the morphism~$v$, we obtain
the morphism $\tilde v=\Psi^+(v)\:\Psi^+(M)[1]\rarrow\Psi^+(L)$,
which can be viewed as a closed morphism $\tilde v\in
\Hom^{-1}_\bE(\Psi^+(M),\Psi^+(L))$ of degree~$-1$ in
the DG\+category~$\bF$.
 The morphism $\sigma_L\tilde q-\tilde v\in
\Hom^{-1}_\bE(\Psi^+(M),\Psi^+(L))$ is another contracting
homotopy for the morphism~$\tilde q$, i.~e., we have
$d(\sigma_L\tilde q-\tilde v)=\tilde q$.
 The contracting homotopy $\sigma_L\tilde q-\tilde v$ also forms
a commutative square with the closed morphisms $\tilde p$,
\,$\tilde p\tilde q$, and the contracting homotopy~$\sigma$
for the object $C\in\bF_n$:
$$
 \xymatrix{
  \Psi^+(M) \ar@{->>}[r]^-{\tilde p\tilde q}
  \ar[d]_{\sigma_L\tilde q-\tilde v} & C \ar[d]^\sigma \\
  \Psi^+(L) \ar@{->>}[r]^-{\tilde p} & C
 }
$$
 Indeed, $\tilde p(\sigma_L\tilde q-\tilde v)=
(\tilde p\sigma_L)\tilde q-\tilde p\tilde v=
\sigma\tilde p\tilde q+\tilde u\tilde q-\tilde p\tilde v=
\sigma\tilde p\tilde q$, as the equation $\tilde u\tilde q
=\tilde p\tilde v$ follows from $uq=pv$ by adjunction.

 Let $D\in\bF_{n-1}$ be the kernel of the admissible epimorphism
$\tilde p\tilde q\:\Psi^+(M)\rarrow C$ and $A\in\bF_{n-1}$ be the kernel
of the admissible epimorphism $\tilde p\:\Psi^+(L)\rarrow C$ in
the exact category $\sZ^0(\bF_n)$ (with $\Psi^+(M)$,
$\Psi^+(L)\in\bF$).
 Then the natural admissible epimorphism $r\:D\rarrow A$ in
the exact category $\sZ^0(\bF_{n-1})$ is homotopic to zero in
the DG\+category $\bF_{n-1}$, with a contracting homotopy
$\tau\in\Hom^{-1}_\bE(D,A)$ induced by the contracting homotopy
$\sigma_L\tilde q-\tilde v\in\Hom^{-1}_\bE(\Psi^+(M),\Psi^+(L))$:
$$
 \xymatrix{
  D \ar@{>->}[r] \ar@{->>}[d]_r & \Psi^+(M)
  \ar@{->>}[r]^-{\tilde p\tilde q} \ar@{->>}[d]_{\tilde q}
  & C \ar@{=}[d] \\
  A \ar@{>->}[r] & \Psi^+(L) \ar@{->>}[r]^-{\tilde p} & C
 }
 \qquad\quad
 \xymatrix{
  D \ar@{>->}[r] \ar[d]_\tau & \Psi^+(M)
  \ar@{->>}[r]^-{\tilde p\tilde q}
  \ar[d]_{\sigma_L\tilde q-\tilde v} & C \ar[d]^\sigma \\
  A \ar@{>->}[r] & \Psi^+(L) \ar@{->>}[r]^-{\tilde p} & C
 }
$$

 Furthermore, the functor $\Psi^+\:\sL\rarrow\sZ^0(\bF)$ is exact,
so we have a short exact sequence $0\rarrow\Psi^+(\ker q)\rarrow
\Psi^+(M)\rarrow\Psi^+(L)\rarrow0$ in $\sZ^0(\bF)$, and consequently
a short exact sequence $0\rarrow\Psi^+(\ker q)\rarrow D\rarrow A
\rarrow0$ in $\sZ^0(\bF_{n-1})$.
 Taking the pullback of the latter short exact sequence along
the natural admissible epimorphism $\Xi_{\bE}(A)\rarrow A$, we obtain
a short exact sequence
$$
 0\lrarrow\Psi^+(\ker q)\lrarrow\cone(r)[-1]\lrarrow\Xi_\bE(A)
 \lrarrow0
$$
in the exact category $\sZ^0(\bF_{n-1})$ (cf.\ the alternative proof
of Corollary~\ref{cone-sequences-universally-admissible-exact}).
 On the other hand, the closed morphism~$r$ is homotopic to zero,
so its cone is isomorphic to $A\oplus D[1]$ in $\sZ^0(\bF_{n-1})$.
 Since the objects $\Psi^+(\ker q)$ and $\Xi_\bE(A)$ are contractible
and $\Tot(\Psi^+(\ker q)\to\cone(r)[-1]\to\Xi_\bE(A))\in
\Ac^\abs(\bF_{n-1})$, it follows that $\cone(r)\in\Ac^\abs(\bF_{n-1})$
and therefore $A\in\Ac^\abs(\bF_{n-1})$.

 Finally, the short exact sequence $0\rarrow A\rarrow\Psi^+(L)\rarrow
C\rarrow0$ in $\sZ^0(\bF_n)$ with $A\in\Ac^\abs(\bF_{n-1})$ and
$\Psi^+(L)$ contractible in $\bF$ shows that the object $C\in\bF_n$
is absolutely $(n-1)$\+resolvable, as desired.
\end{proof}

\begin{lem} \label{homotopy-equivalent-abs-resolvable}
 The class of all absolutely $(n-1)$\+resolvable objects in\/ $\bF_n$
is preserved by the homotopy equivalences (i.~e., the isomorphisms
in\/ $\sH^0(\bF_n)$).
\end{lem}

\begin{proof}
 This is our version of~\cite[Lemma~3.2.D]{PP2}.
 Let $C\rarrow C'$ be a closed morphism which is a homotopy equivalence
in the DG\+category $\bF_n$ such that the object $C'$ is absolutely
$(n-1)$\+resolvable.
 Then the object $\cone(C\to C')\in\bF_n$ is contractible, so it is
absolutely $(n-1)$\+resolvable by
Lemma~\ref{contractibles-abs-resolvable}.

 By Lemma~\ref{cones-cokernels-abs-resolvable}(a), it follows that
the objects $\Xi_\bE(C')=\cone(\id_{C'}[-1])$ and
$C''=\cone(C'\to\cone(C\to C'))[-1]\in\bF_n$ are also absolutely
$(n-1)$\+resolvable.
 The object $C''$ is isomorphic to the direct sum $C\oplus\Xi_\bE(C')$
in the additive category $\sZ^0(\bF_n)$.
 Hence, by Lemma~\ref{cones-cokernels-abs-resolvable}(b), the cokernel
$C$ of the direct summand inclusion $\Xi_\bE(C')\rarrow C''$ is
absolutely $(n-1)$\+resolvable as well.
\end{proof}

\begin{lem} \label{direct-summands-almost-abs-resolvable}
 If $D$, $D'\in\bF_n$ are two objects such that the object $D\oplus D'$
is homotopy equivalent to an absolutely $(n-1)$\+resolvable object,
then the object $D\oplus D[1]$ is absolutely $(n-1)$\+resolvable.
\end{lem}

\begin{proof}
 By Lemma~\ref{homotopy-equivalent-abs-resolvable}, the object
$D\oplus D'\in\bF_n$ itself is absolutely $(n-1)$\+resolvable.
 Consider the endomorphism $D\oplus D'\rarrow D\oplus D'$
whose only nonzero component is the identity morphism
$\id_{D'}\:D'\rarrow D'$.
 By Lemma~\ref{cones-cokernels-abs-resolvable}(a), it follows that
the cone of this closed endomorphism of degree~$0$ in $\bF_n$
is absolutely $(n-1)$\+resolvable.
 This cone is homotopy equivalent to $D\oplus D[1]$; so applying
Lemma~\ref{homotopy-equivalent-abs-resolvable} again we conclude that
the object $D\oplus D[1]\in\bF_n$ is absolutely $(n-1)$\+resolvable.
\end{proof}

 The following lemma summarizes the results of the previous ones,
as it was promised before the beginning of this series of lemmas.

\begin{lem} \label{abs-acyclic-are-almost-abs-resolvable}
 For any object $D\in\Ac^\abs(\bF_n)$, the object $D\oplus D[1]
\in\bF_n$ is absolutely $(n-1)$\+resolvable.
\end{lem}

\begin{proof}
 Follows from Lemmas~\ref{totalizations-abs-resolvable},
\ref{cones-cokernels-abs-resolvable}(a),
\ref{homotopy-equivalent-abs-resolvable},
and~\ref{direct-summands-almost-abs-resolvable} by the definition
of the full subcategory of absolutely acyclic objects
$\Ac^\abs(\bF_n)\subset\sH^0(\bF_n)$.
\end{proof}

 Now we can deduce the equality~\eqref{one-step-Ac-abs-intersection}.
 Let $D\in\bF_{n-1}$ be an object belonging to $\Ac^\abs(\bF_n)$.
 Then, by Lemma~\ref{abs-acyclic-are-almost-abs-resolvable},
the object $D\oplus D[1]\in\bF_{n-1}$ is absolutely
$(n-\nobreak1)$\+resolvable.
 By Lemma~\ref{absolutely-resolvable-is-sufficient}, we can conclude
that $D\oplus D[-1]\in\Ac^\abs(\bF_{n-1})$; hence
$D\in\Ac^\abs(\bF_{n-1})$.
 This finishes the proof of
the equalites~\eqref{one-step-Ac-abs-intersection}
and~\eqref{Ac-abs-intersection}, and of part~(a) of the theorem
for absolute derived categories.

 We still have to prove part~(a) for the coderived and contraderived
categories.
 Concerning the coderived categories, assume that infinite coproducts
are well-behaved in the exact DG\+pair $(\bE,\sK)$ and its exact
DG\+subpair $(\bF,\sL)$, as per the formulation of the theorem.
 First of all, we observe that the full subcategories $\sL_n\subset\sK$
of objects of resolution dimension~$\le n$ with respect to $\sL$ are
closed under infinite coproducts in this case (since the coproducts
are exact in the exact category $\sK$ and the full subcategory
$\sL\subset\sK$ is closed under coproducts).
 Consequently, the full DG\+subcategories $\bF_n\subset\bE$ are
closed under coproducts.
 
 Furthermore, we claim that the resolution dimensions of objects
of the exact category $\sK$ with respect to its resolving
subcategory $\sL$ are uniformly bounded in this case, that is,
there is an integer $m\ge0$ such that $\sK=\sL_m$.
 Indeed, suppose the contrary; then there exists a sequence of objects
$(K_n\in\sK)_{n\ge0}$ such that $K_n\notin\sL_n$.
 Consider the object $K=\coprod_{n\ge0}K_n\in\sK$; by assumption,
there exists $m\ge0$ such that $K\in\sL_m$.
 However, the object $K_n$ is a direct summand of $K$ for every~$n$.
 In order to come to a contradiction, it remains to notice that
the full subcategory $\sL_m\subset\sK$ is closed under direct summands,
as in fact any weakly idempotent-complete additive category with
countable coproducts, and $\sL_m$ in particular, is idempotent-complete
(by the cancellation trick).

 In order to prove~\eqref{Ac-co-intersection}, one shows that
\begin{equation} \label{one-step-Ac-co-intersection}
 \Ac^\co(\bF_n)\cap\sH^0(\bF_{n-1})=\Ac^\co(\bF_{n-1})
\end{equation}
for every $n\ge1$.
 Then it follows by induction in~$n$ that $\Ac^\co(\bF_n)\cap
\sH^0(\bF)=\Ac^\co(\bF)$ for all $0\le n\le m$.

 Let us say that an object $C\in\bF_n$ is
\emph{coacycl-$(n-1)$-resolvable} if there exists a short exact
sequence $0\rarrow A\rarrow B\rarrow C\rarrow0$ in $\sZ^0(\bF_n)$
with $A$, $B\in\Ac^\co(\bF_{n-1})$.
 Similarly to the previous argument, it is presumed here that $A$
and $B$ are objects of $\bF_{n-1}$ (and \emph{not} only homotopy
equivalent to such objects).
 The proof is based on the following straightforward analogue
of Lemma~\ref{absolutely-resolvable-is-sufficient}.

\begin{lem} \label{coacycl-resolvable-is-sufficient}
 If an object $C\in\bF_{n-1}$ is coacycl-$(n-1)$-resolvable,
then it belongs to $\Ac^\co(\bF_{n-1})$.  \qed
\end{lem}

\begin{lem} \label{coproducts-coacycl-resolvable}
 The class of all coacycl-$(n-1)$-resolvable objects is closed
under infinite coproducts in\/ $\sH^0(\bF_n)$.
\end{lem}

\begin{proof}
 Follows from the assumption that the coproducts are exact in
the exact category $\sZ^0(\bF_n)$ and the fact that the full
subcategory $\Ac^\co(\bF_{n-1})\subset\sH^0(\bF_{n-1})$ is
closed under coproducts by the definition.
\end{proof}

\begin{lem} \label{cones-cokernels-coacycl-resolvable}
\textup{(a)} The class of all coacycl-$(n-1)$-resolvable objects
is preserved by the cones of closed morphisms in\/~$\bF_n$. \par
\textup{(b)} The class of all coacycl-$(n-1)$-resolvable objects
is preserved by the cokernels of admissible monomorphisms in\/
$\sZ^0(\bF_n)$.
\end{lem}

\begin{proof}
 Similar to the proof of
Lemma~\ref{cones-cokernels-abs-resolvable}.
\end{proof}

\begin{lem} \label{homotopy-equivalent-coacycl-resolvable}
 The class of all coacycl-$(n-1)$-resolvable objects in\/ $\bF_n$
is preserved by the homotopy equivalences.
\end{lem}

\begin{proof}
 Similar to the proof of
Lemma~\ref{homotopy-equivalent-abs-resolvable}.
\end{proof}

\begin{lem} \label{coacyclic-coacycl-resolvable}
 All the objects in\/ $\Ac^\co(\bF_n)$ are coacycl-$(n-1)$-resolvable.
\end{lem}

\begin{proof}
 Follows from Lemmas~\ref{totalizations-abs-resolvable},
\ref{cones-cokernels-coacycl-resolvable}(a),
\ref{coproducts-coacycl-resolvable},
and~\ref{homotopy-equivalent-coacycl-resolvable} by the definition
of the full subcategory of coacyclic objects
$\Ac^\co(\bF_n)\subset\sH^0(\bF_n)$.
\end{proof}

 Now we can deduce the equality~\eqref{one-step-Ac-co-intersection}.
 Let $C\in\bF_{n-1}$ be an object belonging to $\Ac^\co(\bF_n)$.
 Then, by Lemma~\ref{coacyclic-coacycl-resolvable}, the object $C$
is coacycl-$(n-1)$\+resolvable.
 By Lemma~\ref{coacycl-resolvable-is-sufficient}, it follows that
that $C\in\Ac^\co(\bF_{n-1})$, as desired.
 This finishes the proof of
the equalites~\eqref{one-step-Ac-co-intersection}
and~\eqref{Ac-co-intersection}, and of part~(a) of the theorem
for coderived categories.

 The argument for the contraderived categories in part~(a), under
the respective assumptions that infinite products are well-behaved
in the exact DG\+pair $(\bE,\sK)$ and its exact DG\+subpair
$(\bF,\sL)$, is completely similar to the proof of the assertion for
the coderived categories above.
 One defines \emph{contraacycl-$(n-1)$-resolvable} objects in $\bF_n$,
proves the respective versions of
Lemmas~\ref{coacycl-resolvable-is-sufficient}\+-%
\ref{coacyclic-coacycl-resolvable}, and deduces the desired
equality~\eqref{Ac-ctr-intersection} from the equality
\begin{equation} \label{one-step-Ac-ctr-intersection}
 \Ac^\ctr(\bF_n)\cap\sH^0(\bF_{n-1})=\Ac^\ctr(\bF_{n-1})
\end{equation}
by induction in $0\le n\le m$, using the fact that there exists
$m\ge0$ such that $\sK=\sL_m$.
 We omit the straightforward details.

 Alternatively, the assertion of part~(a) for the contraderived
categories and the assertion of part~(b) for the coderived categories
can be obtained as particular cases of the more general result of
Theorem~\ref{contraderived-resolving-equivalence-theorem}, which
does not require finite (co)resolution dimension.
 Moreover, after we already know that the triangulated functor in
question is essentially surjective, which was explained in
the very beginning of this proof, it becomes enough to show that
the functor is fully faithful.
 So the assertions of parts~(a) and~(b) for the absolute derived
categories can be deduced from the quite general result of
Theorem~\ref{full-and-faithfulness-main-theorem}, which does not
require finite (co)resolution dimension, either.

 Notice, however, that the assertion of part~(a) for the coderived
categories and the assertion of part~(b) for the contraderived
categories do not seem to be provable in this alternative way
(cf.~\cite[Remark~1.5]{EP}).
\end{proof}

\subsection{Examples} \label{finite-resol-dim-examples-subsecn}
 In this section we formulate the particular cases of
Theorem~\ref{finite-resolution-dimension-main-theorem} arising
in the context of Examples~\ref{exact-dg-category-of-complexes}\+-%
\ref{exact-dg-category-of-factorizations}
and~\ref{exact-DG-pairs-examples}.

 For any exact category $\sE$, we put $\sD^\abs(\sE)=
\sD^\abs(\bC(\sE))$, where $\bC(\sE)$ is the exact DG\+category
of complexes in $\sE$ constructed in
Example~\ref{exact-dg-category-of-complexes}.
 If the exact category $\sE$ has exact coproducts, then so does
the DG\+category $\bC(\sE)$.
 In this case, we put $\sD^\co(\sE)=\sD^\co(\bC(\sE))$.
 Dually, if the exact category $\sE$ has exact products, then
so does the DG\+category $\bC(\sE)$, and in this case we put
$\sD^\ctr(\sE)=\sD^\ctr(\bC(\sE))$.

 The following result can be found in~\cite[Proposition~A.5.8]{Pcosh}.

\begin{cor} \label{complexes-finite-resol-dim-corollary}
 Let\/ $\sE$ be an exact category and\/ $\sF$ be
a resolving subcategory in\/ $\sE$ such that all the objects of\/
$\sE$ have finite resolution dimensions with respect to\/~$\sF$.
 If the grading group\/ $\Gamma$ is infinite, assume additionally
that the\/ $\sF$\+resolution dimensions of the objects of\/ $\sE$
are uniformly bounded, i.~e., there exists an integer $m\ge0$
such that the\/ $\sF$\+resolution dimension of any object of\/ $\sE$
does not exceed~$m$.
 Then the triangulated functor
$$
 \sD^\abs(\sF)\lrarrow\sD^\abs(\sE)
$$
induced by the inclusion of exact categories\/ $\sF\rightarrowtail\sE$
is an equivalence of triangulated categories.

 Furthermore, if the exact category\/ $\sE$ has exact coproducts and
the full subcategory\/ $\sF\subset\sE$ is closed under coproducts,
then the triangulated functor
$$
 \sD^\co(\sF)\lrarrow\sD^\co(\sE)
$$
induced by the inclusion of exact categories\/ $\sF\rightarrowtail\sE$
is a triangulated equivalence.
 If the exact category\/ $\sE$ has exact products and the full 
subcategory\/ $\sF\subset\sE$ is closed under products, then
the triangulated functor
$$
 \sD^\ctr(\sF)\lrarrow\sD^\ctr(\sE)
$$
induced by the inclusion\/ $\sF\rightarrowtail\sE$ is a triangulated
equivalence.
\end{cor}

\begin{proof}
 Put $\bE=\bC(\sE)$, \,$\bF=\bC(\sF)$, \,$\sK=\sG(\sE)$,
\,$\sL=\sG(\sF)$, as in Examples~\ref{exact-dg-category-of-complexes}
and~\ref{exact-DG-pairs-examples}(4), and apply
Theorem~\ref{finite-resolution-dimension-main-theorem}(a).
\end{proof}

 The following corollary is a generalization
of~\cite[Theorem~3.2]{PP2}.

\begin{cor} \label{cdg-ring-finite-resol-dim-corollary}
 Let $\biR^\cu=(R^*,d,h)$ be a CDG\+ring and $\sK\subset R^*\smodl$ be
a full subcategory in the category of graded $R^*$\+modules such that\/
$\sK$ is preserved by the shift functors $[n]$, \,$n\in\Gamma$,
inherits an exact category structure from the abelian exact structure
of $R^*\smodl$, and the underlying graded $R^*$\+module of
the CDG\+module $G^+(K^*)$ belongs to\/ $\sK$ for any $K^*\in\sK$.
 Denote by\/ $\bE\subset\biR^\cu\bmodl$ the full DG\+subcategory
consisting of all the left CDG\+modules over $\biR^\cu$
whose underlying graded $R^*$\+modules belong to\/ $\sK$,
and endow\/ $\bE$ with the exact DG\+category structure inherited
from the abelian exact DG\+category structure of $\biR^\cu\bmodl$. \par
 \textup{(a)} Let\/ $\sL\subset\sK$ be a resolving subcategory
such that\/ $\sL$ is preserved by the shift functors and all
the objects of\/ $\sK$ have finite resolution dimensions
with respect to\/~$\sL$.
 Denote by\/ $\bF\subset\bE$ the full DG\+subcategory consisting
of all the CDG\+modules whose underlying graded $R^*$\+modules
belong to\/ $\sL$, and endow\/ $\bF$ with the inherited exact
DG\+category structure.
 Then the triangulated functor
$$
 \sD^\abs(\bF)\lrarrow\sD^\abs(\bE)
$$
induced by the inclusion of exact DG\+categories\/ $\bF\rightarrowtail
\bE$ is an equivalence of triangulated categories.

 Furthermore, if the full subcategories\/ $\sL$ and\/ $\sK$ are
preserved by the infinite direct sums in $R^*\smodl$, then
the triangulated functor
$$
 \sD^\co(\bF)\lrarrow\sD^\co(\bE)
$$
induced by the inclusion of exact DG\+categories\/ $\bF\rightarrowtail
\bE$ is a triangulated equivalence.
 If the full subcategories\/ $\sL$ and\/ $\sK$ are preserved by
the infinite products in $R^*\smodl$, then the triangulated functor
$$
 \sD^\ctr(\bF)\lrarrow\sD^\ctr(\bE)
$$
induced by the inclusion\/ $\bF\rightarrowtail\bE$ is a triangulated
equivalence. \par
 \textup{(b)} Let\/ $\sM\subset\sK$ be a coresolving subcategory
such that\/ $\sM$ is preserved by the shift functors and all
the objects of\/ $\sK$ have finite coresolution dimensions
with respect to\/~$\sM$.
 Denote by\/ $\bG\subset\bE$ the full DG\+subcategory consisting
of all the CDG\+modules whose underlying graded $R^*$\+modules
belong to\/ $\sM$, and endow\/ $\bG$ with the inherited exact
DG\+category structure.
 Then the triangulated functor
$$
 \sD^\abs(\bG)\lrarrow\sD^\abs(\bE)
$$
induced by the inclusion of exact DG\+categories\/ $\bG\rightarrowtail
\bE$ is an equivalence of triangulated categories.

 Furthermore, if the full subcategories\/ $\sM$ and\/ $\sK$ are
preserved by the infinite direct sums in $R^*\smodl$, then
the triangulated functor
$$
 \sD^\co(\bG)\lrarrow\sD^\co(\bE)
$$
induced by the inclusion of exact DG\+categories\/ $\bG\rightarrowtail
\bE$ is a triangulated equivalence.
 If the full subcategories\/ $\sM$ and\/ $\sK$ are preserved by
the infinite products in $R^*\smodl$, then the triangulated functor
$$
 \sD^\ctr(\bG)\lrarrow\sD^\ctr(\bE)
$$
induced by the inclusion\/ $\bG\rightarrowtail\bE$ is a triangulated
equivalence.
\end{cor}

\begin{proof}
 It was explained in Example~\ref{abelian-dg-category-of-cdg-modules}
that the DG\+category $\biR^\cu\bmodl$ of left CDG\+modules over
$(R^*,d,h)$ is an abelian DG\+category.
 It follows from Proposition~\ref{subcategory-L-prop}(c) and
the discussion in Example~\ref{exact-DG-pairs-examples}(1) that,
under the assumptions from the first paragraph of the corollary,
$(\bE,\sK)$ is a strict exact DG\+subpair in
$(\biR^\cu\bmodl,R^*\smodl)$.
 Furthermore, following Example~\ref{exact-DG-pairs-examples}(2),
in the assumptions of part~(a) \,$(\bF,\sL)$ is a strict
exact DG\+subpair in $(\bE,\sK)$, while in the assumptions of
part~(b) \,$(\bG,\sM)$ is a strict exact DG\+subpair in $(\bE,\sK)$.
 It remains to recall that both the infinite coproducts and
infinite products are exact in the abelian category $R^*\smodl$
in order to conclude that the assertions of parts~(a) and~(b) of
the corollary are particular cases of those of
Theorem~\ref{finite-resolution-dimension-main-theorem}(a\+-b).
\end{proof}

 Part~(a) of the next corollary is a generalization
of~\cite[Theorem~1.4]{EP}.

\begin{cor} \label{cdg-quasi-algebra-finite-resol-dim-corollary}
 Let $X$ be a scheme and $\biB^\cu=(B^*,d,h)$ be a quasi-coherent
CDG\+quasi-algebra over~$X$.
 Let $\sK\subset B^*\sqcoh$ be a full subcategory in the category
of quasi-coherent graded $B^*$\+modules such that\/ $\sK$ is preserved
by the shift functors $[n]$, \,$n\in\Gamma$, inherits an exact category
structure from the abelian exact structure of $B^*\sqcoh$, and
the underlying quasi-coherent graded $B^*$\+module of the quasi-coherent
CDG\+module $G^+(K^*)$ belongs to\/ $\sK$ for any $K^*\in\sK$.
 Denote by\/ $\bE\subset\biB^\cu\bqcoh$ the full DG\+subcategory
consisting of all the quasi-coherent left CDG\+modules over $\biB^\cu$
whose underlying quasi-coherent graded $B^*$\+modules belong to\/ $\sK$,
and endow\/ $\bE$ with the exact DG\+category structure inherited
from the abelian exact DG\+category structure of $\biB^\cu\bqcoh$. \par
 \textup{(a)} Let\/ $\sL\subset\sK$ be a resolving subcategory
such that\/ $\sL$ is preserved by the shift functors and all
the objects of\/ $\sK$ have finite resolution dimensions
with respect to\/~$\sL$.
 Denote by\/ $\bF\subset\bE$ the full DG\+subcategory consisting
of all the quasi-coherent CDG\+modules whose underlying quasi-coherent
graded $B^*$\+modules belong to\/ $\sL$, and endow\/ $\bF$ with
the inherited exact DG\+category structure.
 Then the triangulated functor
$$
 \sD^\abs(\bF)\lrarrow\sD^\abs(\bE)
$$
induced by the inclusion of exact DG\+categories\/ $\bF\rightarrowtail
\bE$ is an equivalence of triangulated categories.

 Furthermore, if the full subcategories\/ $\sL$ and\/ $\sK$ are
preserved by the infinite direct sums in $B^*\sqcoh$, then
the triangulated functor
$$
 \sD^\co(\bF)\lrarrow\sD^\co(\bE)
$$
induced by the inclusion of exact DG\+categories\/ $\bF\rightarrowtail
\bE$ is a triangulated equivalence. \par
 \textup{(b)} Let\/ $\sM\subset\sK$ be a coresolving subcategory
such that\/ $\sM$ is preserved by the shift functors and all
the objects of\/ $\sK$ have finite coresolution dimensions
with respect to\/~$\sM$.
 Denote by\/ $\bG\subset\bE$ the full DG\+subcategory consisting
of all the quasi-coherent CDG\+modules whose underlying quasi-coherent
graded $B^*$\+modules belong to\/ $\sM$, and endow\/ $\bG$ with
the inherited exact DG\+category structure.
 Then the triangulated functor
$$
 \sD^\abs(\bG)\lrarrow\sD^\abs(\bE)
$$
induced by the inclusion of exact DG\+categories\/ $\bG\rightarrowtail
\bE$ is an equivalence of triangulated categories.

 Furthermore, if the full subcategories\/ $\sM$ and\/ $\sK$ are
preserved by the infinite direct sums in $B^*\sqcoh$, then
the triangulated functor
$$
 \sD^\co(\bG)\lrarrow\sD^\co(\bE)
$$
induced by the inclusion of exact DG\+categories\/ $\bG\rightarrowtail
\bE$ is a triangulated equivalence.
\end{cor}

\begin{proof}
 According to
Example~\ref{abelian-dg-category-of-quasi-coh-cdg-modules},
the DG\+category $\biB^\cu\bqcoh$ of quasi-coherent left CDG\+modules over
$(B^*,d,h)$ is an abelian DG\+category.
 It follows from Proposition~\ref{subcategory-L-prop}(c) and
the discussion in Example~\ref{exact-DG-pairs-examples}(1) that,
under the assumptions from the first paragraph of the corollary,
$(\bE,\sK)$ is a strict exact DG\+subpair in
$(\biB^\cu\bqcoh,B^*\sqcoh)$.
 Following Example~\ref{exact-DG-pairs-examples}(2), in the assumptions
of part~(a) \,$(\bF,\sL)$ is a strict exact DG\+subpair in $(\bE,\sK)$,
while in the assumptions of part~(b) \,$(\bG,\sM)$ is a strict exact
DG\+subpair in $(\bE,\sK)$.
 It remains to recall that the infinite coproducts are exact in
the abelian category $B^*\sqcoh$ and apply
Theorem~\ref{finite-resolution-dimension-main-theorem}(a\+-b)
in order to deduce parts~(a) and~(b) of the corollary.
\end{proof}

 Notice that assertions about the contraderived categories are missing
from
Corollary~\ref{cdg-quasi-algebra-finite-resol-dim-corollary}(a\+-b),
because the functors of infinite product of quasi-coherent sheaves
are not exact.
 If one wishes to assign some contraderived categories to
a quasi-coherent CDG\+quasi-algebra $\biB^\cu$ over a nonaffine
scheme $X$, then the way to proceed is to consider contraherent
cosheaves of CDG\+modules over $\biB^\cu$ instead of quasi-coherent
sheaves (see~\cite{Pcosh}).

 The following corollary is a generalization
of~\cite[Corollaries~2.3(a,b,c,e,g) and~2.6]{EP}.

\begin{cor} \label{factorizations-finite-resol-dim-corollary}
 Let\/ $\sE$ be an exact category and\/ $\Lambda\:\sE\rarrow\sE$ be
an autoequivalence preserving and reflecting short exact sequences.
 Assume that either\/ $\Gamma=\boZ$, or\/ $\Lambda$ is involutive
and\/ $\Gamma=\boZ/2$.
 Let $w\:\Id_\sE\rarrow\Lambda^2$ be a potential (as in
Section~\ref{factorizations-subsecn}).
 Let\/ $\sH$ be a resolving subcategory in\/ $\sE$ preserved by
the autoequivalences\/ $\Lambda$ and\/ $\Lambda^{-1}$ and such that
all the objects of\/ $\sE$ have finite resolution dimensions with
respect to\/~$\sH$.
 Then the triangulated functor
$$
 \sD^\abs(\bF(\sH,\Lambda,w))\lrarrow\sD^\abs(\bF(\sE,\Lambda,w))
$$
induced by the inclusion of exact categories\/ $\sH\rightarrowtail\sE$
is an equivalence of triangulated categories.

 Furthermore, if the exact category\/ $\sE$ has exact coproducts and
the full subcategory\/ $\sH\subset\sE$ is closed under coproducts,
then the triangulated functor
$$
 \sD^\co(\bF(\sH,\Lambda,w))\lrarrow\sD^\co(\bF(\sE,\Lambda,w))
$$
induced by the inclusion of exact categories\/ $\sH\rightarrowtail\sE$
is a triangulated equivalence.
 If the exact category\/ $\sE$ has exact products and the full 
subcategory\/ $\sH\subset\sE$ is closed under products, then
the triangulated functor
$$
 \sD^\ctr(\bF(\sH,\Lambda,w))\lrarrow\sD^\ctr(\bF(\sE,\Lambda,w))
$$
induced by the inclusion\/ $\sH\rightarrowtail\sE$ is a triangulated
equivalence.
\end{cor}

\begin{proof}
 Put $\bE=\bF(\sE,\Lambda,w)$, \,$\bF=\bF(\sH,\Lambda,w)$,
\,$\sK=\sP(\sE,\Lambda)$, \,$\sL=\sP(\sH,\Lambda)$,
as in Examples~\ref{exact-dg-category-of-factorizations}
and~\ref{exact-DG-pairs-examples}(5), and apply
Theorem~\ref{finite-resolution-dimension-main-theorem}(a).
\end{proof}

\Section{Full-and-Faithfulness Theorem}
\label{full-and-faithfulness-secn}

\subsection{Self-resolving subcategories in exact categories}
\label{self-resolving-subcategories-subsecn}
 Let $\sE$ be an exact category.
 The following condition is fairly well-known
(see~\cite[Section~12]{Kel2}), but we are not aware of any terminology
for it in the literature.
 We will say that a full subcategory $\sF\subset\sE$ is
\emph{self-resolving} if
\begin{itemize}
\item $\sF$ inherits an exact category structure from~$\sE$;
\item $\sF$ is closed under the kernels of admissible epimorphisms
in~$\sE$;
\item for any objects $F\in\sF$ and $E\in\sE$, and any admissible
epimorphism $E\rarrow F$ in $\sE$, there exist an object $F'\in\sF$,
an admissible epimorphism $F'\rarrow F$, and a morphism $F'\rarrow E$
such that the triangle diagram $F'\rarrow E\rarrow F$ is commutative
in~$\sE$.
\end{itemize}
 Clearly, any resolving subcategory is self-resolving.

 Dually, a full subcategory $\sG\subset\sE$ is said to be
\emph{self-coresolving} if
\begin{itemize}
\item $\sG$ inherits an exact category structure from~$\sE$;
\item $\sG$ is closed under the cokernels of admissible monomorphisms
in~$\sE$;
\item for any objects $G\in\sG$ and $E\in\sE$, and any admissible
monomorphism $G\rarrow E$ in $\sE$, there exist an object $G'\in\sG$,
an admissible monomorphism $G\rarrow G'$, and a morphism $E\rarrow G'$
such that the triangle diagram $G\rarrow E\rarrow G'$ is commutative
in~$\sE$.
\end{itemize}
 Clearly, any coresolving subcategory is self-coresolving.

 The following lemma is our version of~\cite[Remark~2.2]{PSch}.

\begin{lem}
 Any self-resolving or self-coresolving subcategory in an exact
category is closed under extensions.
\end{lem}

\begin{proof}
 Let $0\rarrow G\rarrow E\rarrow F\rarrow0$ be a short exact sequence
in an exact category $\sE$ with the objects $F$ and $G$ belonging to
a self-resolving subcategory $\sF\subset\sE$.
 Then there exists an object $F'\in\sF$, an admissible epimorphism
$F'\rarrow F$, and a morphism $F'\rarrow E$ such that the triangle
diagram $F'\rarrow E\rarrow F$ is commutative.
 Since $\sF$ is closed under the kernels of admissible epimorphisms
in $\sE$, the kernel $H$ of the admissible epimorphism $F'\rarrow F$
belongs to~$\sF$.
 So we have a commutative diagram in the exact category~$\sE$
$$
 \xymatrix{
  H \ar@{>->}[r] \ar[d] & F' \ar@{->>}[r] \ar[d] & F \\
  G \ar@{>->}[r] & E \ar@{->>}[ur]
 }
$$
with the objects $H$, $F'$, $F$, $G\in\sF$ and admissible short
exact sequences $0\rarrow H\rarrow F'\rarrow F\rarrow0$ and
$0\rarrow G\rarrow E\rarrow F\rarrow0$.
 By~\cite[Proposition~A.2 and Corollary~A.3]{Partin}, it follows that
$H\rarrow F'\rarrow E$, \ $H\rarrow G\rarrow E$ is a pushout square
in~$\sE$.
 Now Lemma~\ref{inheriting-exact-structure-lemma}(ii) tells that
$E\in\sF$.
\end{proof}

\subsection{Approachability in triangulated categories}
\label{approachability-subsecn}
 The terminology ``approachable object'' goes back
to~\cite[proof of Theorem~5.5]{PSch}.

 Let $\sT$ be a triangulated category and $\sS$, $\sY\subset\sT$ be
two full subcategories.
 We will say that an object $X\in\sT$ is \emph{approachable from\/
$\sS$ via\/~$\sY$} if every morphism $S\rarrow X$ in $\sT$ with
an object $S\in\sS$ factorizes through an object from~$\sY$.

 Equivalently, an object $X$ is approachable from $\sS$ via $\sY$ if
and only if, for every morphism $S\rarrow X$ in $\sT$ with $S\in\sS$
there exists an object $S'\in\sT$ and a morphism $S'\rarrow S$ in $\sT$
with a cone belonging to $\sY$ such that the composition $S'\rarrow S
\rarrow X$ vanishes in~$\sT$.
 When $\sS$ is a full triangulated subcategory in $\sT$ and
$\sY\subset\sS$, the conditions in the criterion above imply that
$S'\in\sS$.

 Dually, an object $X\in\sT$ is \emph{coapproachable to\/ $\sS$
via\/~$\sY$} if every morphism $X\rarrow S$ in $\sT$ with
an object $S\in\sS$ factorizes through an object from~$\sY$.
 The following lemma constitutes a standard technique for proving
full-and-faithfulness of triangulated functors induced by inclusions.

\begin{lem} \label{approachable-implies-fully-faithful-lemma}
 Let\/ $\sT$ be a triangulated category and\/ $\sS$, $\sX\subset\sT$
be two (strictly) full triangulated subcategories.
 Let\/ $\sY\subset\sS\cap\sX$ be a full triangulated subcategory in
the intersection.
 Assume that all objects of\/ $\sX$ are approachable from\/ $\sS$
via\/~$\sY$.
 Then the triangulated functor between the Verdier quotient categories
$$
 \sS/\sY\lrarrow\sT/\sX
$$
induced by the inclusion of triangulated categories\/
$\sS\rightarrowtail\sT$ is fully faithful.
\end{lem}

\begin{proof}
 Straightforward from the definitions.
\end{proof}

 In the following three lemmas, we list the basic properties of
the class of all approachable objects.

\begin{lem} \label{summands-approachable-lemma}
 For any two full subcategories\/ $\sS$, $\sY\subset\sT$, the class of
all objects approachable from\/ $\sS$ via\/ $\sY$ is closed under
direct summands in\/~$\sT$.
 If the full subcategory\/ $\sY\subset\sT$ is closed under finite
direct sums, then so is the full subcategory of all objects
approachable from\/ $\sS$ via\/~$\sY$.
\end{lem}

\begin{proof}
 This lemma does not depend on the assumption that $\sT$ is
triangulated, but holds in any additive category.
 We omit the largely straightforward proof (see the proof of the next
Lemma~\ref{co-products-approachable-lemma} for some details).
\end{proof}

\begin{lem} \label{co-products-approachable-lemma}
\textup{(a)} Assume that all infinite products exist in\/ $\sT$ and
the full subcategory\/ $\sY\subset\sT$ is closed under products.
 Then the full subcategory of all objects approachable from\/ $\sS$
via\/ $\sY$ is closed under products in\/~$\sT$. \par
\textup{(b)} Assume that all infinite coproducts exists in\/ $\sT$ and
the full subcategory\/ $\sY\subset\sT$ is closed under coproducts.
 Then the full subcategory of all objects coapproachable to\/ $\sS$
via\/ $\sY$ is closed under coproducts in\/~$\sT$.
\end{lem}

\begin{proof}
 This lemma also does not depend on the assumption that $\sT$ is
triangulated.
 Let us prove part~(a) (part~(b) is dual).
 Let $X_\alpha$ be a family of objects in $\sT$ such that $X_\alpha$
is approachable from $\sS$ via $\sY$ for every~$\alpha$.
 Let $S\rarrow\prod_\alpha X_\alpha$ be a morphism into the product of
$X_\alpha$ in $\sT$ from an object $S\in\sS$.
 Let $S\rarrow X_\alpha$ be the components of the morphism
$S\rarrow\prod_\alpha X_\alpha$.
 Then for every~$\alpha$ there exists an object $Y_\alpha\in\sY$
such that the morphism $S\rarrow X_\alpha$ factorizes
through~$Y_\alpha$.
 Hence the morphism $S\rarrow\prod_\alpha X_\alpha$ factorizes through
the object $\prod_\alpha Y_\alpha\in\sY$.
\end{proof}

\begin{lem} \label{shifts-cones-approachable-lemma}
 Let\/ $\sT$ be a triangulated category and\/ $\sS$, $\sY\subset\sT$
be two full triangulated subcategories such that\/ $\sY\subset\sS$.
 Then the full subcategory of all objects approachable from\/ $\sS$
via $\sY$ is triangulated, that is, closed under shifts and cones
in\/~$\sT$.
\end{lem}

\begin{proof}
 First of all, the class of all approachable objects is preserved
by the shifts in~$\sT$, because the full subcategories
$\sS$, $\sY\subset\sT$ are.

 Furthermore, let $X\rarrow W\rarrow Z\rarrow X[1]$ be a distinguished
triangle in $\sT$ with approachable objects $X$ and~$Z$,
and let $S\rarrow W$ be a morphism in $\sT$ from an object $S\in\sS$.
 Consider the composition $S\rarrow W\rarrow Z$ in~$\sT$.
 By assumption, there exists an object $S'\in\sT$ and a morphism
$S'\rarrow S$ with a cone belonging to $\sY$ such that the composition
$S'\rarrow S\rarrow W\rarrow Z$ vanishes in~$\sT$.

 Hence the composition $S'\rarrow S\rarrow W$ factorizes through
the morphism $X\rarrow W$ in~$\sT$; so we obtain a morphism
$S'\rarrow X$.
 Since $\sS\subset\sT$ is a full triangulated subcategory and
$\cone(S'\to S)\in\sY\subset\sS$, we have $S'\in\sS$.

 By assumption, there exists an object $S''\in\sT$ (in fact,
$S''\in\sS$) and a morphism $S''\rarrow S'$ with a cone belonging
to $\sY$ such that the composition $S''\rarrow S' \rarrow X$ vanishes.
 Now, since $\sY\subset\sT$ is a full triangulated subcategory,
the composition $S''\rarrow S'\rarrow S$ has a cone belonging to~$\sY$.
 The composition $S''\rarrow S'\rarrow S\rarrow W$ vanishes in
$\sT$, and we are done.
\end{proof}

\subsection{Three lemmas} \label{lemmas-E-and-F-subsecn}
 Let $\bA$ be a DG\+category with shifts and cones, and let
$U\overset j\rarrow V\overset k\rarrow W$ be a three-term complex
in the additive category of closed morphisms $\sZ^0(\bA)$.
 Denote by $T=\Tot(U\to V\to W)$ be the corresponding total object
in~$\bA$.

 By the definition, for any object $A\in\bA$, the complex of morphisms
$\Hom_\bA^\bu(A,T)$ is naturally isomorphic to the totalization of
the bicomplex with three rows $\Hom_\bA^\bu(A,U)\rarrow
\Hom_\bA^\bu(A,V)\rarrow\Hom_\bA^\bu(A,W)$.
 In particular, for any $n\in\Gamma$, elements of the abelian group
$\Hom_\bA^n(A,T)$ can be represented by triples $(f,g,h)$, where
$f\in\Hom_\bA^{n+1}(A,U)$, \ $g\in\Hom_\bA^n(A,V)$, and
$h\in\Hom_\bA^{n-1}(A,W)$.

 The following Lemmas~\ref{Hom-complex-into-Tot-lemma}
and~\ref{liftability-preserved-by-Phi-Psi-adjunction} form a useful
technique for proving theorems about derived categories of
the second kind.

\begin{lem} \label{Hom-complex-into-Tot-lemma}
\textup{(a)} The differential of an element of\/ $\Hom_\bA^n(A,T)$
represented by a triple $(f,g,h)$ is given by the rule
$$
 d(f,g,h)=(-df,-jf+dg,kg-dh).
$$ \par
\textup{(b)} Assume that $j=\ker(k)$ in the additive category\/
$\sZ^0(\bA)$.
 Then a closed morphism $(f,g,h)\in\Hom_\bA^n(A,T)$ is homotopic to
zero whenever the morphism $h\in\Hom_\bA^{n-1}(A,W)$ can be lifted
to a morphism $t\in\Hom_\bA^{n-1}(A,V)$, that is, $h=kt$. \par
\textup{(c)} Moreover, in the assumptions of~(b), to any closed
morphism $(f,g,h)\in\Hom_\bA^n(A,T)$ and a lifting
$t\in\Hom_\bA^{n-1}(A,V)$ of the morphism $h\in\Hom_\bA^{n-1}(A,W)$
one can assign a naturally defined contracting homotopy for
the morphism $(f,g,h)$.
\end{lem}

\begin{proof}
 This is a straightforward generalization of~\cite[Lemma~1.5.E]{EP}.
 Part~(a) holds by the definition of the differential in a bicomplex
with three rows.

 To prove part~(b), one can observe that the identity inclusion
functor $\sZ^0(\bA)\rarrow\bA^0$ preserves kernels, because $\bA^0$
is a full subcategory in $\sZ^0(\bA^\bec)$
(see Lemma~\ref{Phi-Psi-extended-to-nonclosed-morphisms}) and
the functor $\Phi_\bA\:\sZ^0(\bA)\rarrow\sZ^0(\bA^\bec)$ preserves
kernels since it is a right adjoint.
 So $0\rarrow\Hom_\bA^\bu(A,U)\rarrow\Hom_\bA^\bu(A,V)\rarrow
\Hom_\bA^\bu(A,W)$ is a left exact sequence of complexes of abelian
groups.
 Denoting by $\Hom_\bA^\bu(A,W)'\subset\Hom_\bA^\bu(A,W)$
the subcomplex constructed as the image of the map $\Hom_\bA^\bu(A,V)
\rarrow\Hom_\bA^\bu(A,W)$, we have a short exact sequence of
complexes $0\rarrow\Hom_\bA^\bu(A,U)\rarrow\Hom_\bA^\bu(A,V)\rarrow
\Hom_\bA^\bu(A,W)'\rarrow0$.
 The total complex of this short exact sequence is an acyclic
subcomplex in $\Hom_\bA^\bu(A,T)$.
 Hence any cocycle in $\Hom_\bA^\bu(A,T)$ that belongs to this
subcomplex is a coboundary.

 To prove part~(c), one only needs to make the above proof of~(b)
more explicit.
 The equations $kt=h$, $dk=0$, and $kg-dh=0$ imply that $k(dt-g)=
d(kt)-kg=dh-kg=0$.
 Since the short sequence of complexes of abelian groups in the proof
of part~(b) is left exact, it follows that there exists a unique
element $s\in\Hom^n_\bA(A,U)$ such that $dt-g=js$.
 Then the equations $dj=0$ and $-jf+dg=0$ imply that $jds=d(js)=-dg=
-jf$; hence $ds=-f$ and $d(s,t,0)=(f,g,h)$.
 So $(s,t,0)$ is the desired natural contracting homotopy.
\end{proof}

 Let $H^*$ be a graded abelian group.
 As a particular case of the notation in
Proposition~\ref{G-plus-minus-prop} and 
Example~\ref{complexes-bec-example}, we denote by $G^-(H^*)$
the contractible complex of abelian groups with the grading
components $G^-(H^*)^n=H^n\oplus H^{n+1}$ and the differential
$d_{G^-,n}\:H^n\oplus H^{n+1}\rarrow H^{n+1}\oplus H^{n+2}$
whose only nonzero component is the identity map
$H^{n+1}\rarrow H^{n+1}$.
 Notice that there is an obvious natural split epimorphism of graded
abelian groups $G^-(H^*)\rarrow H^*$.

 Recall the notation $\Hom_{\cZ(\bB)}^*(X,Y)$ for the graded abelian
group of cocycles in the complex of morphisms $\Hom_\bB^\bu(X,Y)$
between two objects $X$ and $Y$ in a DG\+category~$\bB$.
 Given two objects $X^\bec$ and $Y^\bec\in\bA^\bec$, denote by
$\Hom_{\cZ(\bA^\bec)}^{-*}(X^\bec,Y^\bec)$ the graded abelian group
$\Hom_{\cZ(\bA^\bec)}^*(X^\bec,Y^\bec)$ with the sign of the grading
inverted, i.~e., the component of degree~$n$ in the former graded
abelian group is the component of degree~$-n$ in the latter one.

\begin{lem} \label{Hom-complex-from-Psi-computed}
 For any DG\+category\/ $\bA$ with shifts and cones, and any objects
$X^\bec\in\bA^\bec$ and $A\in\bA$, there is a natural isomorphism of
complexes of abelian groups
\begin{equation} \label{Hom-complex-from-Psi-eqn}
 \Hom_\bA^\bu(\Psi^+_\bA(X^\bec),A)\simeq
 G^-(\Hom_{\cZ(\bA^\bec)}^{-*}(X^\bec,\Phi_\bA(A))).
\end{equation}
\end{lem}

\begin{proof}
 Recall the notation $X^\bec=(X,\sigma_X)$ for a generic object
in~$\bA^\bec$ (see Section~\ref{main-construction-subsecn}).
 By the definition, one has $\Psi^+_\bA(X^\bec)=X$.

 The object $E^\bec=\Phi_\bA(A)\in\bA^\bec$ is constructed as
$E^\bec=(E,\sigma_E)$ with $E=\cone(\id_A[-1])$ and $\sigma_E=
\iota'\pi'$ (see Lemma~\ref{G-functors-in-DG-categories}).
 For any object $B\in\bA$, the complex of morphisms $\Hom_\bA^\bu(B,E)$
is naturally isomorphic to the cone of the identity endomorphism of
the complex $\Hom_\bA^\bu(B,A)[-1]$.
 In other words, the elements of the abelian group $\Hom_\bA^n(B,E)$
are pairs $(g_{n-1},g_n)$ with the $g_{n-1}\in\Hom_\bA^{n-1}(B,A)$
and $g_n\in\Hom_\bA^n(B,A)$, and the differential $d\:\Hom_\bA^n(B,E)
\rarrow\Hom_\bA^{n+1}(B,A)$ is given by the formula $d(g_{n-1},g_n)
=(-dg_{n-1}+g_n,dg_n)$.

 The subgroup $\Hom^{-n}_{\cZ(\bA^\bec)}(X^\bec,E^\bec)\subset
\Hom^n_\bA(X,E)$ consists of all the pairs $g=(g_{n-1},g_n)$ with
$g_{n-1}\in\Hom_\bA^{n-1}(X,A)$ and $g_n\in\Hom_\bA^n(X,A)$
satisfying the equations $dg=0$ and $d^\bec g=
\sigma_Eg-(-1)^ng\sigma_X=0$.
 Explicitly, we have $\sigma_E(g_{n-1},g_n)=(0,g_{n-1})$ and
$(g_{n-1},g_n)\sigma_X=(g_{n-1}\sigma_X,g_n\sigma_X)$;
so the equation $\sigma_Eg=(-1)^ng\sigma_X$ means that
$g_{n-1}\sigma_X=0$ and $g_n\sigma_X=(-1)^ng_{n-1}$.
 Assuming that $dg=0$, that is $g_n=dg_{n-1}$, we have
$d(g_{n-1}\sigma_X)=g_n\sigma_X+(-1)^{n-1}g_{n-1}d(\sigma_X)=
g_n\sigma_X+(-1)^{n-1}g_{n-1}$; so the equation $g_{n-1}\sigma_X=0$
implies $g_n\sigma_X=(-1)^ng_{n-1}$.
 We have computed the group $\Hom^{-n}_{\cZ(\bA^\bec)}(X^\bec,E^\bec)$
as the subgroup in $\Hom^{n-1}_\bA(X,A)$ consisting of all the elements
$g_{n-1}$ such that $g_{n-1}\sigma_X=0$.

 Now the grading component
$G^-(\Hom^{-*}_{\cZ(\bA^\bec)}(X^\bec,E^\bec))^n$ is the group
of all pairs $(g,h)$ with $g\in\Hom^{n-1}(X,A)$ and $h\in\Hom^n(X,A)$
such that $g\sigma_X=0=h\sigma_X$.
 The differential on the complex
$G^-(\Hom^{-n}_{\cZ(\bA^\bec)}(X^\bec,E^\bec))$ is given by the rule
$d_{G^-}(g,h)=(h,0)$.
 Finally, the desired isomorphism $\Hom_\bA^\bu(X,A)\rarrow
G^-(\Hom_{\cZ(\bA^\bec)}^{-*}(X^\bec,E^\bec))$ is provided by
the map taking an element $f\in\Hom_\bA^n(X,A)$ to the pair $(g,h)$
with $g=f\sigma_X$ and $h=d(f)\sigma_X$.
 The inverse map takes a pair $(g,h)$ to the element
$f=(-1)^ndg+(-1)^{n+1}h$.
 One easily computes that the two maps are mutually inverse and
commute with the differentials.
\end{proof}

 Notice that, by Lemmas~\ref{Phi-Psi-extended-to-nonclosed-morphisms}
and~\ref{twists-into-isomorphisms}, in the context of
Lemma~\ref{Hom-complex-from-Psi-computed} there is
a natural isomorphism of graded abelian groups
\begin{equation} \label{Hom-group-from-Psi-computed-via-tilde-Phi}
 \Hom_\bA^*(\Psi^+_\bA(X^\bec),A)\simeq
 \Hom_{\cZ(\bA^\bec)}^{-*}(\Phi_\bA\Psi^+_\bA(X^\bec),\Phi_\bA(A)).
\end{equation}
 The composition with the adjunction morphism $X^\bec\rarrow
\Phi_\bA\Psi^+_\bA(X^\bec)$ provides a map of graded abelian groups
$$
 \Hom_{\cZ(\bA^\bec)}^{-*}(\Phi_\bA\Psi^+_\bA(X^\bec),\Phi_\bA(A))
 \lrarrow\Hom_{\cZ(\bA^\bec)}^{-*}(X^\bec,\Phi_\bA(A)).
$$
 On the other hand, we have the natural split epimorphism of graded
abelian groups
$$
 G^-(\Hom_{\cZ(\bA^\bec)}^{-*}(X^\bec,\Phi_\bA(A)))\lrarrow
 \Hom_{\cZ(\bA^\bec)}^{-*}(X^\bec,\Phi_\bA(A)).
$$
 The two maps of graded abelian groups above form a commutative
rhombus diagram
\begin{equation} \label{rhombus-diagram}
\begin{gathered}
 \xymatrix@C-=-0.6cm{
 & \Hom_\bA^*(\Psi^+_\bA(X^\bec),A) \ar@{=}[ld] \ar@{=}[rd] \\
 \Hom_{\cZ(\bA^\bec)}^{-*}(\Phi_\bA\Psi^+_\bA(X^\bec),\Phi_\bA(A))
 \ar[rd] &&
 G^-(\Hom_{\cZ(\bA^\bec)}^{-*}(X^\bec,\Phi_\bA(A))) \ar[ld] \\
 & \Hom_{\cZ(\bA^\bec)}^{-*}(X^\bec,\Phi_\bA(A))}
\end{gathered}
\end{equation}
with the isomorphisms~\eqref{Hom-complex-from-Psi-eqn}
and~\eqref{Hom-group-from-Psi-computed-via-tilde-Phi}.

\begin{lem} \label{liftability-preserved-by-Phi-Psi-adjunction}
 Let $T$ be the total object of a three-term complex $U\rarrow V
\rarrow W$ in the DG\+category\/ $\bA$, and let $X^\bec$ be
an object of the DG\+category\/~$\bA^\bec$.
 Let $p\:X^\bec\rarrow\Phi(T)\simeq\Phi(U[1]\oplus V\oplus W[-1])$ be
a closed morphism of degree~$-n$ in the DG\+category\/ $\bA^\bec$ with
the components $(f,g,h)$, so $f\:X^\bec\rarrow\Phi(U)$, \
$g\:X^\bec\rarrow\Phi(V)$, and $h\:X^\bec\rarrow\Phi(W)$ are closed 
morphisms of degrees $-n-1$, $-n$, and~$-n+1$, respectively.
 Let $\tilde p\:\Psi^+(X^\bec)\rarrow T$ be the closed morphism of
degree~$n$ in the DG\+category\/ $\bA$ corresponding by adjunction
to~$p$, and let $(\tilde f,\tilde g,\tilde h)$ be the three components
of~$\tilde p$, as in Lemma~\ref{Hom-complex-into-Tot-lemma}.
 Assume that the component $h\:X^\bec\rarrow\Phi(W)$ can be lifted
to a closed morphism $t\:X^\bec\rarrow\Phi(V)$ of degree~$-n+1$
in\/~$\bA^\bec$.
 Then the morphism $\tilde h\in\Hom^{n-1}_\bA(\Psi^+(X^\bec),W)$ can be
lifted to some morphism $\tilde t\in\Hom^{n-1}_\bA(\Psi^+(X^\bec),V)$.
\end{lem}

\begin{proof}
 This is our version of~\cite[Lemma~1.5.F]{EP}.
 The assumptions of the lemma only require $U\rarrow V\rarrow W$ to be
a complex in a DG\+category, but in all the applications of this lemma
below $0\rarrow U\rarrow V\rarrow W\rarrow0$ will be a short exact
sequence; in particular, the morphism $U\rarrow V$ will be the kernel
of the morphism $V\rarrow W$ in the additive category $\sZ^0(\bA)$.
 In this context, by Lemma~\ref{Hom-complex-into-Tot-lemma}(b),
the assertion of
Lemma~\ref{liftability-preserved-by-Phi-Psi-adjunction} implies
that the morphism $\tilde p\:\Psi^+(X^\bec)\rarrow T$ is homotopic to
zero in the DG\+category~$\bA$.
 This is trivial, however: \emph{any} closed morphism from
$\Psi^+(X^\bec)$ is homotopic to zero in $\bA$, as the object
$X=\Psi^+(X^\bec)\in\bA$ is contractible by definition.
 The preservation of liftability by adjunction claimed in the lemma
is stronger than this trivial observation, and we will need its full
strength.

 Denote by $B^*$, $C^*$, and $D^*$ the graded abelian groups
$\Hom_{\cZ(\bA^\bec)}^{-*}(X^\bec,\Phi(U))$,
$\Hom_{\cZ(\bA^\bec)}^{-*}(X^\bec,\Phi(V))$, and
$\Hom_{\cZ(\bA^\bec)}^{-*}(X^\bec,\Phi(W))$, respectively.
 Then we have a three-term complex of graded abelian groups
$B^*\rarrow C^*\rarrow D^*$.
 Applying the functor $G^-$ produces a bicomplex with three rows
$G^-(B^*)\rarrow G^-(C^*)\rarrow G^-(D^*)$.
 Lemma~\ref{Hom-complex-from-Psi-computed} implies a natural
isomorphism of complexes of abelian groups
$$
 \Hom_\bA^\bu(\Psi^+(X^\bec),T)\simeq
 \Tot(G^-(B^*)\to G^-(C^*)\to G^-(D^*)).
$$
 Using this isomorphism, the cocycle
$\tilde p\in\Hom_\bA^\bu(\Psi^+(X^\bec),T)$ can be viewed as
an element of the abelian group $G^-(B^*)^{n+1}\oplus G^-(C^*)^n
\oplus G^-(D^*)^{n-1}$.

 The elements $\tilde f\in G^-(B^*)^{n+1}$, \ $\tilde g\in
G^-(C^*)^n$, and $\tilde h\in G^-(D^*)^{n-1}$ are the respective
projections of the direct sum element~$\tilde p$ to
the direct summands.
 Furthermore, the elements $f\in B^{n+1}$, \ $g\in C^n$, and
$h\in D^{n-1}$ are the respective projections of $\tilde f$,
$\tilde g$, and~$\tilde h$ onto the first direct summands in
$G^-(B^*)^{n+1}=B^{n+1}\oplus B^{n+2}$, \ $G^-(C^*)^n=
C^n\oplus C^{n+1}$, and $G^-(D^*)^{n-1}=D^{n-1}\oplus D^n$.
 This is essentially explained in the discussion of the commutativity 
of the rhombus diagram~\eqref{rhombus-diagram} in the paragraph
preceding the lemma.

$$
 \xymatrix{
  B^{n-1}\oplus B^n \ar[r] \ar[d] 
  & C^{n-1}\oplus C^n \ar[r] \ar[d]
  & D^{n-1}\oplus D^n \ar[d] \\
  B^n\oplus B^{n+1} \ar[r] \ar[d]
  & C^n\oplus C^{n+1} \ar[r] \ar[d]
  & D^n\oplus D^{n+1} \ar[d] \\
  B^{n+1}\oplus B^{n+2} \ar[r]
  & C^{n+1}\oplus C^{n+2} \ar[r]
  & D^{n+1}\oplus D^{n+2}
 }
$$

 Now the cocycle equation for the element~$\tilde p$ in the total
complex implies that the $D^n$\+component of~$\tilde h$ can be lifted
to the element $g\in C^n$.
 Thus the element $\tilde h\in D^{n-1}\oplus D^n$ is liftable to
$C^{n-1}\oplus C^n$ whenever the element $h\in D^{n-1}$ is liftable
to~$C^{n-1}$.
\end{proof}

\subsection{Full-and-faithfulness theorem}
\label{full-and-faithfulness-subsecn}
 The idea of the following theorem, which is the main result of
Section~\ref{full-and-faithfulness-secn}, goes back
to~\cite[Proposition~1.5(a\+-c) and Remark~1.5]{EP}
and~\cite[Proposition~A.2.1]{Pcosh}
(see also~\cite[Theorem~5.5]{PSch}).

\begin{thm} \label{full-and-faithfulness-main-theorem}
 Let $(\bE,\sK)$ be an exact DG\+pair, and let $(\bF,\sL)$, $(\bG,\sM)
\subset(\bE,\sK)$ be exact DG\+subpairs. \par
\textup{(a)} Assume that the full subcategory\/ $\sL$ is self-resolving
in the exact category\/~$\sK$.
 Then the triangulated functor
$$
 \sD^\abs(\bF)\lrarrow\sD^\abs(\bE)
$$
induced by the inclusion of exact DG\+categories\/ $\bF\rightarrowtail
\bE$ is fully faithful.

 Furthermore, if the exact DG\+category\/ $\bE$ has exact products
and the full DG\+sub\-cat\-e\-gory\/ $\bF\subset\bE$ is closed under
products,
then the triangulated functor
$$
 \sD^\ctr(\bF)\lrarrow\sD^\ctr(\bE)
$$
induced by the inclusion of exact DG\+categories\/ $\bF\rightarrowtail
\bE$ is fully faithful. \par
\textup{(b)} Assume that the full subcategory\/ $\sM$ is
self-coresolving in the exact category\/~$\sK$.
 Then the triangulated functor
$$
 \sD^\abs(\bG)\lrarrow\sD^\abs(\bE)
$$
induced by the inclusion of exact DG\+categories\/ $\bG\rightarrowtail
\bE$ is fully faithful.

 Furthermore, if the exact DG\+category\/ $\bE$ has exact coproducts
and the full DG\+sub\-cat\-e\-gory\/ $\bG\subset\bE$ is closed under
coproducts,
then the triangulated functor
$$
 \sD^\co(\bG)\lrarrow\sD^\co(\bE)
$$
induced by the inclusion of exact DG\+categories\/ $\bG\rightarrowtail
\bE$ is fully faithful.
\end{thm}

 The proof of the theorem is based on the following proposition.
 Given an exact DG\+category $\bA$, we denote by $\Ac^0(\bA)\subset
\sH^0(\bA)$ the class of all totalizations of short exact sequences
in $\sZ^0(\bA)$.
 So $\Ac^\abs(\bA)$ is the thick subcategory generated by the full
subcategory $\Ac^0(\bA)$ in the triangulated category $\sH^0(\bA)$.
 We use the terminology ``approachable objects'' defined in
Section~\ref{approachability-subsecn}.

\begin{prop} \label{totalizations-approachable-prop}
 Let $(\bF,\sL)$ and $(\bG,\sM)$ be an exact DG\+subpairs in
an exact DG\+pair $(\bE,\sK)$. \par
\textup{(a)} Assume that the full subcategory\/ $\sL$ is self-resolving
in the exact category\/~$\sK$.
 Then all the objects of\/ $\Ac^0(\bE)$ are approachable from\/
$\sH^0(\bF)$ via\/ $\Ac^0(\bF)$ in the triangulated category\/
$\sH^0(\bE)$. \par
\textup{(b)} Assume that the full subcategory\/ $\sM$ is
self-coresolving in the exact category\/~$\sK$.
 Then all the objects of\/ $\Ac^0(\bE)$ are coapproachable to\/
$\sH^0(\bG)$ via\/ $\Ac^0(\bG)$ in the triangulated category\/
$\sH^0(\bE)$.
\end{prop}

\begin{proof}
 We will prove part~(a) (part~(b) is dual).
 Let $0\rarrow U\rarrow V\rarrow W\rarrow0$ be a short exact sequence
in $\sZ^0(\bE)$ and $T$ be its total object; so $T\in\Ac^0(\bE)$.
 Let $F\rarrow T$ be a closed morphism of degree~$0$ in $\bE$ from
an object $F\in\bF$ to the object~$T$.
 We have to show that there exists an object $G\in\bF$ and a morphism
$G\rarrow F$ in $\sH^0(\bF)$ with a cone belonging to $\Ac^0(\bF)$
such that the composition $G\rarrow F\rarrow T$ vanishes in
$\sH^0(\bE)$.
 Following the discussion in the beginning of
Section~\ref{lemmas-E-and-F-subsecn}, the morphism $F\rarrow T$ is
represented by a triple of (not necessarily closed) morphisms
$f\in\Hom_\bE^1(F,U)$, \ $g\in\Hom_\bE^0(F,V)$, and
$h\in\Hom_\bE^{-1}(F,W)$ in the DG\+category~$\bE$.

 Viewing $h$ as a morphism $F\rarrow W[-1]$ in the additive category
$\bE^0$ and applying the functor $\widetilde\Phi_\bE$ from
Lemma~\ref{Phi-Psi-extended-to-nonclosed-morphisms}, we produce
a morphism $\Phi(F)\rarrow\Phi(W[-1])$ in the exact category $\sK$
with the object $\Phi(F)\in\sL$.
 Applying the functor $\Phi_\bE$ to the admissible epimorphism
$V[-1]\rarrow W[-1]$ in the exact category $\sZ^0(\bE)$, we obtain
an admissible epimorphism $\Phi(V[-1])\rarrow\Phi(W[-1])$ in~$\sK$.
 Let $K$ be the related pullback object; so we have the commutative
square part of a diagram in~$\sK$
$$
 \xymatrix{
  L \ar[rr] \ar@{->>}[rrd] && K \ar[rr] \ar@{->>}[d]
  && \Phi(V[-1]) \ar@{->>}[d] \\
  && \Phi(F) \ar[rr]^-{\widetilde\Phi(h)} && \Phi(W[-1])
 }
$$
with an admissible epimorphism $K\rarrow\Phi(F)$.
 By assumption, there exists an object $L\in\sL$ together with
an admissible epimorphism $L\rarrow\Phi(F)$ in $\sL$ and
a morphism $L\rarrow K$ in $\sK$ such that the triangle diagram
$L\rarrow K\rarrow\Phi(F)$ is commutative.

 By Lemma~\ref{mono-epi-exact-DG-adjunction}(b), the morphism
$\Psi^+(L)\rarrow F$ in $\sZ^0(\bF)$ corresponding by adjunction
to the admissible epimorphism $L\rarrow\Phi(F)$ in $\sL$ is
also an admissible epimorphism.
 Let $R\in\sZ^0(\bF)$ denote the kernel of $\Psi^+(L)\rarrow F$;
so we have a short exact sequence $0\rarrow R\rarrow\Psi^+(L)
\rarrow F\rarrow0$ in $\sZ^0(\bF)$.
 Denote by $G$ the cone of the closed morphism $R\rarrow\Psi^+(L)$
in the DG\+category~$\bF$.
 Then we have a natural closed morphism $G\rarrow F$ in $\bF$
whose cone $\Tot(R\to\Psi^+(L)\to F)$ belongs to $\Ac^0(\bF)$.

 It remains to show that the composition $G\rarrow F\rarrow T$
is homotopic to zero in the DG\+category~$\bE$.
 For this purpose, we apply Lemma~\ref{Hom-complex-into-Tot-lemma},
which tells that it suffices to show that the component
$\tilde l'\in\Hom_\bE^0(G,W[-1])$ of our morphism $G\rarrow T$
can be lifted to an element of $\Hom_\bE^0(G,V[-1])$.
 In the category $\bE^0$, the object $G$ is naturally isomorphic
to the direct sum $R[1]\oplus\Psi^+(L)$, and the morphism~$\tilde l'$
(and in fact, the whole morphism $G\rarrow T$, and indeed the morphism
$G\rarrow F$) vanishes on the component~$R[1]$.
 Thus it suffices to show that the related element $\tilde l\in
\Hom_\bE^0(\Psi^+(L),W[-1])$ can be lifted to an element of
$\Hom_\bE^0(\Psi^+(L),V[-1])$.

 The element~$\tilde l$ is the $W$\+component of the composition
$\Psi^+(L)\rarrow F\rarrow T$, which is a closed morphism in~$\bE$.
 According to Lemma~\ref{liftability-preserved-by-Phi-Psi-adjunction},
we only need to check that the $\Phi(W[-1])$\+component~$l$ of
the composition $L\rarrow\Phi\Psi^+(L)\rarrow\Phi(F)\rarrow\Phi(T)$
can be lifted to a morphism $L\rarrow\Phi(V[-1])$ in $\sZ^0(\bE^\bec)$.
 But we are given such a lifting by construction; see the commutative
diagram above.
\end{proof}

\begin{proof}[Proof of Theorem~\ref{full-and-faithfulness-main-theorem}]
 Let us prove part~(a) (part~(b) is dual).
 By Lemma~\ref{approachable-implies-fully-faithful-lemma}, in order to
prove that the triangulated functor $\sD^\abs(\bF)\lrarrow\sD^\abs(\bE)$
is fully faithful, it suffices to show that all the objects of the full
subcategory $\Ac^\abs(\bE)$ are approachable from $\sH^0(\bF)$ via
$\Ac^\abs(\bF)$ in $\sH^0(\bE)$.

 Indeed, by Proposition~\ref{totalizations-approachable-prop}(a),
all the objects of $\Ac^0(\bE)$ are approachable from $\sH^0(\bF)$
via $\Ac^0(\bF)$, hence also via $\Ac^\abs(\bF)$.
 By Lemmas~\ref{summands-approachable-lemma}
and~\ref{shifts-cones-approachable-lemma}, the class of all objects
approachable from $\sH^0(\bF)$ via $\Ac^\abs(\bF)$ is a thick
subcategory in $\sH^0(\bE)$.
 It follows that thick subcategory contains $\Ac^\abs(\bE)$, and
we are done.

 Similarly, in order to prove that the triangulated functor
$\sD^\ctr(\bF)\lrarrow\sD^\ctr(\bE)$ is fully faithful, it suffices to
show that all the objects of the full subcategory $\Ac^\ctr(\bE)$ are
approachable from $\sH^0(\bF)$ via $\Ac^\ctr(\bF)$ in $\sH^0(\bE)$.
 By Lemmas~\ref{co-products-approachable-lemma}(a)
and~\ref{shifts-cones-approachable-lemma}, the class of all objects
approachable from $\sH^0(\bF)$ via $\Ac^\ctr(\bF)$ is a full
triangulated subcategory closed under infinite products in $\sH^0(\bE)$.
 It follows that this triangulated subcategory contains $\Ac^\ctr(\bE)$.
\end{proof}

\subsection{Triangulated equivalence for a resolving subcategory}
 The idea of the following result goes back to~\cite[Remark~1.5]{EP},
\cite[Theorem~4.2.1]{Pweak}, \cite[Proposition~A.3.1(b)]{Pcosh},
and~\cite[Proposition~2.1]{Pfp}.

\begin{thm} \label{contraderived-resolving-equivalence-theorem}
 Let $(\bE,\sK)$ be an exact DG\+pair, and let $(\bF,\sL)$, $(\bG,\sM)
\subset(\bE,\sK)$ be exact DG\+subpairs. \par
\textup{(a)} Assume that the full subcategory\/ $\sL$ is resolving
in the exact category\/~$\sK$.
 Assume further that the exact DG\+category\/ $\bE$ has twists and
exact products, and the full DG\+sub\-cat\-e\-gory\/ $\bF\subset\bE$
is closed under twists and products.
 Then the triangulated functor
$$
 \sD^\ctr(\bF)\lrarrow\sD^\ctr(\bE)
$$
induced by the inclusion of exact DG\+categories\/ $\bF\rightarrowtail
\bE$ is an equivalence of triangulated categories. \par
\textup{(b)} Assume that the full subcategory\/ $\sM$ is
coresolving in the exact category\/~$\sK$.
 Assume further that the exact DG\+category\/ $\bE$ has twists and
exact coproducts, and the full DG\+sub\-cat\-e\-gory\/ $\bG\subset\bE$
is closed under twists and coproducts.
 Then the triangulated functor
$$
 \sD^\co(\bG)\lrarrow\sD^\co(\bE)
$$
induced by the inclusion of exact DG\+categories\/ $\bG\rightarrowtail
\bE$ is an equivalence of triangulated categories.
\end{thm}

\begin{proof}
 Let us prove part~(a) (part~(b) is dual).
 By Theorem~\ref{full-and-faithfulness-main-theorem}(a),
the triangulated functor $\sD^\ctr(\bF)\rarrow\sD^\ctr(\bE)$ is fully
faithful; so we only need to show that it is essentially surjective.
 For this purpose, we will construct for any object $E\in\bE$
an object $F\in\bF$ together with a closed morphism $F\rarrow E$
of degree~$0$ in $\bE$ whose cone belongs to $\Ac^\ctr(\bE)$.

 By Lemma~\ref{one-step-exact-subcategory-resolution-lemma}(a),
there exists a short exact sequence $0\rarrow E_1\rarrow F_0\rarrow
E\rarrow0$ in $\sZ^0(\bE)$ with an object $F_0\in\bF$.
 Applying the same lemma to the object $E_1\in\bE$, we obtain
a short exact sequence $0\rarrow E_2\rarrow F_1\rarrow E_1\rarrow0$
in $\sZ^0(\bE)$ with an object $F_1\in\bF$.
 Proceeding in this way, we construct an exact complex
$\dotsb\rarrow F_2\rarrow F_1\rarrow F_0\rarrow E\rarrow0$ in
$\sZ^0(\bE)$ with objects $F_i\in\sZ^0(\bF)$.

 Denote by $F=\Tot^\sqcap(F_\bu)$ the product totalization of
the complex $\dotsb\rarrow F_2\rarrow F_1\rarrow F_0\rarrow0$ in
the DG\+category~$\bE$.
 Then the object $F\in\bE$ is a twist of the object
$\prod_{i=0}^\infty F_i[i]$, hence $F\in\bF$ by assumptions.
 The cone of the natural morphism $F\rarrow E$ in $\bE$ is the product
totalization of the bounded above exact complex
$\dotsb\rarrow F_2\rarrow F_1\rarrow F_0\rarrow E\rarrow0$
in $\sZ^0(\bE)$, so it is contraacyclic by
Lemma~\ref{totalization-of-one-sided-bounded-co-contra-acyclic}(b).
\end{proof}

\subsection{Examples}
 Let us formulate here the particular cases of
Theorems~\ref{full-and-faithfulness-main-theorem}
and~\ref{contraderived-resolving-equivalence-theorem} arising in
the context of Examples~\ref{exact-dg-category-of-complexes}\+-%
\ref{exact-dg-category-of-factorizations}
and~\ref{exact-DG-pairs-examples}.

 We refer to Section~\ref{finite-resol-dim-examples-subsecn} for
the definitions of the absolute derived, coderived, and contraderived
categories of an exact category~$\sE$.
 The following result can be found in~\cite[Propositions~A.2.1
and~A.3.1(b)]{Pcosh}.

\begin{cor} \label{complexes-full-and-faithfulness-corollary}
 Let\/ $\sE$ be an exact category and\/ $\sF$ be
a self-resolving subcategory in\/~$\sE$.
 Then the triangulated functor
$$
 \sD^\abs(\sF)\lrarrow\sD^\abs(\sE)
$$
induced by the inclusion of exact categories\/ $\sF\rightarrowtail\sE$
is fully faithful.

 Furthermore, assume that the exact category\/ $\sE$ has exact products
and the full subcategory\/ $\sF\subset\sE$ is closed under products.
 Then the triangulated functor
$$
 \sD^\ctr(\sF)\lrarrow\sD^\ctr(\sE)
$$
induced by the inclusion of exact categories\/ $\sF\rightarrowtail\sE$
is fully faithful.
 Assuming additionally that\/ $\sF$ is a resolving subcategory in\/
$\sE$, the latter functor is a triangulated equivalence.
\end{cor}

\begin{proof}
 Put $\bE=\bC(\sE)$, \,$\bF=\bC(\sF)$, \,$\sK=\sG(\sE)$,
\,$\sL=\sG(\sF)$, as in Examples~\ref{exact-dg-category-of-complexes}
and~\ref{exact-DG-pairs-examples}(4), and apply
Theorems~\ref{full-and-faithfulness-main-theorem}(a)
(for the first two assertions of the corollary)
and~\ref{contraderived-resolving-equivalence-theorem}(a)
(for the third one).
\end{proof}

 Part~(a) of the following corollary is a generalization of
an assertion from~\cite[Remark~1.5]{EP}.

\begin{cor} \label{cdg-ring-full-and-faithfulness-corollary}
 Let $\biR^\cu=(R^*,d,h)$ be a CDG\+ring.
 Let\/ $\sK\subset R^*\smodl$ be a full subcategory as in
the first paragraph of
Corollary~\ref{cdg-ring-finite-resol-dim-corollary}, and let\/
$\bE\subset\biR^\cu\bmodl$ the related full DG\+subcategory
endowed with the inherited exact DG\+category structure. \par
 \textup{(a)} Let\/ $\sL$ be a self-resolving subcategory in\/~$\sK$.
 Denote by\/ $\bF\subset\bE$ the full DG\+sub\-cat\-e\-gory consisting
of all the CDG\+modules whose underlying graded $R^*$\+modules
belong to\/ $\sL$, and endow\/ $\bF$ with the inherited exact
DG\+category structure.
 Then the triangulated functor
$$
 \sD^\abs(\bF)\lrarrow\sD^\abs(\bE)
$$
induced by the inclusion of exact DG\+categories\/ $\bF\rightarrowtail
\bE$ is fully faithful.

 Furthermore, assume that the full subcategories\/ $\sL$ and\/ $\sK$
are preserved by the infinite products in $R^*\smodl$.
 Then the triangulated functor
$$
 \sD^\ctr(\bF)\lrarrow\sD^\ctr(\bE)
$$
induced by the inclusion of exact DG\+categories\/ $\bF\rightarrowtail
\bE$ is fully faithful.
 Assuming additionally that\/ $\sL$ is a resolving subcategory in\/
$\sK$, the latter functor is a triangulated equivalence. \par
 \textup{(b)} Let\/ $\sM$ be a self-coresolving subcategory in\/~$\sK$.
 Denote by\/ $\bG\subset\bE$ the full DG\+subcategory consisting
of all the CDG\+modules whose underlying graded $R^*$\+modules
belong to\/ $\sM$, and endow\/ $\bG$ with the inherited exact
DG\+category structure.
 Then the triangulated functor
$$
 \sD^\abs(\bG)\lrarrow\sD^\abs(\bE)
$$
induced by the inclusion of exact DG\+categories\/ $\bG\rightarrowtail
\bE$ is fully faithful.

 Furthermore, assume that the full subcategories\/ $\sM$ and\/ $\sK$
are preserved by the infinite direct sums in $R^*\smodl$.
 Then the triangulated functor
$$
 \sD^\co(\bG)\lrarrow\sD^\co(\bE)
$$
induced by the inclusion of exact DG\+categories\/ $\bG\rightarrowtail
\bE$ is fully faithful.
 Assuming additionally that\/ $\sM$ is a coresolving subcategory in\/
$\sK$, the latter functor is a triangulated equivalence. \par
\end{cor}

\begin{proof}
 Similar to the proof of
Corollary~\ref{cdg-ring-finite-resol-dim-corollary}.
\end{proof}

 Part~(a) of the next corollary is a generalization
of~\cite[Proposition~1.5(a\+-c)]{EP}.

\begin{cor} \label{cdg-quasi-algebra-full-and-faithfulness-corollary}
 Let $X$ be a scheme and $\biB^\cu=(B^*,d,h)$ be a quasi-coherent
CDG\+quasi-algebra over~$X$.
 Let $\sK\subset B^*\sqcoh$ be a full subcategory as in
the first paragraph of
Corollary~\ref{cdg-quasi-algebra-finite-resol-dim-corollary},
and let\/ $\bE\subset\biB^\cu\bqcoh$ the related full DG\+subcategory
endowed with the inherited exact DG\+category structure. \par
 \textup{(a)} Let\/ $\sL$ be a self-resolving subcategory in\/~$\sK$.
 Denote by\/ $\bF\subset\bE$ the full DG\+subcategory consisting
of all the quasi-coherent CDG\+modules whose underlying quasi-coherent
graded $B^*$\+modules belong to\/ $\sL$, and endow\/ $\bF$ with
the inherited exact DG\+category structure.
 Then the triangulated functor
$$
 \sD^\abs(\bF)\lrarrow\sD^\abs(\bE)
$$
induced by the inclusion of exact DG\+categories\/ $\bF\rightarrowtail
\bE$ is fully faithful. \par
 \textup{(b)} Let\/ $\sM$ be a self-coresolving subcategory in\/~$\sK$.
 Denote by\/ $\bG\subset\bE$ the full DG\+subcategory consisting
of all the quasi-coherent CDG\+modules whose underlying quasi-coherent
graded $B^*$\+modules belong to\/ $\sM$, and endow\/ $\bG$ with
the inherited exact DG\+category structure.
 Then the triangulated functor
$$
 \sD^\abs(\bG)\lrarrow\sD^\abs(\bE)
$$
induced by the inclusion of exact DG\+categories\/ $\bG\rightarrowtail
\bE$ is fully faithful.

 Furthermore, assume that the full subcategories\/ $\sM$ and\/ $\sK$
are preserved by the infinite direct sums in $B^*\sqcoh$.
 Then the triangulated functor
$$
 \sD^\co(\bG)\lrarrow\sD^\co(\bE)
$$
induced by the inclusion of exact DG\+categories\/ $\bG\rightarrowtail
\bE$ is fully faithful.
  Assuming additionally that\/ $\sM$ is a coresolving subcategory in\/
$\sK$, the latter functor is a triangulated equivalence.
\end{cor}

\begin{proof}
 Similar to the proof of
Corollary~\ref{cdg-quasi-algebra-finite-resol-dim-corollary}.
\end{proof}

 Let us point out once again that assertions about the contraderived
categories are absent in
Corollary~\ref{cdg-quasi-algebra-full-and-faithfulness-corollary}(a),
because the functors of infinite product of quasi-coherent sheaves
are not exact.
 If one wants to consider contraderived categories of module objects
over a quasi-coherent CDG\+quasi-algebra over a nonaffine scheme,
then one should work with contraherent cosheaves~\cite{Pcosh} instead
of quasi-coherent sheaves.

 The first assertion of the following corollary is a generalization
of~\cite[Corollary~2.3(h\+-k)]{EP}.

\begin{cor} \label{factorizations-full-and-faithfulness-corollary}
 Let\/ $\sE$ be an exact category and\/ $\Lambda\:\sE\rarrow\sE$ be
an autoequivalence preserving and reflecting short exact sequences.
 Assume that either\/ $\Gamma=\boZ$, or\/ $\Lambda$ is involutive
and\/ $\Gamma=\boZ/2$.
 Let $w\:\Id_\sE\rarrow\Lambda^2$ be a potential (as in
Section~\ref{factorizations-subsecn}).
 Let\/ $\sH$ be a self-resolving subcategory in\/ $\sE$ preserved by
the autoequivalences\/ $\Lambda$ and\/~$\Lambda^{-1}$.
 Then the triangulated functor
$$
 \sD^\abs(\bF(\sH,\Lambda,w))\lrarrow\sD^\abs(\bF(\sE,\Lambda,w))
$$
induced by the inclusion of exact categories\/ $\sH\rightarrowtail\sE$
is fully faithful.

 Furthermore, if the exact category\/ $\sE$ has exact products and
the full subcategory\/ $\sH\subset\sE$ is closed under products,
then the triangulated functor
$$
 \sD^\ctr(\bF(\sH,\Lambda,w))\lrarrow\sD^\ctr(\bF(\sE,\Lambda,w))
$$
induced by the inclusion of exact categories\/ $\sH\rightarrowtail\sE$
is fully faithful.
 Assuming additionally that\/ $\sH$ is a resolving subcategory in\/
$\sE$, the latter functor is a triangulated equivalence.
\end{cor}

\begin{proof}
 Put $\bE=\bF(\sE,\Lambda,w)$, \,$\bF=\bF(\sH,\Lambda,w)$,
\,$\sK=\sP(\sE,\Lambda)$, \,$\sL=\sP(\sH,\Lambda)$,
as in Examples~\ref{exact-dg-category-of-factorizations}
and~\ref{exact-DG-pairs-examples}(5), and apply
Theorems~\ref{full-and-faithfulness-main-theorem}(a)
(for the first two assertions of the corollary)
and~\ref{contraderived-resolving-equivalence-theorem}(a)
(for the third one).
\end{proof}

\Section{Finite Homological Dimension Theorem}
\label{finite-homol-dim-secn}

\subsection{Strongly generated thick subcategories}
\label{strongly-generated-thick-subsecn}
 Let $\sT$ be a triangulated category and $\sM$, $\sN\subset\sT$
be two classes of objects.
 The following concept and notation goes back to~\cite[1.3.9]{BBD}.

 One says that an object $X\in\sT$ is an \emph{extension} of two
objects $M$ and $N\in\sT$, if there exists a distinguished
triangle $M\rarrow X\rarrow N\rarrow M[1]$ in~$\sT$.
 The class of all objects $X\in\sT$ for which there exists
a distinguished triangle $M\rarrow X\rarrow N\rarrow M[1]$ with
$M\in\sM$ and $N\in\sN$ is denoted by $\sM*\sN\subset\sT$.
 According to~\cite[Lemme~1.3.10]{BBD}, the operation~$*$ on
the subclasses of objects in~$\sT$ is associative.

 In particular, the notation $\sM^{*n}$ for the class of objects
$\sM*\sM*\dotsb*\sM$ ($n$~factors) is unambiguous for any given
integer $n\ge1$.
 Assuming for simplicity that the class of objects $\sM\subset\sT$
is closed under the shifts $[1]$, $[-1]$ and finite direct sums, and
following~\cite[Section~2.2]{BvdB}, we will denote by
$\langle\sM\rangle_n\subset\sT$ the class of all direct summands
of objects from $\sM^{*n}$.

 Let $\sX$ be a thick subcategory in a triangulated category $\sT$,
and let $\sM\subset\sX$ be a class of objects.
 Following the terminology in~\cite[Section~2.2]{BvdB}, we will say
that the class $\sM$ \emph{strongly generates} the thick subcategory
$\sX\subset\sT$ if there exists a finite integer $n\ge1$ such that
$\sX=\langle\sM\rangle_n$.

\begin{lem} \label{strongly-generated-closed-under-co-products}
\textup{(a)} Let\/ $\sT$ be a triangulated category with infinite
coproducts and\/ $\sM$, $\sN\subset\sT$ be two classes of objects
closed under coproducts in\/~$\sT$.
 Then the classes of objects\/ $\sM*\sN$ and $\langle\sM\rangle_n$
are also closed under coproducts in\/ $\sT$ (for all $n\ge1$). \par
\textup{(b)} Let\/ $\sT$ be a triangulated category with infinite
products and\/ $\sM$, $\sN\subset\sT$ be two classes of objects
closed under products in\/~$\sT$.
 Then the classes of objects\/ $\sM*\sN$ and $\langle\sM\rangle_n$
are also closed under products in\/ $\sT$ (for all $n\ge1$).
\end{lem}

\begin{proof}
 It suffices to prove part~(b) (part~(a) is dual).
 The assertion about the class $\sM*\nobreak\sN$ holds due to
the fact that the products of distinguished trianges are
distinguished triangles~\cite[Proposition~1.2.1]{Neem2}.
 The assertion about the class $\langle\sM\rangle_n$ easily
follows, because products of direct summands are direct summands
of products.
\end{proof}

\subsection{Absolutely acyclic objects as totalizations of
finite exact complexes}
 The following proposition provides a converse assertion to
Lemma~\ref{totalization-of-finite-absolutely-acyclic}.
 We will use the notation $\Ac^0(\bE)\subset\Ac^\abs(\bE)\subset
\sH^0(\bE)$ from Sections~\ref{derived-second-kind-definitions-subsecn}
and~\ref{full-and-faithfulness-subsecn}.

\begin{prop} \label{abs-acyclic-as-totalizations}
 Let\/ $\bE$ be an exact DG\+category.
 Then for any absolutely acyclic object $X\in\Ac^\abs(\bE)$
there exists a finite exact complex $Y_\bu$ in the exact
category\/ $\sZ^0(\bE)$ such that $X$ is a direct summand of
the object\/ $\Tot(Y_\bu)$ in the homotopy category\/ $\sH^0(\bE)$.
 More precisely, if $n\ge1$ is an integer such that
$X\in\langle\Ac^0(\bE)\rangle_n\subset\sH^0(\bE)$, then
there exists an exact complex\/ $0\rarrow Y_{n+1}\rarrow Y_n\rarrow
\dotsb\rarrow Y_1\rarrow Y_0\rarrow0$ in\/ $\sZ^0(\bE)$ such that
$X$ is a direct summand of\/ $\Tot(Y_\bu)$ in\/ $\sH^0(\bE)$.
\end{prop}

 The proof of the proposition is based on the following lemma.

\begin{lem} \label{triangle-with-a-totalization-in-vertex-lemma}
 Let\/ $\bE$ be an exact DG\+category and\/ $0\rarrow U\rarrow V\rarrow
W\rarrow0$ be a short exact sequence in the exact category\/
$\sZ^0(\bE)$, and $T=\Tot(U\to V\to W)\in\bE$ be its totalization.
 Let $S\rarrow X\rarrow T\rarrow S[1]$ be a distinguished triangle
in the homotopy category\/ $\sH^0(\bE)$, let $P\in\bE$ be
an object, and let $P\rarrow X$ be a morphism in\/ $\sH^0(\bE)$.
 Then there exists a short exact sequence\/ $0\rarrow R\rarrow Y\rarrow
P\rarrow0$ in\/ $\sZ^0(\bE)$ and a morphism $R[1]\rarrow S$ in\/
$\sH^0(\bE)$ such that the following composition $C\rarrow P\rarrow X$
is equal to the composition $C\rarrow R[1]\rarrow S\rarrow X$ in
the homotopy category\/ $\sH^0(\bE)$, i.~e., the pentagon on
the right-hand side of the diagram is commutative up to homotopy:
$$
 \xymatrix{
  S[1] & T \ar[l] & & X \ar[ll] & S \ar[l] \\
  & & P \ar[ur] \\
  R \ar@{>..>}[r] & Y \ar@{..>>}[ru] \ar@{..>}[rr]
  & & C \ar@{..>}[ul] \ar@{..>}[r] & R[1] \ar@{..>}[uu]
 }
$$
 Here $C$ is the cone of the closed morphism $R\rarrow Y$ in\/ $\bE$,
and $C\rarrow P$ and $C\rarrow R[1]$ are the natural closed morphisms
of degree\/~$0$.
 So the triangle in the lower middle part of the diagram is
commutative in\/ $\sZ^0(\bE)$, while the bottom line is the standard
distinguished triangle in\/ $\sH^0(\bE)$.
\end{lem}

\begin{proof}
 This is a generalization of~\cite[Lemma~1.6.G]{EP}.
 Let us first briefly discuss why the complicated assertion of
the lemma has any chance to be true at all.
 Specifically, let us explain why the desired commutative diagram
exists \emph{in the absolute derived category\/ $\sD^\abs(\bE)$}.
 Indeed, in $\sD^\abs(\bE)$, the morphisms $S\rarrow X$ and
$C\rarrow P$ are isomorphisms.
 Choosing the object $Y\in\bE$ to be contractible, one makes
the morphism $C\rarrow R[1]$ an isomorphism (in $\sD^\abs(\bE)$
and even in $\sH^0(\bE)$) as well.
 Now it is clear that the composition $C\rarrow P\rarrow X$ factorizes
through the isomorphisms  $C\rarrow R[1]$ and $S\rarrow X$ in
the absolute derived category.

 To prove the lemma (i.~e., construct the desired commutative pentagon
diagram in the homotopy category $\sH^0(\bE)$), we present a simplified
version of the argument from~\cite{EP}.
 The argument starts with choosing a closed morphism $P\rarrow X$ of
degree~$0$ representing the given homotopy class.

 Consider the composition $P\rarrow X\rarrow T$ and represent it
by a triple $(f,g,h)$ with $f\in\Hom_\bE^1(P,U)$, \
$g\in\Hom_\bE^0(P,V)$, and $h\in\Hom_\bE^{-1}(P,W)$, as in
Section~\ref{lemmas-E-and-F-subsecn}.
 Arguing similarly to the proof of
Proposition~\ref{totalizations-approachable-prop}, we view~$h$ as
a morphism $P\rarrow W[-1]$ in the additive category $\bE^0$ and,
applying the functor $\widetilde\Phi_\bE$, produce a morphism
$\Phi(P)\rarrow\Phi(W[-1])$ in the category $\sZ^0(\bE^\bec)$.
 Applying the functor $\Phi_\bE$ to the admissible epimorphism
$V[-1]\rarrow W[-1]$ in the exact category $\sZ^0(\bE)$, we obtain
an admissible epimorphism $\Phi(V[-1])\rarrow\Phi(W[-1])$ in
the exact category $\sZ^0(\bE^\bec)$.
 Let $K$ be the related pullback object in $\sZ^0(\bE^\bec)$;
so we have a commutative square diagram as in the proof of
Proposition~\ref{totalizations-approachable-prop} with
an admissible epimorphism $K\rarrow\Phi(P)$ and a morphism
$K\rarrow\Phi(V[-1])$ in $\sZ^0(\bE^\bec)$.

 Put $Y=\Psi^+_\bE(K)\in\sZ^0(\bE)$.
 By Lemma~\ref{mono-epi-exact-DG-adjunction}(b), the morphism
$Y=\Psi^+(K)\rarrow P$ corresponding by adjunction to the admissible
epimorphism $K\rarrow\Phi(P)$ in $\sZ^0(\bE^\bec)$ is an admissible 
epimorphism in $\sZ^0(\bE)$.
 Let $R$ be kernel of the admissible epimorphism $Y\rarrow P$; so
we have a short exact sequence $0\rarrow R\rarrow Y\rarrow P\rarrow0$
in $\sZ^0(\bE)$.
 Put $C=\cone(R\rightarrowtail Y)\in\bE$, and let $C\rarrow P$ and
$C\rarrow R[1]$ be the natural closed morphisms of degree~$0$.
 We claim that the composition of closed morphisms
$C\rarrow P\rarrow X\rarrow T$ is homotopic to zero in~$\bE$.

 Indeed, one can continue to follow the proof of
Proposition~\ref{totalizations-approachable-prop}.
 By Lemma~\ref{Hom-complex-into-Tot-lemma}, it suffices to show that
the component $\tilde k'\in\Hom^0_\bE(C,W[-1])$ of our morphism
$C\rarrow T$ can be lifted to an element of $\Hom_\bE^0(C,V[-1])$.
 In the category $\bE^0$, the object $C$ is naturally isomorphic to
the direct sum $R[1]\oplus Y$, and the morphism $\tilde k'$ vanishes
on the component~$R[1]$.
 Hence it suffices to show that the related element $\tilde k\in
\Hom_\bE^0(Y,W[-1])$ can be lifted to an element of
$\Hom_\bE^0(Y,V[-1])$.

 The element~$\tilde k$ is the $W$\+component of the composition
$Y=\Psi^+(K)\rarrow P\rarrow X\rarrow T$, which is a closed morphism
in~$\bE$.
 According to Lemma~\ref{liftability-preserved-by-Phi-Psi-adjunction},
we only need to check that the component $K\rarrow\Phi(W[-1])$ of
the composition $K\rarrow\Phi\Psi^+(K)\rarrow\Phi(P)\rarrow\Phi(T)$
can be lifted to a morphism $K\rarrow\Phi(V[-1])$ in $\sZ^0(\bE^\bec)$.
 But we are given such a lifting by the construction of
the object $K\in\sZ^0(\bE^\bec)$.

 We have shown that the composition $C\rarrow P\rarrow X\rarrow T$
vanishes in $\sH^0(\bE)$, and it follows that the composition
$C\rarrow P\rarrow X$ factorizes through the morphism $S\rarrow X$
in the distringuished triangle $S\rarrow X\rarrow T\rarrow S[1]$.
 We have constructed the diagram of undotted arrows
$$
 \xymatrix{
  S[1] & T \ar[l] & & X \ar[ll] & S \ar[l] \\
  & & P \ar[ur] \\
  R \ar@{>->}[r] & Y \ar@{->>}[ru] \ar[rr] & & C \ar[ul] 
  \ar[uur] \ar[r] & R[1] \ar@{..>}[uu]
 }
$$
with the quadrangle in the right-hand side commutative up to homotopy.

 It remains to deduce the dotted factorization through $R[1]$
in $\sH^0(\bE)$.
 Here we simply notice that the object $Y=\Psi^+(K)$ is contractible
in $\bE$ by the definition of the functor $\Psi^+_\bE$, hence
$C\rarrow R[1]$ is an isomorphism in the homotopy category $\sH^0(\bE)$.
 (See~\cite[proof of Lemma~1.6.G]{EP} for a fancier argument not using
the observation that the functor $\Psi^+$ produces contractible
objects.)
\end{proof}

\begin{proof}[Proof of Proposition~\ref{abs-acyclic-as-totalizations}]
 We follow~\cite[first half of the proof of Theorem~1.6]{EP}.
 Let us prove the following more general claim: for any objects
$X\in\langle\Ac^0(\bE)\rangle_n$ and $P\in\sH^0(\bE)$, and any
morphism $f\:P\rarrow X$ in $\sH^0(\bE)$, there exists an exact complex
$0\rarrow Y_n\rarrow Y_{n-1}\rarrow\dotsb\rarrow Y_0\rarrow P\rarrow0$
in $\sZ^0(\bE)$ such that the composition $\Tot(Y_\bu)\rarrow P\rarrow
X$ of the natural closed morphism $\Tot(Y_\bu)\rarrow P$ with
the morphism~$f$ vanishes in $\sH^0(\bE)$.

 Indeed, it suffices to consider the case when $X\in\Ac^0(\bE)^{*n}$,
as the passage to a direct summand is trivial.
 So we have a sequence of objects $T_1$,~\dots, $T_n\in\Ac^0(\bE)$
and distinguished triangles $X_{i-1}\rarrow X_i\rarrow T_i\rarrow
X_{i-1}[1]$, \ $i=1$,~\dots,~$n$ in $\sH^0(\bE)$ such that $X_0=0$
and $X_n=X$.

 Applying Lemma~\ref{triangle-with-a-totalization-in-vertex-lemma}
to the distinguished triangle $X_{n-1}\rarrow X_n\rarrow T_n\rarrow
X_{n-1}[1]$ and the morphism $P\rarrow X_n$ in $\sH^0(\bE)$, we obtain
a short exact sequence $0\rarrow R_1\rarrow Y_0\rarrow P\rarrow0$
in $\sZ^0(\bE)$ and a morphism $R_1[1]\rarrow X_{n-1}$ in $\sH^0(\bE)$
such that the composition $\cone(R_1\to Y_0)\rarrow P\rarrow X_n$
is equal to the composition $\cone(R_1\to Y_0)\rarrow R_1[1]\rarrow
X_{n-1}\rarrow X_n$ in $\sH^0(\bE)$.
 Applying Lemma~\ref{triangle-with-a-totalization-in-vertex-lemma}
again to the morphism $R_1[1]\rarrow X_{n-1}$ and the distinguished
triangle $X_{n-2}\rarrow X_{n-1}\rarrow T_{n-1}\rarrow X_{n-2}[1]$,
we construct a short exact sequence $0\rarrow R_2\rarrow Y_1\rarrow R_1
\rarrow0$ in $\sZ^0(\bE)$ and a morphism $R_2[2]\rarrow X_{n-2}$ in
$\sH^0(\bE)$, etc.
 Finally we produce a short exact sequence $0\rarrow R_n\rarrow Y_{n-1}
\rarrow R_{n-1}\rarrow0$ in $\sZ^0(\bE)$ and a morphism $R_n[n]\rarrow
X_0=0$ in $\sH^0(\bE)$.
 Put $Y_n=R_n$.

 Now we have a commutative diagram in $\sH^0(\bE)$ as depicted on
Figure~1,
\begin{figure}[htb]
$$
 \xymatrix@C-=0.5cm{
 \Tot(Y_n\to\dotsb\to Y_1\to Y_0) \ar[rr] \ar[rd] \ar[dd]
 & & P \ar[r] & {X=X_n} \\
 & \cone(R_1\to Y_0) \ar[ru] \ar[rd] \\
 \Tot(Y_n\to\dotsb\to Y_2\to Y_1) \ar[rr] \ar[rd] \ar[dd]
 & & R_1[1] \ar[r] & X_{n-1} \ar[uu] \\
 & \cone(R_2\to Y_1)[1] \ar[ru] \ar[rd] \\
 \Tot(Y_n\to\dotsb\to Y_3\to Y_2) \ar[rr] \ar[dd]
 & & R_2[2] \ar[r] & X_{n-2} \ar[uu] \\ \\
 \vdots \ar[dd] & & & \vdots \ar[uu] \\ \\
 \Tot(Y_n\to Y_{n-1}) \ar[rr] \ar@{=}[rd]
 & & R_{n-1}[n-1] \ar[r] & X_1 \ar[uu] \\
 & \cone(R_n\to Y_{n-1})[n-1] \ar[ru] \ar[rd] \\
 & & R_n[n] \ar[r] & {X_0=0} \ar[uu]
 }
$$
\caption{The composition of morphisms vanishes in $\sH^0(\bE)$}
\end{figure}
showing that the composition of morphisms
$\Tot(Y_n\to\dotsb\to Y_0)\rarrow P\rarrow X$ vanishes in the homotopy
category, since it factorizes through a zero object.
 Indeed, the pentagons on the diagram are commutative in $\sH^0(\bE)$
by Lemma~\ref{triangle-with-a-totalization-in-vertex-lemma}, while
the triangles and the quadrangles are obviously commutative in
$\sZ^0(\bE)$ already.

 Having proved the more general claim, let us deduce the assertion of
proposition as it is stated.
 Given an object $X\in\langle\Ac^0(\bE)\rangle_n$, we put $P=X$ and
consider the identity morphism $f=\id_X$.
 As we have shown, there exists an exact complex
$0\rarrow Y_n\rarrow Y_{n-1}\rarrow\dotsb\rarrow Y_0\rarrow X\rarrow0$
in $\sZ^0(\bE)$ such that the natural closed morphism
$\Tot(Y_n\to\dotsb\to Y_0)\rarrow X$ is homotopic to zero in~$\bE$.
 It remains to consider the distinguished triangle
$$
 X[-1]\lrarrow\Tot(Y_n\to\dotsb\to Y_0\to X)\lrarrow
 \Tot(Y_n\to\dotsb\to Y_0)\lrarrow X
$$
in the homotopy category $\sH^0(\bE)$, and observe that the object
$X[-1]$ becomes a direct summand of the object $\Tot(Y_n\to\dotsb
\to Y_0\to X)$ in the triangulated category $\sH^0(\bE)$ whenever
the morphism $\Tot(Y_n\to\dotsb \to Y_0)\rarrow X$ vanishes
in $\sH^0(\bE)$.
\end{proof}

\subsection{Preliminaries on Yoneda extensions}
\label{prelim-Yoneda-Ext-subsecn}
 Let us recall the basics of the Yoneda Ext construction in exact
categories (see, e.~g., \cite[Sections~A.7\+-A.8]{Partin}).

 Given two objects $X$ and $Y$ in an exact category $\sK$, by
an \emph{$n$\+extension of $X$ by $Y$} one means an exact complex
$0\rarrow Y\rarrow A_n\rarrow\dotsb\rarrow A_1\rarrow X\rarrow0$
in~$\sK$.
 This means that the complex $0\rarrow Y\rarrow A_\bu\rarrow X
\rarrow0$ can be obtained by splicing (admissible) short exact
sequences in $\sK$, which we will denote by $0\rarrow Z^A_i\rarrow A_i
\rarrow Z^A_{i-1}\rarrow0$, \ $i=1$,~\dots, $n$, with $Z^A_n=Y$ and
$Z^A_0=X$.

 An \emph{elementary equivalence} acting from an $n$\+extension
$A_\bu$ to an $n$\+extension $B_\bu$ of $X$ by $Y$ is a morphism of
complexes $A_\bu\rarrow B_\bu$ inducing the identity maps on
the homology objects $X$ and~$Y$.
 The equivalence relation on the $n$\+extensions of $X$ by $Y$ is,
by the definition, generated by the elementary equivalences.
 When $n=1$, all elementary equivalences of $1$\+extensions are
isomorphisms, hence all equivalent $1$\+extensions are isomorphic;
but for $n>1$ this is no longer the case.

 The equivalence classes of $n$\+extensions of $X$ and $Y$ form
the abelian group $\Ext_\sK^n(X,Y)$.
 In fact, it is straightforward to see from the definition of
the derived category $\sD^\bb(\sK)$ (as presented in~\cite{Neem})
that the set of equivalence classes of $n$\+extensions of $X$ by $Y$
is naturally bijective to the group $\Ext_\sK^n(X,Y)=
\Hom_{\sD^\bb(\sK)}(X,Y[n])$.

 The class of all morphisms with exact cones in the homotopy category
of complexes in $\sK$ is localizing, i.~e., it satisfies the left and
right Ore conditions.
 This implies a similar property of the Yoneda $n$\+extensions.
 Namely, two $n$\+extensions $A_\bu$ and $B_\bu$ of $X$ by $Y$ are
equivalent if and only if there exists an $n$\+extension $C_\bu$ of $X$
by $Y$ together with a fraction (``roof'') of elementary equivalences
$A_\bu\larrow C_\bu\rarrow B_\bu$, and if and only if there exists
an $n$\+extension $D_\bu$ of $X$ by $Y$ together with a fraction of
elementary equivalences $A_\bu\rarrow D_\bu\larrow B_\bu$.

 The \emph{zero $n$\+extension} (corresponding to the zero element
$0\in\Ext_\sK^n(X,Y)$) is represented by the complex
$0\rarrow Y\rarrow Y\rarrow 0\rarrow\dotsb\rarrow 0\rarrow X
\rarrow X\rarrow0$ (for $n\ge2$) or $0\rarrow Y\rarrow Y\oplus X
\rarrow X\rarrow0$ (for $n=1$).
 Now, given an $n$\+extension $C_\bu$ of $X$ by $Y$, one easily
observes that an elementary equivalence acting from $C_\bu$ to
the zero $n$\+extension exists if and only if the admissible
monomorphism $Y\rarrow C_n$ is split.
 Dually, an elementary equivalence acting to an $n$\+extension
$D_\bu$ from the zero $n$\+extension exists if and only if
the admissible epimorphism $D_1\rarrow X$ is split.

 One says that an $n$\+extension $A_\bu$ of $X$ by $Y$ is
\emph{split} (or \emph{trivial}) if it is equivalent to the zero
$n$\+extension.
 We have proved the following lemma.

\begin{lem} \label{split-n-extension-lemma}
 Let\/ $\sK$ be an exact category, $X$, $Y\in\sK$ be two objects,
and $A_\bu$ be an $n$\+extension of $X$ by~$Y$.
 Then the $n$\+extension $A_\bu$ is split if and only if there exists
an elementary equivalence of $n$\+extensions $C_\bu\rarrow A_\bu$
of $X$ by $Y$ such that the admissible monomorphism $Y\rarrow C_n$ is
split, or equivalently, if and only if there exists an elementary
equivalence of $n$\+extensions $A_\bu\rarrow D_\bu$ of $X$ by $Y$
such that the admissible epimorphism $D_1\rarrow X$ is split.  \qed
\end{lem}

 We will say that an elementary equivalence $C_\bu\rarrow A_\bu$ of
$n$\+extensions of $X$ by $Y$ in $\sK$ is an \emph{elementary
epiequivalence} if the induced morphisms of the objects of cocycles
$Z^C_i\rarrow Z^A_i$ are admissible epimorphisms for all
$1\le i\le n-1$.
 Clearly, it follows that the morphisms $C_i\rarrow A_i$ are also
admissible epimorphisms in this case for all $1\le i\le n$.
 Conversely, if the morphisms $C_i\rarrow A_i$ are admissible
epimorphisms for all $1\le i\le n-1$ and the exact category $\sK$
is weakly idempotent-complete, then an elementary equivalence
$C_\bu\rarrow A_\bu$ is an elementary epiequivalence.

 The next lemma can be used in conjunction with the previous one.
 
\begin{lem} \label{epic-elementary-equivalence-lemma}
 Let\/ $\sK$ be an exact category, $X$, $Y\in\sK$ be two objects,
and $f\:B_\bu\rarrow A_\bu$ be an elementary equivalence of
$n$\+extensions of $X$ by~$Y$.
 Then there exists an $n$\+extension $C_\bu$ of $X$ by $Y$ together
with elementary epiequivalences of $n$\+extensions $C_\bu\rarrow A_\bu$
and $C_\bu\rarrow B_\bu$.
\end{lem}

 Notice that there is \emph{no} claim of commutativity of
the triangle diagram of elementary equivalences of $n$\+extensions
$C_\bu\rarrow B_\bu\rarrow A_\bu$ in
Lemma~\ref{epic-elementary-equivalence-lemma}.
 This triangle is essentially \emph{never} commutative, and we do not
need it to be.
 Rather, it will be important for us that there exists an elementary
equivalence of $n$\+extensions acting from $C_\bu$ to the zero
$n$\+extension of $X$ by $Y$ whenever there is such an elementary
equivalence acting from $B_\bu$ to the zero $n$\+extension.

\begin{proof}
 This is~\cite[first paragraph of the proof of Lemma~4.4]{Partin}.
 The case $n=1$ is trivial, so we assume $n>1$.
 Put $C_i=B_i\oplus A_i\oplus A_{i+1}$ for $1<i<n$, \
$C_n=B_n\oplus A_n$, and $C_1=B_1\oplus A_2$.
 In the exact sequence $0\rarrow Y\rarrow C_n\rarrow\dotsb\rarrow
C_1\rarrow X\rarrow0$, the components of the differential $Y\rarrow C_n$
are the differential $Y\rarrow B_n$ and the zero morphism
$Y\rarrow A_n$.
 The only nonzero components of the differential $C_i\rarrow C_{i-1}$
are the differential $B_i\rarrow B_{i-1}$ and the identity map
$A_i\rarrow A_i$.
 The components of the differential $C_1\rarrow X$ are the differential
$B_1\rarrow X$ and the zero morphism $A_2\rarrow X$.
 So the exact complex $0\rarrow Y\rarrow C_\bu\rarrow X\rarrow0$ is
constructed as the direct sum of the exact complex $0\rarrow Y\rarrow
B_\bu\rarrow X\rarrow0$ and the contractible two-term complexes
$0\rarrow A_i\rarrow A_i\rarrow0$, where $1<i\le n$.

 The elementary equivalence $C_\bu\rarrow B_\bu$ is the direct summand
projection; so the only nonzero component of the morphism $C_i\rarrow
B_i$ is the identity morphism $B_i\rarrow B_i$.
 The elementary equivalence $C_\bu\rarrow A_\bu$ is formed by
the following morphisms $C_i\rarrow A_i$.
 The components of the morphism $C_n\rarrow A_n$ are the morphism
$f_n\:B_n\rarrow A_n$ and the identity morphism $A_n\rarrow A_n$.
 For $1<i<n$, the components of the morphism $C_i\rarrow A_i$ are
the morphism $f_i\:B_i\rarrow A_i$, the identity map $A_i\rarrow A_i$,
and the differential $A_{i+1}\rarrow A_i$.
 For $i=1$, the components of the morphism $C_1\rarrow A_1$ are
the morphism $f_1\:B_1\rarrow A_1$ and the differential
$A_2\rarrow A_1$.

 To show that the elementary equivalence $C_\bu\rarrow A_\bu$ is
an elementary epiequivalence, one can compute the objects of cocycles
$Z_i^C\in\sK$.
 In fact, one has $Z_i^C=Z_i^B\oplus A_{i+1}$ for all $1\le i\le n-1$.
 The components of the induced morphism $Z_i^C\rarrow Z_i^A$
are the induced morphism $Z_i^B\rarrow Z_i^A$ and the admissible
epimorphism $A_{i+1}\rarrow Z_i^A$; so the morphism $Z_i^C\rarrow
Z_i^A$ is an admissible epimorphism by the dual version
of~\cite[first assertion of Exercise~3.11(i)]{Bueh}.
\end{proof}

\begin{lem} \label{epi-onto-pullback-lemma}
 Let $(0\to W'\to A'\to Z'\to 0)\rarrow (0\to W\to A\to Z\to0)$ be
a morphism of short exact sequences in an exact category\/~$\sK$:
$$
 \xymatrix{
 0 \ar[r] & W' \ar[r] \ar[d] & A' \ar[r] \ar[d]
 & Z' \ar[r] \ar[d] & 0 \\
 0 \ar[r] & W \ar[r] & A \ar[r] & Z \ar[r] & 0
 }
$$ \par
\textup{(a)} If the morphism $W'\rarrow W$ is an admissible epimorphism,
then the induced morphism to the pullback in the rightmost square
$$
 A'\lrarrow A\sqcap_Z Z'
$$
is an admissible epimorphism. \par
\textup{(b)} If the morphism $Z'\rarrow Z$ is an isomorphism, then
the leftmost square is (both a pushout and) a pullback,
$W'=W\sqcap_A A'$.
\end{lem}

\begin{proof}
 Part~(a): put $A''=A\sqcap_Z Z'$, and consider the pullback of
the short exact sequence $0\rarrow W\rarrow A\rarrow Z\rarrow0$ by
the morphism $Z'\rarrow Z$.
 Then we get a commutative triangle diagram of morphisms of short exact
sequences with identity morphisms on some of the leftmost and rightmost
terms:
$$
 \xymatrix{
 0 \ar[r] & W' \ar[r] \ar[d] & A' \ar[d] \ar[rd] \\
 0 \ar[r] & W \ar[r] \ar[rd] & A'' \ar[r] \ar[d]
 & Z' \ar[r] \ar[d] & 0 \\
 & & A \ar[r] & Z \ar[r] & 0
 }
$$
 Since both the short sequences in the upper half of the diagram are
exact, it follows that the upper leftmost square is a pushout (so
$A'\sqcup_{W'}W=A''=A\sqcap_ZZ'$).
 Finally, since $W'\rarrow W$ is an admissible epimorphism by
assumption, we can conclude that $A'\rarrow A''$ is also an admissible
epimorphism with the same kernel (by the dual assertion
to~\cite[Proposition~2.15]{Bueh}).

 Part~(b) is~\cite[Proposition~2.12(iv)$\Rightarrow$(iii)]{Bueh}.
\end{proof}

\subsection{Graded split Yoneda Ext classes}
\label{graded-split-Yoneda-Ext-subsecn}
 Now we pass to the following setting.
 Let $\sE$ and $\sK$ be two exact categories and $\Phi\:\sE\rarrow\sK$
be an exact functor.
 The aim of this Section~\ref{graded-split-Yoneda-Ext-subsecn} is to
prove the following proposition, which will be useful in
Section~\ref{finite-homological-dimension-subsecn}.

\begin{prop} \label{graded-split-Yoneda-extensions-prop}
 Assume that an exact functor\/ $\Phi\:\sE\rarrow\sK$ satisfies
the following condition:
\begin{itemize}
\item[($\spadesuit$)] For any two objects $B\in\sE$ and $K\in\sK$, and
any admissible epimorphism $K\rarrow\Phi(B)$ in\/~$\sK$, there exist
an admissible epimorphism $C\rarrow B$ in\/ $\sE$ and a morphism
$\Phi(C)\rarrow K$ in\/ $\sK$ making the triangle diagram
$\Phi(C)\rarrow K\rarrow\Phi(B)$ commutative in\/~$\sK$.
\end{itemize}
 Let $X$, $Y\in\sE$ be two objects, and let $A_\bu$ be
an $n$\+extension of $X$ by $Y$ in the exact category\/~$\sE$.
 Then the $n$\+extension\/ $\Phi(A_\bu)$ of the object\/ $\Phi(X)$ by
the object\/ $\Phi(Y)$ is split in the exact category\/ $\sK$ if and
only if there exists an $n$\+extension $C_\bu$ of $X$ by $Y$ in\/ $\sE$
together with an elementary equivalence of $n$\+extensions $C_\bu
\rarrow A_\bu$ in\/ $\sE$ such that the functor\/ $\Phi$ takes
the admissible monomorphism $Y\rarrow C_n$ in\/ $\sE$ to a split
monomorphism in\/~$\sK$.
\end{prop}

 Condition~($\spadesuit$) plays an important role in the theory of
the Yoneda Ext functor in exact categories;
see~\cite[Section~4.4]{Partin} and~\cite[condition~(i$'$) in
Section~0.1]{Pred}.
 In the special case of a fully faithful functor $\Phi$, this
condition is also a part of the definition of a self-resolving
subcategory in Section~\ref{self-resolving-subcategories-subsecn}.

\begin{proof}
 An exact functor $\Phi$ takes $n$\+extensions in $\sE$ to
$n$\+extensions in $\sK$, and it also takes elementary equivalences
of $n$\+extensions in $\sE$ to elementary equivalences of
$n$\+extensions in~$\sK$.
 This suffices to prove the ``if'' implication of the proposition.
 The nontrivial part is the ``only if''.

 We follow~\cite[second half of the proof of Theorem~1.6]{EP} with
some details added (as the exposition in~\cite{EP} is rather sketchy).
 So let $A_\bu$ be an $n$\+extension of $X$ by $Y$ in $\sE$ such that
the $n$\+extension $\Phi(A_\bu)$ in $\sK$ is split.
 By Lemma~\ref{split-n-extension-lemma}, there exists an elementary
$n$\+extension $K_\bu$ of $\Phi(X)$ by $\Phi(Y)$ in $\sK$ together with
an elementary equivalence of $n$\+extensions $K_\bu\rarrow\Phi(A_\bu)$
such that the admissible monomonorphism $\Phi(Y)\rarrow K_n$ is split.
 By Lemma~\ref{epic-elementary-equivalence-lemma}, we can assume that
$K_\bu\rarrow\Phi(A_\bu)$ is an elementary epiequivalence, i.~e.,
the induced morphisms $Z^K_i\rarrow\Phi(Z^A_i)$ are admissible
epimorphisms in~$\sK$ for all $1\le i\le n-1$.

 Applying condition~($\spadesuit$) to the admissible epimorphism
$K_1\rarrow\Phi(A_1)$ in $\sK$, we obtain an admissible epimorphism
$C_1\rarrow A_1$ in $\sE$ and a morphism $\Phi(C_1)\rarrow K_1$
in $\sK$ such that the triangle diagram
$\Phi(C_1)\rarrow K_1\rarrow\Phi(A_1)$ is commutative.
 It follows that the induced morphism $Z^C_1\rarrow Z^A_1$ is
an admissible epimorphism, too (with the same kernel as
the morphism $C_1\rarrow A_1$).
 We want to continue this procedure, producing objects $C_i\in\sE$
for all $2\le i\le n$ forming an $n$\+extension $C_\bu$ of $X$ by $Y$
in $\sE$, together with elementary equivalences of $n$\+extensions
$C_\bu\rarrow A_\bu$ in $\sE$ and $\Phi(C_\bu)\rarrow K_\bu$ in $\sK$
(forming a commutative diagram of elementary equivalences of
$n$\+extensions $\Phi(C_\bu)\rarrow K_\bu\rarrow\Phi(A_\bu)$ in~$\sK$).

 The next step of the construction can be formulated in the same
terms as in~\cite{EP}.
 In the commutative diagram with exact rows in~$\sK$
$$
 \xymatrix{
 0 \ar[r] & Z^K_2 \ar[r] \ar@{->>}[d] & K_2 \ar[r] \ar[d]
 & K_1 \ar[r] \ar[d] & \Phi(X) \ar[r] \ar@{=}[d] & 0 \\
 0 \ar[r] & \Phi(Z^A_2) \ar[r] & \Phi(A_2) \ar[r]
 & \Phi(A_1) \ar[r] & \Phi(X) \ar[r] & 0
 }
$$
the morphism to the pullback in the middle square,
$$
 K_2\lrarrow \Phi(A_2)\sqcap_{\Phi(A_1)}K_1,
$$
is an admissible epimorphism, as one can see by viewing
the diagram as a splice of two morphisms of short exact sequences
and using parts~(a) and~(b) of Lemma~\ref{epi-onto-pullback-lemma}.

 Consider the commutative diagram in~$\sK$
$$
 \xymatrix{
 & \Phi(C_1) \ar[d] \\
 K_2 \ar[r] \ar[d] & K_1 \ar[d] \\
 \Phi(A_2) \ar[r] & \Phi(A_1) 
 }
$$
 Any pullback of admissible epimorphism is an admissible epimorphism,
hence the morphism
$$
 K_2\sqcap_{K_1}\Phi(C_1) \lrarrow \Phi(A_2)\sqcap_{\Phi(A_1)}\Phi(C_1)
 \,=\, (\Phi(A_2)\sqcap_{\Phi(A_1)}K_1)\sqcap_{K_1}\Phi(C_1)
$$
is an admissible epimorphism.

 Applying condition~($\spadesuit$) to the object
$A_2\sqcap_{A_1}C_1\in\sE$ and the admissible epimorphism 
$K_2\sqcap_{K_1}\Phi(C_1)\rarrow \Phi(A_2\sqcap_{A_1}C_1)$ in $\sK$,
we obtain an admissible epimorphism $C_2\rarrow A_2\sqcap_{A_1}C_1$
in $\sE$ and a morphism $\Phi(C_2)\rarrow K_2\sqcap_{K_1}\Phi(C_1)$
in $\sK$ such that the triangle diagram $\Phi(C_2)\rarrow
K_2\sqcap_{K_1}\Phi(C_1)\rarrow\Phi(A_2\sqcap_{A_1}C_1)$ is
commutative in~$\sK$.
 Commutativity of the latter triangle trianslates into commutativity
of the triangle $\Phi(C_2)\rarrow K_2\rarrow\Phi(A_2)$ together
with commutativity of the two leftmost squares on the diagram
\begin{equation} \label{first-two-steps-diagram}
\begin{gathered}
 \xymatrix{
 \Phi(C_2) \ar[r] \ar[d] & \Phi(C_1) \ar[rd] \ar[d] \\
 K_2 \ar[r] \ar[d] & K_1 \ar[r] \ar[d] & \Phi(X) \ar[r] & 0 \\
 \Phi(A_2) \ar[r] & \Phi(A_1) \ar[ru] 
 }
\end{gathered}
\end{equation}
in $\sK$, where the outer square in the left-hand side is obtained by
applying $\Phi$ to a commutative square in~$\sE$.
 Furthermore, since $C_1\rarrow A_1$ and $C_2\rarrow A_2
\sqcap_{A_1}C_1$ are admissible epimorphisms and the sequence
$0\rarrow Z^A_2\rarrow A_2\rarrow A_1\rarrow X\rarrow0$ is exact,
it follows that the upper line of~\eqref{first-two-steps-diagram}
can be extended to an exact sequence $0\rarrow Z^C_2\rarrow C_2\rarrow
C_1\rarrow X\rarrow0$ in~$\sE$, and the induced morphism $Z^C_2
\rarrow Z^A_2$ is an admissible epimorphism.

 Let us step back and reinterpret our approach at this point, or
at least adjust the notation.
 Commutativity of the right-hand side of
the diagram~\eqref{first-two-steps-diagram} in $\sK$, with the same
object $\Phi(X)$ in the rightmost column in all the three rows,
together with  commutativity of the related outer triangle in $\sE$,
imply isomorphisms
$$
 A_2\sqcap_{A_1}C_1 = A_2\sqcap_{Z^A_1}Z^C_1 \quad\text{and}\quad
 \Phi(A_2)\sqcap_{\Phi(A_1)}K_1 = \Phi(A_2)\sqcap_{\Phi(Z^A_1)}Z^K_1.
$$
 So, instead of speaking about applying condition~($\spadesuit$) to
the admissible epimorphism $K_2\sqcap_{K_1}\Phi(C_1)\rarrow
\Phi(A_2\sqcap_{A_1}C_1)$, we could have said that we were applying it
to the admissible epimorphism $K_2\sqcap_{Z^K_1}\Phi(Z^C_1)\rarrow
\Phi(A_2\sqcap_{Z^A_1}Z^C_1)$.

 It is the latter point of view that we will carry over to the next
step.
 Indeed, one can observe that a morphism
$C_3\rarrow A_3\sqcap_{A_2}C_2$ in our construction \emph{cannot} be
an admissible epimorphism, generally speaking, because the composition
$A_3\sqcap_{A_2}C_2\rarrow C_2\rarrow C_1$ need not vanish.
 It is only the composition $A_3\sqcap_{A_2}C_2\rarrow C_2\rarrow C_1
\rarrow A_1$ that always vanishes.

 Let us spell out the construction of the object $C_{i+1}$ for
an arbitrary $1\le i\le n-2$.
 Suppose that an exact complex $0\rarrow Z^C_i\rarrow C_i\rarrow\dotsb
\rarrow C_1\rarrow X\rarrow0$ in $\sE$ together with a morphism of
complexes $(Z^C_i\to C_i\to\dotsb\to C_1\to X)\rarrow(Z^A_i\to A_i\to
\dotsb\to A_1\to X)$ acting by the identity morphism on the rightmost
terms~$X$ have been constructed already, together with a commutative
diagram in~$\sK$
\begin{equation} \label{first-i-steps-diagram}
\begin{gathered}
 \xymatrix{
 0 \ar[r] & \Phi(Z^C_i) \ar[r] \ar[d] & \Phi(C_i) \ar[r] \ar[d]
 & \dotsb \ar[r] & \Phi(C_1) \ar[rd] \ar[d] \\
 0 \ar[r] & Z^K_i \ar[r] \ar[d] & K_i \ar[r] \ar[d] & \dotsb \ar[r]
 & K_1 \ar[r] \ar[d] & \Phi(X) \ar[r] & 0 \\
 0 \ar[r] & \Phi(Z^A_i) \ar[r] & \Phi(A_i) \ar[r] & \dotsb \ar[r]
 & \Phi(A_1) \ar[ru] 
 }
\end{gathered}
\end{equation}

 Applying Lemma~\ref{epi-onto-pullback-lemma}(a) to the morphism of
short exact sequences in~$\sK$
$$
 \xymatrix{
 0 \ar[r] & Z^K_{i+1} \ar[r] \ar@{->>}[d] & K_{i+1} \ar[r] \ar[d]
 & Z^K_i \ar[r] \ar[d] & 0 \\
 0 \ar[r] & \Phi(Z^A_{i+1}) \ar[r] & \Phi(A_{i+1}) \ar[r]
 & \Phi(Z^A_i) \ar[r] & 0
 }
$$
we see that the morphism to the pullback in the rightmost square,
$$
 K_{i+1}\lrarrow\Phi(A_{i+1})\sqcap_{\Phi(Z^A_i)}Z^K_i,
$$
is an admissible epimorphism.

 Consider the commutative diagram in~$\sK$
$$
 \xymatrix{
 & \Phi(Z^C_i) \ar[d] \\
 K_{i+1} \ar[r] \ar[d] & Z^K_i \ar[d] \\
 \Phi(A_{i+1}) \ar[r] & \Phi(Z^A_i) 
 }
$$
 As above, we argue that any pullback of admissible epimorphism is
an admissible epimorphism, hence the morphism
$$
 K_{i+1}\sqcap_{Z^K_i}\Phi(Z^C_i) \lrarrow
 \Phi(A_{i+1})\sqcap_{\Phi(Z^A_i)}\Phi(Z^C_i)
 \,=\, (\Phi(A_{i+1})\sqcap_{\Phi(Z^A_i)}Z^K_i)
 \sqcap_{Z^K_i}\Phi(Z^C_i)
$$
is an admissible epimorphism.

 Applying condition~($\spadesuit$) to the object
 $A_{i+1}\sqcap_{Z^A_i}Z^C_i\in\sE$ and the admissible epimorphism
$K_{i+1}\sqcap_{Z^K_i}\Phi(Z^iC_i)\rarrow
\Phi(A_{i+1}\sqcap_{Z^A_i}Z^C_i)$ in $\sK$, we obtain an admissible
epimorphism
\begin{equation} \label{C-to-pullback-A-C-admissible-epi}
 C_{i+1}\lrarrow A_{i+1}\sqcap_{Z^A_i}Z^C_i
\end{equation}
in $\sE$ and a morphism
\begin{equation} \label{Phi-C-to-pullback-K-Phi-C-morphism}
 \Phi(C_{i+1})\lrarrow K_{i+1}\sqcap_{Z^K_i}\Phi(Z^C_i)
\end{equation}
in $\sK$ such that the triangle diagram
\begin{equation} \label{Phi-C-to-pullback-to-pullback-triangle}
 \Phi(C_{i+1})\lrarrow K_{i+1}\sqcap_{Z^K_i}\Phi(Z^C_i)
 \lrarrow\Phi(A_{i+1}\sqcap_{Z^A_i}Z^C_i)
\end{equation}
is commutative in~$\sK$.

 The morphism $A_{i+1}\sqcap_{Z^A_i}Z^C_i\rarrow Z^C_i$ is an admissible
epimorphism in $\sE$, because the morphism $A_{i+1}\rarrow Z^A_i$ is.
 Hence the composition $C_{i+1}\lrarrow A_{i+1}\sqcap_{Z^A_i}Z^C_i
\rarrow Z^C_i$ is an admissible epimorphism; denote its kernel
by $Z^C_{i+1}\in\sE$.
 Now the morphism~\eqref{C-to-pullback-A-C-admissible-epi} translates
into a commutative square, and passing to the kernels we obtain
a morphism of short exact sequences in~$\sE$
$$
 \xymatrix{
 0 \ar[r] & Z^C_{i+1} \ar[r] \ar[d] & C_{i+1} \ar[r] \ar[d]
 & Z^C_i \ar[r] \ar[d] & 0 \\
 0 \ar[r] & Z^A_{i+1} \ar[r] & A_{i+1} \ar[r]
 & Z^A_i \ar[r] & 0
 }
$$
 The morphism~\eqref{Phi-C-to-pullback-K-Phi-C-morphism} translates
into a commutative square as well, and passing to the kernels we obtain
a morphism of short exact sequences in~$\sK$
$$
 \xymatrix{
 0 \ar[r] & \Phi(Z^C_{i+1}) \ar[r] \ar[d] & \Phi(C_{i+1}) \ar[r] \ar[d]
 & \Phi(Z^C_i) \ar[r] \ar[d] & 0 \\
 0 \ar[r] & Z^K_{i+1} \ar[r] & K_{i+1} \ar[r]
 & Z^K_i \ar[r] & 0
 }
$$
 Finally, commutativity of the triangle
diagram~\eqref{Phi-C-to-pullback-to-pullback-triangle} translates into
commutativity of the triangle diagram of morphisms of short exact
sequences in~$\sK$
$$
 \xymatrix{
 0 \ar[r] & \Phi(Z^C_{i+1}) \ar[r] \ar[d] & \Phi(C_{i+1}) \ar[r] \ar[d]
 & \Phi(Z^C_i) \ar[r] \ar[d] & 0 \\
 0 \ar[r] & Z^K_{i+1} \ar[r] \ar[d] & K_{i+1} \ar[r] \ar[d]
 & Z^K_i \ar[r] \ar[d] & 0 \\
 0 \ar[r] & \Phi(Z^A_{i+1}) \ar[r] & \Phi(A_{i+1}) \ar[r]
 & \Phi(Z^A_i) \ar[r] & 0
 }
$$

 We have managed to extend the commutative triangle diagram of
morphisms of exact complexes~\eqref{first-i-steps-diagram} one step
further to the left.
 Proceeding in this way, we construct a diagram
like~\eqref{first-i-steps-diagram} for $i=n-1$, that is
$$
 \xymatrix{
 0 \ar[r] & \Phi(Z^C_{n-1}) \ar[r] \ar[d] & \Phi(C_{n-1}) \ar[r] \ar[d]
 & \dotsb \ar[r] & \Phi(C_1) \ar[rd] \ar[d] \\
 0 \ar[r] & Z^K_{n-1} \ar[r] \ar[d] & K_{n-1} \ar[r] \ar[d]
 & \dotsb \ar[r] & K_1 \ar[r] \ar[d] & \Phi(X) \ar[r] & 0 \\
 0 \ar[r] & \Phi(Z^A_{n-1}) \ar[r] & \Phi(A_{n-1}) \ar[r]
 & \dotsb \ar[r] & \Phi(A_1) \ar[ru] 
 }
$$

 Now we are coming close to an end of the construction, and
a different procedure is needed for $i=n$.
 The commutative diagram in~$\sK$
$$
 \xymatrix{
  & & K_n \ar[d] \ar[r] & Z^K_{n-1} \ar[d] \ar[r] & 0 \\
  0 \ar[r] & \Phi(Y) \ar[r] \ar[ru] & \Phi(A_n) \ar[r]
  & \Phi(Z^A_{n-1}) \ar[r] & 0
 }
$$
with short exact sequences $0\rarrow\Phi(Y)\rarrow K_n\rarrow
Z^K_{n-1}\rarrow0$ and $0\rarrow\Phi(Y)\rarrow\Phi(A_n)\rarrow
\Phi(Z^A_{n-1})\rarrow0$ implies that the square in
the right-hand side is a pullback.
 The object $C_n\in\sE$ is constructed by taking a pullback of
the short exact sequence $0\rarrow Y\rarrow A_n\rarrow Z^A_n\rarrow0$
by the morphism $Z^C_{n-1}\rarrow Z^A_{n-1}$, as on the diagram
$$
 \xymatrix{
  & & C_n \ar[d] \ar[r] & Z^C_{n-1} \ar[d] \ar[r] & 0 \\
  0 \ar[r] & Y \ar[r] \ar[ru] & A_n \ar[r]
  & Z^A_{n-1} \ar[r] & 0
 }
$$
 Then functoriality of the pullbacks implies commutativity of
the diagram in~$\sK$
$$
 \xymatrix{
 & & \Phi(C_n) \ar[r] \ar[d] & \Phi(Z^C_{n-1}) \ar[d] \\
 0 \ar[r] & \Phi(Y) \ar[r] \ar[ru] \ar[rd]
 & K_n \ar[r] \ar[d] & Z^K_{n-1} \ar[d] \\
 & & \Phi(A_n) \ar[r] & \Phi(Z^A_{n-1})
 }
$$
finishing the construction of the $n$\+extension $C_\bu$ and
the elementary equivalences of $n$\+extensions $C_\bu\rarrow A_\bu$
and $\Phi(C_\bu)\rarrow K_\bu$.

 It remains to recall that the $n$\+extension $K_\bu$ was chosen
at the beginning of this proof so that the admissible monomorphism
$\Phi(Y)\rarrow K_n$ were split.
 In other words, there was an elementary equivalence acting from
$K_\bu$ to the zero $n$\+extension of $\Phi(X)$ by $\Phi(Y)$.
 Hence there is also an elementary equivalence acting from
the $n$\+extension $\Phi(C_\bu)$ to the zero $n$\+extension.
 This means that the admissible monomorphism $\Phi(Y)\rarrow\Phi(C_n)$
is also split, as desired.
\end{proof}

\subsection{Finite homological dimension theorem}
\label{finite-homological-dimension-subsecn}
 The following proposition, formulated in the terminology and notation
of the paper~\cite{BvdB} (see
Section~\ref{strongly-generated-thick-subsecn}),
is the main result of Section~\ref{finite-homol-dim-secn}.

\begin{prop} \label{finite-homol-dim-abs-acycl-strongly-generated-prop}
 Let $(\bE,\sK)$ be an exact DG\+pair in which the exact category\/
$\sK$ has finite homological dimension.
 Then the full triangulated subcategory of absolutely acyclic
objects\/ $\Ac^\abs(\bE)\subset\sH^0(\bE)$ is strongly generated by
the totalizations of short exact sequences in\/ $\sZ^0(\bE)$.
 More precisely, if the homological dimension of\/ $\sK$ does not
exceed~$n$, then\/ $\Ac^\abs(\bE)=\langle\Ac^0(\bE)\rangle_n$.
\end{prop}

 The next theorem, whose idea goes back to~\cite[Theorem~1.6]{EP}, 
follows easily.

\begin{thm} \label{finite-homological-dimension-main-theorem}
 Let $(\bE,\sK)$ be an exact DG\+pair in which the exact category\/
$\sK$ has finite homological dimension. \par
\textup{(a)} Assume that the exact DG\+category $\bE$ has exact
coproducts.
 Then any coacyclic object in\/ $\bE$ is absolutely acyclic,
$\Ac^\co(\bE)=\Ac^\abs(\bE)$, and therefore the coderived category
of\/ $\bE$ coincides with the absolute derived category,
$$ 
 \sD^\abs(\bE)=\sD^\co(\bE).
$$ \par
\textup{(b)} Assume that the exact DG\+category $\bE$ has exact
products.
 Then any contraacyclic object in\/ $\bE$ is absolutely acyclic,
$\Ac^\ctr(\bE)=\Ac^\abs(\bE)$, and therefore the contraderived category
of\/ $\bE$ coincides with the absolute derived category,
$$ 
 \sD^\abs(\bE)=\sD^\ctr(\bE).
$$
\end{thm}

\begin{proof}
 One observes that, for any exact DG\+category $\bE$ with exact
coproducts, the class $\Ac^0(\bE)$ of all totalizations of short
exact sequences in $\sZ^0(\bE)$ is closed under coproducts
in $\sH^0(\bE)$.
 Dually, for any exact DG\+category $\bE$ with exact products,
the class $\Ac^0(\bE)$ of all totalizations of short exact sequences
in $\sZ^0(\bE)$ is closed under products in $\sH^0(\bE)$.
 Now, if the exact category $\sK$ has finite homological dimension~$n$,
then $\Ac^\abs(\bE)=\langle\Ac^0(\bE)\rangle_n$ according to
Proposition~\ref{finite-homol-dim-abs-acycl-strongly-generated-prop}.
 Hence, by Lemma~\ref{strongly-generated-closed-under-co-products},
the thick subcategory of all absolutely acyclic objects $\Ac^\abs(\bE)$
is closed under coproducts in $\sH^0(\bE)$ under the assumptions of
part~(a), and it is closed under products in $\sH^0(\bE)$
under the assumptions of part~(b).
\end{proof}

 Before proving
Proposition~\ref{finite-homol-dim-abs-acycl-strongly-generated-prop},
let us collect a couple of lemmas.
 The following lemma is a more precise version of
Lemma~\ref{totalization-of-finite-absolutely-acyclic}.

\begin{lem} \label{totalization-of-finite-length-n-has-dimension-n}
 Let\/ $\bE$ be an exact DG\+category, and let\/ $0\rarrow Y_{n+1}
\rarrow Y_n\rarrow\dotsb\rarrow Y_1\rarrow Y_0\rarrow0$ be
a finite exact complex with $n$~middle terms in the exact category\/
$\sZ^0(\bE)$.
 Then the totalization\/ $\Tot(Y_\bu)$ belongs to the class\/
$\Ac^0(\bE)^{*n}\subset\sH^0(\bE)$.
\end{lem}

\begin{proof}
 See the proof of Lemma~\ref{totalization-of-finite-absolutely-acyclic}.
\end{proof}

 The next lemma establishes applicability of
Proposition~\ref{graded-split-Yoneda-extensions-prop} in our
context.

\begin{lem} \label{exact-right-adjoint-implies-self-resolving-lemma}
 Let\/ $\sE$ and\/ $\sK$ be exact categories, and let\/
$\Phi\:\sE\rarrow\sK$ be an additive functor.
 Assume that the functor\/ $\Phi$ has an exact right adjoint functor\/
$\Psi^-\:\sK\rarrow\sE$.
 Let $B\in\sE$ and $K\in\sK$ be two objects, and let $K\rarrow
\Phi(B)$ be an admissible epimorphism in\/~$\sK$.
 Then there exist an admissible epimorphism $C\rarrow B$ in\/ $\sE$
and a morphism\/ $\Phi(C)\rarrow K$ in\/ $\sK$ making the triangle
diagram\/ $\Phi(C)\rarrow K\rarrow\Phi(B)$ commutative in\/~$\sK$.
 In other words, condition~\textup{($\spadesuit$)} is satisfied.
\end{lem}

\begin{proof}
 This is a generalization of~\cite[Lemma~1.6.H]{EP}.
 Applying the exact functor $\Psi^-$ to the admissible epimorphism
$K\rarrow\Phi(B)$ in $\sK$, we obtain an admissible epimorphism
$\Psi^-(K)\rarrow\Psi^-\Phi(B)$ in~$\sE$.
 Consider the adjunction morphism $B\rarrow\Psi^-\Phi(B)$ in $\sE$,
and take the pullback, producing a commutative square diagram in $\sE$
with an admissible epimorphism $C\rarrow B$:
$$
 \xymatrix{
  C \ar[r] \ar@{->>}[d] & \Psi^-(K) \ar@{->>}[d] \\
  B \ar[r] & \Psi^-\Phi(B)
 }
$$
 Applying the functor $\Phi$, we obtain a commutative square in
the left-hand side of the diagram in~$\sK$
$$
 \xymatrix{
  \Phi(C) \ar[r] \ar[d] & \Phi\Psi^-(K) \ar[d] \ar[r] & K \ar[d] \\
  \Phi(B) \ar[r] & \Phi\Psi^-\Phi(B) \ar[r] & \Phi(B)
 }
$$
 Here the horizontal arrows in the right-hand side are the adjunction
morphisms, the square in the right-hand side is commutative because
the adjunction counit is a natural transformation, and the composition
in the lower line is the identity morphism by the general property of
adjunctions.
 It follows that the triangle $\Phi(C)\rarrow K\rarrow\Phi(B)$ is
commutative in~$\sK$.
\end{proof}

\begin{proof}[Proof of
Proposition~\ref{finite-homol-dim-abs-acycl-strongly-generated-prop}]
 Notice first of all that if the exact category $\sK$ has either
enough projective or enough injective objects, then the assertions of
the proposition follow from the proof of
Theorem~\ref{inj-proj-resolutions-finite-homol-dim-theorem}
(with Remark~\ref{exact-DG-pairs-work-in-proj-inj-section-remark}).
 In fact, one obtains a stronger result that any object of
$\Ac^\abs(\bE)$ is homotopy equivalent to the totalization of
an exact complex $0\rarrow Y_{n+1}\rarrow Y_{n-1}\rarrow\dotsb\rarrow
Y_1\rarrow Y_0\rarrow0$.
 By Lemma~\ref{totalization-of-finite-length-n-has-dimension-n}, it
follows that $\Ac^\abs(\bE)=\Ac^0(\bE)^{*n}$ in this case.

 In the general case, we continue to follow~\cite[proof of
Theorem~1.6]{EP} and proceed within the context of the proof of
Proposition~\ref{abs-acyclic-as-totalizations}.
 Let $X\in\Ac^\abs(\bE)$ and $P\in\sH^0(\bE)$ be two objects, and
let $f\:P\rarrow X$ be a morphism in $\sH^0(\bE)$.
 As it was shown in the proof of
Proposition~\ref{abs-acyclic-as-totalizations}, there exists
an exact complex $0\rarrow Y'_{m+1}\rarrow Y'_m\rarrow\dotsb\rarrow
Y'_0\rarrow P\rarrow0$ in $\sZ^0(\bE)$ such that the composition
$\Tot(Y'_\bu)\rarrow P\rarrow X$ vanishes in $\sH^0(\bE)$.

 Our aim is to show that one can have $m<n$.
 Assuming that $m\ge n$, we will explain how to shorten the complex
$Y'_\bu$ so as to obtain an exact complex $0\rarrow Y_m\rarrow
Y_{m-1}\rarrow\dotsb\rarrow Y_0\rarrow P\rarrow0$ in $\sZ^0(\bE)$
with the same property that the composition
$\Tot(Y_\bu)\rarrow P\rarrow X$ vanishes in $\sH^0(\bE)$.

 Applying the exact functor $\Phi_\bE^\sK\:\sZ^0(\bE)\rarrow\sK$, we
obtain an exact complex $0\rarrow\Phi(Y'_{m+1})\rarrow\Phi(Y'_m)
\rarrow\dotsb\rarrow\Phi(Y'_0)\rarrow\Phi(P)\rarrow0$ in $\sK$.
 This exact complex can be viewed as a Yoneda $(m+1)$\+extension
representing some element of $\Ext_\sK^{m+1}(\Phi(P),\Phi(Y'_{m+1}))$.
 By assumption, any such Ext group vanishes in~$\sK$; so the this
$(m+1)$\+extension splits in the exact category~$\sK$.

 Lemma~\ref{exact-right-adjoint-implies-self-resolving-lemma} tells
that Proposition~\ref{graded-split-Yoneda-extensions-prop} is
applicable to the exact functor $\Phi_\bE^\sK\:\sZ^0(\bE)\allowbreak
\rarrow\sK$.
 Applying the proposition, we obtain a morphism of exact complexes
in $\sE=\sZ^0(\bE)$ depicted in the lower half of the diagram
$$
 \xymatrix{
  & 0 \ar[r] & Y_m \ar[r] & Y_{m-1} \ar[r] & \dotsb \ar[r]
  & Y_0 \ar[rd] \\
  0 \ar[r] & Y'_{m+1} \ar[r] \ar[u] \ar[rd] & Y''_m \ar[r] \ar[u] \ar[d]
  & Y_{m-1} \ar[r] \ar@{=}[u] \ar[d] & \dotsb \ar[r]
  & Y_0 \ar[r] \ar@{=}[u] \ar[d] & P \ar[r] & 0 \\
  & & Y'_m \ar[r] & Y'_{m-1} \ar[r] & \dotsb \ar[r] & Y'_0 \ar[ru]
 }
$$
with the property the exact complex in the middle row is obtained
by splicing short exact sequences, the leftmost one of which is
taken to a split short exact sequence by the functor~$\Phi_\bE^\sK$.
 
 Let us express the latter property in notation with formulas.
 The exact complex $0\rarrow Y'_{m+1}\rarrow Y''_m\rarrow Y_{m-1}
\rarrow\dotsb\rarrow Y_0\rarrow P\rarrow0$ is obtained by splicing
a short exact sequence $0\rarrow Y'_{m+1}\rarrow Y''_m\rarrow
Y_m\rarrow0$ with an exact complex $0\rarrow Y_m\rarrow Y_{m-1}
\rarrow\dotsb\rarrow Y_0\rarrow P\rarrow0$ in $\sZ^0(\bE)$ (depicted
in the upper line of the diagram).
 The result of Proposition~\ref{graded-split-Yoneda-extensions-prop}
allows us to choose the exact complex in the middle row of
the diagram in such a way that the admissible monomorphism
$\Phi(Y'_{m+1})\rarrow\Phi(Y''_m)$ is split, or in other words,
the short exact sequence $0\rarrow\Phi(Y'_{m+1})\rarrow\Phi(Y''_m)
\rarrow\Phi(Y_m)\rarrow0$ is split in~$\sK$.

 In view of Lemma~\ref{Phi-Psi-extended-to-nonclosed-morphisms},
the latter property means that the short sequence $0\rarrow Y'_{m+1}
\rarrow Y''_m\rarrow Y_m\rarrow0$ in $\sZ^0(\bE)$ is split exact in
the additive category~$\bE^0$.
 Consequently, the totalization $\Tot(Y'_{m+1}\to Y''_m\to Y_m)$ is
contractible, i.~e., it represents a zero object in $\sH^0(\bE)$.
 It follows that the morphism of totalizations induced by
the morphism of complexes in the upper half of the diagram,
$$
 \Tot(Y'_{m+1}\to Y''_m\to Y_{m-1}\to\dotsb\to Y_0)
 \lrarrow \Tot(Y_m\to Y_{m-1}\to\dotsb\to Y_0)
$$
is a homotopy equivalence, i.~e., an isomorphism in $\sH^0(\bE)$.

 Since the composition 
$$
 \Tot(Y'_{m+1}\to Y'_m\to\dotsb\to Y'_0)\lrarrow P\lrarrow X
$$
vanishes in $\sH^0(\bE)$ by the choice of the complex $Y'_\bu$,
it follows that the composition
\begin{multline*}
 \Tot(Y'_{m+1}\to Y''_m\to Y_{m-1}\to\dotsb\to Y_0) \\ \lrarrow
 \Tot(Y'_{m+1}\to Y'_m\to\dotsb\to Y'_0)\lrarrow P\lrarrow X
\end{multline*}
vanishes in $\sH^0(\bE)$ as well.
 Since the composition
\begin{multline*}
 \Tot(Y'_{m+1}\to Y''_m\to Y_{m-1}\to\dotsb\to Y_0) \\
 \lrarrow\Tot(Y_m\to Y_{m-1}\to\dotsb\to Y_0) \lrarrow P\lrarrow X
\end{multline*}
vanishes and the leftmost map is an isomorphism in $\sH^0(\bE)$,
we can conclude that the composition
$$
 \Tot(Y_m\to Y_{m-1}\to\dotsb\to Y_0)
 \lrarrow P\lrarrow X
$$
vanishes in $\sH^0(\bE)$.

 We have managed to cut our original complex $Y'_\bu$ down to a shorter
complex $Y_\bu$ with the same property, as we wished.
 By decreasing induction in~$m$, we eventually arrive to an exact
complex $0\rarrow Y_n\rarrow Y_{n-1}\rarrow \dotsb \rarrow Y_0
\rarrow P\rarrow0$ in $\sZ^0(\bE)$ for which the composition
$\Tot(Y_\bu)\rarrow P\rarrow X$ vanishes in $\sH^0(\bE)$.

 Finally, given an absolutely acyclic object $X\in\Ac^\abs(\bE)$,
we take $P=X$ and $f=\id_X$, and find an exact complex
$0\rarrow Y_n\rarrow Y_{n-1}\rarrow \dotsb \rarrow Y_0
\rarrow X\rarrow0$ in $\sZ^0(\bE)$ for which the natural closed
morphism $\Tot(Y_\bu)\rarrow X$ is homotopic to zero.
 Then the argument spelled out at the very end of the proof of
Proposition~\ref{abs-acyclic-as-totalizations} shows that the object $X$
is a direct summand of the object $\Tot(Y_\bu\to X)$ in $\sH^0(\bE)$.

 It remains to apply
Lemma~\ref{totalization-of-finite-length-n-has-dimension-n} to
the effect that $\Tot(Y_\bu\to X)\in\Ac^0(\bE)^{*n}$ and
therefore $X\in\langle\Ac^0(\bE)\rangle_n$.
\end{proof}

\subsection{Examples}
 In the specific case of the exact DG\+category of complexes in
an exact category, as per Example~\ref{exact-dg-category-of-complexes},
a result stronger than the respective specialization of
Theorem~\ref{finite-homological-dimension-main-theorem} is long known.
 In this case, the absolute derived category coincides not only with
the coderived and/or contraderived category, but also with
the conventional unbounded derived category;
see~\cite[Remark~2.1]{Psemi}.

 Let us formulate here the particular cases of
Theorem~\ref{finite-homological-dimension-main-theorem} arising in
the context of Examples~\ref{abelian-dg-category-of-cdg-modules}\+-%
\ref{exact-dg-category-of-factorizations}
and~\ref{exact-DG-pairs-examples}.

\begin{cor}
 Let $\biR^\cu=(R^*,d,h)$ be a CDG\+ring.
 Let\/ $\sK\subset R^*\smodl$ be a full subcategory as in
the first paragraph of
Corollary~\ref{cdg-ring-finite-resol-dim-corollary}, and let\/
$\bE\subset\biR^\cu\bmodl$ the related full DG\+subcategory
endowed with the inherited exact DG\+category structure.
 Assume that the exact category\/ $\sK$ has finite homological
dimension. \par
 \textup{(a)} If the full subcategory\/ $\sK$ is preserved by
infinite direct sums in $R^*\smodl$, then all coacyclic objects
in\/ $\sH^0(\bE)$ are absolutely acyclic, that is\/
$\sA^\co(\bE)=\sA^\abs(\bE)$, and therefore the coderived 
category of\/ $\bE$ coincides with the absolute derived category,
$$
 \sD^\abs(\bE)=\sD^\co(\bE).
$$ \par
 \textup{(b)} If the full subcategory\/ $\sK$ is preserved by
infinite products in $R^*\smodl$, then all contraacyclic objects
in\/ $\sH^0(\bE)$ are absolutely acyclic, that is\/
$\sA^\ctr(\bE)=\sA^\abs(\bE)$, and therefore the contraderived 
category of\/ $\bE$ coincides with the absolute derived category,
$$
 \sD^\abs(\bE)=\sD^\co(\bE).
$$
\end{cor}

\begin{proof}
 Similar to the proofs of
Corollaries~\ref{cdg-ring-finite-resol-dim-corollary}
and~\ref{cdg-ring-full-and-faithfulness-corollary}.
\end{proof}

 The next corollary is a generalization of~\cite[Theorem~1.6]{EP}.

\begin{cor}
 Let $X$ be a scheme and $\biB^\cu=(B^*,d,h)$ be a quasi-coherent
CDG\+quasi-algebra over~$X$.
 Let $\sK\subset B^*\sqcoh$ be a full subcategory as in
the first paragraph of
Corollary~\ref{cdg-quasi-algebra-finite-resol-dim-corollary},
and let\/ $\bE\subset\biB^\cu\bqcoh$ the related full DG\+subcategory
endowed with the inherited exact DG\+category structure. \par
 Assume that the full subcategory\/ $\sK$ is preserved by
infinite direct sums in $B^*\sqcoh$ and the exact category\/ $\sK$
has finite homological dimension.
 Then all coacyclic objects in\/ $\sH^0(\bE)$ are absolutely acyclic,
that is\/ $\sA^\co(\bE)=\sA^\abs(\bE)$, and therefore the coderived 
category of\/ $\bE$ coincides with the absolute derived category,
$$
 \sD^\abs(\bE)=\sD^\co(\bE).
$$
\end{cor}

\begin{proof}
 Similar to the proofs of
Corollaries~\ref{cdg-quasi-algebra-finite-resol-dim-corollary}
and~\ref{cdg-quasi-algebra-full-and-faithfulness-corollary}.
\end{proof}

 The following corollary is a generalization
of~\cite[Corollary~2.3(d,f)]{EP}.

\begin{cor} \label{factorization-finite-homol-dim-corollary}
 Let\/ $\sE$ be an exact category and\/ $\Lambda\:\sE\rarrow\sE$ be
an autoequivalence preserving and reflecting short exact sequences.
 Assume that either\/ $\Gamma=\boZ$, or\/ $\Lambda$ is involutive
and\/ $\Gamma=\boZ/2$.
 Let $w\:\Id_\sE\rarrow\Lambda^2$ be a potential (as in
Section~\ref{factorizations-subsecn}).
 Assume that the exact category\/ $\sE$ has finite homological
dimension and exact coproducts.
 Then all coacyclic objects in the DG\+category of factorizations\/
$\bF(\sE,\Lambda,w)$ are absolutely acyclic, that is\/
$\Ac^\co(\bF(\sE,\Lambda,w))=\Ac^\abs(\bF(\sE,\Lambda,w))$,
and therefore the coderived category of\/ $\bF(\sE,\Lambda,w)$
coincides with the absolute derived category,
$$
 \sD^\abs(\bF(\sE,\Lambda,w))=\sD^\co(\bF(\sE,\Lambda,w)).
$$
\end{cor}

\begin{proof}
 Similar to the proofs of
Corollaries~\ref{factorizations-finite-resol-dim-corollary}
and~\ref{factorizations-full-and-faithfulness-corollary}.
\end{proof}

\subsection{Generalizations}
 The following result, whose idea goes back
to~\cite[Proposition~A.6.1]{Pcosh}, is a generalization of
Theorem~\ref{finite-homological-dimension-main-theorem} with
a flavor of Theorem~\ref{full-and-faithfulness-main-theorem}.

\begin{thm} \label{finite-homol-dim-full-and-faithful-theorem}
 Let $(\bE,\sK)$ be an exact DG\+pair and $(\bF,\sL)\subset(\bE,\sK)$
be an exact DG\+subpair.
 Assume that the exact DG\+category\/ $\bE$ has exact products,
the full subcategory\/ $\sL$ is self-resolving in the exact category\/
$\sK$, and the exact category\/ $\sL$ has finite homological dimension.
 Then the triangulated functor
$$
 \sD^\abs(\bF)\lrarrow\sD^\ctr(\bE)
$$
induced by the inclusion of exact DG\+categories\/ $\bF\rightarrowtail
\bE$ is fully faithful.
\end{thm}

 Setting $\bF=\bE$ and $\sL=\sK$, one obtains
Theorem~\ref{finite-homological-dimension-main-theorem}(b) as
a particular case of
Theorem~\ref{finite-homol-dim-full-and-faithful-theorem}.

\begin{proof}
 When the full DG\+subcategory $\bF$ is closed under products in $\bE$,
one can use Theorem~\ref{finite-homological-dimension-main-theorem}(b)
to the effect that $\sD^\abs(\bF)=\sD^\ctr(\bF)$, and
Theorem~\ref{full-and-faithfulness-main-theorem}(a) tells that
the functor $\sD^\ctr(\bF)\rarrow\sD^\ctr(\bE)$ is fully faithful;
so the desired assertion follows.
 The proof in the general case is a combination of the arguments
proving the two mentioned previous theorems.

 According to Lemma~\ref{approachable-implies-fully-faithful-lemma},
it suffices to show that all objects of $\Ac^\ctr(\bE)$ are
approachable from $\sH^0(\bF)$ via $\Ac^\abs(\bF)$ in $\sH^0(\bE)$.
 By Proposition~\ref{totalizations-approachable-prop}(a), all objects
of $\Ac^0(\bE)$ are approachable from $\sH^0(\bF)$ via $\Ac^0(\bF)$,
hence also via $\Ac^\abs(\bF)$.
 According to Lemma~\ref{shifts-cones-approachable-lemma},
the class of all objects approachable from $\sH^0(\bF)$ via
$\Ac^\abs(\bF)$ is a full triangulated subcategory in $\sH^0(\bE)$.
 It remains to show that the class of all objects approachable
from $\sH^0(\bF)$ via $\Ac^\abs(\bF)$ is closed under infinite
products in $\sH^0(\bE)$. 

 Let $F\rarrow\prod_\alpha X_\alpha$ be a morphism in $\sH^0(\bE)$
with $F\in\sH^0(\bF)$ and all the objects $X_\alpha\in\sH^0(\bE)$
approachable from $\sH^0(\bF)$ via $\Ac^\abs(\bF)$.
 Consider the components $F\rarrow X_\alpha$ of the morphism
$F\rarrow\prod_\alpha X_\alpha$.
 Then for every~$\alpha$ there exists an object
$Y_\alpha\in\Ac^\abs(\bF)$ such that the morphism $F\rarrow X_\alpha$
factorizes through~$Y_\alpha$.
 It follows that the morphism $F\rarrow\prod_\alpha X_\alpha$
factorizes through $\prod_\alpha Y_\alpha$ in $\sH^0(\bE)$.

 According to
Proposition~\ref{finite-homol-dim-abs-acycl-strongly-generated-prop},
we have $\Ac^\abs(\bF)=\langle\Ac^0(\bF)\rangle_n$, where
$n$~is the homological dimension of the exact category~$\sL$.
 Furthermore, $\langle\Ac^0(\bF)\rangle_n\subset
\langle\Ac^0(\bE)\rangle_n$, and the class $\langle\Ac^0(\bE)\rangle_n$
is closed under products in $\sH^0(\bE)$ by
Lemma~\ref{strongly-generated-closed-under-co-products}(b).
 Thus $\prod_\alpha Y_\alpha\in\langle\Ac^0(\bE)\rangle_n$;
and by the definition $\langle\Ac^0(\bE)\rangle_n\subset\Ac^\abs(\bE)$.

 Following the proof of
Theorem~\ref{full-and-faithfulness-main-theorem}(a), all objects
of $\Ac^\abs(\bE)$ are approachable from $\sH^0(\bF)$ via
$\Ac^\abs(\bF)$ in $\sH^0(\bE)$.
 Therefore, any morphism $F\rarrow\prod_\alpha Y_\alpha$ factorizes
through an object from $\Ac^\abs(\bF)$, and we are done.
\end{proof}

 The next result, whose idea goes back
to~\cite[Corollary~A.6.2 and the subsequent paragraph]{Pcosh},
is a common generalization of
Theorems~\ref{inj-proj-resolutions-star-conditions-theorem}
and~\ref{finite-homological-dimension-main-theorem} with a flavor of
Theorem~\ref{contraderived-resolving-equivalence-theorem}.

\begin{thm} \label{finite-homol-dim-star-condition-theorem}
 Let $(\bE,\sK)$ be an exact DG\+pair and $(\bF,\sL)\subset(\bE,\sK)$
be a strict exact DG\+subpair.
 Assume that the exact DG\+category\/ $\bE$ has twists and
exact products, the full subcategory\/ $\sL$ is resolving in the exact
category\/ $\sK$, and the exact category\/ $\sL$ has finite homological
dimension.
 Furthermore, assume that there is a resolving subcategory\/ $\sM
\subset\sK$ preserved by the shift functors\/~$[n]$, \,$n\in\Gamma$
such that countable products of objects from\/ $\sM$ in\/
$\sZ^0(\bE^\bec)$ belong to\/ $\sK$ and have finite\/ $\sL$\+resolution
dimensions in\/~$\sK$.
 Then the triangulated functor
$$
 \sD^\abs(\bF)\lrarrow\sD^\ctr(\bE)
$$
induced by the inclusion of exact DG\+categories\/ $\bF\rightarrowtail
\bE$ is a triangulated equivalence.
\end{thm}

 Taking $\bF=\bE_\bproj$ and $\sM=\sL=\sZ^0(\bE)_\proj$, one obtains
Theorem~\ref{inj-proj-resolutions-star-conditions-theorem}(b) as
a particular case of
Theorem~\ref{finite-homol-dim-star-condition-theorem}.
 Setting $\bF=\bE$ and $\sM=\sL=\sK$, one (almost) recovers
Theorem~\ref{finite-homological-dimension-main-theorem}(b) as
a particular case of
Theorem~\ref{finite-homol-dim-star-condition-theorem}.

\begin{proof}
 The triangulated functor $\sD^\abs(\bF)\rarrow\sD^\ctr(\bE)$ is
fully faithful by
Theorem~\ref{finite-homol-dim-full-and-faithful-theorem}.
 So it suffices to find for any object $E\in\bE$ an object
$F\in\bF$ together with a morphism $F\rarrow E$ in $\sH^0(\bE)$
with a cone belonging to $\Ac^\ctr(\bE)$.
 This is provable similarly to the proof of
Theorem~\ref{inj-proj-resolutions-star-conditions-theorem}(b).

 Specifically, choose an object $M_0\in\sM$ together with
an admissible epimorphism $M_0\rarrow\Phi_\bE^K(E)$ in~$\sK$.
 By Lemma~\ref{mono-epi-exact-DG-adjunction}(b), the corresponding
morphism $\Psi^+(M_0)\rarrow E$ is an admissible epimorphism
in $\sZ^0(\bE)$.
 Similarly to the proof of
Lemma~\ref{one-step-inj-proj-resolution-lemma}(b), we have a natural
short exact sequence $0\rarrow M_0\rarrow\Phi\Psi^+(M_0)\rarrow M_0[1]
\rarrow0$ in the exact category~$\sK$.
 Since the full subcategory $\sM\subset\sK$ is closed under shifts
and extensions, it follows that $\Phi\Psi^+(M_0)\in\sM$.

 Put $N_0=\Psi^+(M_0)$.
 Then we have a short exact sequence $0\rarrow E_1\rarrow N_0
\rarrow E\rarrow0$ in $\sZ^0(\bE)$ with $\Phi(N_0)\in\sM$.
 The same construction produces a short exact sequence
$0\rarrow E_2\rarrow N_1\rarrow E_1\rarrow0$ in $\sZ^0(\bE)$ with
$\Phi(N_1)\in\sM$, etc.
 Proceeding in this way, we construct an exact complex
$\dotsb\rarrow N_2\rarrow N_1\rarrow N_0\rarrow E\rarrow0$ in
$\sZ^0(\bE)$ with $\Phi(N_i)\in\sM$.

 Denote by $T=\Tot^\sqcap(N_\bu)$ the product totalization of
the complex $\dotsb\rarrow N_2\rarrow N_1\rarrow N_0\rarrow0$
in the DG\+category~$\bE$.
 By Lemma~\ref{twists-into-isomorphisms}, the object $\Phi(T)$
is isomorphic to the product $\prod_{i=0}^\infty\Phi(N_i)[-i]$
in the additive category $\sZ^0(\bE^\bec)$.
 By assumption, it follows that the object $\Phi(T)$ has finite
$\sL$\+resolution dimension in~$\sK$.
 Denote this resolution dimension by~$n$.
 The cone of the natural closed morphism $T\rarrow E$ in $\bE$
is contraacyclic by
Lemma~\ref{totalization-of-one-sided-bounded-co-contra-acyclic}(b).

 Choose an object $L_0\in\sL$ together with an admissible
epimorphism $L_0\rarrow\Phi(T)$ in~$\sK$.
 By Lemma~\ref{mono-epi-exact-DG-adjunction}(b), the corresponding
morphism $\Psi^+(L_0)\rarrow T$ is an admissible epimorphism
in $\sZ^0(\bE)$.
 Furthermore, we have $F_0=\Psi^+(L_0)\in\bF$.
 So we get a short exact sequence $0\rarrow T_1\rarrow F_0
\rarrow T\rarrow0$ in $\sZ^0(\bE)$ with $F_0\in\bF$.
 The same construction produces a short exact sequence
$0\rarrow T_2\rarrow F_1\rarrow T_1\rarrow0$ in $\sZ^0(\bE)$
with $F_1\in\bF$, etc.

 Proceeding in this way and arguing similarly to the beginning of
the proof of Theorem~\ref{finite-resolution-dimension-main-theorem}(a),
we construct an exact complex $0\rarrow D\rarrow F_{n-1}\rarrow
\dotsb\rarrow F_0\rarrow T\rarrow0$ in $\sZ^0(\bE)$ with $F_i\in\bF$.
 Then the complex $0\rarrow\Phi(D)\rarrow\Phi(F_{n-1})\rarrow\dotsb
\rarrow\Phi(F_0)\rarrow\Phi(T)\rarrow0$ is exact in $\sK$ with
$\Phi(F_i)\in\sL$.
  By Proposition~\ref{resol-dim-well-defined}(a), it follows that
$\Phi(D)\in\sL$ (we recall that $n$~is the $\sL$\+resolution
dimension of $\Phi(T)$ in~$\sK$).

 By the assumption that $(\bF,\sL)$ is a strict exact DG\+subpair
in $(\bE,\sK)$, we can conclude that $D\in\bF$.
 Put $F_n=D$, and denote by $F$ the totalization of the finite
complex $F_n\rarrow F_{n-1}\rarrow\dotsb\rarrow F_0$ in
the DG\+category~$\bF$.
 Then the cone of the natural closed morphism $F\rarrow T$ in $\bE$
is absolutely acyclic
by Lemma~\ref{totalization-of-finite-absolutely-acyclic}.
 We have $F\in\bF$, and the cone of the composition of closed
morphisms $F\rarrow T\rarrow E$ is contraacyclic in~$\bE$.
\end{proof}

\Section{Compact Generation Theorem}  \label{compact-generation-secn}

\subsection{Grothendieck DG-categories}
 We recall from Section~\ref{abelian-dg-categories-subsecn} that
an additive DG\+category $\bA$ with shifts and cones is said to be
\emph{abelian} if both the additive categories $\sZ^0(\bA)$ and
$\sZ^0(\bA^\bec)$ are abelian.
 In fact, according to
Proposition~\ref{A-abelian-implies-A-bec-abelian}, it suffices that
the additive category $\sZ^0(\bA)$ be abelian.
 Any abelian DG\+category carries the abelian exact DG\+category
structure, which consists of the abelian exact structures on both
the abelian categories $\sZ^0(\bA)$ and $\sZ^0(\bA^\bec)$.

 Let $\sC$ be a category.
 We recall that an object $G\in\sC$ is said to be a \emph{generator} of
$\sC$ if for any two parallel morphisms $f'$, $f''\:C\rarrow D$ in
$\sC$ the equation $f'g=f''g$ for all morphisms $g\:G\rarrow C$ in
$\sC$ implies $f'=f''$.
 A generator $G$ is called a \emph{strong generator} if for any
morphism $f\:C\rarrow D$ in $\sC$ bijectivity of the map of sets
$\Hom(G,f)\:\Hom_\sC(G,C)\rarrow\Hom_\sC(G,D)$ implies that
$f$~is an isomorphism (see~\cite[Section~0.6]{AR} for an equivalent
definition).
 Similarly one defines a multi-object generator (a \emph{generating
set of objects}) of a category, as in~\cite{AR}.

\begin{lem} \label{Phi-Psi-preserve-generators}
 Let\/ $\bA$ be an additive DG\+category with shifts and cones. \par
\textup{(a)} Let $A\in\bA$ be an object which is a generator of
the additive category\/ $\sZ^0(\bA)$.
 Then the object\/ $\Phi_\bA(A)\in\bA^\bec$ is a generator of
the additive category\/ $\sZ^0(\bA^\bec)$.
 Besides, the functor\/ $\Phi_\bA$ also takes strong generators to
strong generators. \par
\textup{(b)} Let $X\in\bA^\bec$ be an object which is a generator of
the additive category\/ $\sZ^0(\bA^\bec)$.
 Then the object\/ $\Psi^+_\bA(X)\in\bA$ is a generator of the additive
category\/ $\sZ^0(\bA)$.
 Besides, the functor\/ $\Psi^+_\bA$ also takes strong generators
to strong generators.
\end{lem}

\begin{proof}
 By Lemma~\ref{G-functors-in-DG-categories}, the functors $\Phi$ and
$\Psi^+$ have faithful right adjoint functors ($\Psi^-$ and $\Phi$, 
respectively).
 A functor left adjoint to a faithful functor takes generators to
generators, as one can see immediately from the definitions.
 Furthermore, the functors $\Psi^-$ and $\Phi$ are conservative by
Lemma~\ref{Phi-Psi-conservative}.
 A functor left adjoint to a faithful conservative functor takes
strong generators to strong generators.
\end{proof}

 The definition of the \emph{coproduct} of an (infinite) family
of objects in a DG\+category was given in
Section~\ref{dg-categories-subsecn}.

\begin{lem} \label{coproducts-in-cocycles-and-in-DG}
 Let\/ $\bA$ be an additive DG\+category with shifts and cones.
 Then an object $A\in\bA$ is the coproduct of a family of objects
$A_\alpha\in\bA$ if and only if the object $A\in\sZ^0(\bA)$ is
the coproduct of the family of objects $A_\alpha\in\sZ^0(\bA)$.
 So the DG\+category\/ $\bA$ has coproducts if and only if
the additive category\/ $\sZ^0(\bA)$ does.
\end{lem}

\begin{proof}
 The ``only if'' assertion follows directly from the fact that
the functor assigning to a complex of abelian groups its group of
degree~$0$ cocycles preserves products.
 This implication, mentioned in Section~\ref{dg-categories-subsecn}, 
does not require the assumptions of $\bA$ having shifts or cones.

 To prove the ``if'', assume that an object $A$ is the coproduct of
objects $A_\alpha$ in $\sZ^0(\bA)$.
 Let $\iota_\alpha\:A_\alpha\rarrow A$ be the structure morphisms of
$A$ as the coproduct of $A_\alpha$ in $\sZ^0(\bA)$; so $\iota_\alpha
\in\Hom^0_\bA(A_\alpha,A)$ are closed morphisms of degree~$0$ in
$\bA$, that is $d(\iota_\alpha)=0$.
 Then, for every object $B\in\bA$, the morphisms~$\iota_\alpha$
induce a morphism of complexes of abelian groups
\begin{equation} \label{coproduct-Hom-comparison}
 \Hom_\bA^\bu(A,B)\lrarrow
 \prod\nolimits_\alpha\Hom_\bA^\bu(A_\alpha,B).
\end{equation}
 We need to prove that~\eqref{coproduct-Hom-comparison} is
an isomorphism of complexes of abelian groups.
 For this purpose, it suffices to show that the same map is
an isomorphism of the underlying graded abelian groups of
the complexes in question.
 Replacing $B$ with its shifts $B[n]$, \,$n\in\Gamma$, it remains
to show that~\eqref{coproduct-Hom-comparison} is an isomorphism
of the underling abelian groups in the cohomological degree~$0$.

 The latter fact is established by the observation that the functor
$\Phi_\bA$, being a left adjoint, preserves coproducts, so
$\Phi(A)$ is the coproduct of $\Phi(A_\alpha)$ in $\sZ^0(\bA^\bec)$,
together with Lemma~\ref{Phi-Psi-extended-to-nonclosed-morphisms}.
\end{proof}

 We refer to~\cite[Theorem~1.5]{AR} or~\cite{AN}
for a discussion of \emph{directed} and \emph{filtered} colimits.

\begin{lem} \label{directed-colimits-exact-in-A-iff-in-A-bec}
 Let\/ $\bA$ be an abelian DG\+category with coproducts.
 Then the coproduct functors are exact in the abelian category\/
$\sZ^0(\bA)$ if and only if they are exact in the abelian category\/
$\sZ^0(\bA^\bec)$.
 The directed colimit functors are exact in\/ $\sZ^0(\bA)$ if and only
if they are exact in\/ $\sZ^0(\bA^\bec)$.
\end{lem}

\begin{proof}
 The assertion about coproducts is a particular case of
Lemma~\ref{exact-co-products-lemma}(a).
 The assertion about directed colimits is provable similarly: it follows
from the observation that the functors $\Phi_\bA$ and $\Psi^+_\bA$,
being left adjoints, preserve all colimits, together with the fact that
these functors are exact and faithful (so they take nonzero objects
to nonzero objects).
\end{proof}

 In an abelian category, any generator (or generating set) is strong,
and existence of coproducts implies existence of colimits.
 A \emph{Grothendieck category} is an abelian category with a generator,
infinite coproducts, and exact directed/filtered colimits.

\begin{prop}
 Let\/ $\bA$ be an abelian DG\+category. \par
\textup{(a)} If the abelian category\/ $\sZ^0(\bA)$ is Grothendieck,
then so is the abelian category\/ $\sZ^0(\bA^\bec)$.
 The DG\+category\/ $\bA$ has coproducts in this case. \par
\textup{(b)} If the DG\+category $\bA$ has coproducts and
the abelian category\/ $\sZ^0(\bA^\bec)$ is Grothendieck, then
the abelian category\/ $\sZ^0(\bA)$ is Grothendieck.
\end{prop}

\begin{proof}
 Combine the results of Lemmas~\ref{Phi-Psi-preserve-generators},
\ref{coproducts-in-cocycles-and-in-DG},
and~\ref{directed-colimits-exact-in-A-iff-in-A-bec}.
\end{proof}

 We will say that an abelian DG\+category is \emph{Grothendieck} if
the abelian categories $\sZ^0(\bA)$ and $\sZ^0(\bA^\bec)$ are
Grothendieck.

\begin{exs} \label{Grothendieck-abelian-DG-examples}
 (1)~Let $\sA$ be a Grothendieck abelian category.
 Then the abelian DG\+category $\bC(\sA)$ of (unbounded) complexes
in $\sA$, as in Example~\ref{exact-dg-category-of-complexes}, is
a Grothendieck DG\+category.

 (2)~Let $\biR^\cu=(R^*,d,h)$ be a CDG\+ring.
 Then the abelian DG\+category $\biR^\cu\bmodl$ of left CDG\+modules
over $\biR^\cu$, as in Example~\ref{abelian-dg-category-of-cdg-modules},
is a Grothendieck DG\+category.
 In fact, $\biR^\cu\bmodl$ is even a locally finitely presentable
abelian DG\+category in the sense of the definition in
the next Section~\ref{finitely-generated-presentable-subsecn}.

 (3)~Let $X$ be a scheme and $\biB^\cu=(B^*,d,h)$ be a quasi-coherent
CDG\+quasi-algebra over~$X$.
 Then the abelian DG\+category $\biB^\cu\bqcoh$ of quasi-coherent left
CDG\+modules over $\biB^\cu$, as in
Example~\ref{abelian-dg-category-of-quasi-coh-cdg-modules}, is
a Grothendieck DG\+category.

 (4)~Let $\sA$ be a Grothendieck abelian category and
$\Lambda\:\sA\rarrow\sA$ be an autoequivalence.
 Assume that either $\Gamma=\boZ$, or $\Lambda$ is involutive
and $\Gamma=\boZ/2$.
 Let $w\:\Id_\sA\rarrow\Lambda^2$ be a potential.
 Then the abelian DG\+category of factorizations $\bF(\sA,\Lambda,w)$,
as in Example~\ref{exact-dg-category-of-factorizations},
is a Grothendieck DG\+category.
\end{exs}

\subsection{Finitely generated and finitely presentable objects}
\label{finitely-generated-presentable-subsecn}
 Let $\sC$ be a category with directed colimits.
 An object $P\in\sC$ is said to be \emph{finitely
presentable}~\cite[Definition~1.1]{AR} if the functor $\Hom_\sC(P,{-})$
from $\sC$ to the category of sets preserves directed colimits.
 The object $P$ is said to be \emph{finitely
generated} (\cite[Definitions~1.13(1) and~1.67]{AR} for
$\lambda=\aleph_0$) if the functor $\Hom_\sC(P,{-})$ preserves
the colimits of directed diagrams of monomorphisms.

\begin{lem} \label{Phi-Psi-preserve-generatedness-presentability}
 Let\/ $\bA$ be an additive DG\+category with shifts and cones
such that both the additive categories\/ $\sZ^0(\bA)$ and\/
$\sZ^0(\bA^\bec)$ have directed colimits.
 Then both the additive functors\/ $\Phi_\bA\:\sZ^0(\bA)\rarrow
\sZ^0(\bA^\bec)$ and\/ $\Psi^+_\bA\:\sZ^0(\bA^\bec)\rarrow\sZ^0(\bA)$
preserve finite generatedness and finite presentability of objects.
\end{lem}

\begin{proof}
 Any functor left adjoint to a functor preserving directed colimits
takes finitely generated objects to finitely generated objects and
finitely presentable objects to finitely presentable objects
(notice that any right adjoint functor takes monomorphisms to
monomorphisms).
 Furthermore, both the right adjoint functors $\Psi^-_\bA$ and
$\Phi_\bA$ preserve directed colimits, since they, in turn, have
right adjoints.
\end{proof}

\begin{lem} \label{Phi-Psi-reflect-generatedness-presentability}
 Let\/ $\bA$ be an additive DG\+category with shifts and cones
such that both the additive categories\/ $\sZ^0(\bA)$ and\/
$\sZ^0(\bA^\bec)$ have directed colimits.
 Then both the additive functors\/ $\Phi_\bA\:\sZ^0(\bA)\rarrow
\sZ^0(\bA^\bec)$ and\/ $\Psi^+_\bA\:\sZ^0(\bA^\bec)\rarrow\sZ^0(\bA)$
\emph{reflect} finite generatedness and finite presentability
of objects.
 In other words, an object $A\in\bA$ is finitely generated
(respectively, finitely presentable) in\/ $\sZ^0(\bA)$ whenever
the object\/ $\Phi_\bA(A)$ is finitely generated (resp., finitely
presentable) in\/ $\sZ^0(\bA^\bec)$.
 An object $X\in\bA^\bec$ is finitely generated (resp., finitely
presentable) in\/ $\sZ^0(\bA^\bec)$ whenever the object\/
$\Psi^+_\bA(X)$ is finitely generated (resp., finitely presentable)
in\/ $\sZ^0(\bA)$.
\end{lem}

\begin{proof}
 Let $A\in\bA$ be an object such that the object $\Phi(A)$ is finitely
presentable in $\sZ^0(\bA^\bec)$.
 Then, by Lemma~\ref{Phi-Psi-preserve-generatedness-presentability},
the object $\Psi^+_\bA\Phi_\bA(A)$ is finitely presentable
in $\sZ^0(\bA)$.
 According to Lemma~\ref{Xi-decomposed}, we have
$\Psi^+_\bA\Phi_\bA(A)\simeq\Xi_\bA(A)$.
 By Lemma~\ref{cone-kernel-cokernel}, we have a natural kernel-cokernel
pair of morphisms $A[-1]\rarrow\Xi_\bA(A)\rarrow A$ in $\sZ^0(\bA)$;
hence the object $A$ is the cokernel of a natural morphism $\Xi_\bA(A)
[-1]\rarrow\Xi_\bA(A)$ in $\sZ^0(\bA)$.
 It remains to observe that the class of finitely presentable objects
is preserved by cokernels.
 All the other assertions are provable similarly.
\end{proof}

 A category $\sC$ is said to be \emph{locally finitely presentable} if
it has all colimits and a strong generating set consisting of finitely 
presentable objects~\cite[Definition~1.9 and Theorem~1.11]{AR}.
 The class (or full subcategory) of all finitely presentable objects
in $\sC$ is denoted by $\sC_\fp\subset\sC$.
 Classical references on locally finitely presentable (and more
generally, finitely accessible) additive categories include
the papers~\cite{Len,CB,Kr0}; see also the discussion
in~\cite[Section~8.1]{PS5}.

\begin{prop}
 Let\/ $\bA$ be an additive DG\+category with shifts and cones. \par
\textup{(a)} If the additive category\/ $\sZ^0(\bA)$ is locally\/
finitely presentable, then so is the additive category\/
$\sZ^0(\bA^\bec)$. \par
\textup{(b)} If the additive category\/ $\sZ^0(\bA)$ has all
colimits and the additive category\/ $\sZ^0(\bA^\bec)$ is locally\/
finitely presentable, then the additive category\/ $\sZ^0(\bA)$
is locally finitely presentable as well.
\end{prop}

\begin{proof}
 Let us show that the additive category $\sZ^0(\bA^\bec)$ has colimits
whenever the additive category $\sZ^0(\bA)$ does.
 Indeed, the DG\+category $\bA$ has coproducts in this case by
Lemma~\ref{coproducts-in-cocycles-and-in-DG}, hence so does
the DG\+category $\bA^\bec$ and the additive category $\sZ^0(\bA^\bec)$.
 Furthermore, by Lemma~\ref{Psi-creates-co-kernels} the additive
category $\sZ^0(\bA^\bec)$ has cokernels whenever the additive
category $\sZ^0(\bA)$ does.
 Therefore, all colimits exist in $\sZ^0(\bA^\bec)$.

 Now we can assume that both the categories $\sZ^0(\bA)$ and
and $\sZ^0(\bA^\bec)$ have colimits.
 Lemma~\ref{Phi-Psi-preserve-generators} has an obvious version for
single generators replaced with generating sets, which can be applied
together with Lemma~\ref{Phi-Psi-preserve-generatedness-presentability}
to prove that the category $\sZ^0(\bA)$ has a set of finitely
presentable generators if and only if the category
$\sZ^0(\bA^\bec)$ does.
\end{proof}

 An additive DG\+category $\bA$ with shifts and cones is said to be
\emph{locally finitely presentable} if both the additive categories
$\sZ^0(\bA)$ and $\sZ^0(\bA^\bec)$ are locally finitely presentable.
 Similarly one can define the notion of a \emph{locally
$\lambda$\+presentable DG\+category} for a regular cardinal~$\lambda$
(based on the notion of a locally $\lambda$\+presentable
category~\cite[Definition~1.17 and Theorem~1.20]{AR}).
 We refer to~\cite[Section~6.2]{PS5} for a detailed discussion of
locally presentable DG\+categories.

 Let $\sA$ be a Grothendieck abelian category.
 Then an object of $\sA$ is finitely generated if and only if it
cannot be produced as the union of an infinite ascending chain of its
subobjects.
 Equivalently, an object $A\in\sA$ is finitely generated if and only if,
whenever $A$ is the sum of a family of its subobjects $(A_\alpha)$,
there exists a finite subfamily in $(A_\alpha)$ whose sum is also equal
to~$A$.
 The full subcategory $\sA_\fg\subset\sA$ of all finitely generated
objects in $\sA$ is closed under quotients and extensions.
 All finitely generated objects are quotients of finite direct sums of
copies of the generator of the category; so up to isomorphism, there is
only a set of objects in~$\sA_\fg$.

 We will only define \emph{local finite generatedness} in the context
of \emph{abelian} categories, referring to~\cite[Section~1.E]{AR} for
the general case and taking into account the known fact that any
locally finitely generated abelian category is Grothendieck (see,
e.~g., \cite[Corollary~9.6]{PS1} for a more general result).
 A Grothendieck abelian category is said to be \emph{locally finitely
generated} if it has a generating set consisting of finitely generated
objects.
 Any object in a locally finitely generated abelian category is
the sum of its finitely generated subobjects.

\begin{prop} \label{loc-fin-gen-abelian-DG-category-prop}
 Let\/ $\bA$ be a Grothendieck abelian DG\+category.
 The the abelian category\/ $\sZ^0(\bA)$ is locally finitely generated
if and only if the category\/ $\sZ^0(\bA)$ is locally
finitely generated.
\end{prop}

\begin{proof}
 Follows from Lemmas~\ref{Phi-Psi-preserve-generators}
and~\ref{Phi-Psi-preserve-generatedness-presentability}.
\end{proof}

 A Grothendieck abelian DG\+category is said to be \emph{locally
finitely generated} if it satisfies the equivalent conditions of
Proposition~\ref{loc-fin-gen-abelian-DG-category-prop}.

\subsection{Exactly embedded full abelian DG\+subcategories}
 Let $\sA$ be an abelian category and $\sB\subset\sA$ be a full
subcategory.

 Assume that $\sB$ is an abelian category.
 We will say that $\sB$ is an \emph{exactly embedded full abelian
subcategory} in $\sA$ if the inclusion $\sB\rightarrowtail\sA$ is
an exact functor between abelian categories.
 Equivalently, this means that $\sB$ inherits an exact category
structure from the abelian exact category structure of $\sA$ and
the inherited exact category structure on $\sB$ is the abelian
exact category structure.

 A full subcategory $\sB$ in an abelian category $\sA$ is an exactly
embedded full abelian subcategory if and only if it is closed under
finite direct sums, kernels, and cokernels.

 Let $\bA$ be an abelian DG\+category, and $\bB\subset\bA$ be
a full additive DG\+subcategory closed under shifts and cones.
 Assume that $\bB$ is an abelian DG\+category.

 We will say that $\bB$ is an \emph{exactly embedded full abelian
DG\+subcategory} in $\bA$ if both the fully faithful inclusions
$\sZ^0(\bB)\rightarrowtail\sZ^0(\bA)$ and $\sZ^0(\bB^\bec)
\rightarrowtail\sZ^0(\bA^\bec)$ are exact functors between abelian
categories.
 Equivalently, this means that the full DG\+subcategory $\bB$ inherits
an exact DG\+category structure from the abelian exact DG\+category
structure of $\bA$ and the inherited exact DG\+category structure on
$\bB$ is the abelian exact DG\+category structure.

 A more detailed discussion of exactly embedded full abelian
DG\+subcategories in abelian DG\+categories can be found
in~\cite[Section~3.3]{PS5}; see, in particular,
\cite[Proposition~3.11]{PS5}.

\subsection{Locally Noetherian DG-categories}
\label{locally-Noetherian-DG-categories-subsecn}
 Let $\sA$ be a Grothendieck abelian category.
 An object $N\in\sA$ is said to be \emph{Noetherian} if it has no
infinite ascending chain of subobjects.
 Equivalently, an object in $\sA$ is Noetherian if and only if all
its subobjects are finitely generated.
 The class of all Noetherian objects in $\sA$ is closed under
subobjects, quotients, and extensions; so the full subcategory of
Noetherian objects is an exactly embedded full abelian subcategory
in~$\sA$.

\begin{lem} \label{Phi-Psi-preserve-reflect-Noetherianity}
 Let\/ $\bA$ be a Grothendieck abelian DG\+category.
 Then both the additive functors\/ $\Phi_\bA\:\sZ^0(\bA)\rarrow
\sZ^0(\bA^\bec)$ and\/ $\Psi^+_\bA\:\sZ^0(\bA^\bec)\rarrow\sZ^0(\bA)$
\emph{preserve} and \emph{reflect} Noetherianity of objects.
 The class of all Noetherian objects in\/ $\sZ^0(\bA)$ is preserved
by shifts and twists (hence also by cones).
\end{lem}

\begin{proof}
 Any faithful exact functor reflects Noetherianity of objects,
because it takes an infinite ascending chain of subobjects to
an infinite ascending chain of subobjects.
 Now let $A\in\sZ^0(\bA)$ be a Noetherian object.
 By Lemma~\ref{cone-kernel-cokernel}, we have a short exact sequence
$0\rarrow A[-1]\rarrow\Xi_\bA(A)\rarrow A\rarrow0$ in the abelian
category $\sZ^0(\bA)$, implying that $\Xi_\bA(A)$ is also a Noetherian
object in $\sZ^0(\bA)$.
 According to Lemma~\ref{Xi-decomposed}, we have $\Xi_\bA(A)\simeq
\Psi^+_\bA\Phi_\bA(A)$.
 Since the functor $\Psi^+_\bA$ reflects Noetherianity, we can conclude
that $\Phi_\bA(A)$ is a Noetherian object in $\sZ^0(\bA^\bec)$.
 The preservation of Noetherianity by the functor $\Psi^+_\bA$ is
provable similarly.

 Concerning the second assertion of the lemma, the preservation by
shifts is obvious, and the preservation by cones can be deduced
from the result obtained in the proof of
Lemma~\ref{cone-kernel-cokernel}.
 More generally, the preservation of Noetherianity by twists follows
from the facts that the functor $\Phi_\bA$ preserves and reflects
Noetherianity and takes twists to isomorphisms (see
Lemma~\ref{twists-into-isomorphisms}).
\end{proof}

 A Grothendieck abelian category $\sA$ is said to be \emph{locally
Noetherian} if it has a generating set of Noetherian objects.
 Equivalently, this means that any object of $\sA$ is the sum of its
Noetherian subobjects.
 An object of a locally Noetherian category is finitely generated
if and only if it is finitely presentable and if and only if it is
Noetherian.

\begin{prop} \label{locally-Noetherian-DG-category-prop}
 Let\/ $\bA$ be a Grothendieck abelian DG\+category.
 Then the category\/ $\sZ^0(\bA)$ is locally Noetherian if and only if
the category\/ $\sZ^0(\bA^\bec)$ is locally Noetherian.
\end{prop}

\begin{proof}
 Follows from Lemmas~\ref{Phi-Psi-preserve-generators}
and~\ref{Phi-Psi-preserve-reflect-Noetherianity}.
\end{proof}

 We will say that a DG\+category $\bA$ is \emph{locally Noetherian}
if it satisfies the equivalent conditions of
Proposition~\ref{locally-Noetherian-DG-category-prop}.
 Assuming this is the case, consider the full DG\+sub\-cat\-e\-gory
$\bA_\bfg\subset\bA$ whose objects are all the finitely generated
(equivalently, Noetherian) objects in $\sZ^0(\bA)$, that is
$\sZ^0(\bA_\bfg)=\sZ^0(\bA)_\fg$.
 It follows from Lemma~\ref{Phi-Psi-preserve-reflect-Noetherianity}
that the objects of $\sZ^0((\bA_\bfg)^\bec)$ are all the finitely
generated (equivalently, Noetherian) objects of $\sZ^0(\bA^\bec)$
(so the notation $\bA_\bfg^\bec$ is unambigous).
 The full DG\+subcategory of Noetherian objects $\bA_\bfg$ is
an exactly embedded full abelian DG\+subcategory in~$\bA$.

\begin{exs} \label{locally-Noetherian-examples}
 (1)~Let $\biR^\cu=(R^*,d,h)$ be a CDG\+ring and $\bA=\biR^\cu\bmodl$ be
the DG\+category of left CDG\+modules over~$\biR^\cu$.
 Consider the DG\+ring $\widehat\biR^\bu=(\widehat\biR^*,\d)$
from Section~\ref{cdg-revisited-subsecn}; so $\widehat\biR^*$ is
the graded ring $R^*[\delta]$ with the sign of the grading changed.

 According to Example~\ref{CDG-ring-bec-example}, the abelian
category $\sZ^0(\bA^\bec)$ is equivalent to $R^*\smodl$.
 On the other hand, it was mentioned in the proof of
Proposition~\ref{G-plus-minus-prop} that the abelian category
$\sZ^0(\bA)$ is obviously equivalent to $R^*[\delta]\smodl$.

 The category $R^*\smodl$ is locally Noetherian if and only if
the graded ring $R^*$ is graded left Noetherian.
 Similarly, the category $R^*[\delta]\smodl$ is locally Noetherian
if and only if the graded ring $R^*[\delta]$ is graded left
Noetherian.
 Hence Proposition~\ref{locally-Noetherian-DG-category-prop} provides
a fancy proof of the assertion that the graded ring $R^*[\delta]$
is graded left Noetherian if and only if the graded ring $R^*$~is.
 These equivalent conditions characterize CDG\+rings $\biR^\cu$
for which the DG\+category $\biR^\cu\bmodl$ is locally Noetherian.

 Given a regular cardinal~$\lambda$, one can similarly show that
every homogeneous left ideal in $R^*$ has~$<\lambda$ generators
if and only if every homogeneous left ideal in $R^*[\delta]$
has~$<\lambda$ generators.

\smallskip
 (2)~Let $X$ be a quasi-compact quasi-separated scheme and $B^*$
be a quasi-coherent graded quasi-algebra over~$X$.
 Then the Grothendieck abelian category $B^*\sqcoh$ of quasi-coherent
graded left modules over $B^*$ is locally Noetherian if and only if,
for every affine open subscheme $U\subset X$, the graded ring
$B^*(U)$ is graded left Noetherian.
 If this is the case, the quasi-coherent graded quasi-algebra $B^*$
is said to be \emph{graded left Noetherian}.
 Furthermore, a quasi-coherent graded left $B^*$\+module $M^*$ is
Noetherian if and only if, for every affine open subscheme
$U\subset X$, the graded $B^*(U)$\+module $M^*(U)$ is Noetherian.

 Let $\biB^\cu=(B^*,d,h)$ be a quasi-coherent CDG\+quasi-algebra
over $X$ and $\bA=\biB^\cu\bqcoh$ be the DG\+category of quasi-coherent
left CDG\+modules over~$\biB^\cu$.
 According to Example~\ref{qcoh-CDG-bec-example}, the abelian
category $\sZ^0(\bA^\bec)$ is equivalent to $B^*\sqcoh$.
 Thus the Grothendieck DG\+category $\biB^\cu\bqcoh$ is locally Noetherian
if and only if the quasi-coherent graded quasi-algebra $B^*$ is graded
left Noetherian.
\end{exs}

\subsection{Locally coherent DG-categories}
\label{locally-coherent-DG-categories-subsecn}
 Let $\sA$ be a locally finitely generated abelian category.
 A finitely generated object $A\in\sA$ is called \emph{coherent} if,
for any morphism $f\:B\rarrow A$ to $A$ from a finitely generated
object $B\in\sA$, the kernel of~$f$ is also finitely generated.
 The class of all coherent objects in $\sA$ is closed under kernels,
cokernels, and extensions.

\begin{lem} \label{Phi-Psi-preserve-reflect-coherence}
 Let\/ $\bA$ be a locally finitely generated abelian DG\+category.
 Then both the additive functors\/ $\Phi_\bA\:\sZ^0(\bA)\rarrow
\sZ^0(\bA^\bec)$ and\/ $\Psi^+_\bA\:\sZ^0(\bA^\bec)\rarrow\sZ^0(\bA)$
\emph{preserve} and \emph{reflect} coherence of objects.
 The class of all coherent objects in\/ $\sZ^0(\bA)$ is preserved
by shifts and twists (hence also by cones).
\end{lem}

\begin{proof}
 Let us show that the functor $\Phi_\bA$ reflects coherence.
 Let $A\in\bA$ be an object for which the object $\Phi(A)$ is
coherent in $\sZ^0(\bA^\bec)$.
 Let $B\in\sZ^0(\bA)$ be a finitely generated object and
$f\:B\rarrow A$ be a morphism in $\sZ^0(\bA)$.
 Put $K=\ker(f)\in\sZ^0(\bA)$.
 By Lemma~\ref{Phi-Psi-preserve-generatedness-presentability},
the object $\Phi(B)$ is finitely generated in $\sZ^0(\bA^\bec)$.
 Hence the kernel of the morphism $\Phi(f)\:\Phi(B)\rarrow\Phi(A)$
is finitely generated in $\sZ^0(\bA^\bec)$ as well.
 Since the functor $\Phi_\bA$ is exact, we have $\ker\Phi(f)=
\Phi(K)$, so the object $\Phi(K)$ is finitely generated.
 Using Lemma~\ref{Phi-Psi-reflect-generatedness-presentability},
we can conclude that the object $K\in\sZ^0(\bA)$ is finitely
generated.
 Thus the object $A\in\bA$ is coherent.

 The rest of the proof is similar to that of the proof of
Lemma~\ref{Phi-Psi-preserve-reflect-Noetherianity} and based on
the fact that the class of all coherent objects is closed under
extensions.
 For an alternative exposition, see~\cite[Lemmas~8.5\+-8.6]{PS5}.
\end{proof}

 A locally finitely generated abelian category is said to be
\emph{locally coherent} if it has a generating set consisting of
coherent objects~\cite[Section~2]{Ro}.
 Any locally coherent category is locally finitely presentable.
 An object of a locally coherent category is coherent if and only if
it is finitely presentable.
 We refer to~\cite[Section~13]{PS3} and~\cite[Section~8.2]{PS5} for
a further discussion. 

\begin{prop} \label{locally-coherent-DG-category-prop}
 Let\/ $\bA$ be a locally finitely generated abelian DG\+category.
 Then the category\/ $\sZ^0(\bA)$ is locally coherent if and only if
the category\/ $\sZ^0(\bA^\bec)$ is locally coherent.
\end{prop}

\begin{proof}
 Follows from Lemmas~\ref{Phi-Psi-preserve-generators}
and~\ref{Phi-Psi-preserve-reflect-coherence}.
\end{proof}

 We will say that a DG\+category $\bA$ is \emph{locally coherent}
if it satisfies the equivalent conditions of
Proposition~\ref{locally-coherent-DG-category-prop}.
 If this is the case, consider the full DG\+sub\-cat\-e\-gory
$\bA_\bfp\subset\bA$ whose objects are all the finitely presentable
(equivalently, coherent) objects in $\sZ^0(\bA)$, that is
$\sZ^0(\bA_\bfp)=\sZ^0(\bA)_\fp$.
 It follows from Lemma~\ref{Phi-Psi-preserve-reflect-coherence}
that the objects of $\sZ^0((\bA_\bfp)^\bec)$ are all the finitely
presentable (equivalently, coherent) objects of $\sZ^0(\bA^\bec)$
(so the notation $\bA_\bfp^\bec$ is unambigous).
 The full DG\+subcategory of coherent objects $\bA_\bfp$ is
an exactly embedded full abelian DG\+subcategory in~$\bA$.

\begin{ex} \label{locally-coherent-cdg-ring-example}
 Let $\biR^\cu=(R^*,d,h)$ be a CDG\+ring and $\bA=\biR^\cu\bmodl$ be
the DG\+category of left CDG\+modules over~$\biR^\cu$.
 Continuing the discussion in
Example~\ref{locally-Noetherian-examples}(1), we observe that
the category of graded modules $R^*\smodl$ is locally coherent if
and only if the graded ring $R^*$ is graded left coherent (i.~e.,
all finitely generated homogeneous left ideals in $R^*$ are finitely
presented as left $R^*$\+modules).
 Similarly, the category $R^*[\delta]\smodl$ is locally coherent
if and only if the graded ring $R^*[\delta]$ is graded left
coherent.
 Thus Proposition~\ref{locally-coherent-DG-category-prop} provides
a fancy proof of the assertion that the graded ring $R^*[\delta]$
is graded left coherent if and only if the graded ring $R^*$~is.
 These equivalent conditions characterize CDG\+rings $\biR^\cu$
for which the DG\+category $\biR^\cu\bmodl$ is locally coherent.
 (See also~\cite[Corollary~8.9]{PS5}.)
\end{ex}

\subsection{Small objects}
 Let $\sA$ be an additive category with infinite coproducts,
and let $(A_\upsilon\in\sA)_{\upsilon\in\Upsilon}$ be a family of 
objects in $\sA$ indexed by some set~$\Upsilon$.
 Then, for any subset $\Xi\subset\Upsilon$, the coproduct
$\coprod_{\xi\in\Xi}A_\xi$ can be viewed as a split subobject
of the coproduct $\coprod_{\upsilon\in\Upsilon}A_\upsilon$.
 In other words, there is a natural split monomorphism
$\coprod_{\xi\in\Xi}A_\xi\rarrow
\coprod_{\upsilon\in\Upsilon}A_\upsilon$ in~$\sA$, whose cokernel
is the natural split epimorphism $\coprod_{\upsilon\in\Upsilon}
A_\upsilon\rarrow\coprod_{\zeta\in\Upsilon\setminus\Xi}A_\zeta$.

 For any object $S\in\sA$, there is a natural map of abelian groups
\begin{equation} \label{Hom-into-coproduct-comparison}
 \bigoplus\nolimits_{\upsilon\in\Upsilon}\Hom_\sA(S,A_\upsilon)
 \lrarrow\Hom_\sA(S,\coprod\nolimits_{\upsilon\in\Upsilon}A_\upsilon).
\end{equation}
 Using the fact that the morphism $\coprod_{\xi\in\Xi}A_\xi\rarrow
\coprod_{\upsilon\in\Upsilon}A_\upsilon$ is a split monomorphism
for any finite subset $\Xi\subset\Upsilon$, one can immediately see
that the map~\eqref{Hom-into-coproduct-comparison} is injective
for all objects $S$ and $A_\upsilon\in\sA$.
 
 We will say that an object $S\in\sA$ is \emph{small} if, for any
family of objects $(A_\upsilon\in\sA)_{\upsilon\in\Upsilon}$,
the map~\eqref{Hom-into-coproduct-comparison} is an isomorphism.
 Equivalently, an object $S\in\sA$ is small if and only if, for
every family of objects $(A_\upsilon)_{\upsilon\in\Upsilon}$ and
every morphism $f\:S\rarrow\coprod_{\upsilon\in\Upsilon}A_\upsilon$
in $\sA$, there exists a finite subset $\Xi\subset\Upsilon$ such that
the morphism~$f$ factorizes through the split monomorphism
$\coprod_{\xi\in\Xi}A_\xi\rarrow\coprod_{\upsilon\in\Upsilon}
A_\upsilon$; or equivalently, the composition of~$f$ with the split
epimorphism $\coprod_{\upsilon\in\Upsilon}A_\upsilon\rarrow
\coprod_{\zeta\in\Upsilon\setminus\Xi}A_\zeta$ vanishes.

 Clearly, any finitely generated object in an additive category
with directed colimits is small.
 Small modules are otherwise known as ``dually slender''; and small
objects were called ``weakly finitely generated'' in~\cite{PS1};
we refer to~\cite[Remark~9.4]{PS1} for a further discussion with
references.
 Small objects in triangulated categories are known as \emph{compact}.

\begin{lem} \label{small-objects-closure-properties-lemma}
 Let\/ $\sA$ be an additive category with coproducts. \par
\textup{(a)} If $S\rarrow T$ is an epimorphism in\/ $\sA$ and $S\in\sA$
is a small object, then $T$ is also a small object in\/~$\sA$. \par
\textup{(b)} If $S\overset j\rarrow T\overset k\rarrow R\rarrow0$ is
a composable pair of morphisms in\/ $\sA$ such that $k=\coker j$,
and both $S$ and $R$ are small objects, then $T$ is also a small
object in\/~$\sA$.
\end{lem}

\begin{proof}
 Part~(a) is the particular case of part~(b) for $R=0$.
 To prove part~(b), consider a family of objects
$(A_\upsilon\in\sA)_{\upsilon\in\Upsilon}$ and a morphism
$f\:T\rarrow\coprod_{\upsilon\in\Upsilon}A_\upsilon$ in~$\sA$.
 Consider the composition $S\overset j\rarrow T\overset f\rarrow
\coprod_{\upsilon\in\Upsilon}A_\upsilon$.
 Since the object $S\in\sA$ is small, there exists a finite
subset $\Xi'\subset\Upsilon$ such that the composition $S\overset j
\rarrow T\overset f\rarrow\coprod_{\upsilon\in\Upsilon}A_\upsilon
\rarrow\coprod_{\zeta\in\Upsilon\setminus\Xi'}A_\zeta$ vanishes.
 Since the morphism~$k$ is a cokernel of~$j$, it follows that
the morphism $T\rarrow\coprod_{\zeta\in\Upsilon\setminus\Xi'}A_\zeta$
factorizes through the epimorphism~$k$.

 So we obtain a morphism $g\:R\rarrow
\coprod_{\zeta\in\Upsilon\setminus\Xi'}A_\zeta$.
 Since the object $R\in\sA$ is small, there exists a finite subset
$\Xi''\subset\Upsilon\setminus\Xi'$ such that the composition
$R\overset g\rarrow\coprod_{\zeta\in\Upsilon\setminus\Xi'}A_\zeta
\rarrow\coprod_{\zeta\in\Upsilon\setminus(\Xi'\cup\Xi'')}A_\zeta$
vanishes.
 Now one easily concludes that the composition
$T\overset f\rarrow\coprod_{\zeta\in\Upsilon}A_\upsilon\rarrow
\coprod_{\zeta\in\Upsilon\setminus(\Xi'\cup\Xi'')}A_\zeta$ vanishes,
and it remains to say that the union $\Xi'\cup\Xi''$ of two finite
subsets $\Xi'$ and $\Xi''$ in $\Upsilon$ is finite.
\end{proof}

\begin{lem} \label{Phi-Psi-preserve-reflect-smallness}
 Let\/ $\bA$ be an additive DG\+category with shifts, cones, and
infinite coproducts.
 Then both the additive functors\/ $\Phi_\bA\:\sZ^0(\bA)\rarrow
\sZ^0(\bA^\bec)$ and\/ $\Psi^+_\bA\:\sZ^0(\bA^\bec)\rarrow\sZ^0(\bA)$
\emph{preserve} and \emph{reflect} smallness of objects.
 The class of all small object in\/ $\sZ^0(\bA)$ is preserved
by shifts and twists (hence also by cones).
\end{lem}

\begin{proof}
 Let us show that the functor $\Phi_\bA$ reflects smallness.
 Let $S\in\bA$ be an object for which the object $\Phi(S)$ is small
in $\sZ^0(\bA^\bec)$.
 Let $(A_\upsilon)_{\upsilon\in\Upsilon}$ be a family of objects
in $\bA$ and $S\rarrow\coprod_{\upsilon\in\Upsilon}A_\upsilon$ be
a morphism in $\sZ^0(\bA)$.
 The functor $\Phi_\bA$, being a left adjoint, preserves coproducts;
so we have the induced morphism $\Phi(f)\:\Phi(S)\rarrow
\coprod_{\upsilon\in\Upsilon}\Phi(A_\upsilon)$ in $\sZ^0(\bA^\bec)$.
 By assumption, there exists a finite subset $\Xi\subset\Upsilon$
for which the composition $\Phi(S)\rarrow\coprod_{\upsilon\in\Upsilon}
\Phi(A_\upsilon)\rarrow\coprod_{\zeta\in\Upsilon\setminus\Xi}
\Phi(A_\zeta)$ is zero.
 This composition can be obtained by applying the functor $\Phi$
to the composition $S\rarrow\coprod_{\upsilon\in\Upsilon}A_\upsilon
\rarrow\coprod_{\zeta\in\Upsilon\setminus\Xi}A_\zeta$ of morphisms
in the category $\sZ^0(\bA)$.
 Since the functor $\Phi_\bA$ is faithful, it follows that the latter
composition vanishes in $\sZ^0(\bA)$, as desired.

 The rest of the proof is similar to that of the proofs of
Lemmas~\ref{Phi-Psi-preserve-reflect-Noetherianity}
and~\ref{Phi-Psi-preserve-reflect-coherence}, and uses
Lemmas~\ref{Xi-decomposed}, \ref{cone-kernel-cokernel},
and~\ref{small-objects-closure-properties-lemma}(b).
\end{proof}

\begin{lem} \label{small-in-cocycles-implies-compact-in-cohomology}
 Let\/ $\bA$ be an additive DG\+category with shifts, cones, and
infinite coproducts, and let $S\in\bA$ be an object which is small
in the additive category\/ $\sZ^0(\bA)$.
 Then the object $S$ is also small (compact) in the triangulated
category\/ $\sH^0(\bA)$.
\end{lem}

\begin{proof}
 Let $(A_\upsilon)_{\upsilon\in\Upsilon}$ be a family of objects 
in~$\bA$.
 Then there is a natural map of complexes of abelian groups
\begin{equation} \label{small-object-comparison-map}
 \bigoplus\nolimits_{\upsilon\in\Upsilon}\Hom_\bA^\bu(S,A_\upsilon)
 \lrarrow\Hom_\bA^\bu(S,\coprod\nolimits_{\upsilon\in\Upsilon}
 A_\upsilon).
\end{equation}
induced by the natural closed morphisms $A_\xi\rarrow\coprod_\upsilon
A_\upsilon$ of degree~$0$ in~$\bA$.
 We would like to show that~\eqref{small-object-comparison-map}
is an isomorphism of complexes of abelian groups.

 For this purpose, it suffices to check that the underlying map
of~\eqref{small-object-comparison-map} is an isomorphism of
graded abelian groups.
 Replacing all $A_\upsilon$ with $A_\upsilon[n]$ for a given
$n\in\Gamma$, the question is reduced to showing
that~\eqref{small-object-comparison-map} is an isomorphism on
the grading components of degree zero.

 It remains to interpret the map induced
by~\eqref{small-object-comparison-map} on the grading components
of degree zero as the map
\begin{equation} \label{small-object-degree-zero-comparison-map}
 \bigoplus\nolimits_{\upsilon\in\Upsilon}
 \Hom_{\sZ^0(\bA^\bec)}(\Phi(S),\Phi(A_\upsilon))\lrarrow
 \Hom_{\sZ^0(\bA^\bec)}(\Phi(S),
 \coprod\nolimits_{\upsilon\in\Upsilon}\Phi(A_\upsilon))
\end{equation}
using Lemma~\ref{Phi-Psi-extended-to-nonclosed-morphisms} and
the fact that the functor $\Phi$ preserves coproducts.
 Then the map~\eqref{small-object-degree-zero-comparison-map} is
an isomorphism of abelian groups, since the object $\Phi(S)\in
\sZ^0(\bA^\bec)$ is small by
Lemma~\ref{Phi-Psi-preserve-reflect-smallness}.
 Thus the map~\eqref{small-object-comparison-map} is an isomorphism
on the grading components of degree zero, as desired.

 We have shown that the map~\eqref{small-object-comparison-map}
is an isomorphism of complexes of abelian groups.
 The passage to the degree~$0$ cohomology groups of the complexes
in~\eqref{small-object-comparison-map} proves that $S$ is a small
object in $\sH^0(\bA)$.

 Alternatively, one can simply say that if every morphism
$S\rarrow\coprod_{\upsilon\in\Upsilon}A_\upsilon$ in the category
$\sZ^0(\bA)$ factorizes through $\coprod_{\xi\in\Xi}A_\xi$ for some
finite subset $\Xi\subset\Upsilon$, then the same holds for every
morphism $S\rarrow\coprod_{\upsilon\in\Upsilon}A_\upsilon$ in
the category $\sH^0(\bA)$, because any morphism in $\sH^0(\bA)$
can be lifted to a morphism between the same objects $\sZ^0(\bA)$
and the coproducts in $\sZ^0(\bA)$ and $\sH^0(\bA)$ agree.
\end{proof}

\subsection{Full-and-faithfulness and compactness theorem}
 We start with two lemmas, which are formulated in terms
of the notion of approachability in triangulated categories
(see Section~\ref{approachability-subsecn}).
 The following one should be compared with
Lemma~\ref{co-products-approachable-lemma}.

\begin{lem} \label{coproducts-approachable-from-compacts-lemma}
 Let\/ $\sT$ be a (triangulated) category with infinite coproducts
and\/ $\sS$, $\sY\subset\sT$ be two full subcategories.
 Assume that the full subcategory\/ $\sS$ consists of (some)
compact objects in\/ $\sT$, while the full subcategory\/ $\sY$
is closed under finite direct sums in\/~$\sT$.
 Then the full subcategory of all objects approachable from\/ $\sS$
via\/ $\sY$ is closed under coproducts in\/~$\sT$.
\end{lem}

\begin{proof}
 Let $(X_\upsilon)_{\upsilon\in\Upsilon}$ be a family of objects
in $\sT$ such that, for every $\upsilon\in\Upsilon$, the object
$X_\upsilon$ is approachable from $\sS$ via~$\sY$.
 Let $S\in\sS$ be an object and $f\:S\rarrow
\coprod_{\upsilon\in\Upsilon}X_\upsilon$ be a morphism in~$\sT$.
 Then, since the object $S$ is compact in $\sT$, the morphism~$f$
factorizes as $S\rarrow\bigoplus_{\xi\in\Xi}X_\xi\rarrow
\coprod_{\upsilon\in\Upsilon}X_\upsilon$, where $\Xi$ is
a finite subset of~$\Upsilon$.
 Since every component $S\rarrow X_\xi$, \,$\xi\in\Xi$ factorizes
through an object of $\sY$ and the full subcategory $\sY$ is closed
under finite direct sums in $\sT$, we can conclude that
the morphism~$f$ factorizes through an object of~$\sY$.
\end{proof}

\begin{lem} \label{compacts-in-quotient-category-lemma}
 Let\/ $\sT$ be a triangulated category with infinite coproducts and\/
$\sS$, $\sX\subset\sT$ be (strictly) full triangulated subcategories
such that the full subcategory\/ $\sX$ is closed under coproducts
in\/~$\sT$.
 Let\/ $\sY\subset\sS\cap\sX$ be a full triangulated subcategory in
the intersection.
 Assume that all the objects of\/ $\sS$ are compact in\/ $\sT$ and
all the objects of\/ $\sX$ are approachable from\/ $\sS$ via\/~$\sY$.
 Then the triangulated functor between the Verdier quotient categories
$$
 \sS/\sY\lrarrow\sT/\sX
$$
induced by the inclusion of the triangulated categories\/ $\sS
\rightarrowtail\sT$ is fully faithful, and all the objects in its
image are compact in\/ $\sT/\sX$.
\end{lem}

\begin{proof}
 The full-and-faithfulness assertion is a particular case
of Lemma~\ref{approachable-implies-fully-faithful-lemma}.
 Before proving compactness, we need to explain why coproducts
exist in the quotient category~$\sT/\sX$.
 Indeed, since the triangulated category $\sT$ has coproducts and
the full subcategory $\sX\subset\sT$ is closed under coproducts,
it follows that the quotient category $\sT/\sX$ also has coproducts
and the Verdier quotient functor $\sT\rarrow\sT/\sX$ preserves
coproducts by~\cite[Lemma~1.5]{BN} or~\cite[Lemma~3.2.10]{Neem2}.

 To establish the compactness, we have to recall some details of
the proof of Lemma~\ref{approachable-implies-fully-faithful-lemma}.
 Consider a family of objects $(T_\upsilon\in\sT)_{\upsilon\in\Upsilon}$
and an object $S\in\sS$.
 By the construction of the triangulated Verdier quotient category,
any morphism $S\rarrow\coprod_{\upsilon\in\Upsilon}T_\upsilon$ in
the quotient category $\sT/\sX$ is represented by a fraction
(or ``roof'') $S\larrow D\rarrow\coprod_{\upsilon\in\Upsilon}
T_\upsilon$ of two morphisms $D\rarrow S$ and $D\rarrow
\coprod_{\upsilon\in\Upsilon} T_\upsilon$ in $\sT$, where a cone $X$
of the morphism $D\rarrow S$ belongs to~$\sX$.
 So we have a distinguished triangle $D\rarrow S\rarrow X\rarrow D[1]$
in $\sT$ with $S\in\sS$ and $X\in\sX$.
 By the approachability assumption of the lemma, the morphism
$S\rarrow X$ factorizes as $S\rarrow Y\rarrow X$, where $Y\in\sY$.
 It is important for our argument that $Y\in\sS\cap\sX$.

 Put $G=\cone(S\to Y)[-1]$, so we have a distinguished triangle
$G\rarrow S\rarrow Y\rarrow G[1]$ in $\sS\subset\sT$.
 Then the pair of morphisms $\id_S\:S\rarrow S$ and $Y\rarrow X$
can be completed to a morphism of distinguished triangles in~$\sT$;
so we get a morphism $G\rarrow D$ making the triangle diagram
$G\rarrow D\rarrow S$ commutative.
$$
 \xymatrix{
  G \ar[r] \ar@{..>}[d] & S \ar[r] \ar@{=}[d] & Y \ar[r] \ar[d]
  & G[1] \ar@{..>}[d] \\
  D \ar[r] & S \ar[r] & X \ar[r] & D[1]
 }
$$
 By the octahedron axiom, since both the morphisms $D\rarrow S$ and
$G\rarrow S$ have cones belonging to $\sX$, a cone of the morphism
$G\rarrow D$ also belongs to~$\sX$.

 Now our morphism $S\rarrow\coprod_{\upsilon\in\Upsilon}T_\upsilon$
in $\sT/\sX$ can be represented by the fraction
$$
 S\llarrow D\llarrow G\lrarrow D\lrarrow
 \coprod\nolimits_{\upsilon\in\Upsilon}T_\upsilon
$$
of morphisms in $\sT$ with an object $G\in\sS$.
 Consider the composition $G\rarrow D\rarrow
\coprod_{\upsilon\in\Upsilon}T_\upsilon$ in~$\sT$.
 Since the object $G\in\sS$ is compact in $\sT$, it follows that there
exists a finite subset of indices $\Xi\subset\Upsilon$ such that
the morphism $G\rarrow\coprod_{\upsilon\in\Upsilon}T_\upsilon$
factorizes as $G\rarrow\coprod_{\xi\in\Xi}T_\xi\rarrow
\coprod_{\upsilon\in\Upsilon}T_\upsilon$ in~$\sT$.
 Thus the original morphism $S\rarrow\coprod_{\upsilon\in\Upsilon}
T_\upsilon$ in the quotient category $\sT/\sX$ factorizes as
$S\rarrow\coprod_{\xi\in\Xi}T_\xi\rarrow
\coprod_{\upsilon\in\Upsilon}T_\upsilon$.
\end{proof}

 The following theorem is a generalization
of~\cite[Theorem~3.11.1]{Pkoszul} and~\cite[Corollary~2.6(a)]{Pfp}.

\begin{thm} \label{full-and-faithfulness-and-compactness-main-theorem}
 Let $(\bE,\sK)$ be an exact DG\+pair and $(\bF,\sL)\subset(\bE,\sK)$
be an exact DG\+subpair.
 Assume that the DG\+category\/ $\bE$ has exact coproducts,
the full subcategory\/ $\sK\subset\sZ^0(\bE^\bec)$ is closed under
coproducts, the full subcategory\/ $\sL$ is self-resolving in\/ $\sK$,
and all the objects of\/ $\sL$ are small in\/~$\sK$.
 Then the triangulated functor
$$
 \sD^\abs(\bF)\lrarrow\sD^\co(\bE)
$$
induced by the inclusion of exact DG\+categories\/ $\bF\rightarrowtail
\bE$ is fully faithful, and all the objects in its essential image
are compact in\/ $\sD^\co(\bE)$.
\end{thm}

\begin{proof}
 First of all, let us explain once again why coproducts exist in
the coderived category $\sD^\co(\bE)$ (so the compactness assertion
makes sense).
 Indeed, coproducts exist in the homotopy category $\sH^0(\bE)$
(induced by the coproducts in the DG\+category~$\bE$), and
the full subcategory of coacyclic objects $\Ac^\co(\bE)\subset
\sH^0(\bE)$ is closed under coproducts by the definition.
 By~\cite[Lemma~1.5]{BN} or~\cite[Lemma~3.2.10]{Neem2},
it follows that coproducts exist in the quotient category
$\sD^\co(\bE)=\sH^0(\bE)/\Ac^\co(\bE)$, and the Verdier quotient
functor $\sH^0(\bE)\rarrow\sD^\co(\bE)$ preserves coproducts.

 Furthermore, following the proof of the ``$\Phi_\bA$ reflects
smallness'' assertion in
Lemma~\ref{Phi-Psi-preserve-reflect-smallness}, one can see
that the assumption that all the objects of $\sL$ are small in $\sK$
implies that all the objects of $\sZ^0(\bF)$ are small in $\sZ^0(\bE)$.
 Hence, by Lemma~\ref{small-in-cocycles-implies-compact-in-cohomology},
all objects of $\sH^0(\bF)$ are compact in $\sH^0(\bE)$.

 Let us show that all the objects of the full subcategory $\Ac^\co(\bE)$
are approachable from $\sH^0(\bF)$ via $\Ac^\abs(\bF)$ in $\sH^0(\bE)$.
 The argument follows the proof of
Theorem~\ref{full-and-faithfulness-main-theorem} with suitable
variations.
 By Proposition~\ref{totalizations-approachable-prop}(a), all
the objects of $\Ac^0(\bE)$ are approachable from $\sH^0(\bF)$
via $\Ac^0(\bF)$, hence also via $\Ac^\abs(\bF)$.
 By Lemma~\ref{shifts-cones-approachable-lemma}, the class of all
objects approachable from $\sH^0(\bF)$ via $\Ac^\abs(\bF)$ is
a full triangulated subcategory in $\sH^0(\bE)$.
 Finally, according to
Lemma~\ref{coproducts-approachable-from-compacts-lemma}, the class
of all objects approachable from $\sH^0(\bF)$ via $\Ac^\abs(\bF)$ is
closed under coproducts in $\sH^0(\bE)$.
 We can conclude that all the objects of $\Ac^\co(\bE)$ belong to
this class.

 Thus Lemma~\ref{compacts-in-quotient-category-lemma} is applicable
to $\sT=\sH^0(\bE)$, \ $\sS=\sH^0(\bF)$, \ $\sX=\Ac^\co(\bE)$,
and $\sY=\Ac^\abs(\bF)$, providing both the full-and-faithfulness
and compactness assertions of the theorem.
\end{proof} 

\subsection{Compact generation: Noetherian case}
 Let $\bA$ be an abelian DG\+category.
 Recall the notation $\bA_\binj\subset\bA$ for the full DG\+subcategory
formed by all the graded-injective objects, i.~e., such objects
$J\in\bA$ that the object $\Phi_\bA(J)$ is injective in the abelian
category $\sZ^0(\bA^\bec)$ (see
Section~\ref{injective-projective-semiorthogonality-subsecn}).

 In this section we consider a locally Noetherian (Grothendieck
abelian) DG\+cat\-e\-gory~$\bA$.
 The following proposition is a particular case of
Theorem~\ref{inj-proj-resolutions-star-conditions-theorem}(a).

\begin{prop} \label{Noetherian-coderived-category-prop}
 Let\/ $\bA$ be a locally Noetherian DG\+category.
 Then the triangulated functor
$$
 \sH^0(\bA_\binj)\lrarrow\sD^\co(\bA)
$$
induced by the inclusion of exact/abelian DG\+categories\/ $\bA_\binj
\rightarrowtail\bA$ is an equivalence of triangulated categories.
\end{prop}

\begin{proof}
 It suffices to check that the assumptions of
Theorem~\ref{inj-proj-resolutions-star-conditions-theorem}(a)
are satisfied for the DG\+category $\bA$ with its abelian exact
DG\+category structure.
 Any Grothendieck abelian category has exact coproducts, hence
any Grothendieck abelian DG\+category also has exact coproducts
(cf.\ Lemma~\ref{coproducts-in-cocycles-and-in-DG}).
 Moreover, any Grothendieck abelian category has enough injective
objects.
 Furthermore, any abelian DG\+category has twists by
Proposition~\ref{abelian-DG-categories-have-twists}.
 Finally, coproducts of injective objects are injective in any
locally Noetherian Grothendieck category; so condition~($*$)
is trivially satisfied for locally Noetherian categories.
\end{proof}

 The following theorem, provable by an argument due to Arinkin,
is our generalization of~\cite[Theorem~3.11.2]{Pkoszul}
and~\cite[Proposition~1.5(d)]{EP}.

\begin{thm} \label{Noetherian-compact-generation-theorem}
 Let\/ $\bA$ be a locally Noetherian DG\+category and\/
$\bA_\bfg\subset\bA$ be its full DG\+subcategory of finitely
generated (Noetherian) objects.
 Then the triangulated functor\/
$$
 \sD^\abs(\bA_\bfg)\lrarrow\sD^\co(\bA)
$$
induced by the inclusion of abelian DG\+categories\/
$\bA_\bfg\rarrow\bA$ is fully faithful, and (representatives of
the isomorphism classes of all) the objects in its image form a set of
compact generators of the triangulated category\/ $\sD^\co(\bA)$.
\end{thm}

\begin{proof}
 To prove the full-and-faithfulness of the triangulated functor
in question and compactness of all the objects in its essential image,
we apply
Theorem~\ref{full-and-faithfulness-and-compactness-main-theorem}.
 Consider the exact DG\+pair $(\bA,\sZ^0(\bA^\bec))$ and its exact
DG\+subpair $(\bA_\bfg,\sZ^0(\bA^\bec_\bfg))$ (recall that
$\sZ^0(\bA^\bec_\bfg)=\sZ^0(\bA^\bec)_\fg$, as per the discussion in
Section~\ref{locally-Noetherian-DG-categories-subsecn}).
 Recall further that all finitely generated objects are small, and
observe that the full subcategory of Noetherian objects is
self-resolvable in any locally Noetherian category.

 In order to prove that the Noetherian objects generate the coderived
category $\sD^\co(\bA)$, we use the result of
Proposition~\ref{Noetherian-coderived-category-prop}, telling that
any object of $\sD^\co(\bA)$ can be represented by a graded-injective
object of~$\bA$.
 Furthermore, the semiorthogonality assertion of
Theorem~\ref{inj-proj-resolutions-semiorthogonality-theorem}(a) implies
that the Verdier quotient functor $\sH^0(\bA)\rarrow\sD^\co(\bA)$
induces an isomorphism of the Hom groups $\Hom_{\sH^0(\bA)}(A,J)\simeq
\Hom_{\sD^\co(\bA)}(A,J)$ for any objects $A\in\bA$ and
$J\in\bA_\binj$.
 It remains to use the following lemma, going back to Arinkin, in
order to finish the proof of the theorem.
\end{proof}

\begin{lem} \label{Arinkin-lemma}
 Let\/ $\bA$ be a locally Noetherian DG\+category and $J\in\bA_\binj$
be a graded-injective object.
 Assume that all closed morphisms $E\rarrow J$ from Noetherian objects
$E\in\bA_\bfg$ to $J$ are homotopic to zero in\/~$\bA$.
 Then the object $J$ is contractible.
\end{lem}

\begin{proof}
 Consider the following poset~$P$.
 The elements of $P$ are pairs $(B,t_B)$, where $B\subset J$ is
a subobject of the object $J$ in the abelian category $\sZ^0(\bA)$,
and $h\in\Hom^{-1}_\bA(B,J)$ is a contracting homotopy for
the identity embedding $\iota_B\:B\rarrow J$, that is $\iota_B=d(t)$.
 The partial order on $P$ is defined in the obvious way: we say that
$(B,t_B)\le (C,t_C)$ in $P$ if $B\subset C\subset J$ and $t_B=t_C|_B$.

 It is important to observe that, since directed colimits are exact
in $\sZ^0(\bA)$, the union of a chain of subobjects coincides with
their direct limit.
 Moreover, the identity functor $\sZ^0(\bA)\rarrow\bA^0$ preserves
direct limits (since so does the functor $\Phi_\bA\:\sZ^0(\bA)
\rarrow\sZ^0(\bA^\bec)$, which can be extended to a fully faithful
embedding $\widetilde\Phi_\bA\:\bA^0\rarrow\sZ^0(\bA^\bec)$; see
Lemma~\ref{Phi-Psi-extended-to-nonclosed-morphisms}).
 Hence, interpreting the homotopies~$t_B$ as morphisms
$B\rarrow J[-1]$ in $\bA^0$ and using the direct limit property of
the union, one can show that any chain (linearly ordered subset)
in $P$ has an upper bound.

 By Zorn's lemma, the poset $B$ has a maximal element $(B,t_B)$.
 It suffices to prove that the subobject $B\subset J$ cannot be
different from the whole of~$J$.
 Indeed, otherwise there exists a subobject $C\subset J$ such
$B\subset C$ and the quotient object $C/B$ is a nonzero Noetherian
object in $\sZ^0(\bA)$.
 It remains to show that the homotopy $t_B\in\Hom^{-1}_\bA(B,J)$
can be extended to a compatible homotopy $t_C\in\Hom^{-1}_\bA(C,J)$,
i.~e., an element $t_C$~can be chosen so that $d(t_C)=\iota_C$
and $t_B=t_C|_B$.

 Consider the element $\widetilde\Phi(t_B)\in\Hom_{\sZ^0(\bA^\bec)}
(\Phi(B),\Phi(J[-1]))$.
 Since $\Phi(B)$ is a subobject of $\Phi(C)$ in $\sZ^0(\bA^\bec)$ and
the object $\Phi(J[-1])=\Phi(J)[1]$ is injective in $\sZ^0(\bA^\bec)$,
the morphism $\widetilde\Phi(t_B)$ can be extended to a morphism
$s'_C\in\Hom_{\sZ^0(\bA^\bec)}(\Phi(C),\Phi(J[-1]))$.
 Since the functor $\widetilde\Phi$ is fully faithful, we have
an element $t'_C\in\Hom^0_\bA(C,J[-1])\simeq\Hom^{-1}_\bA(C,J)$
such that $s'_C=\widetilde\Phi(t'_C)$.
 Simply put, we can extend our morphism $t_B\in\Hom^{-1}_\bA(B,J)$ to
some morphism $t'_C\in\Hom^{-1}_\bA(C,J)$ such that $t_B=t'_C|_B$.

 Now the inclusion $B\rightarrowtail C$ is a morphism in $\sZ^0(\bA)$,
so it is a closed morphism in~$\bA$.
 Therefore, we have $d(t'_C)|_B=d(t'_C|_B)=d(t_B)=\iota_B$.
 Hence, the difference $\iota_C-d(t'_C)$ vanishes in restriction
to~$B$.
 Furthermore, we have $d(\iota_C-d(t'_C))=d(\iota_C)-d^2(t'_C)=0$;
so $\iota_C-d(t'_C)\:C\rarrow J$ is a morphism in $\sZ^0(\bA)$
vanishing in restriction to~$B$.

 Denote by $E$ the quotient object $C/B$ in the abelian category
$\sZ^0(\bA)$; so $E\in\bA_\bfg$.
 Let $\pi\:C\rarrow E$ be the natural epimorphism in $\sZ^0(\bA)$.
 Then there exists a morphism $f\:E\rarrow J$ in $\sZ^0(\bA)$
such that $\iota_C-d(t'_C)=f\pi$.
 By the assumption of the lemma, the morphism $f$ must be homotopic
to zero in~$\bA$; so there is a homotopy $t_E\in\Hom_\bA^{-1}(E,J)$
such that $d(t_E)=f$.
 Finally, we have $\iota_C=d(t'_C)+f\pi=d(t'_C+t_E\pi)$ and
$(t'_C+t_E\pi)|_B=t'_C|_B+t_E\pi|_B=t_B$.
 It remains to put $t_C=t'_C+t_E\pi$.
\end{proof}

\begin{rem} \label{compact-objects-are-direct-summands-remark}
 Let us warn the reader that
Theorem~\ref{Noetherian-compact-generation-theorem} does \emph{not} 
completely describe the full triangulated subcategory of compact
objects in the compactly generated triangulated category $\sD^\co(\bA)$.
 Rather, the theorem claims that the full triangulated subcategory
$\sD^\abs(\bA_\bfg)$ consists of compact objects and generates
the triangulated category $\sD^\co(\bA)$.
 It follows that the full triangulated subcategory of compact objects
in $\sD^\co(\bA)$ is the thick closure (coinciding with
the idempotent completion) of $\sD^\abs(\bA_\bfg)$, that is,
any compact object of $\sD^\co(\bA)$ is a \emph{direct summand} of
an object coming from $\sD^\abs(\bA_\bfg)$ \,\cite[Lemma~2.2]{Neem1}.
\end{rem}

\subsection{Examples}
 Let us formulate here the particular cases of
Theorem~\ref{Noetherian-compact-generation-theorem} arising in
the context of Examples~\ref{Grothendieck-abelian-DG-examples}
and~\ref{locally-Noetherian-examples}.

 The following well-known result is due to
Krause~\cite[Proposition~2.3]{Kr}.

\begin{cor} \label{complexes-Noetherian-compact-generators}
 Let\/ $\sA$ be a locally Noetherian Grothendieck abelian category
and\/ $\sA_\fg\subset\sA$ be its full abelian subcategory of
Noetherian objects.
 Then the triangulated functor
$$
 \sD^\bb(\sA_\fg)\lrarrow\sD^\co(\sA)
$$
induced by the inclusion of abelian categories\/ $\sA_\fg\rightarrowtail
\sA$ is fully faithful, and its essential image coincides with
the full subcategory of compact objects in a compactly generated
triangulated category\/ $\sD^\co(\sA)$.
\end{cor}

\begin{proof}
 The notation $\sD^\bb(\sB)$ for the bounded derived category of
an exact/abelian category $\sB$ presumes that the complexes are graded
by the usual grading group of integers $\Gamma=\boZ$.
 To obtain this corollary as a particular case of
Theorem~\ref{Noetherian-compact-generation-theorem} one observes,
first of all, that the abelian category of unbounded complexes
$\sC(\sA)$ is locally Noetherian for any locally Noetherian abelian
category~$\sA$.
 The Noetherian objects of the category of complexes $\sC(\sA)$ are
the bounded complexes of Noetherian objects in $\sA$, that is
$\sC(\sA)_\fg=\sC^\bb(\sA_\fg)$.

 Furthermore, for any abelian category $\sB$ one can consider
the exactly embedded full abelian DG\+subcategory of bounded complexes
in $\bC(\sB)$, denoted by $\bC^\bb(\sB)\subset\bC(\sB)$.
 Then the absolute derived category $\sD^\abs(\bC^\bb(\sB))$ can be
easily seen to agree with the bounded derived category of $\sB$,
i.~e., $\sD^\abs(\bC^\bb(\sB))=\sD^\bb(\sB)$
(cf.\ Lemma~\ref{totalization-of-finite-absolutely-acyclic}
or~\cite[Section~A.1]{Pcosh}).
 In particular, $\sD^\abs(\bC(\sA)_\bfg)=\sD^\bb(\sA_\fg)$.

 Finally, Theorem~\ref{Noetherian-compact-generation-theorem} still
gives a weaker result since is leaves the question of direct summands
open (as per Remark~\ref{compact-objects-are-direct-summands-remark}).
 To deduce the full assertion formulated in the corollary, one needs
to check that the bounded derived category $\sD^\bb(\sB)$ is
idempotent-complete for any abelian category~$\sB$.
 A more general result, for idempotent-complete exact categories
rather than abelian ones, can be found in~\cite[Theorem~2.8]{BS}.
\end{proof}

 The next assertion can be found in~\cite[Theorems~3.11.1
and~3.11.2]{Pkoszul}.

\begin{cor} \label{cdg-modules-Noetherian-compact-generators}
 Let $\biR^\cu=(R^*,d,h)$ be a CDG\+ring such that the graded ring
$R^*$ is graded left Noetherian.
 Let $\bA=\biR^\cu\bmodl$ be the locally Noetherian DG\+category of
left CDG\+modules over $\biR^\cu$, and let\/
$\bA_\bfg=\biR^\cu\bmodl_\bfg\subset\biR^\cu\bmodl$ be the exactly
embedded full abelian DG\+subcategory whose objects are
the CDG\+modules with Noetherian underlying graded $R^*$\+modules.
 Then the triangulated functor
$$
 \sD^\abs(\biR^\cu\bmodl_\bfg)\lrarrow\sD^\co(\biR^\cu\bmodl)
$$
induced by the inclusion of abelian DG\+categories $\biR^\cu\bmodl_\bfg
\rightarrowtail\biR^\cu\bmodl$ is fully faithful, and the objects in
its image form a set of compact generators of the coderived
category\/ $\sD^\co(\biR^\cu\bmodl)$.
\end{cor}

\begin{proof}
 Notice that, by
Lemmas~\ref{Phi-Psi-preserve-generatedness-presentability}\+-%
\ref{Phi-Psi-reflect-generatedness-presentability} or by
Lemma~\ref{Phi-Psi-preserve-reflect-Noetherianity}, a left
CDG\+module $(M^*,d_M)$ over $(R^*,d,h)$ is Noetherian as
an object of $\sZ^0(\biR^\cu\bmodl)$ if and only if the graded
$R^*$\+module $M^*$ is Noetherian as an object of the category
of graded modules $R^*\smodl$.
 The rest is explained in Example~\ref{locally-Noetherian-examples}(1)
and Theorem~\ref{Noetherian-compact-generation-theorem}.
\end{proof}

 The following corollary is a generalization
of~\cite[Proposition~1.5(d)]{EP}.

\begin{cor} \label{quasi-coh-cdg-modules-Noetherian-compact-generators}
 Let $X$ be a quasi-compact quasi-separated scheme and
$\biB^\cu=(B^*,d,h)$ be a quasi-coherent CDG\+quasi-algebra over~$X$
such that the quasi-coherent graded quasi-algebra $B^*$ is graded
left Noetherian.
 Let $\bA=\biB^\cu\bqcoh$ be the locally Noetherian DG\+category of
quasi-coherent left CDG\+modules over $\biB^\cu$, and let\/ $\bA_\bfg=
\biB^\cu\bqcoh_\bfg\subset\biB^\cu\bqcoh$ be the exactly embedded full abelian
DG\+subcategory whose objects are the CDG\+modules with Noetherian
underlying quasi-coherent graded $B^*$\+modules.
 Then the triangulated functor
$$
 \sD^\abs(\biB^\cu\bqcoh_\bfg)\lrarrow\sD^\co(\biB^\cu\bqcoh)
$$
induced by the inclusion of abelian DG\+categories $\biB^\cu\bqcoh_\bfg
\rightarrowtail\biB^\cu\bqcoh$ is fully faithful, and the objects in
its image form a set of compact generators of the coderived
category\/ $\sD^\co(\biB^\cu\bqcoh)$.
\end{cor}

\begin{proof}
 Similarly to the proof of
Corollary~\ref{cdg-modules-Noetherian-compact-generators}, the same
lemmas imply that a quasi-coherent CDG\+module $(M^*,d_M)$ over
$(B^*,d,h)$ is Noetherian as an object of $\sZ^0(\biB^\cu\bqcoh)$ if and
only if the quasi-coherent graded $B^*$\+module $M^*$ is Noetherian as
an object of the category of quasi-coherent graded modules $B^*\sqcoh$.
 Then it remains to refer to
Example~\ref{locally-Noetherian-examples}(2) and
Theorem~\ref{Noetherian-compact-generation-theorem}.
\end{proof}

 The next corollary generalizes~\cite[Corollary~2.3(l)]{EP}.

\begin{cor} \label{factorizations-Noetherian-compact-generators}
 Let\/ $\sA$ be a locally Noetherian abelian category and\/
$\Lambda\:\sA\rarrow\sA$ be an autoequivalence.
 Assume that either\/ $\Gamma=\boZ$, or\/ $\Lambda$ is involutive
and\/ $\Gamma=\boZ/2$.
 Let $w\:\Id_\sA\rarrow\Lambda^2$ be a potential (as in
Section~\ref{factorizations-subsecn}).
 Then the triangulated functor
$$
 \sD^\abs(\bF(\sA_\fg,\Lambda,w))\lrarrow
 \sD^\co(\bF(\sA,\Lambda,w))
$$
induced by the inclusion of abelian categories\/ $\sA_\fg\rightarrowtail
\sA$ is fully faithful, and the objects in its image form a set of
compact generators of the coderived category\/
$\sD^\co(\bF(\sA,\Lambda,w)$.
\end{cor}

\begin{proof}
 Here, as usual, we denote the restrictions of $\Lambda$ and~$w$
to $\sA_\fg\subset\sA$ simply by $\Lambda$ and~$w$.
 First of all, we observe that the abelian DG\+category of
factorizations $\bF(\sA,\Lambda,w)$ is locally Noetherian whenever
an abelian category $\sA$ is locally Noetherian.
 Using, e.~g.,
Lemmas~\ref{Phi-Psi-preserve-generatedness-presentability}\+-%
\ref{Phi-Psi-reflect-generatedness-presentability} or
Lemma~\ref{Phi-Psi-preserve-reflect-Noetherianity}, one shows that
the full DG\+subcategory of Noetherian objects
$\bF(\sA,\Lambda,w)_\bfg\subset\bF(\sA,\Lambda,w)$ is
the abelian DG\+category of factorizations in $\sA_\fg$, that is
$\bF(\sA,\Lambda,w)_\bfg=\bF(\sA_\fg,\Lambda,w)$.
 Then it remains to apply
Theorem~\ref{Noetherian-compact-generation-theorem}.
\end{proof}

\subsection{Fp-injective and fp-projective objects}
\label{fp-injective-projective-subsecn}
 We start the discussion in the context of locally finitely presentable
categories before passing to the better behaved special case of
locally coherent categories.
 Recall from the discussion in
Section~\ref{finitely-generated-presentable-subsecn} that any locally
finitely generated abelian category is Grothendieck (a more general
result can be found in~\cite[Corollary~9.6]{PS1}); in particular,
any locally finitely presentable abelian category is Grothendieck.
 For a discussion of fp\+injective and fp\+projective modules over
a ring, in a context closely related to our exposition in this paper,
see~\cite[Section~1]{Pfp} and the references therein.

 Let $\sA$ be a locally finitely presentable abelian category.
 An object $J\in\sA$ is said to be \emph{fp\+injective} if
$\Ext_\sA^1(E,J)=0$ for all finitely presentable objects $E\in\sA$.
 An object $P\in\sA$ is said to be \emph{fp\+projective} if
$\Ext_\sA^1(P,J)=0$ for all fp\+injective objects $J\in\sA$.
 We will denote the full subcategories of fp\+injective and
fp\+projective objects by $\sA_\fpinj$ and $\sA_\fproj\subset\sA$,
respectively.

 Notice that the condition $\Ext_\sA^1(S,J)=0$ for all \emph{finitely
generated} objects in a locally finitely generated abelian category
$\sA$ implies injectivity of the object~$J$.
 Hence in a locally Noetherian category $\sA$, where all finitely
generated objects are finitely presentable, one has that all
fp\+injective objects are injective and all objects are fp\+projective.
 This is \emph{not} true in a locally coherent category which is not
locally Noetherian.
 For locally coherent categories, fp\+injectivity and fp\+projectivity
are important technical concepts.

 Let $\sA$ be a Grothendieck category and $A\in\sA$ be an object.
 An (ordinal-indexed increasing) \emph{filtration} $F$ on $A$ is
a family of subobjects $F_i\subset A$, indexed by the ordinals
$0\le i\le\alpha$, where $\alpha$~is a fixed ordinal, and satisfying
the following conditions:
\begin{itemize}
\item $F_0=0$ and $F_\alpha=A$;
\item $F_i\subset F_j$ for all ordinals $i\le j\le\alpha$;
\item $F_j=\bigcup_{i<j}F_i$ for all limit ordinals $j\le\alpha$.
\end{itemize}
 Given a filtration $F$ on an object $A$, one says that the object $A$
is \emph{filtered by} the quotient objects $F_{i+1}/F_i$, where
$0\le i<\alpha$.
 In an alternative terminology, the object $A$ is said to be
a \emph{transfinitely iterated extension} (\emph{in the sense of
the directed colimit}) of the objects $F_{i+1}/F_i\in\sA$.
 Notice that, in particular, an infinite coproduct is a transfinitely
iterated extension.

\begin{prop} \label{fp-Eklof-Trlifaj}
 Let\/ $\sA$ be a locally finitely generated abelian category. \par
\textup{(a)} For any object $A\in\sA$ there exists a short exact
sequence
$$
 0\rarrow J'\rarrow P\rarrow A\rarrow0
$$
in\/ $\sA$ with an fp\+injective object $J'$ and an fp\+projective
object~$P$. \par
\textup{(b)} For any object $A\in\sA$ there exists a short exact
sequence
$$
 0\rarrow A\rarrow J\rarrow P'\rarrow0
$$
in\/ $\sA$ with an fp\+injective object $J$ and an fp\+projective
object~$P'$. \par
\textup{(c)} An object $P\in\sA$ is fp\+projective if and only if it is
a direct summand of an object filtered by finitely presentable objects.
\end{prop}

\begin{proof}
 This is a version of a classical theorem of Eklof and
Trlifaj~\cite{ET}.
 For a generalization covering the situation at hand, see,
e.~g., \cite[Theorems~3.3 and~3.4]{PS4}.
 It is important here that there are enough fp\+injective objects
in $\sA$, in the sense that any object is a subobject of
an fp\+injective one (and in fact, even a subobject of an injective
object), and there are enough fp\+projective objects, in the sense
that any object is a quotient of an fp\+projective object.
 Indeed, coproducts of fp\+projective objects are easily seen to
be fp\+projective; in particular, coproducts of finitely presentable
objects are fp\+projective; and any object in a locally finitely
presentable category is a quotient of a coproduct of finitely
presentable ones.
\end{proof}

\begin{lem} \label{fp-inj-proj-closure-properties-in-loc-presentable}
 Let\/ $\sA$ be a locally finitely presentable abelian category.
 Then \par
\textup{(a)} both the classes of fp\+injective and fp\+projective
objects in\/ $\sA$ are closed under extensions and direct summands;
\par
\textup{(b)} the class of all fp\+projective objects in\/ $\sA$ is
closed under transfinitely iterated extensions; \par
\textup{(c)} both the classes of fp\+projective and fp\+injective
objects in\/ $\sA$ are closed under coproducts.
 The class of all fp\+injective objects is also closed under products.
\end{lem}

\begin{proof}
 Part~(a) follows immediately from the definitions.
 Part~(b) is a version of a classical
\emph{Eklof lemma}~\cite[Lemma~1]{ET} and essentially a part of
Proposition~\ref{fp-Eklof-Trlifaj} (cf.~\cite[Proposition~1.3]{PS4}).
 In part~(c), the coproducts of fp\+projective objects are
fp\+projective and the products of fp\+injective objects are
fp\+injective by a general property of $\Ext^1$\+orthogonal
classes in abelian categories with (co)products.

 The fact that coproducts of fp\+injective objects are fp\+injective
is a property of right $\Ext^1$\+orthogonal clases to sets of
finitely presentable objects in locally finitely presentable
categories.
 Notice that, for any finitely presentable object $E\in\sA$ and
any object $B\in\sA$, the group $\Ext^1_\sA(E,B)$ can be computed
as the filtered colimit of cokernels of the restriction maps
$\Hom_\sA(F,B)\rarrow\Hom_\sA(G,B)$ taken over all the short exact
sequences $0\rarrow G\rarrow F\rarrow E\rarrow0$ in $\sA$ with
a \emph{finitely presentable} object~$F$.
 Furthermore, the cokernel $G$ of any epimorphism of finitely
presentable objects $F\rarrow E$ is a finitely generated (hence
small) object.
\end{proof}

\begin{lem} \label{fp-inj-proj-closure-properties-in-loc-coherent}
 Let\/ $\sA$ be a locally coherent abelian category.  Then \par
\textup{(a)} the class of all fp\+injective objects in\/ $\sA$
is closed under directed colimits; \par
\textup{(b)} the class of all fp\+injective objects in\/ $\sA$
is closed under the cokernels of monomorphisms; \par
\textup{(c)} the class of all fp\+projective objects in\/ $\sA$
is closed under the kernels of epimorphisms; \par
\textup{(d)} one has\/ $\Ext_\sA^n(P,J)=0$ for any fp\+projective
object $P\in\sA$, any fp\+injective object $J\in\sA$, and all
integers $n\ge1$.
\end{lem}

\begin{proof}
 The proof of part~(a) is similar to the second paragraph of
the proof of
Lemma~\ref{fp-inj-proj-closure-properties-in-loc-presentable} above;
one only needs to observe that in a locally coherent category
$\sA$ the object $G$ is finitely presentable (coherent), too.
 Parts~(b\+-d) are equivalent expressions of the assertion that
\emph{the fp cotorsion pair is hereditary} in a locally
coherent category (see~\cite[Lemma~1.4]{PS4} for the general
theory) and follow from the same observation as part~(a).
\end{proof}

 Notice that
Lemmas~\ref{fp-inj-proj-closure-properties-in-loc-presentable}(a)
and~\ref{fp-inj-proj-closure-properties-in-loc-coherent}(b\+-c) and
the observations in the proof of Proposition~\ref{fp-Eklof-Trlifaj}
imply, in particular, that the class of all fp\+projective objects
is resolving, while the class of all fp\+injective objects is
coresolving in a locally coherent category~$\sA$.
 Hence one can speak about the \emph{fp\+projective} (resolution)
and the \emph{fp\+injective} (coresolution) \emph{dimensions} of
objects of~$\sA$.

\begin{lem} \label{fp-projective-dimension-lemma}
 Given a locally coherent category\/ $\sA$ and an integer $n\ge0$,
the following two conditions are equivalent:
\begin{enumerate}
\item all the objects of\/ $\sA$ have fp\+projective dimensions
not exceeding~$n$;
\item all the fp\+injective objects in\/ $\sA$ have injective
dimensions not exceeding~$n$.
\end{enumerate}
\end{lem}

\begin{proof}
 Both the conditions are equivalent to vanishing of the Ext groups
$\Ext^i_\bA(A,J)$ with $i>n$ for all objects $A\in\sA$ and
$J\in\sA_\fpinj$.
\end{proof}

 We will say that the \emph{fp\+projective dimension} of a locally
coherent category $\sA$ is~$\le n$ if the equivalent conditions
of Lemma~\ref{fp-projective-dimension-lemma} hold.
 In particular, a locally coherent category has fp\+projective
dimension zero if and only if it is locally Noetherian.

\begin{lem} \label{Ext-exact-adjunction-lemma}
 Let\/ $\sE$ and\/ $\sF$ be exact categories, and let\/ $\Phi\:\sE
\rarrow\sF$ and\/ $\Psi\:\sF\rarrow\sE$ be an adjoint pair of
exact functors, with the functor\/ $\Psi$ right adjoint to\/~$\Phi$.
 Then for any two objects $E\in\sE$ and $F\in\sF$ there are natural
isomorphisms of Yoneda Ext groups
$$
 \Ext^n_\sE(E,\Psi(F))\simeq\Ext^n_\sF(\Phi(E),F)
 \qquad\text{for all \ $n\ge0$.}
$$
\end{lem}

\begin{proof}
 One can, e.~g., show that the functors induced by $\Phi$ and $\Psi$
on the (bounded or unbounded) derived categories of $\sE$ and $\sF$
are adjoint, by constructing the natural morphisms of adjunction unit
and counit for complexes and noticing that the required equations for
the compositions hold.
 Then the desired isomorphism of the Ext groups follows.
\end{proof}

\begin{lem} \label{Phi-Psi-preserve-fp-injectivity-projectivity}
 Let\/ $\bA$ be a locally coherent DG\+category.  Then \par
\textup{(a)} both the additive functors\/ $\Phi_\bA\:\sZ^0(\bA)\rarrow
\sZ^0(\bA^\bec)$ and\/ $\Psi^+_\bA\:\sZ^0(\bA^\bec)\rarrow\sZ^0(\bA)$
preserve fp\+injectivity of objects; \par
\textup{(b)} both the functors also preserve fp\+projectivity of
objects.
\end{lem}

\begin{proof}
 Using Lemma~\ref{Ext-exact-adjunction-lemma} for $n=1$, part~(a)
follows from the fact that both the functors preserve finite 
presentability of objects
(Lemma~\ref{Phi-Psi-preserve-generatedness-presentability}), and
part~(b) follows from part~(a).
 Here, as usual, one should keep in mind that the functor
$\Psi^-_\bA$ right adjoint to $\Phi_\bA$ only differs from
the functor $\Psi^+_\bA$ by a shift, $\Psi^-(X)=\Psi^+(X)[1]$.
\end{proof}

 Let $\bA$ be a locally coherent (Grothendieck abelian) DG\+category.
 An object $J\in\bA$ is said to be \emph{graded-fp-injective} if
$\Phi_\bA(J)\in\sZ^0(\bA^\bec)$ is an fp\+injective object of
the abelian category $\sZ^0(\bA^\bec)$.
 We will denote the full DG\+subcategory of graded-fp-injective
objects by $\bA_\bfpinj\subset\bA$.
 The following lemma is similar to Lemma~\ref{graded-proj-inj-lemma}.

\begin{lem}
 For any locally coherent DG\+category\/ $\bA$, the full DG\+subcategory
of graded-fp-injective objects\/ $\bA_\bfpinj\subset\bA$ is additive
and closed under shifts and twists (hence also under cones) as well as
under direct summands, infinite products, and infinite coproducts.
 The full DG\+subcategory\/ $\bA_\bfpinj$ inherits an exact DG\+category
structure from the abelian exact DG\+category structure of\/~$\bA$.
 The inclusion\/ $\sZ^0(\bA^\bec)_\fpinj\subset
\sZ^0((\bA_\bfpinj)^\bec)$ holds in\/ $\sZ^0(\bA^\bec)$.
\end{lem}

\begin{proof}
 Let $\sL=\sZ^0(\bA^\bec)_\fpinj$ be the full subcategory of
fp\+injective objects in the locally coherent abelian category
$\sZ^0(\bA^\bec)$.
 The full subcategory $\sL$ is additive, closed under extensions,
direct summands, and the cokernels of monomorphisms in $\sZ^0(\bA^\bec)$
(Lemmas~\ref{fp-inj-proj-closure-properties-in-loc-presentable}(a)
and~\ref{fp-inj-proj-closure-properties-in-loc-coherent}(b)), and
preserved by the shift functors acting on $\sZ^0(\bA^\bec)$.
 By Proposition~\ref{subcategory-L-prop}(a\+-d), all the assertions
of the lemma follow except the ones concerning infinite products
and coproducts.
 To prove these, one recalls that the full subcategory
$\sL=\sZ^0(\bA^\bec)_\fpinj\subset\sZ^0(\bA^\bec)$ is also closed
under products and coproducts
(Lemma~\ref{fp-inj-proj-closure-properties-in-loc-presentable}(c)),
and the functor $\Phi_\bA$ preserves products and coproducts.
\end{proof}

\subsection{Compact generation: $\aleph_n$-Noetherian coherent case}
 The following proposition is a generalization
of~\cite[Theorem~2.2]{Pfp}.

\begin{prop} \label{coherent-two-coderived-categories-prop}
 Let\/ $\bA$ be a locally coherent DG\+category.
 Then the triangulated functor
$$
 \sD^\co(\bA_\bfpinj)\lrarrow\sD^\co(\bA)
$$
induced by the inclusion of exact/abelian DG\+categories\/ $\bA_\bfpinj
\rightarrowtail\bA$ is an equivalence of triangulated categories.
\end{prop}

\begin{proof}
 Apply Theorem~\ref{contraderived-resolving-equivalence-theorem}(b) to
the exact DG\+subpair $(\bG,\sM)=(\bA_\bfpinj,\sZ^0(\bA^\bec)_\fpinj)$
in the exact DG\+pair $(\bE,\sK)=(\bA,\sZ^0(\bA^\bec))$.
 To check the assumptions, recall that any abelian DG\+category has
twists by Proposition~\ref{abelian-DG-categories-have-twists}.
 Furthermore, the class of all fp\+injective objects is coresolving
in a locally coherent abelian category $\sZ^0(\bA^\bec)$ according to
Lemmas~\ref{fp-inj-proj-closure-properties-in-loc-presentable}(a)
and~\ref{fp-inj-proj-closure-properties-in-loc-coherent}(b), and
because there are enough injective (hence enough fp\+injective) objects.
\end{proof}

 The next proposition is a generalization
of~\cite[Theorem~2.4]{Pfp}.
 It is also a generalization of
Proposition~\ref{Noetherian-coderived-category-prop} above.

\begin{prop} \label{coherent-finite-fp-dimen-coderived-category-prop}
 Let\/ $\bA$ be a locally coherent DG\+category.
 Assume that the locally coherent abelian category\/ $\sZ^0(\bA^\bec)$
has finite fp\+projective dimension.
 Then the triangulated functor
$$
 \sH^0(\bA_\binj)\lrarrow\sD^\co(\bA)
$$
induced by the inclusion of exact/abelian DG\+categories\/ $\bA_\binj
\rightarrowtail\bA$ is an equivalence of triangulated categories.
\end{prop}

\begin{proof}
 The point is that condition~($*$) is satisfied for locally coherent
abelian categories of finite fp\+projective dimension, and consequently
Theorem~\ref{inj-proj-resolutions-star-conditions-theorem}(a)
is applicable.
 Indeed, in a locally coherent category of finite fp\+projective
dimension, all coproducts of injective objects are fp\+injective
by Lemma~\ref{fp-inj-proj-closure-properties-in-loc-presentable}(c),
and all fp\+injective objects have finite injective dimensions
by the definition.
 The rest was already explained in the proof of
Proposition~\ref{Noetherian-coderived-category-prop}. 
\end{proof}

 The following theorem is our version of~\cite[Corollary~6.13]{Sto2}
and~\cite[Corollary~2.6]{Pfp}.
 It is also a generalization of
Theorem~\ref{Noetherian-compact-generation-theorem} above.

\begin{thm} \label{coherent-finite-fp-dimen-compact-generation-theorem}
 Let\/ $\bA$ be a locally coherent DG\+category and\/
$\bA_\bfp\subset\bA$ be its full DG\+subcategory of finitely
presentable (coherent) objects.
 Assume that \emph{both} the locally coherent abelian categories\/
$\sZ^0(\bA)$ and\/ $\sZ^0(\bA^\bec)$ have finite fp\+projective
dimensions.
 Then the triangulated functor\/
$$
 \sD^\abs(\bA_\bfp)\lrarrow\sD^\co(\bA)
$$
induced by the inclusion of abelian DG\+categories\/
$\bA_\bfp\rarrow\bA$ is fully faithful, and the objects in its image
form a set of compact generators of the coderived category\/
$\sD^\co(\bA)$.
\end{thm}

\begin{proof}
 To prove the full-and-faithfulness of the triangulated inclusion
functor and compactness of all the objects in its essential image,
we once again apply
Theorem~\ref{full-and-faithfulness-and-compactness-main-theorem}.
 Consider the exact DG\+pair $(\bE,\sK)=(\bA,\sZ^0(\bA^\bec))$ and its
exact DG\+subpair $(\bF,\sL)=(\bA_\bfp,\sZ^0(\bA^\bec_\bfp))$
(recall that $\sZ^0(\bA^\bec_\bfp)=\sZ^0(\bA^\bec)_\fp$, according to
the discussion in Section~\ref{locally-coherent-DG-categories-subsecn}).
 Recall further that all finitely generated (hence all finitely
presentable) objects are small, and observe that the full subcategory
of coherent objects is self-resolvable in any locally coherent category.

 In order to prove that the coherent objects generate the coderived
category $\sD^\co(\bA)$, we use the result of
Proposition~\ref{coherent-finite-fp-dimen-coderived-category-prop},
telling that any object of $\sD^\co(\bA)$ can be represented by
a graded-injective object of $\bA$ under the assumption of finite
fp\+projective dimension of $\sZ^0(\bA^\bec)$.
 The semiorthogonality assertion of
Theorem~\ref{inj-proj-resolutions-semiorthogonality-theorem}(a) implies
that the Verdier quotient functor $\sH^0(\bA)\rarrow\sD^\co(\bA)$
induces an isomorphism of the Hom groups $\Hom_{\sH^0(\bA)}(A,J)\simeq
\Hom_{\sD^\co(\bA)}(A,J)$ for all objects $A\in\bA$ and
$J\in\bA_\binj$.

 Let $J\in\bA_\binj$ be a graded-injective object such that all
closed morphisms $E\rarrow J$ from coherent objects $E\in\bA_\bfp$
to $J$ are homotopic to zero in~$\bA$.
 We have to prove that the object $J$ is contractible.
 For this purpose, we will show that all closed morphisms $A\rarrow J$
from an arbitrary object $A\in\bA$ to the object $J$ are homotopic
to zero.

 We start with picking an epimorphism $P_0\rarrow A$ in $\sZ^0(\bA)$
with an fp\+projective object~$P_0$.
 Denoting the kernel of this epimorphism by~$A_1$, we choose
an epimorphism $P_1\rarrow A_1$ with an fp\+projective object~$P_1$,
etc.
 Proceeding in this way, we construct an exact sequence
$0\rarrow Q\rarrow P_{n-1}\rarrow\dotsb\rarrow P_0\rarrow A\rarrow0$
in the abelian category $\sZ^0(\bA)$, where $n$~is the fp\+projective
dimension of $\sZ^0(\bA)$.
 Then the object $Q\in\sZ^0(\bA)$ is also fp\+projective.
 Put $P_n=Q$.

 By Lemma~\ref{totalization-of-finite-absolutely-acyclic}, the object
$X=\Tot(P_\bu\to A)$ is absolutely acyclic in~$\bA$.
 Put $B=\Tot(P_\bu)\in\bA$.
 Then we have a distinguished triangle $X\rarrow B\rarrow A\rarrow
X[1]$ in $\sH^0(\bA)$.
 Using
Theorem~\ref{inj-proj-resolutions-semiorthogonality-theorem}(a) again,
we see that it suffices to show that all closed morphisms $B\rarrow J$
are homotopic to zero in~$\bA$.

 Now the object $B\in\sZ^0(\bA)$ is fp\+projective as a finitely
iterated extension of fp\+projective objects
(use the proof of Lemma~\ref{cone-kernel-cokernel} together with
Lemma~\ref{fp-inj-proj-closure-properties-in-loc-presentable}(a)
or~(b)).
 By Proposition~\ref{fp-Eklof-Trlifaj}(c), the object $B\in\sZ^0(\bA)$
is a direct summand of an object $F$ filtered by coherent objects
in~$\sZ^0(\bA)$.
 Clearly, it is enough to show that all closed morphisms $F\rarrow J$
are homotopic to zero in~$\bA$.

 The following lemma, which is a rather straightforward generalization
of Arinkin's Lemma~\ref{Arinkin-lemma}, finishes the proof of
the theorem.
\end{proof}

\begin{lem}
 Let\/ $\bA$ be a locally finitely presentable abelian DG\+category
and $J\in\bA_\binj$ be a graded-injective object.
 Assume that all closed morphisms $E\rarrow J$ from finitely
presentable objects $E$ of\/ $\sZ^0(\bA)$ to $J$ are homotopic to zero
in\/~$\bA$.
 Let $F\in\bA$ be an object filtered by finitely presentable objects
in\/ $\sZ^0(\bA)$.
 Then all closed morphisms $F\rarrow J$ are also homotopic to zero
in\/~$\bA$.
\end{lem}

\begin{proof}
 Instead of adopting the proof of Lemma~\ref{Arinkin-lemma} to
the new situation (which is quite easy to do), we will structure
the argument differently in order to demonstrate a connection
with the Eklof lemma.
 The next lemma provides natural isomorphisms of abelian groups
$\Ext^1_{\sZ^0(\bA)}(E,J[-1])\simeq\Hom_{\sH^0(\bA)}(E,J)$ and
$\Ext^1_{\sZ^0(\bA)}(F,J[-1])\simeq\Hom_{\sH^0(\bA)}(F,J)$.
 The former isomorphism shows that the object $J$ is fp\+injective in
$\sZ^0(\bA)$ in our assumptions, and then the latter isomorphism
implies the desired vanishing of closed morphisms $F\rarrow J$
up to homotopy in~$\bA$.
\end{proof}

\begin{lem}
 Let\/ $\bE$ be an exact DG\+category.
 Then, for any two objects $A$ and $B\in\bE$, the kernel of
the map of abelian groups
\begin{equation} \label{Phi-Ext1-map}
 \Phi_\bE\:\Ext^1_{\sZ^0(\bE)}(A,B)\lrarrow
 \Ext^1_{\sZ^0(\bE^\bec)}(\Phi_\bE(A),\Phi_\bE(B))
\end{equation}
induced by the exact functor\/ $\Phi_\bE\:\sZ^0(\bE)\rarrow
\sZ^0(\bE^\bec)$ is naturally isomorphic to the group\/
$\Hom_{\sH^0(\bE)}(A,B[1])$.
 In particular, if either $A$ is a graded-projective object
in\/ $\bE$, or $B$ is a graded-injective object in\/ $\bE$, then
$$
 \Ext^1_{\sZ^0(\bE)}(A,B)\simeq\Hom_{\sH^0(\bA)}(A,B[1]).
$$
\end{lem}

\begin{proof}
 This is a quite well-known observation (formulated explicitly
for complexes in abelian categories in~\cite[Lemma~5.1]{PS4}
and for abelian DG\+categories in~\cite[Lemma~6.1]{PS5}).
 The point is that the kernel of the map~\eqref{Phi-Ext1-map}
is the abelian group of equivalence classes of short sequences
$0\rarrow B\rarrow C\rarrow A\rarrow0$ in $\sZ^0(\bE)$ that are
split exact in~$\bE^0$ (in view of
Lemma~\ref{Phi-Psi-extended-to-nonclosed-morphisms}).
 Irrespectively of any exact structure on $\bE$, in any DG\+category
with shifts and cones such ``graded split'' short exact sequences
are indexed by the closed morphisms $f\:A\rarrow B[1]$ up to homotopy,
and the object $C$ is recovered as $C=\cone(f)[-1]$ (as was mentioned
already in Section~\ref{dg-categories-subsecn}).
\end{proof}

\begin{cor} \label{coherent-aleph-n-Noetherian-cdg-ring-corollary}
 Let $\biR^\cu=(R^*,d,h)$ be a CDG\+ring.
 Assume that the graded ring $R^*$ is graded left coherent and there
exists an integer $n\ge0$ such that all homogenenous left ideals in
$R^*$ have less than\/ $\aleph_n$ generators.
 Let\/ $\bA=\biR^\cu\bmodl$ be the locally coherent DG\+category of
left CDG\+modules over $\biR^\cu$, and let $\bA_\bfp=\biR^\cu\bmodl_\bfp
\subset\biR^\cu\bmodl$ be the exactly embedded full abelian
DG\+subcategory whose objects are the CDG\+modules with coherent
(finitely presentable) underlying graded $R^*$\+modules.
 Then the triangulated functor
$$
 \sD^\abs(\biR^\cu\bmodl_\bfp)\lrarrow\sD^\co(\biR^\cu\bmodl)
$$
induced by the inclusion of abelian DG\+categories $\biR^\cu\bmodl_\bfp
\rightarrowtail\biR^\cu\bmodl$ is fully faithful, and the objects in
its image form a set of compact generators of the coderived
category\/ $\sD^\co(\biR^\cu\bmodl)$.
\end{cor}

\begin{proof}
 According to Example~\ref{locally-coherent-cdg-ring-example},
the abelian DG\+category $\biR^\cu\bmodl$ is locally coherent under our
assumptions.
  Notice that, by
Lemmas~\ref{Phi-Psi-preserve-generatedness-presentability}\+-%
\ref{Phi-Psi-reflect-generatedness-presentability} or by
Lemma~\ref{Phi-Psi-preserve-reflect-coherence}, a left CDG\+module
$(M^*,d_M)$ over $(R^*,d,h)$ is coherent as an object of
$\sZ^0(\biR^\cu\bmodl)$ if and only if the graded $R^*$\+module $M^*$
is coherent as an object of the category of graded modules $R^*\smodl$.

 Let us explain how the $\aleph_n$\+Noetherianity assumption in
the formulation of the corollary implies both the fp\+projective
dimension assumptions of
Theorem~\ref{coherent-finite-fp-dimen-compact-generation-theorem}.
 As stated in Example~\ref{locally-Noetherian-examples}(1), all
homogeneous left ideals in the graded ring $R^*[\delta]$ have less
than\/ $\aleph_n$ generators if and only if all homogeneous left
ideals in the graded ring $R^*$ have this property.
 According to the graded version of~\cite[Proposition~2.3]{Pfp},
the fp\+projective dimension of the locally coherent category
of graded left $R^*$\+modules $\sZ^0(\biR^\cu\bmodl^\bec)$ does not
exceed~$n$ in this case.
 The same, of course, applies to the locally coherent category
$\sZ^0(\biR^\cu\bmodl)$ of graded left modules over $R^*[\delta]$.

 We have checked all the assumptions of
Theorem~\ref{coherent-finite-fp-dimen-compact-generation-theorem}
in the situation at hand, and invoking this theorem proves
the corollary.
\end{proof}

\begin{rem}
 The results of the paper~\cite{PS5} allow to drop the finite
fp\+dimension conditions in
Theorem~\ref{coherent-finite-fp-dimen-compact-generation-theorem}
and the $\aleph_n$\+Noetherianity condition in
Corollary~\ref{coherent-aleph-n-Noetherian-cdg-ring-corollary}
at the cost of replacing the coderived category in the sense
of the present author (as defined in
Section~\ref{derived-second-kind-definitions-subsecn}) by
the coderived category in the sense of Becker
(as defined in~\cite[Section~7.3]{PS5}).
 The respective results in~\cite{PS5} are~\cite[Theorem~0.2
or~8.19]{PS5} and~\cite[Corollary~0.3 or~8.20]{PS5}.

 The terminology of Positselski's vs.\ Becker's co/contraderived
categories was established in~\cite[Remark~9.2]{PS4}
and~\cite[Section~7.9]{Pksurv}.
 It is an open problem whether the present author's and Becker's
approaches to defining coderived and contraderived categories
\emph{ever} produce different outputs within their common domain
of definition (see~\cite[Examples~2.5(3) and~2.6(3)]{Pps} for
a discussion).

 Let $\bA$ be a locally coherent DG\+category.
 Assuming that the locally coherent abelian category $\sZ^0(\bA^\bec)$
has finite fp\+projective dimension,
Proposition~\ref{coherent-finite-fp-dimen-coderived-category-prop} tells
that Positselski's and Becker's coderived categories of $\bA$ agree.
 So, even if one is interested the coderived category in the sense of
the present author, the methods of the paper~\cite{PS5} provide
an improvement over our results in that they imply that the condition
of finite fp\+projective dimension of the category $\sZ^0(\bA)$ can be
dropped in
Theorem~\ref{coherent-finite-fp-dimen-compact-generation-theorem},
and only for the category $\sZ^0(\bA^\bec)$ this condition is needed.
 This conclusion is essentially based on~\cite[Corollary~7.19]{PS5},
which is a more powerful technique for proving that finitely presentable
objects generate the coderived category, as compared to our argument
in the proof of
Theorem~\ref{coherent-finite-fp-dimen-compact-generation-theorem} above.
\end{rem}

\end{document}